\documentclass[letterpaper,11pt]{amsart}

\usepackage{amsmath,amssymb,amsthm,amsfonts}
\usepackage{hyperref}
\usepackage{graphicx,color,colortbl}
\usepackage{upgreek}
\usepackage[mathscr]{euscript}
\usepackage{hhline}
\usepackage{mathbbol,mathtools,calc}
\usepackage[labelfont={rm}]{subfig}
\usepackage{bm}
\usepackage{framed}
\usepackage{musicography}
\usepackage{tikz-cd} 
\usepackage[export]{adjustbox}
\usepackage{blkarray}

\usepackage{nomencl}
\makenomenclature

\DeclareMathAlphabet{\mathpzc}{OT1}{pzc}{m}{it}
\DeclareMathAlphabet{\mathdutchcal}{U}{dutchcal}{m}{n}
\SetMathAlphabet{\mathdutchcal}{bold}{U}{dutchcal}{b}{n}

\newcommand{\ts}[1]{{\color{red} #1}}

\newcommand{\matteo}[1]{{\color{brown} #1}}

\allowdisplaybreaks
\numberwithin{equation}{section}
\usepackage{ytableau}
\usepackage{mathdots}
\usepackage{tikz}
\usetikzlibrary{
	shapes,
	arrows,
	positioning,
	decorations.markings,
	decorations.pathmorphing,
	circuits.logic.US,
	circuits.logic.IEC,
	fit,
	calc,
	plotmarks,
	matrix
}

\tikzset{
>=stealth',
help lines/.style={dashed, thick},
axis/.style={<->},
important line/.style={thick},
connection/.style={thick, dotted},
punkt/.style={
rectangle,
rounded corners,
draw=black, thick,
text width=4.5em,
minimum height=2em,
text centered,
},
pil/.style={
->,
thick,
gray,
shorten <=2pt,
shorten >=2pt,}
}

\usepackage{array}
\usepackage{adjustbox}
\usepackage{cleveref}
\usepackage{enumitem}

\usepackage[tmargin=3cm,lmargin=3cm,rmargin=3cm,bmargin=3cm,headheight=1cm,footskip=1.5cm]{geometry}

\usepackage{fancyhdr}

\usepackage{comment}
\newtheorem{proposition}{Proposition}[section]
\newtheorem{lemma}[proposition]{Lemma}
\newtheorem{corollary}[proposition]{Corollary}
\newtheorem{theorem}[proposition]{Theorem}
\newtheorem*{theorem*}{Theorem}

\theoremstyle{definition}
\newtheorem{definition}[proposition]{Definition}

\newtheorem{remark}[proposition]{Remark}
\newtheorem*{remark*}{Remark}
\newtheorem{example}[proposition]{Example}

\newtheorem{question}[]{Question}

\newcommand{\Qbinomial}[3]{\binom{#1}{#2}_{\hspace{-2pt} #3}}

\newcommand{\E}[1]{\widetilde{e}_{#1}}
\newcommand{\F}[1]{\widetilde{f}_{#1}}

\newcommand{\pr}{\mathrm{pr}}
\newcommand{\evac}{\mathrm{evac}}

\newcommand{\wt}{\mathrm{wt}}
\newcommand{\skwRSK}{\mathbf{SS}}
\newcommand{\RS}{\mathbf{RS}}
\newcommand{\RSK}{\mathbf{RSK}}
\newcommand{\V}{\mathbf{V}}

\newcommand{\rc}{\mathrm{rc}}
\newcommand{\std}{\mathrm{std}}
\newcommand{\overlap}{\mathrm{ov}}

\newcommand{\cupdot}{\mathbin{\mathaccent\cdot\cup}}
\newcommand{\smallvdots}{\scalebox{0.5}{$\vdots$}}

\begin{document}
\title{Skew RSK dynamics: Greene invariants, affine crystals and applications to $q$-Whittaker polynomials}

\author[T. Imamura]{Takashi Imamura}\address{T. Imamura, 
Department of Mathematics and Informatics, Chiba University, Chiba, 263-8522,Japan}\email{imamura@math.s.chiba-u.ac.jp}

\author[M. Mucciconi]{Matteo Mucciconi}\address{M. Mucciconi, 
Department of Physics,
Tokyo Institute of Technology, Tokyo, 152-8551 Japan}\email{matteomucciconi@gmail.com}

\author[T. Sasamoto]{Tomohiro Sasamoto}\address{T. Sasamoto, 
Department of Physics,
Tokyo Institute of Technology, Tokyo, 152-8551 Japan}\email{sasamoto@phys.titech.ac.jp}

\date{}

\maketitle

\begin{abstract}
    Iterating the skew RSK correspondence discovered by Sagan and Stanley in the late '80s, we define a deterministic dynamics on the space of pairs of skew Young tableaux $(P,Q)$. We find that this skew RSK dynamics displays conservation laws which, in the picture of Viennot's shadow line construction, identify generalizations of Greene invariants.
    The introduction of a novel realization of $0$-th Kashiwara operators reveals that the skew RSK dynamics possesses symmetries induced by an affine bicrystal structure,
    which, combined with connectedness properties of Demazure crystals,  leads to its linearization. 
    Studying asymptotic evolution of the dynamics started from a pair of skew tableaux $(P,Q)$, we discover a new bijection $\Upsilon : (P,Q) \mapsto (V,W; \kappa, \nu)$. Here $(V,W)$ is a pair of  
    vertically strict tableaux, i.e., column strict filling of Young diagrams with no condition on rows, with shape prescribed by the Greene invariant, $\kappa$ is an array of non-negative weights and $\nu$ is a partition. 
    An application of this construction is the first bijective proof of Cauchy and Littlewood identities involving $q$-Whittaker polynomials. New identities relating sums of $q$-Whittaker and Schur polynomials are also presented.
\end{abstract}

\tableofcontents

\section{Introduction}

\subsection{The goal of this paper} \label{subs:goals}

    The Robinson-Schensted-Knuth (RSK) correspondence is a fundamental bijection between matrices $M$ with non-negative integer entries, sometimes encoded by biwords $\pi$,
    and pairs of semi-standard tableaux $(P,Q)$ \cite{robinson1938representations,Schensted1961,Knuth1970}. It represents one of the central tools in combinatorics and its applications range from representation theory to probability. Along with a simple algorithmic description, the RSK correspondence possesses a surprising number of properties and symmetries. These have been central object of study throughout the twentieth century receiving contributions from a number of celebrated combinatorialists. A detailed account on the theory of RSK correspondence can be found in classical books as \cite{fulton1997young,Stanley1999,sagan2001symmetric,lothaire_2002}.
    
    The RSK correspondence provides powerful tools to prove various identities involving symmetric functions. For instance the Cauchy identity for the Schur polynomials $s_{\lambda}$, with $x=(x_1,\cdots,x_n)$, $y=(y_1,\cdots,y_n)$,
    \begin{equation} \label{eq:Cauchy_id_Schur0}
        \sum_{\lambda} s_{\lambda}(x) s_{\lambda}(y) =  \prod_{i,j=1}^n \frac{1}{1-x_i y_j}, 
    \end{equation}
    which can be proved in various ways, may be also seen as a consequence of the RSK correspondence. On the left hand side the Schur polynomial appears as a result of the combinatorial formula $s_{\lambda}(x)=\sum_{T:{\rm sh} T=\lambda}x^T$ where the sum is over semi-standard tableaux with shape $\lambda$, whereas each factor in the right hand side is a geometric sum corresponding to each matrix element of an integral matrix $M$ of size $n \times n$. 
    An advantage of finding a bijective proof is that by 
    leveraging symmetries it leads to a number of related identities, see for instance 
    \cite{Stanley1999} 
    
    A well known property of the RSK is the Schensted's theorem \cite{Schensted1961}. It says that, assuming $\pi \xleftrightarrow[]{\RSK} (P,Q)$, the length of the first row of tableaux $P,Q$ equals the length of the longest increasing subsequence of the biword $\pi$. Noticeably this property became a crucial tool in the solution of the Ulam's problem \cite{logan_shepp1977variational,VershikKerov_LimShape1077,baik1999distribution}. A generalization of Schensted's theorem was found by Greene \cite{Greene1974}, who proved that the full shape of tableaux $P,Q$ can be identified maximizing disjoint increasing subsequences of $\pi$. 
    Greene's characterization has found uses in the discovery of universal objects in probability theory such as the Directed Landscape \cite{directed_landscape}, which is generalization the Airy process \cite{PhahoferSpohn2002}.

    In \cite{sagan1990robinson} Sagan and Stanley discovered a generalization of the RSK correspondence which 
    relates pairs $(\overline{M};\nu)$ consisting of a matrix of non-negative integer sequences $\overline{M}= (\overline{M}_{i,j}(k)\in \mathbb{N}_0 :i,j\in\{1,\dots,n\},k\in\mathbb{N}_0)$ and a partition $\nu$ with pairs $(P,Q)$ of semi-standard tableaux of generic skew shape. Throughout we will use the convention $\mathbb{N}=\{1,2,\dots\}$ and $\mathbb{N}_0=\mathbb{N} \cup \{ 0 \}$.  In this paper we will refer to this as \emph{Sagan-Stanley correspondence} and we will often use the short-hand $(\overline{M};\nu) \xleftrightarrow[]{\skwRSK\,} (P,Q)$. 
    Naturally they also discussed an application of their correspondence 
    to prove bijectively a Cauchy identity for skew Schur polynomials $s_{\lambda/\rho}$ \cite{Macdonald1995}. 
    Fixing a parameter $|q|<1$ and variables $|x_i y_j|<1$, $i,j=1,\dots,n$, it reads
    \begin{equation} \label{eq:Cauchy_id_Schur}
        \sum_{\rho \subseteq \lambda} q^{|\rho|} s_{\lambda/\rho}(x) s_{\lambda/\rho}(y) = \frac{1}{(q;q)_\infty} \prod_{i,j=1}^n \frac{1}{(x_i y_j ; q)_\infty}, 
    \end{equation}
    where $(z;q)_n =(1-z)(1-qz)\cdots (1-q^{n-1}z) ,n\in \mathbb{N}_0 \cup \{+\infty\}$ is the $q$-Pochhammer symbol.
    
     Unlike for the classical RSK correspondence, a detailed description of properties of Sagan and Stanley's algorithm has proven to be more challenging to obtain. Powerful tools such as Sch{\"u}tzenberger's theory of jeu de taquin \cite{schutzenberger_RS} do not admit straightforward ``skew" analogs and extensions of Greene invariants in skew setting have also remained unexplored. For instance, if we assume $(\overline{M};\nu) \xleftrightarrow[]{\skwRSK\,} (P,Q)$, then a simple characterization of the shape of $P,Q$ in terms of $\overline{M},\nu$ is at the moment not available. 
    In this paper we fill this void and provide a generalization of Greene's theorem in this skew setting,
    as a consequence of the theory we develop, as described below. 

    To people with experience in symmetric polynomials, the factorized expression in the right hand side of identity \eqref{eq:Cauchy_id_Schur} should look familiar. In fact, a closely resembling expression arises when considering the Cauchy identity for $q$-Whittaker polynomials $\mathscr{P}_\mu(x;q)$ \cite{GLO2010}, that are Macdonald polynomials $\mathscr{P}_\mu(x;q,t)$  \cite[Chapter VI]{Macdonald1995} with parameter $t=0$. We have 
    \begin{equation} \label{eq:Cauchy_id_qW_intro}
        \sum_\mu \mathdutchcal{b}_\mu (q) \mathscr{P}_\mu(x;q) \mathscr{P}_\mu(y;q) = \prod_{i,j=1}^n \frac{1}{(x_i y_j ; q)_\infty},
    \end{equation}
where $\mathdutchcal{b}_\mu$ is a normalization factor and its explicit definition can be found in \eqref{eq:b_mu} in the text. The Macdonald polynomials $\mathscr{P}_\mu(x;q,t)$ are widely considered as a central object in the theory of special functions and play prominent roles in various fields such as affine Hecke algebras \cite{Cherednik_DAHA}, Hilbert schemes \cite{Haiman_Hilber_Schemes}, combinatorics \cite{haglund_haiman_loehr} and more recently in integrable probability \cite{BorodinCorwin2011Macdonald} and integrable systems \cite{Cantini_2015}.   
The particular case of the $q$-Whittaker polynomials have also attracted special attention in recent years because of their importance in integrable probability \cite{BorodinCorwin2011Macdonald,OConnellPei2012,MatveevPetrov2014,OrrPetrov2016,imamura2019stationary}, representation theory \cite{Garsia_Procesi_1992,Sanderson_Demazure_q_Whittaker,schilling_tingley,naito_sagaki_Demazure}, and combinatorics \cite{BorodinWheelerSpinq, Cantini_2015,Garbali_Wheeler_2020} and a few other subjects.
A proof of the Cauchy identity \eqref{eq:Cauchy_id_qW_intro} is explained in \cite{Macdonald1995}. Several different proofs have appeared in literature in recent years, which are based on the Yang-Baxter equation \cite{BorodinWheelerSpinq}, or \emph{randomized} variants of the RSK algorithm \cite{OConnellPei2012,MatveevPetrov2014}.
However, to the best of the authors knowledge, none of the techniques available in existing literature allow for a bijective proof of the Cauchy identity \eqref{eq:Cauchy_id_qW_intro}.
Nevertheless the striking similarity between partition functions \eqref{eq:Cauchy_id_Schur}, \eqref{eq:Cauchy_id_qW_intro}, along with the fact that all terms involved $q^{|\rho|},s_{\lambda/\rho}(x), \mathdutchcal{b}_\mu(q),\mathscr{P}_\mu(x;q)$ possess positive monomials expansions, suggest the possibility of relating the theories concerning the RSK correspondence to $q$-Whittaker polynomials. \emph{The goal of this paper is to develop a combinatorial theory extending the scope of the RSK correspondence and which allows the first bijective proof of the Cauchy identity \eqref{eq:Cauchy_id_qW_intro}.}
As consequence our theory will produce a number of new identities involving $q$-Whittaker polynomials and we envision it playing 
important roles in a wide range of related fields in future.
    
\subsection{Skew RSK dynamics: examples and emerging questions} \label{subs:examples}

    To achieve the goals outlined above we first introduce a new deterministic time evolution on pairs of skew tableaux, which is defined by combining the skew RSK map introduced in \cite{sagan1990robinson} and a novel cyclic operation on tableaux. We call this the \emph{skew RSK dynamics} and in this subsection we will see through an example how it would bring a connection between skew tableaux and $q$-Whittaker polynomials. Looking at time evolution of skew tableaux for some examples, we observe certain properties of the dynamics and a few questions emerge. Indeed results presented in this paper are obtained while proving these properties and answering these questions. 
    
    To define our dynamics, we first recall a basic operation on a tableau called the \emph{internal insertion}, which was introduced in \cite{sagan1990robinson}. From a semi-standard tableau of skew shape $P$ select a row $r$ such that the leftmost cell $(c,r)$ at that row is a corner cell, i.e. both $(c-1,r)$ and $(c,r-1)$ are empty cells. Then, $\mathcal{R}_{[r]}(P)$ is the tableau obtained vacating the cell $(c,r)$ of $P$ and inserting, following the usual Schensted's bumping algorithm, the value $P(c,r)$ at row below. For a more precise description of this procedure see \cref{subs:RSK_product_tableaux} below. In the following example, calling $P$ the tableau in the left hand side, we show, step by step, the computation of $\mathcal{R}_{[2]}(P)$
    \begin{equation}
        \begin{minipage}{.78\linewidth}
            \includegraphics[scale=1.2]{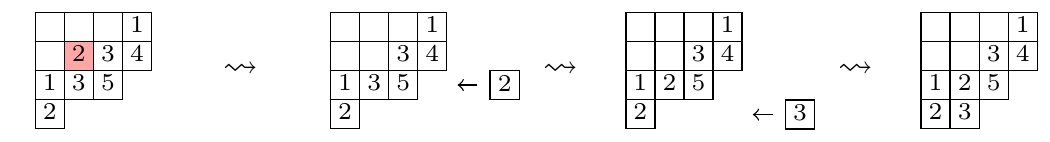}
        \end{minipage}
        \hspace{1cm} .
    \end{equation}
    Using the notion of internal insertion we define a new map, this time acting on pairs of skew tableaux $(P,Q)$ with same shape. We call it $\iota_2$ to emphasize its nontrivial action on the second tableaux $Q$; later in \cref{subs:results} we will also introduce $\iota_1$. Entries of tableaux here are assumed to belong to the alphabet $\{1,\dots,n\}$ for some fixed $n\ge 1$. Define $\iota_2 : (P,Q) \mapsto ( P', Q' )$, where $P' = \mathcal{R}_{[i_k]} \cdots \mathcal{R}_{[i_1]}(P)$, $i_1 \ge \dots \ge i_k$ are all row coordinates of $1$-cells (i.e. cells with label 1) of $Q$ and $Q'$ is obtained from $Q$ vacating all $1$-cells, decreasing by 1 the labels of all remaining cells and creating $n$-cells so to make the shape of $P',Q'$ equal. The following example shows a realization of $\iota_2(P,Q)$ and we assume $n=5$
    \begin{equation}
    \label{ex_iota2}
    \ytableausetup{aligntableaux = center,smalltableaux}
        (P,Q)=
        \left(
        \,
            \begin{ytableau}
                \, &  &  & 1
                \\
                & 2 & 3 & 4
                \\
                1 & 3 & 5
                \\
                2
            \end{ytableau}
            \,\, , \,\,
            \begin{ytableau}
                \,  &  &  & 2
                \\
                & 1 & 3 & 3
                \\
                2 & 2 & 5
                \\
                3
            \end{ytableau}
            \,
        \right)
        \xrightarrow[]{\hspace{.5cm} \iota_2 \hspace{.5cm}}
        \left(
        \,
            \begin{ytableau}
                \,  &  &  & 1
                \\
                & & 3 & 4
                \\
                1 & 2 & 5
                \\
                2 & 3
            \end{ytableau}
            \,\, , \,\,
            \begin{ytableau}
                \,  &  &  & 1
                \\
                & & 2 & 2
                \\
                1 & 1 & 4
                \\
                2 & 5
            \end{ytableau}
            \,
        \right) .
    \end{equation}
    $\iota_2$ is invertible and the inverse $\iota_2^{-1}$ is always well defined, provided we allow cells of tableaux to occupy also non strictly positive rows. To give a reference, while drawing tableaux we will color such cells in gray, so for instance we have
    \begin{equation}
        \left(
        \,
            \begin{ytableau}
                \,  &  &  & 1
                \\
                & 2 & 3 & 4
                \\
                1 & 3 & 5
                \\
                2
            \end{ytableau}
            \,\, , \,\,
            \begin{ytableau}
                \,  &  &  & 2
                \\
                & 1 & 3 & 3
                \\
                2 & 2 & 5
                \\
                3
            \end{ytableau}
            \,
        \right)
        \xrightarrow[]{\hspace{.5cm} \iota_2^{-1} \hspace{.5cm}}
        \left(
        \,
            \begin{ytableau}
                *(gray!35) & *(gray!35) & *(gray!35) & *(gray!35) 1
                \\
                \,  &  &  & 4
                \\
                & 2 & 3 & 5
                \\
                1 & 2
                \\
                2
            \end{ytableau}
            \,\, , \,\,
            \begin{ytableau}
                *(gray!35) & *(gray!35) & *(gray!35) & *(gray!35) 1
                \\
                \, &  &  & 3
                \\
                & 2 & 4 & 4
                \\
                3 & 3
                \\
                4
            \end{ytableau}
            \,
        \right) .
    \end{equation}
    The operation $\iota_2$, in particular the cycling operation on a $Q$ tableau, is new in this paper and represents a dynamical rule preserving semi-standard properties. 
    Iterating $n$ times the application of $\iota_2$ yields a known content preserving map, that in \cite{sagan1990robinson} was called ``skew Knuth map" and that we will call \emph{skew $\RSK$ map}, 
    \begin{equation} \label{eq:intro_RSK_iota_2}
        \RSK(P,Q) \coloneqq \iota_2^n (P,Q). 
    \end{equation}
    For instance we have
    \begin{equation}
        \left(
        \,
            \begin{ytableau}
                \, & & & 1
                \\
                & 2 & 3 & 4
                \\
                1 & 3 & 5
                \\
                2
            \end{ytableau}
            \,\, , \,\,
            \begin{ytableau}
                \, &  &  & 2
                \\
                & 1 & 3 & 3
                \\
                2 & 2 & 5
                \\
                3
            \end{ytableau}
            \,
        \right)
        \xrightarrow[]{\hspace{.5cm} \RSK \hspace{.5cm}}
        \left(
        \,
            \begin{ytableau}
                \, &  &  & 
                \\
                & &  & 
                \\
                & & & 4
                \\
                & 1 & 3
                \\
                1 & 2 & 5
                \\
                2 & 3
            \end{ytableau}
            \,\, , \,\,
            \begin{ytableau}
                \, &  &  & 
                \\
                & &  & 
                \\
                & & & 3
                \\
                & 1 & 2
                \\
                2 & 2 & 3
                \\
                3 & 5
            \end{ytableau}
            \,
        \right) .
    \end{equation}
    
    From (\ref{eq:intro_RSK_iota_2}), $\iota_2$ can be considered as a refinement of the skew $\RSK$ map. It will also play a 
    crucial role when we discuss an affine bicrystal symmetry of the skew $\RSK$ dynamics; see (\ref{eq:intro_0_th_operators}) below. The skew $\RSK$ map is invertible and its inverse $\RSK^{-1}$ comes from the application of $n$ consecutive times of $\iota_2^{-1}$. Continuing with our running example we find
    \begin{equation}
        \left(
        \,
            \begin{ytableau}
                \, &  &  & 1
                \\
                & 2 & 3 & 4
                \\
                1 & 3 & 5
                \\
                2
            \end{ytableau}
            \,\, , \,\,
            \begin{ytableau}
                \, &  &  & 2
                \\
                & 1 & 3 & 3
                \\
                2 & 2 & 5
                \\
                3
            \end{ytableau}
            \,
        \right)
        \xrightarrow[]{\hspace{.5cm} \RSK^{-1} \hspace{.5cm}}
        \left(
        \,
        \begin{ytableau}
            *(gray!35) & *(gray!35) & *(gray!35) & *(gray!35) 1
            \\
            *(gray!35) & *(gray!35) & *(gray!35) & *(gray!35) 4
            \\
            *(gray!35) & *(gray!35) & *(gray!35) 2 & *(gray!35) 5
            \\
            1 & 3 & 3
            \\
            2
        \end{ytableau}
        \,\, , \,\,
        \begin{ytableau}
            *(gray!35) & *(gray!35) & *(gray!35) & *(gray!35) 2
            \\
            *(gray!35) & *(gray!35) & *(gray!35) & *(gray!35) 3
            \\
            *(gray!35) & *(gray!35) & *(gray!35) 1 & *(gray!35) 5
            \\
            2 & 2 & 3
            \\
            3
        \end{ytableau}
        \,
        \right).
    \end{equation}
   
    The skew $\RSK$ map is a map from a pair of skew tableaux to another. Iterating the map for $t$ times
    one can define a time evolution of a pair of skew tableaux by $(P_{t+1},Q_{t+1}) = \RSK^{t}(P,Q)$, with the initial condition given by $(P_1,Q_1)=(P,Q)$. 
    Note that $t$ can be an arbitrary integer, using $\RSK^{-1}$ for a negative $t$.  
    We call this the \emph{skew $\RSK$ dynamics} and it plays the central role in our theory. 
    
   An interesting phenomenon happens when we consider the large $t$ limit.  Tableaux $(P_t,Q_t)$, 
   from a certain $t$ onward become ``stable", in the sense that the application of the skew $\RSK$ map has the only effect of rigidly shifting columns downward. The amplitude of each shift equals the number of labeled cells at the column. Let us show this in our example taking, for instance, $t=10$. With some patience one can compute $\RSK^{10}(P,Q)$ as
    \begin{equation} \label{eq:RSK_10}
        \left(
        \,
            \begin{ytableau}
                \, &  &  & 1
                \\
                & 2 & 3 & 4
                \\
                1 & 3 & 5
                \\
                2
            \end{ytableau}
            \,\, , \,\,
            \begin{ytableau}
                \, &  &  & 2
                \\
                & 1 & 3 & 3
                \\
                2 & 2 & 5
                \\
                3
            \end{ytableau}
            \,
        \right)
        \xrightarrow[]{\hspace{.5cm} \RSK^{10} \hspace{.5cm}}
        \left(
        \,
        \begin{ytableau}
            \none[\smallvdots] & \smallvdots & \smallvdots & \smallvdots & \smallvdots
            \\
            \none[{\textcolor{gray}{\scriptstyle 12\,\,}}] & & & & 4
            \\
            \none[\smallvdots] & \smallvdots & \smallvdots & \smallvdots
            \\
            \none[{\textcolor{gray}{\scriptstyle 22\,\,}}] & & & 3
            \\
            \none[{\textcolor{gray}{\scriptstyle 23\,\,}}] & & 1 & 5
            \\
            \none[{\textcolor{gray}{\scriptstyle 24\,\,}}] & & 2
            \\
            \none[\smallvdots] & \smallvdots
            \\
            \none[{\textcolor{gray}{\scriptstyle 31\,\,}}] & 1
            \\
            \none[{\textcolor{gray}{\scriptstyle 32\,\,}}] & 2
            \\
            \none[{\textcolor{gray}{\scriptstyle 33\,\,}}] & 3
        \end{ytableau}
        \,\, , \,\,
        \begin{ytableau}
            \none[\smallvdots] & \smallvdots & \smallvdots & \smallvdots & \smallvdots
            \\
            \none[{\textcolor{gray}{\scriptstyle 12\,\,}}] & & & & 3
            \\
            \none[\smallvdots] & \smallvdots & \smallvdots & \smallvdots
            \\
            \none[{\textcolor{gray}{\scriptstyle 22\,\,}}] & & & 2
            \\
            \none[{\textcolor{gray}{\scriptstyle 23\,\,}}] & & 2 & 3
            \\
            \none[{\textcolor{gray}{\scriptstyle 24\,\,}}] & & 5
            \\
            \none[\smallvdots] & \smallvdots
            \\
            \none[{\textcolor{gray}{\scriptstyle 31\,\,}}] & 1
            \\
            \none[{\textcolor{gray}{\scriptstyle 32\,\,}}] & 2
            \\
            \none[{\textcolor{gray}{\scriptstyle 33\,\,}}] & 3
        \end{ytableau}
        \,
        \right),
    \end{equation}
    so that applying the skew $\RSK$ map one more time yields
    \begin{equation} \label{eq:RSK_11}
        \left(
        \,
        \begin{ytableau}
            \none[\smallvdots] & \smallvdots & \smallvdots & \smallvdots & \smallvdots
            \\
            \none[{\textcolor{gray}{\scriptstyle 12\,\,}}] & & & & 4
            \\
            \none[\smallvdots] & \smallvdots & \smallvdots & \smallvdots
            \\
            \none[{\textcolor{gray}{\scriptstyle 22\,\,}}] & & & 3
            \\
            \none[{\textcolor{gray}{\scriptstyle 23\,\,}}] & & 1 & 5
            \\
            \none[{\textcolor{gray}{\scriptstyle 24\,\,}}] & & 2
            \\
            \none[\smallvdots] & \smallvdots
            \\
            \none[{\textcolor{gray}{\scriptstyle 31\,\,}}] & 1
            \\
            \none[{\textcolor{gray}{\scriptstyle 32\,\,}}] & 2
            \\
            \none[{\textcolor{gray}{\scriptstyle 33\,\,}}] & 3
        \end{ytableau}
        \,\, , \,\,
        \begin{ytableau}
            \none[\smallvdots] & \smallvdots & \smallvdots & \smallvdots & \smallvdots
            \\
            \none[{\textcolor{gray}{\scriptstyle 12\,\,}}] & & & & 3
            \\
            \none[\smallvdots] & \smallvdots & \smallvdots & \smallvdots
            \\
            \none[{\textcolor{gray}{\scriptstyle 22\,\,}}] & & & 2
            \\
            \none[{\textcolor{gray}{\scriptstyle 23\,\,}}] & & 2 & 3
            \\
            \none[{\textcolor{gray}{\scriptstyle 24\,\,}}] & & 5
            \\
            \none[\smallvdots] & \smallvdots
            \\
            \none[{\textcolor{gray}{\scriptstyle 31\,\,}}] & 1
            \\
            \none[{\textcolor{gray}{\scriptstyle 32\,\,}}] & 2
            \\
            \none[{\textcolor{gray}{\scriptstyle 33\,\,}}] & 3
        \end{ytableau}
        \,
        \right)
        \xrightarrow[]{\hspace{.5cm} \RSK \hspace{.5cm}}
        \left(
        \,
        \begin{ytableau}
            \none[\smallvdots] & \smallvdots & \smallvdots & \smallvdots & \smallvdots
            \\
            \none[{\textcolor{gray}{\scriptstyle 13\,\,}}] & & & & 4
            \\
            \none[\smallvdots] & \smallvdots & \smallvdots & \smallvdots
            \\
            \none[{\textcolor{gray}{\scriptstyle 24\,\,}}] & & & 3
            \\
            \none[{\textcolor{gray}{\scriptstyle 25\,\,}}] & & 1 & 5
            \\
            \none[{\textcolor{gray}{\scriptstyle 26\,\,}}] & & 2
            \\
            \none[\smallvdots] & \smallvdots
            \\
            \none[{\textcolor{gray}{\scriptstyle 34\,\,}}] & 1
            \\
            \none[{\textcolor{gray}{\scriptstyle 35\,\,}}] & 2
            \\
            \none[{\textcolor{gray}{\scriptstyle 36\,\,}}] & 3
        \end{ytableau}
        \,\, , \,\,
        \begin{ytableau}
            \none[\smallvdots] & \smallvdots & \smallvdots & \smallvdots & \smallvdots
            \\
            \none[{\textcolor{gray}{\scriptstyle 13,\,}}] & & & & 3
            \\
            \none[\smallvdots] & \smallvdots & \smallvdots & \smallvdots
            \\
            \none[{\textcolor{gray}{\scriptstyle 24\,\,}}] & & & 2
            \\
            \none[{\textcolor{gray}{\scriptstyle 25\,\,}}] & & 2 & 3
            \\
            \none[{\textcolor{gray}{\scriptstyle 26\,\,}}] & & 5
            \\
            \none[\smallvdots] & \smallvdots
            \\
            \none[{\textcolor{gray}{\scriptstyle 34\,\,}}] & 1
            \\
            \none[{\textcolor{gray}{\scriptstyle 35\,\,}}] & 2
            \\
            \none[{\textcolor{gray}{\scriptstyle 36\,\,}}] & 3
        \end{ytableau}
        \,
        \right).
    \end{equation}
    In the previous two equations grey numbers at the left of tableaux indicate the row coordinates of cells to their right. We notice that in \eqref{eq:RSK_11}, the skew $\RSK$ map had the only effect of shifting columns downward, as an instance of the stabilization phenomenon described just above. Notice again that in such stable states each column travels downward with ``speed" equal to the number of labeled cells it hosts. In the above example the speeds are  3,2,2,1 for the 1st, 2nd, 3rd, 4th columns. Obviously longer columns travel ``faster". 
    This procedure defines an important object. 
    \begin{definition}[Asymptotic increments] \label{def:asymptotic_increment}
        For a pair $(P,Q)$ of semi-standard tableaux of same skew shape let $\lambda^{t+1}/\rho^{t+1}$ be the shape of pair $\RSK^t(P,Q)$. The \emph{asymptotic increment} $\mu(P,Q)$ is the partition defined by
        \begin{equation}\label{eq:intro_asymptotic_increment}
            \mu_i' = \lim_{t \to \infty} (\lambda^t)_i'/t,
        \end{equation}
    where $\lambda'$ means the transpose of $\lambda$, i.e., $\lambda_i'$ is the number of cells in the $i$-th column of $\lambda$.     
    \end{definition}
    \noindent
    In other words partition $\mu$, defined by \eqref{eq:intro_asymptotic_increment}, is such that $\mu_i'$ is the speed 
    of the $i$-th column of $(P_t,Q_t)$, 
    or the number of labeled cells eventually remaining in it, when $t$ becomes large. It is an easy exercise to verify that limits \eqref{eq:intro_asymptotic_increment} always exist and numbers $\mu_i'$ define, in fact, an integer partition. In our example we have
    \begin{equation} \label{eq:intro_greene_inv_example}
        \mu(P,Q) = \ydiagram{4,3,1}.
    \end{equation}
    
    The same stabilization phenomenon happens when iterating the map $\RSK^{-1}$ and we have, for instance
    \begin{equation} \label{eq:RSK_minus_10}
        \left(
        \,
            \begin{ytableau}
                \, &  &  & 1
                \\
                & 2 & 3 & 4
                \\
                1 & 3 & 5
                \\
                2
            \end{ytableau}
            \,\, , \,\,
            \begin{ytableau}
                \, &  &  & 2
                \\
                & 1 & 3 & 3
                \\
                2 & 2 & 5
                \\
                3
            \end{ytableau}
            \,
        \right)
        \xrightarrow[]{\hspace{.5cm} \RSK^{-10} \hspace{.5cm}}
        \left(
        \,\,\,
        \begin{ytableau}
            \none[{\textcolor{gray}{\scriptstyle -29\hspace{2ex}}}] & *(gray!35) & *(gray!35) & *(gray!35) & *(gray!35) 1
            \\
            \none[{\textcolor{gray}{\scriptstyle -28\hspace{2ex}}}] & *(gray!35) & *(gray!35) & *(gray!35) & *(gray!35) 4
            \\
            \none[{\textcolor{gray}{\scriptstyle -27\hspace{2ex}}}] & *(gray!35) & *(gray!35) & *(gray!35) & *(gray!35) 5
            \\
            \none[\smallvdots] & *(gray!35) \smallvdots & *(gray!35) \smallvdots & *(gray!35) \smallvdots
            \\
            \none[{\textcolor{gray}{\scriptstyle -18\hspace{2ex}}}] & *(gray!35) & *(gray!35) & *(gray!35) 2 
            \\
            \none[{\textcolor{gray}{\scriptstyle -17\hspace{2ex}}}] & *(gray!35) & *(gray!35) 1 & *(gray!35) 3
            \\
            \none[{\textcolor{gray}{\scriptstyle -16\hspace{2ex}}}] & *(gray!35) & *(gray!35) 3
            \\
            \none[\smallvdots] & *(gray!35) \smallvdots 
            \\
            \none[{\textcolor{gray}{\scriptstyle -8\hspace{2ex}}}] & *(gray!35) 2
        \end{ytableau}
        \,\, , \,\,\,\,\,\,
        \begin{ytableau}
            \none[{\textcolor{gray}{\scriptstyle -29\hspace{2ex}}}] & *(gray!35) & *(gray!35) & *(gray!35) & *(gray!35) 2
            \\
            \none[{\textcolor{gray}{\scriptstyle -28\hspace{2ex}}}] & *(gray!35) & *(gray!35) & *(gray!35) & *(gray!35) 3
            \\
            \none[{\textcolor{gray}{\scriptstyle -27\hspace{2ex}}}] & *(gray!35) & *(gray!35) & *(gray!35) & *(gray!35) 5
            \\
            \none[\smallvdots] & *(gray!35) \smallvdots & *(gray!35) \smallvdots & *(gray!35) \smallvdots
            \\
            \none[{\textcolor{gray}{\scriptstyle -18\hspace{2ex}}}] & *(gray!35) & *(gray!35) & *(gray!35) 1
            \\
            \none[{\textcolor{gray}{\scriptstyle -17\hspace{2ex}}}] & *(gray!35) & *(gray!35) 2 & *(gray!35) 3
            \\
            \none[{\textcolor{gray}{\scriptstyle -16\hspace{2ex}}}] & *(gray!35) & *(gray!35) 3
            \\
            \none[\smallvdots] & *(gray!35) \smallvdots
            \\
            \none[{\textcolor{gray}{\scriptstyle -8\hspace{2ex}}}] & *(gray!35) 2
        \end{ytableau}
        \,
        \right).
    \end{equation}
    A striking observation is that asymptotic increments of tableaux in the right hand side of \eqref{eq:RSK_10}, \eqref{eq:RSK_minus_10} are equal, if we sort columns by length: in both cases labeled cells eventually arrange themselves into four blocks which propagate with the same fixed speeds 3,2,2,1. 
    This is not a coincidence. For any chosen pair of tableaux $P,Q$, the ``backward" asymptotic increments one computes taking the limit $\RSK^{-t}(P,Q)$ for large $t$ are always equal to $\mu(P,Q)$, after sorting columns by length. 
    This strongly suggests that the asymptotic increments $\mu$ record in fact conserved quantities of 
    the skew $\RSK$ dynamics throughout the time evolution, and the information of $\mu$ may be contained already in $(P,Q)$.
    This leads us to a first major question.
    \begin{question}
        Can we characterize the asymptotic increment $\mu$ in terms of initial data $(P,Q)$?
    \end{question}
    \noindent
    We will answer this question in \cref{thm:intro_Greene_invariants} and in \cref{thm:asymptotic_shape_RSK} in the text. Moreover the equivalence between backward and forward asymptotic increment will be addressed by result in \cref{thm:soliton_conservation}. 
    
     Existence of the conservation laws suggests that the skew $\RSK$ dynamics admits a description as an integrable system. In fact the skew $\RSK$ dynamics shows a clear resemblance to the multi-species Box-Ball system (BBS), which is a well-known discrete classical integrable system  \cite{Takahashi_Satsuma,Takahashi_BBS,TNS_BBS,Hatayama_et_al_AM} (see \cite{Inoue_2012} for a review). 
     We find such perspective particularly insightful. 
     In this language columns of tableaux become solitons. When $t \ll 0$ they are well separated 
     and travel independently with their own speeds. 
     At some point they interact with each other through collisions that momentarily mess up their structure. Once mutual interactions end they recover their original shape and again propagate with the same speed as before.
     A profound result in the theory of classical integrable systems is that the whole time evolution of such a system is fully determined by the knowledge of the \emph{scattering rules}.
     These consist in the precise description of exchange of degrees of freedom (i.e. how content of columns changes between backward and forward aysmptotic states) and of the \emph{phase shift}, which in our context are the shifts in asymptotic positions of solitons as compared to the ones anticipated from initial positions and speeds assuming no interaction occurs. 
   
    A natural question here is the following. 
    \begin{question}
        Can we describe the scattering rules of the skew $\RSK$ dynamics?
    \end{question}
    \noindent
    The answer to such questions from the point of view of solition theory will be provided in \cref{sec:scattering_rules}.
    
    \medskip
    
    The asymptotic increment $\mu$ was defined to be a partition such that $\mu_i'$ is the number of labeled cells of the $i$-th column in $P_t,Q_t$ for large $t$.
    Recording labels eventually remaining on each column of $P_t,Q_t$, we can construct two tableaux $V,W$, each of which consists of columns of increasing numbers. We will refer to these as \emph{vertically strict tableaux} and they differ from semi-standard tableaux in that there is no condition on rows\footnote{
    Note that  ``column strict tableaux", which sounds like a natural term to denote our vertically strict tableaux, is often used as a synonym of semi-standard tableaux.
    In literature ``column strict fillings" is also sometimes used (e.g. \cite{Lenart_Schilling_charge}), but we decided to employ a new term which includes the word ``tableaux".}. This defines a projection map $\Phi$ from a pair of skew tableaux $(P,Q)$ to a pair of vertically strict tableuax $(V, W)$.
    In the example we are considering in this section, from \eqref{eq:RSK_10} we can write $\Phi(P,Q) = ( V,W )$, with
    \begin{equation} \label{eq:intro_vst}
        V = \begin{ytableau}
            1 & 1 & 3 & 4
            \\
            2 & 2 & 5
            \\
            3
        \end{ytableau}
        ,
        \qquad
        W = \begin{ytableau}
            1 & 2 & 2 & 3
            \\
            2 & 5 & 3
            \\
            3
        \end{ytableau}.
    \end{equation}
    Vertically strict tableaux, though much less studied compared to semi-standard tableaux, play an important role in our theory because their generating function, with suitable weights, is known to produce the $q$-Whittaker polynomial \cite{Sanderson_Demazure_q_Whittaker,Nakayashiki_Yamada,schilling_tingley}; see \eqref{qWVST} below.
    This opens up a possibility to understand Cauchy type summation 
    identities involving $q$-Whittaker polynomials in a bijective fashion. 
    In order to do so we need to account for the information we lose while projecting, through $\Phi$, a pair of tableaux $(P,Q)$ to the corresponding vertically strict tableaux $(V,W)$.
    Then we arrive at the following question.
    \begin{question}
        Can we refine projection $\Phi$ into a bijection?
    \end{question}
    \noindent
    The answer to this third question represents a fundamental problem we solve in this paper. 
    The refined map $\Upsilon$, which will be described in \cref{subs:results} and \cref{sec:bijection}, yields a bijection between pairs $(P,Q)$ of skew tableaux and pairs of vertically strict tableaux $(V,W)$ plus some ``additional data" characterizing, for instance, the shape of $\RSK^t(P,Q)$ for large $t$. 

\subsection{Results, Ideas and Tools, and Applications} \label{subs:results}
    
    There are two main results in this paper: the characterization of asymptotic increment $\mu(P,Q)$ as Greene invariants and the construction of the bijection $\Upsilon$. 
    The first result answers Question 1. The second one, while being a direct answer to Question 3, also resolves Question 2. 
    An application of bijection $\Upsilon$, leads to summation identities involving $q$-Whittaker polynomials. We explain these results together with main ideas and tools to obtain them.
    
    \subsubsection{Generalized Greene Invariants}
    
    In the previous subsection we hinted how the asymptotic increment $\mu(P,Q)$ records certain conserved quantities of the skew $\RSK$ dynamics, result that we will prove in \cref{thm:asymptotic_shape_RSK}. In order to explain these conservation laws we find that algorithmic description of the skew $\RSK$ map given in terms of Schensted's bumping algorithm is not particularly insightful. 
    Instead we employ a geometrical visualization of the RSK correspondence through Viennot’s shadow line construction [Vie77]. An analogous geometric realization can be devised for the Sagan-Stanley correspondence, where Viennot’s shadow lines are “drawn” in a lattice with periodic geometry that we call {\it twisted cylinder}. 
    
    For a natural number $n$ the twisted cylinder $\mathscr{C}_n$ can be represented as an infinite vertical strip $\{1,\dots ,n\} \times \mathbb{Z}$, where we impose faces $(n,i)$ and $(1,i+n)$ to be adjacent for all $i$; see \cref{fig:matrix_intro}.
    For a more precise definition, see \cref{subs:periodic_shadow_line}.
    A matrix of non-negative integer sequences
    $\overline{M}=(\overline{M}_{i,j}(k):i,j\in\{1,\dots,n\}, k\in\mathbb{N}_0)$
    can be represented as a map $\overline{M}:\mathscr{C}_n \to \mathbb{N}_0$ by 
    setting\footnote{In the context of RSK correspondence it is common to represent matrices $M$ transposed and with first rows drawn at the bottom. This explains why $\overline{M}_{i,j}(k)$ becomes $\overline{M}(j,i-kn)$.}
    \begin{equation}
        \overline{M}(j,i-kn) = \overline{M}_{i,j}(k),
        \qquad
        \text{for all}
        \qquad
        i,j\in\{1,\dots ,n\}, k\in \mathbb{N}_0,
    \end{equation}
    with a slight abuse of notation. 
    In this new representation the Sagan-Stanley correspondence, described in \cref{subs:periodic_shadow_line} below, gives a bijection between compactly supported fillings $\overline{M}$ of $\mathscr{C}_n$ and partitions $\nu$ to pairs of tableaux. As an example, such correspondence applied to the pair $(P,Q)$ used in \cref{subs:examples} above appears in \cref{fig:matrix_intro} where, for the sake of a cleaner notation we left cells $(j,i)$ of $\mathscr{C}_n$ empty whenever $\overline{M}(j,i)=0$. In the same figure entries of the $\overline{M}$ are taken, for simplicity, to be all 0 or 1, although in general we have $\overline{M}_{i,j}(k)\in \mathbb{N}_0$. In case all entries $\overline{M}_{i,j}(k)=0$ for $k \neq 0$ such representation reduces to the Matrix-Ball construction by Fulton \cite{fulton1997young} and the Sagan-Stanley correspondence becomes the usual RSK.

\begin{figure}[ht]
    \centering
    \includegraphics[scale=.75]{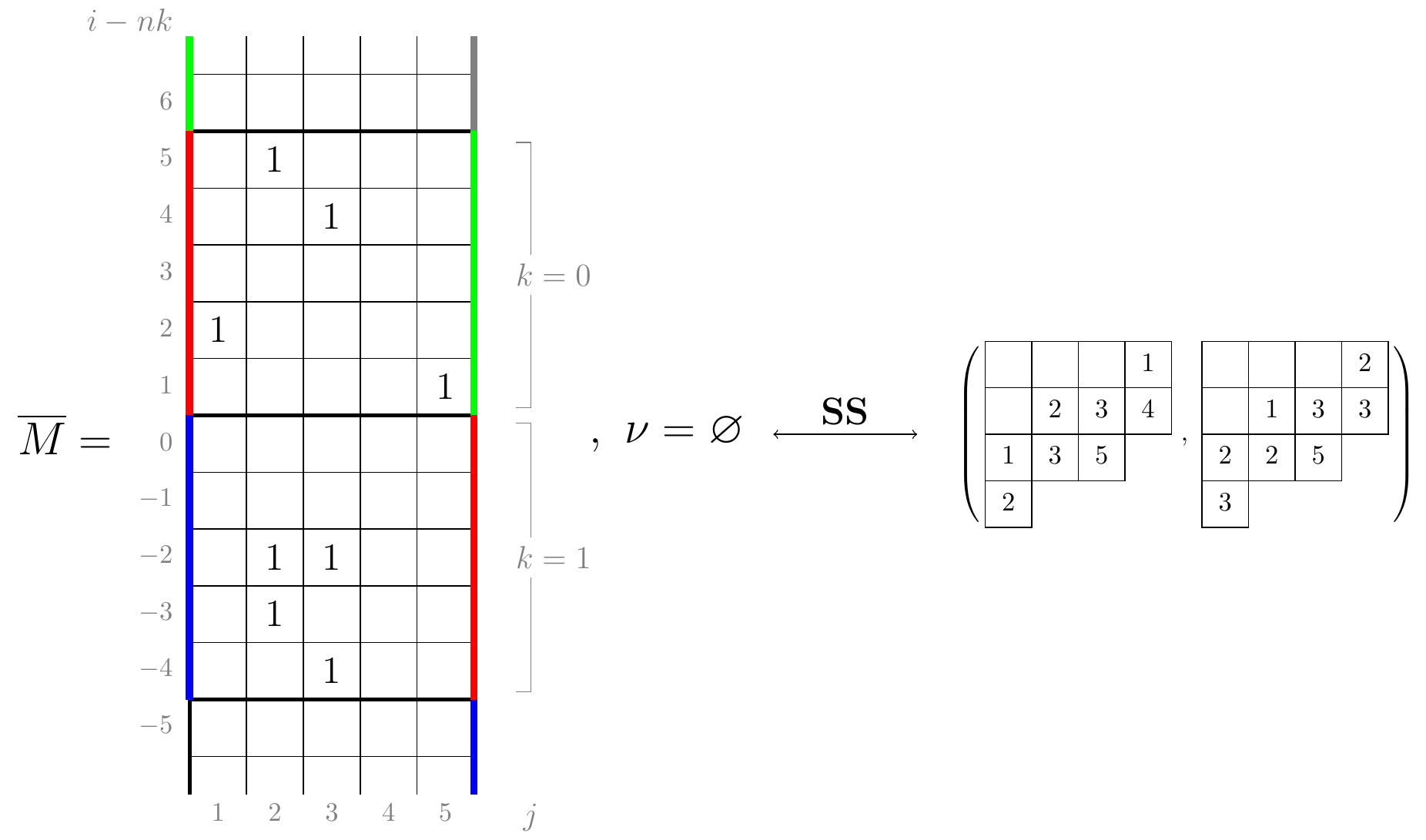}
    \caption{A realization of the Sagan-Stanley correspondence $(\overline{M};\nu) \xleftrightarrow[]{\skwRSK} (P,Q)$. The matrix $\overline{M}$ is represented as a filling of the twisted cylinder $\mathscr{C}_5$. Solid colored lines are identified by periodicity.}
    \label{fig:matrix_intro}
\end{figure}

An \emph{up-right path} $\varpi$ on the twisted cylinder is a sequence $(\varpi_\ell: \ell \in \mathbb{Z})  \subset \mathscr{C}_n$ such that
\begin{equation} \label{eq:up_right_path}
    \varpi_{\ell+1} \sim_n \varpi_\ell + \mathbf{e}_1,
    \qquad
    \text{or}
    \qquad
    \varpi_{\ell+1} \sim_n \varpi_\ell + \mathbf{e}_2,
    \qquad
    \text{for all } \ell \in \mathbb{Z},
\end{equation}
where $\sim_n$ means the equivalence such that $\mathscr{C}_n = \mathbb{Z}^2/\sim_n$; see \cref{def:twisted_cylinder} in the text.
Examples of up-right paths are colored trajectories in \cref{fig:matrix_intro_paths}. Notice that the up-right condition, along with the geometry of $\mathscr{C}_n$ implies that $\varpi$ is not self-intersecting. Define the \emph{passage time} of an up-right path $\varpi$ as 
\begin{equation}
    G(\varpi;\overline{M}) = \sum_{\ell \in \mathbb{Z}} \overline{M}(\varpi_\ell).
\end{equation}
Objects as passage times are standard in the context of RSK correspondence \cite{Greene1974}, although in classical setting endpoints of paths are usually fixed. In our case up-right paths are always infinite and have no endpoints. Define the \emph{last passage time} of $k$ disjoint paths
\begin{equation}
    I_k(\overline{M}) = \max_{\substack{\varpi^{(1)} \cupdot \cdots \cupdot \varpi^{(k)}  \subset \mathscr{C}_n \\ \varpi^{(j)} : \text{up-right path}}} \sum_{j=1}^k G(\varpi^{(j)};\overline{M}),
\end{equation}
 where $\cupdot$ denotes the disjoint union.
For the map $\overline{M}$ given in \cref{fig:matrix_intro} all last passage times can be computed explicitly, as done in \cref{fig:matrix_intro_paths}. It is in general not true that the maximizer is unique or that in order to maximize the passage time for $k+1$ paths it suffices to add an up-right path to a maximizing list of $k$ paths. 

\begin{figure}[ht]
    \centering
    \includegraphics[scale=.75]{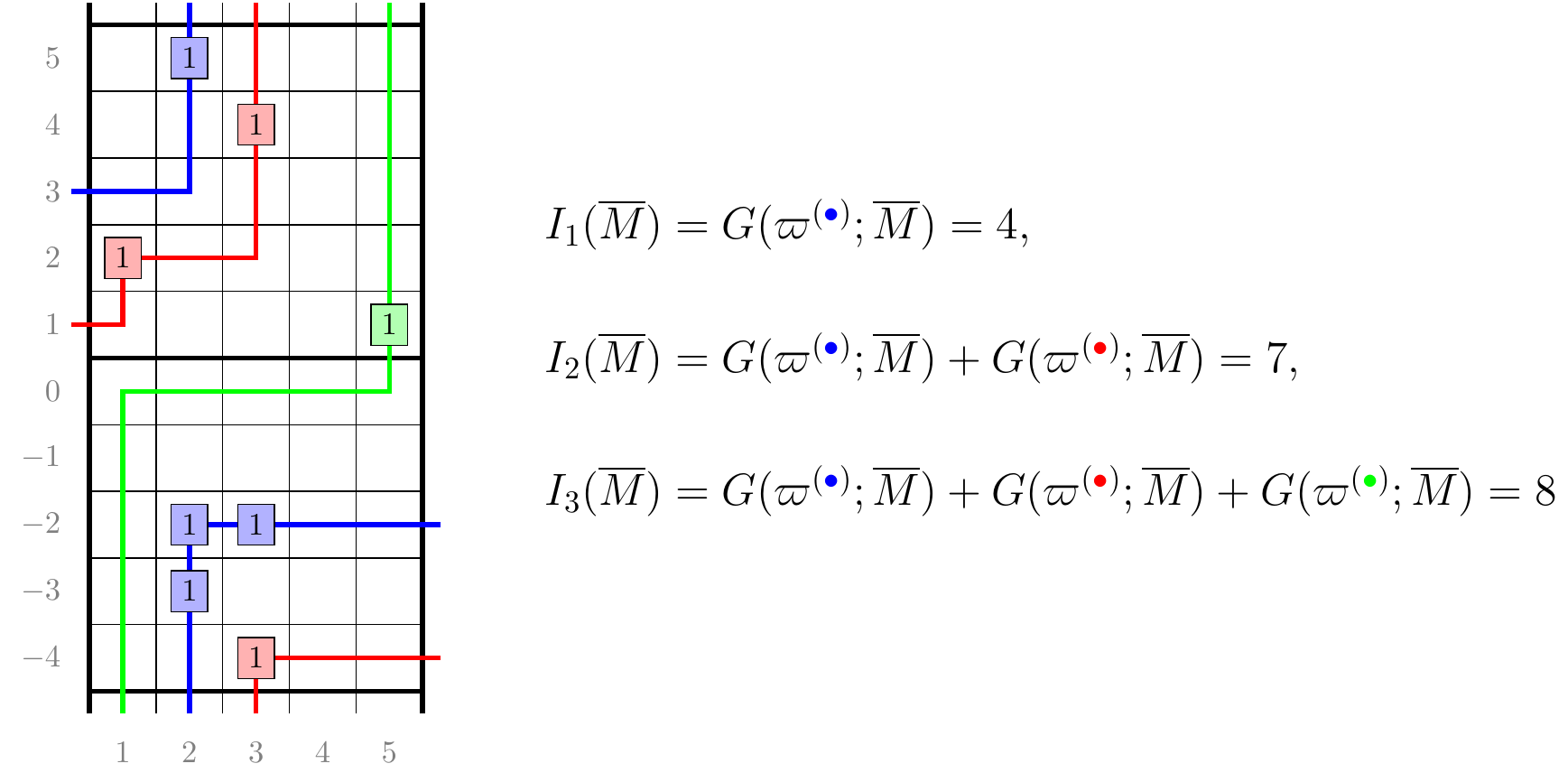}
    \caption{Up-right paths $\varpi^{ ( \textcolor{blue}{\bullet})}, \varpi^{ ( \textcolor{red}{\bullet})}, \varpi^{ ( \textcolor{green}{\bullet})}$ maximize the passage times.}
    \label{fig:matrix_intro_paths}
\end{figure}

\medskip

The first original result we present relates last passage times of a matrix $\overline{M}$ with the asymptotic increment $\mu$ of pair of tableaux $(P,Q)$ corresponding to $\overline{M}$ under Sagan-Stanley correspondence. We think of this as a generalization of Greene's theorem \cite{Greene1974}. 

\begin{theorem}[Corollary of \cref{thm:asymptotic_shape_RSK} in the text] \label{thm:intro_Greene_invariants}
    Consider $(\overline{M};\nu) \xleftrightarrow[]{\skwRSK\,} (P,Q)$ and let $\mu=\mu(P,Q)$ be the asymptotic increment of $(P,Q)$ under the skew $\RSK$ dynamics. Then we have
    \begin{equation} \label{eq:intro_Greene_invariants}
        I_k(\overline{M}) = \mu_1 + \cdots + \mu_k,
    \end{equation}
    for all $k=1,2,\dots$.
\end{theorem}

The reader can check the validity of \cref{thm:intro_Greene_invariants} comparing $\mu$ of \eqref{eq:intro_greene_inv_example} and last passage times in \cref{fig:matrix_intro_paths}. Motivated by \eqref{eq:intro_Greene_invariants}, in the text we will often refer to the partition $\mu(P,Q)$ as \emph{Greene invariant}. We will see in \cref{sec:Greene_invariants} that the statistic $\mu$ is indeed invariant with respect to a number of operations on $(P,Q)$ including (generalized) Knuth relations, Kashiwara operators, skew $\RSK$ map. In \cref{thm:asymptotic_shape_RSK}, an additional characterization of $\mu(P,Q)$ is given, in terms of maximizing closed loops on the twisted cylinder. This represents an additional generalization of the Greene's theorem \cite[Chapter 3]{sagan2001symmetric}.


The following result is an extension in skew setting of the classical Schensted's theorem \cite{Schensted1961}. It is an immediate corollary of \cref{thm:intro_Greene_invariants} along with the fact that the skew $\RSK$ map does not modify the length of the first row of the tableaux.

\begin{corollary}\label{cor:intro_schensted}
    Consider $(\overline{M};\nu) \xleftrightarrow[]{\skwRSK} (P,Q)$ and let $\lambda/\rho$ be the common shape of skew tableaux $(P,Q)$. Then
    \begin{equation}
        \lambda_1 = \nu_1 + I_1(\overline{M}).
    \end{equation}
\end{corollary}
\noindent 
A similar statement was reported in \cite{Betea_etal2019} in the context of ``free boundary Schur processes", although in that paper the quantity $I_1$ was just the last passage time and it was not related to the asymptotic increment of the corresponding pair $(P,Q)$. \Cref{thm:intro_Greene_invariants} and this corollary represent a partial answer to questions (1),(3) of \cite[Section 9]{sagan1990robinson}.

To prove \cref{thm:intro_Greene_invariants} we regard fillings of the twisted cylinder $\mathscr{C}_n$ as ``particle" configurations. We then define a deterministic dynamics $\mathbf{V}:\overline{M} \mapsto \overline{M}'$, transporting particles from sources of Viennot's shadow lines to the sinks. For the definition of the map $\mathbf{V}$ see \cref{def:Viennot_map} while for a quick view of rules of this dynamics see \cref{fig:Viennot_map} at page \pageref{fig:Viennot_map} in the text, where sources (resp. sinks) are denoted as black (resp. red) dots. 
This \emph{Viennot dynamics} is in a sense, ``dual" to the skew $\RSK$ dynamics, but conservation laws are more transparent in this picture. Indeed we will show, in \cref{thm:subsequences_Viennot} below, that last passage times are conserved quantities, i.e., $I_k(\overline{M}) = I_k(\overline{M}')$ holds. To prove this we will utilize a number of well known relations between insertion algorithms and other common operations in combinatorics, as the jeu de taquin, Knuth relations, Kashiwara operators and so on. For the sake of clarity of the exposition these prerequisites will be covered, although not extensively, in \cref{sec:Knuth_and_crystals} and \cref{app:Knuth_rel}. Translating this into the language of the skew $\RSK$ dynamics leads to the proof of \cref{thm:intro_Greene_invariants}.

\subsubsection{The bijection $\Upsilon$: statement of results} \label{subsub:bijection}

We come now to present the main result of this paper: a bijection between pairs of semi-standard tableaux and pairs of vertically strict tableaux equipped with additional weights. For any partition $\mu$ we define the set
\begin{equation}
    \mathcal{K}(\mu) = \{ \kappa=(\kappa_1,\dots,\kappa_{\mu_1}) \in \mathbb{N}_0^{\mu_1} : \kappa_i\ge \kappa_{i+1} \text{ if } \mu_{i}' = \mu_{i+1}' \}.
\end{equation}

\begin{theorem}[\Cref{thm:new_bijection} in the text] \label{thm:main_intro}
    There exists a bijection $(P,Q) \xleftrightarrow[]{\Upsilon \,} (V,W;\kappa;\nu)$ between the set of pairs $(P,Q)$ of semi-standard Young tableaux with same skew shape and quadruples $(V,W;\kappa;\nu)$, with the following properties:
    \begin{enumerate}
        \item[(i)] $V,W$ are a pair of vertically strict tableaux of shape $\mu$ and $\Phi(P,Q)=(V,W)$;
        \item[(ii)] $\kappa \in \mathcal{K}(\mu)$ and $\nu$ is a partition;
        \item[(iii)] if $P,Q$ have skew shape $\lambda/\rho$, then
        \begin{equation} \label{eq:intro_weight_preserving}
            |\rho| = \mathscr{H}(V) + \mathscr{H}(W) + |\kappa| + |\nu|,
        \end{equation}
        where $|\rho|=\rho_1+\rho_2 +\cdots$ and $\mathscr{H}$ is the intrinsic energy function; see \cref{def:energy} in the text.
    \end{enumerate}
\end{theorem}

Note that composing $\Upsilon$ with the Sagan-Stanley correspondence allows to factor out the partition $\nu$ yielding a bijection $\overline{M}\xleftrightarrow[]{\tilde{\Upsilon}\,} (V,W;\kappa)$, more similar to the classical RSK correpondence. In the text such bijection will be denoted by $\tilde{\Upsilon}$.

\medskip

Equality \eqref{eq:intro_weight_preserving} represents the most nontrivial property of $\Upsilon$. The \emph{intrinsic energy} $\mathscr{H}$, discussed more in \cref{subsub:symmetries} and at length in \cref{subs:combinatorial_R}, is a grading statistic on the set of vertically strict tableaux, which was originally introduced in the theory of crystals \cite{Hatayama_et_al_paths_crystals,Okado_Schilling_Schimozono_virtual}. Its precise definition requires the notion of combinatorial $R$-matrix and is not reported here in the introduction, but morally it measures how much a vertically strict tableaux needs to be ``modified" to produce a semi-standard tableaux.

Although the algorithmic definition of the skew $\RSK$ dynamics is not very complicated to apply to specific examples, as we did in \cref{subs:examples}, proving its various properties using only the defining rules poses serious difficulties. To circumvent these issues we will implement a more powerful machinery based on symmetries. More precisely we will show that the skew $\RSK$ dynamics possesses an affine bicrystal symmetry associated with the affine Lie algebra $\widehat{\mathfrak{sl}}_n$. This will allow us to linearize the dynamics, resulting in the precise construction of bijection $\Upsilon$ and in the proof of its various properties. 

\subsubsection{Crystal structure} \label{subsub:symmetries}

In order to establish \cref{thm:main_intro} we import ideas from the theory of crystals \cite{Hong_Kang_book_crystals,BumpSchilling_crystal_book}, which was introduced by Kashiwara and Lusztig \cite{Kashiwara_crystalizing,Kashiwara_On_Crystal_bases,Luszting_crystals} to study representations of quantum groups. In this paper we will only deal with the simple case of the affine Lie algebra $\widehat{\mathfrak{sl}}_n$. Applications of crystals are also common in the context of the Box and Ball system (BBS), which was mentioned after Question 1 in \cref{subs:examples}. For example conserved quantities, scattering rules and phase shifts of the BBS can be studied using affine crystals \cite{fukuda_okado_yamada_BBS_energy,Inoue_2012}. We will apply these ideas to precisely analyze the skew $\RSK$ dynamics.

Many of the combinatorial objects we deal with possess a natural crystal structure. For instance it is a well known fact that many properties of the RSK correspondence can be understood in the language of $\mathfrak{sl}_n$ crystals \cite{lothaire_2002,BumpSchilling_crystal_book}.
In fact, as recalled in \cite{Shimozono_dummies}, even the original algorithm by Robinson \cite{robinson1938representations}, could be stated in terms of crystals. The idea, implicit in \cite{robinson1938representations}, is to assign a permutation $\pi$ to a pair of tableaux in such a way that the assignment commutes with certain transformations, which are nothing but Kashiwara operators $\E{i},\F{i},i=1,\dots,n-1$ in today's language. Kashiwara operators act on a word by changing its content according to certain rules; for instance $\E{i}$ would change a letter $i+1$ into $i$, see \cref{subs:classical_crystals}. In this way, starting from $\pi$, through successive applications of $\E{i},\F{i}$ one would transform it into a Yamanouchi word $\pi_{\mathrm{Yam}}$ whose corresponding tableaux are canonically determined. Then to deduce the pair $(P,Q)$ associated to $\pi$ one would apply in reverse order the inverse of each Kashiwara operator, whose corresponding action on tableaux is defined through their column reading word (see \cref{subs:classical_crystals}).

An example of a well-known $\widehat{\mathfrak{sl}}_n$ crystals, which is relevant for our discussions, is the one of vertically strict tableaux. Here, in addition to $\E{i},\F{i}$ with $i=1,\dots, n-1$ one has to consider two more operators $\E{0},\F{0}$, which act by replacing $1$-labels into $n$-labels and viceversa and they are defined conjugating $\E{1},\F{1}$ by an operation called promotion \cite{shimozono_affine}.
On the set of pairs of vertically strict tableaux $(V,W)$ we may define two commuting families of Kashiwara operators 
\begin{equation} \label{eq:intro_bicrystal}
    \widetilde{E}_i^{(1)} = \E{i} \times \mathbf{1}, 
    \qquad
    \widetilde{E}_i^{(2)} =  \mathbf{1} \times \E{i},
    \qquad
    \widetilde{F}_i^{(1)} = \F{i} \times \mathbf{1}, 
    \qquad
    \widetilde{F}_i^{(2)} =  \mathbf{1} \times \F{i},
\end{equation}
letting $\E{i},\F{i}$ act on single components and this defines an example of $\widehat{\mathfrak{sl}}_n$ \emph{bicrystal}.


To study the skew $\RSK$ dynamics, we want to equip also the space of pairs $(P,Q)$ of semi-standard tableaux of skew shape with an $\widehat{\mathfrak{sl}}_n$ bicrystal structure, with the requirement that projection $\Phi:(P,Q)\mapsto (V,W)$ commutes with the action of respective Kashiwara operators. 
It turns out, however, that a naive action of Kashiwara operators $\E{i},\F{i}$ used above for vertically strict tableaux is not appropriate on skew tableaux. This is because, while $\widetilde{E}_i^{(\epsilon)}, \widetilde{F}_i^{(\epsilon)}$ for $i=1,\dots, n-1$ and $\epsilon=1,2$ commute with the skew $\RSK$ map, the same is not true for the 0-th operators $\widetilde{E}_0^{(\epsilon)}, \widetilde{F}_0^{(\epsilon)}$. One of the key novelties in this paper is that the desired 0-th Kashiwara operators, which commute with the skew $\RSK$ map and make of $\Phi$ a morphism of bicrystal graphs in the sense of \cref{def:morphism_of_crystal_graphs}, can be defined using the operation $\iota_2$. They are given by
\begin{equation} \label{eq:intro_0_th_operators}
    \begin{split}
        &\widetilde{E}_0^{(2)} = \iota_2 \circ ( \mathbf{1} \times \E{1} ) \circ \iota_2^{-1},
        \qquad
        \widetilde{F}_0^{(2)} = \iota_2 \circ ( \mathbf{1} \times \F{1} ) \circ \iota_2^{-1},
        \\
        &\widetilde{E}_0^{(1)} = \iota_1 \circ ( \E{1} \times \mathbf{1} ) \circ \iota_1^{-1},
        \qquad
        \widetilde{F}_0^{(1)} = \iota_1 \circ ( \F{1} \times \mathbf{1} ) \circ \iota_1^{-1},
    \end{split}
\end{equation}
where $\iota_1$ is defined through $\iota_2$ inverting roles of $P,Q$, that is $\iota_1(P,Q) = \mathrm{swap} \circ \iota_2 \circ \mathrm{swap} (P,Q)$. In this way, as we will show in \eqref{subs:pairs_of_tab_affine_bicrystals}, the set of pairs of semi-standard tableaux possess an $\widehat{\mathfrak{sl}}_n$ bicrystal structure. 
    
\subsubsection{The bijection $\Upsilon$: construction}
With these preparations we may now precisely define the correspondence $\Upsilon$ of \cref{thm:main_intro}. For this we study the skew $\RSK$ dynamics for a generic tableaux $(P,Q)$ by generalizing the idea by Robinson. Namely we first bring the pair $(P,Q)$ into a certain canonical form $(T,T)$ through the action of affine crystal operators, then we run the dynamics on such canonical pair, and finally transforms the result back applying inverse crystal transformations. Schematically this procedure is summarized by the following commuting diagram. 
\begin{equation} \label{eq:intro_commuting_diagram}
    \begin{tikzcd}
      (P,Q) \arrow{r}{\mathcal{L}} \arrow[swap]{d}{\RSK} & (T,T) \arrow{d}{ \RSK } \\%
        (P',Q') \arrow{r}{\mathcal{L}}& (T',T')
    \end{tikzcd}    
\end{equation}
Here $(T,T)$ is a pair of identical skew tableaux consisting of generalizations in skew setting of Yamanouchi tableaux\footnote{A Yamanouchi tableaux is a tableaux where content and shape are equal, i.e. each cell in the $i$-th row is an $i$-cell. In the text we denote the Yamanouchi tableau of shape $\mu'$ by $\mu^{\mathrm{lv}}$ and we will call it \emph{leading vector}; see \cref{subs:VST_affine_crystals}}. In the text we will call them \emph{leading tableaux}; see \cref{def:leading_tableaux}. The definition of the canonical transformation $\mathcal{L}$ is delicate and owes to deep results in the theory of affine crystals. If $V$ is a vertically strict tableau with intrinsic energy $\mathscr{H}(V)$, then a result of \cite{schilling_tingley} guarantees that through the action of Kashiwara operators $\E{i},\F{i}$ one can always transform $V$ into a Yamanouchi tableau of the same shape in such a way that the difference $\# \F{0} - \# \E{0}$ of 0-th operators used equals $\mathscr{H}(V)$. We call such transformation \emph{leading map} $\mathcal{L}_V$ 
and in \cref{subs:demazure_subgraph} we construct it in terms of the so-called Demazure arrows.
When $V,W$ are the vertically strict tableaux corresponding to $P,Q$, i.e., when $(V,W)=\Phi(P,Q)$, they can be both transformed into the same Yamanouchi tableau. For instance the ones of \eqref{eq:intro_vst} are transformed to
\begin{equation}
    \begin{ytableau}
        1 & 1 & 1 & 1 \\ 2 & 2 & 2 \\ 3
    \end{ytableau}
\end{equation}
and the respective leading maps are given by the slightly complicated expressions
\begin{equation} \label{eq:leading_maps_VW_intro}
\begin{split}
    \mathcal{L}_{V} &= \E{2} \circ \E{3} \circ \E{4} \circ \E{1} \circ \E{2} \circ \E{3} \circ \E{1} \circ \E{2},
    \\
    \mathcal{L}_{W} &= \E{3} \circ \E{4} \circ \E{1} \circ \F{0} \circ \F{4} \circ \F{3} \circ \F{1}^{\,\, 2} \circ \E{2} \circ \E{1}^{\,\, 3} \circ \E{2}.
\end{split}
\end{equation}
Using the affine bicrystal structure for $(P,Q)$, which is homomorphic to the one for $(V,W)$, we 
can simultaneously lift up the leading maps $\mathcal{L}_V,\mathcal{L}_W$
and define the map $\mathcal{L}$ on $(P,Q)$.
Moreover, our new 0-th operators \eqref{eq:intro_0_th_operators} allow to transport the result of \cite{schilling_tingley} at the level of pairs of skew tableaux.
In particular
the variation in intrinsic energy at the level of vertically strict tableaux yields the removal of $\mathscr{H}(V) + \mathscr{H}(W)$ empty boxes from the skew shape of $(P,Q)$. In the text we will call such map $\mathcal{L}$ the \emph{leading map} of the pair $(P,Q)$. To give an idea of the result of the application of a leading map we consider the pair $(P,Q)$ of \eqref{eq:intro_RSK_iota_2} and we have
\begin{equation} \label{eq:leading_maps_PQ_intro}
    \left(
        \,
            \begin{ytableau}
                \, & & & 1
                \\
                & 2 & 3 & 4
                \\
                1 & 3 & 5
                \\
                2
            \end{ytableau}
            \,\, , \,\,
            \begin{ytableau}
                \, & & & 2
                \\
                & 1 & 3 & 3
                \\
                2 & 2 & 5
                \\
                3
            \end{ytableau}
            \,
        \right)
        \xrightarrow[]{\hspace{.5cm} \mathcal{L} \hspace{.5cm}}
        \left(
        \,
            \begin{ytableau}
                \, & & & 1
                \\
                1 & 1 & 1 & 2
                \\
                2 & 2 & 3
            \end{ytableau}
            \,\, , \,\,
            \begin{ytableau}
                \, & & & 1
                \\
                1 & 1 & 1 & 2
                \\
                2 & 2 & 3
            \end{ytableau}
            \,
        \right).
\end{equation}
For more details consult \cref{sec:Knuth_and_crystals} and \cref{sec:linearization} in the text. From the computation above one can observe how the value $\mathscr{H}(V)+\mathscr{H}(W)=1$, which follows from \eqref{eq:leading_maps_VW_intro} counting the number of $\F{0}$ operators, coincides with the size difference between skew shapes in \eqref{eq:leading_maps_PQ_intro}.

The leading tableau $T$ resulting from the application of a leading map $\mathcal{L}$, as we will prove in \cref{subs:leading_tableaux}, turns out to be uniquely characterized by a triple of data $(\mu,\kappa;\nu)$. Here, $\mu$ is a partition recording the content of $T$ and it is equal to the shape of $V,W$.
$\nu$ is also a partition and it can be easily determined by ``squeezing" $T$, i.e. moving its rows as much as possible to the left without breaking the semi-standard property; see \cref{subs:kernels}.
Finally $\kappa$ is an element of $\mathcal{K}(\mu)$ and it encodes the empty shape of $T$ after the removal of $\nu$. For the tableau $T$ in the right hand side of \eqref{eq:leading_maps_PQ_intro} we have $\nu=\varnothing$ and $\kappa=(0,1,1,1)$.

A crucial observation that motivates such a long construction is that on leading tableaux, the effect of the skew $\RSK$ map becomes purely linear and it modifies the tableaux $T(\mu,\kappa;\nu)$ by just adding $\mu'$ to $\kappa$ as   
\begin{equation}
     T=T(\mu,\kappa;\nu)  \xrightarrow[]{\hspace{.5cm} \RSK \hspace{.5cm}} T'=T(\mu,\kappa+\mu';\nu).
\end{equation}
The reader familiar with discrete integrable systems might notice that the linearization given by map $\mathcal{L}$ resembles the Kerov-Kirillov-Reshetikhin (KKR) algorithm for BBS \cite{Kuniba_et_al_KKR_BBS}, although the precise connections will be explored in future works. 

This parameterization of the leading tableau $T=T(\mu,\kappa;\nu)$, along with the pair $(V,W)$ completes the construction of bijection $\Upsilon$. 
Notice that equality \eqref{eq:intro_weight_preserving} can be understood by carefully analyzing the change of number of empty boxes at each step in the description.

Concluding the example considered throughout the section, we write
\begin{equation} \label{eq:example_correspondence_from_intro}
        \left(
        \,
            \begin{ytableau}
                \, & & & 1
                \\
                & 2 & 3 & 4
                \\
                1 & 3 & 5
                \\
                2
            \end{ytableau}
            \,\, , \,\,
            \begin{ytableau}
                \, & & & 2
                \\
                & 1 & 3 & 3
                \\
                2 & 2 & 5
                \\
                3
            \end{ytableau}
            \,
        \right)
        \xleftrightarrow[]{ \hspace{.4cm} \Upsilon \hspace{.4cm} }
        \left(
        \,
        \begin{ytableau}
            1 & 1 & 3 & 4
            \\
            2 & 2 & 5
            \\
            3
        \end{ytableau}
        \,\, , \,\,
        \begin{ytableau}
            1 & 2 & 2 & 3
            \\
            2 & 5 & 3
            \\
            3
        \end{ytableau}
        ;
        (0,1,1,1); \varnothing
        \right).
    \end{equation}

\subsubsection{Summation identities} 
Finally we present some of the immediate consequences of \cref{thm:main_intro}. We use the well known fact \cite{Sanderson_Demazure_q_Whittaker,Nakayashiki_Yamada,schilling_tingley} that the generating function of vertically strict tableaux of assigned shape $\mu$ and weighted by $\mathscr{H}$ is the $q$-Whittaker polynomial
\begin{equation}
\label{qWVST}
    \sum_{V \in VST (\mu)} q^{\mathscr{H}(V)} x^V = \mathscr{P}_\mu (x;q).
\end{equation}
Bijection $\Upsilon$, or more precisely the one between $\overline{M}$ and $(V,W,\kappa)$ mentioned below \cref{thm:main_intro}, allows us to establish the Cauchy identity for $q$-Whittaker polynomials. \eqref{eq:Cauchy_id_qW_intro}.

\begin{theorem}[Bijective proof of Cauchy identity for $q$-Whittaker polynomials] 
Fix $|q|<1$ and set of variables $x=(x_1,\dots,x_n)$, $y=(y_1,\dots,y_n)$ such that $|x_iy_j|<1$. Then \eqref{eq:Cauchy_id_qW_intro} holds and the normalization term is $\mathdutchcal{b}_\mu (q) = \sum_{\kappa \in \mathcal{K}(\mu)} q^{|\kappa|}$ and is equal to \eqref{eq:b_mu}.
\end{theorem}

The following identity, that refines the Cauchy identity for both skew Schur polynomials \eqref{eq:Cauchy_id_Schur} and $q$-Whittaker polynomials \eqref{eq:Cauchy_id_qW_intro}, was stated and proven in our very recent work \cite{IMS_matching} using integrable probabilistic techniques.

\begin{theorem} \label{thm:intro_summation_qW_skew_Schur}
    Fix $|q|<1$ and sets of variables $x=(x_1,\dots,x_n)$, $y=(y_1,\dots,y_n)$. Then, for all $k=0,1,2,\dots$, we have
    \begin{equation} \label{eq:intro_summation_qW_skew_Schur}         \sum_{\ell=0}^k \frac{q^\ell}{(q;q)_\ell} \sum_{ \mu: \mu_1 = k - \ell} \mathdutchcal{b}_\mu(q) \mathscr{P}_\mu (x;q) \mathscr{P}_\mu (y;q) = \sum_{\lambda,\rho : \lambda_1= k} q^{|\rho|} s_{\lambda / \rho}(x) s_{\lambda / \rho}(y).
    \end{equation}
\end{theorem}
\noindent
Indeed the original motivation of this work was to find a bijective proof of this identity, which 
is now accomplished by the bijection $\Upsilon$ in \cref{thm:main_intro}. 
A number of analogous identities, such as Littlewood-like identities involving summations of single $q$-Whittaker polynomials, will be presented in \cref{subs:summations_qW,subs:summations_qW_Schur}. They all follow from \cref{thm:main_intro} and can be proven bijectively. See \cref{thm:Littlewood_like_id}.

\begin{remark}
    Summation identities of \cref{thm:intro_summation_qW_skew_Schur}, as well as \eqref{eq:Cauchy_id_Schur}, \eqref{eq:Cauchy_id_qW_intro} have been reported assuming set of variables $x,y$ to be $n$-tuples of complex numbers. Such assumption is not necessary to our results and only serves the purpose of keeping the notation as simple as possible. In general $x$ and $y$ in \eqref{eq:intro_summation_qW_skew_Schur} can be arbitrary specializations of the algebra of symmetric functions; see \cref{remark:general_specializations}.
    Analogously we have assumed that matrices $\overline{M}_{i,j}(k)$ are squared; i.e. $i,j \in \{1,\dots,n\}$. This is also not necessary and our results and constructions hold also for rectangular matrices. 
\end{remark}

 Identities such as \eqref{eq:intro_summation_qW_skew_Schur} or \eqref{eq:qW_and_Schur_2} in the text have important consequences in the realm of integrable probability which we will develop in a forthcoming paper \cite{IMS_KPZ_free_fermions}. In fact they provide a new way of solving stochastic integrable systems in the KPZ class connecting $q$-Whittaker polynomials with manifestly determinantal processes, as the ones related to skew Schur polynomials \cite{borodin2007periodic,betea_bouttier_periodic,Betea_etal2019}. This accomplishes a generalization in $q$-deformed setting of the original techniques employed by Johansson to solve the Totally Asymmetric Simple Exclusion Process \cite{johansson2000shape}.  Moreover properties of bijection $\Upsilon$, resulting in identity \eqref{eq:qW_and_Schur_2}, allow us to generalize ideas of Baik and Rains who used symmetries of the RSK correspondence to study, on the same footing, asymptotics of random permutations with various symmetries \cite{baik_rains2001algebraic,baik_rains2001symmetrized}. In particular this will solve, bypassing complicated Bethe Ansatz calculations, the outstanding problem of rigorously establishing pfaffian formulas for solvable models in the KPZ class in restricted environment \cite{borodin_bufetov_corwin_nested,barraquand2018,barraquand_half_space_mac,krajenbrink_le_doussal_half_space}. We will elaborate these results in \cite{IMS_KPZ_free_fermions}.

\subsection{Outline}
    In \cref{sec:preliminary} we fix notation and introduce different useful parameterizations of Young tableaux and other combinatorial objects. In \cref{sec:miscellaneous} we discuss the skew $\RSK$ map in its various formulations both in terms of the insertion algorithms and of edge local rules. In \cref{sec:iterated_RSK} we describe Sagan and Stanley's correspondence and we introduce an integrable dynamics on matrices. We call it \emph{Viennot dynamics}, being based on the shadow line construction. In \cref{sec:Knuth_and_crystals} we endow various combinatorial objects with an affine bicrystal structure.
    In \cref{sec:Greene_invariants} we establish conservation laws for the Viennot dynamics and characterize asymptotic increments $\mu(P,Q)$ as Greene invariants. In \cref{sec:linearization} we discuss the combinatorial $R$-matrix, the intrinsic energy function $\mathscr{H}$. Subsequently we implement a combinatorial transformation that reduces the skew $\RSK$ map to a linear map.
    Such linearization defines a useful class of tableaux, we call \emph{leading tableaux} and which we study in detail. In \cref{sec:bijection} we discuss our bijection $\Upsilon$, proving \cref{thm:main_intro}. 
    In \cref{subs:extensions} we also propose, without entering technical discussion, a few natural extensions of \cref{thm:main_intro}.
    In \cref{sec:scattering_rules} we study the scattering and phase shift of the skew $\RSK$ dynamics. Finally in \cref{sec:summation_identities} we give proofs of a number of summation identities for $q$-Whittaker polynomials and skew Schur polynomials. In \cref{app:Knuth_rel} we review classical notions of Knuth relations and we propose their generalizations in skew setting. In \cref{app:proof_thm} we give a proof of an invariance property of last passage times with respect to crystal operators.

\subsection{Acknowledgments}
    We thank Nikolaos Zygouras and Kirone Mallick for comments and suggestions on an early version of this paper. We are grateful to Shinji Koshida and Ryosuke Sato for discussions and remarks about representation theoretic aspects of this paper. We also thank Rei Inoue 
    for useful remarks 
    about theory of crystals and integrable systems. MM is grateful to Takato Yoshimura for showing interest in this work and to Alexander Garbali for discussions about combinatorics of symmetric polynomials.  
    The work of TS has been supported by JSPS KAKENHI Grants No.JP15K05203, No. JP16H06338, No. JP18H01141, No. JP18H03672, No. JP19L03665, No. JP21H04432. The work of TI has been supported by JSPS KAKENHI Grant Nos. 16K05192, 19H01793,
    and 20K03626.









\section{Preliminary notions} \label{sec:preliminary}
\subsection{Biwords and matrices of integers} \label{subs:biwords}

We introduce the alphabet $\mathcal{A}_n=\{1,\dots,n\}$ and we denote by $\mathcal{A}_n^*$ the set of word of finite length in $\mathcal{A}_n$. The length of a word $p$ is denoted by $\ell(p)$, while its content is recorded by an array $\gamma=(\gamma_1, \dots, \gamma_n)$, where $\gamma_i$ equals the multiplicity of $i$ in $p$.

Given two natural numbers $n,m$ we denote by $\mathbb{A}_{n,m}$ the set of \emph{biwords} in the alphabet $\mathcal{A}_n, \mathcal{A}_m$. A biword $\pi \in \mathbb{A}_{n,m}$ is an array of pairs $\left(\begin{smallmatrix} q_1 & q_2 & \cdots & q_k \\p_1 &p_2 &\cdots &p_k \end{smallmatrix}\right)$, $k\in\mathbb{N}$, where $q_i\in \mathcal{A}_m$, $p_i\in \mathcal{A}_n$ and whose columns are ordered lexicographically. This means that for all $i$ we have $q_i \le q_{i+1}$ and whenever $q_i=q_{i+1}$, then $p_i \le p_{i+1}$. Clearly words $p_1 \dots p_k$ are particular cases of biwords, obtained setting $q_i=i$. Permutations $\sigma \in \mathcal{S}_n$ are also particular cases of biwords where we set $q_i=i$ and $p_i=\sigma_i$. 

Given $\mathcal{A}_n,\mathcal{A}_m$ we consider \emph{weighted biwords} in these alphabets and we denote their set by $\overline{\mathbb{A}}_{n,m}$. Elements of $\overline{\mathbb{A}}_{n,m}$ are arrays of triplets\footnote{In \cite{sagan1990robinson} columns $\left( \begin{smallmatrix} q_i \\ p_i \\ w_i \end{smallmatrix} \right)$ were denoted by $\left( \begin{smallmatrix} q_i \\ p_i^{(w_i)} \end{smallmatrix} \right)$.}
\begin{equation} \label{eq:weighted_biword}
    \overline{\pi} = \left(\begin{matrix} q_1 & q_2 & \cdots & q_k \\ p_1 & p_2 & \cdots & p_k \\w_1 & w_2 &\cdots &w_k \end{matrix}\right),
\end{equation}
where again $q_i\in \mathcal{A}_m$, $p_i \in \mathcal{A}_n$ and weights $w_i \in \mathbb{Z}$. Columns of a weighted biword are arranged so that the biword composed by $q,p$ is lexicographically ordered, with top entries taking precedence, while if $q_i=q_{i+1}$ and $p_i=p_{i+1}$, then $w_i \ge w_{i+1}$. The total weight of a weighted biword $\overline{\pi}$ is the sum of the $w_i$ entries and we denote it by $\wt(\overline{\pi})$, i.e., $\wt(\overline{\pi}) = \sum_i w_i$. In this text weighted biwords will always be denoted by overlined greek letters $\overline{\pi},\overline{\sigma},\overline{\xi},\dots$ so to distinguish them from biwords, whose symbols are never overlined.

For later use we also present an alternative format to express a weighted biword, which we call \emph{timetable ordering}. Given $\overline{\pi}$ its timetable ordering $\overline{\pi}^{\na}$ is the array consisting of the same triplets of $\overline{\pi}$ arranged in such a way that $w_i^{\na} \ge w_{i+1}^{\na}$ and in case $w_i^{\na} = w_{i+1}^{\na}$ then $q_i^{\na} \le q_{i+1}^{\na}$ and if also $q_i^{\na} = q_{i+1}^{\na}$ then $p_i^{\na} \le p_{i+1}^{\na}$. Examples of a weighted biword and of its timetable ordering are 
\begin{equation}
    \overline{\pi}=
    \left( \begin{matrix} 
    1 & 1 & 1 & 1 & 2 & 3 & 3 & 3 & 4 \\
    2 & 3 & 3 & 3 & 1 & 3 & 4 & 4 & 2
    \\
    0 & 2 & 1 & 1 & 0 & 1 & 2 & -1 & 1 \end{matrix} \right),
    \qquad
    \overline{\pi}^{\na}=
    \left( \begin{matrix} 
    1 & 3 & 1 & 1 & 3 & 4 & 1 & 2 & 3 \\
    3 & 4 & 3 & 3 & 3 & 2 & 2 & 1 & 4
    \\
    2 & 2 & 1 & 1 & 1 & 1 & 0 & 0 & -1 \end{matrix} \right).
\end{equation}
Particular cases of weighted biwords are \emph{weighted words}, where we assume $q_i=i$, or \emph{weighted permutations}, where $q_i=i$ and $p_1,\cdots, p_k$ form a permutation of $\{ 1, \dots , k\}$. We also use the notion of \emph{partial (weighted) permutations}, that are weighted biwords where each $q$ and $p$ rows present no repetitions. The set of weighted biwords where all weights are non-negative integers will be important to us and is denoted with $\overline{\mathbb{A}}_{n,m}^+$.

\medskip

Biwords of $\mathbb{A}_{n,m}$ are in natural bijection with the set of rectangular matrices with non-negative integral entries
\begin{equation}
    \mathbb{M}_{n \times m} \coloneqq \{ (M_{i,j}; 1\le i \le n, 1 \le j \le m ): M_{i,j} \in \mathbb{N}_0  \}. 
\end{equation}
Such correspondence is realized assigning to $\pi$ the matrix $m$ with elements 
\begin{equation} \label{eq:matrix_biword}
    M_{i,j}=\# \text{ of } \binom{j}{i} \text{ in } \pi.
\end{equation}
Analogously weighted biwords of $\overline{\mathbb{A}}_{n,m}$ are in correspondence with rectangular matrices whose elements are eventually vanishing sequences of non-negative integers 
\begin{equation}
    \overline{\mathbb{M}}_{n\times m} = \{ (\overline{M}_{i,j}:\mathbb{Z} \to \mathbb{N}_0 : 1\le i \le n, 1 \le j \le m) : \overline{M}_{i,j}(k) = 0 \text{ for } |k| \gg 0  \}.
\end{equation}
Also in this case to a weighted biword $\overline{\pi}$ we assign the matrix $\overline{M}$ defined by
\begin{equation} \label{eq:matrix_weighted_biword}
    \overline{M}_{i,j}(k) = \# \text{ of } \left(
    \begin{matrix}
    j\\i\\k
    \end{matrix}
    \right)
    \text{ in } 
    \overline{\pi}.
\end{equation}
As earlier we will always denote matrices of $\overline{\mathbb{M}}_{n \times n}$ by overlined capital letters, to distinguish them from those of $\mathbb{M}_{n \times n}$.
The subset of $\overline{\mathbb{M}}_{n \times m}$
in bijection with non-negatively weighted biwords $\overline{\mathbb{A}}_{n,m}^+$ is denoted by $\overline{\mathbb{M}}_{n \times m}^+$. 
The weight of a matrix is defined as 
$$
\wt(\overline{M}) = \sum_{k \in\mathbb{Z}}k \sum_{i,j=1}^n M_{i,j}(k),
$$ 
so that under correspondence \eqref{eq:weighted_biword} we have $\wt(\overline{M}) = \wt(\overline{\pi})$.
We will represent matrices $\overline{M}\in \overline{\mathbb{M}}_{n,m}$, with a slight abuse of notation, as compactly supported maps $\overline{M} : \{1, \dots, m\} \times \mathbb{Z} \to \mathbb{N}_0$ via the identification 
\begin{equation}\label{eq:matrices_as_maps}
    \overline{M}(j,i-kn) = \overline{M}_{i,j}(k),
    \qquad
    \text{for all}
    \qquad
    i\in \mathcal{A}_n,\, j\in \mathcal{A}_m, \, k\in \mathbb{Z}.
\end{equation}

Given two weighted biwords $\overline{\pi},\overline{\pi}'$ we consider their disjoint union $\overline{\pi} \cupdot \overline{\pi}'$ formed taking all columns of $\overline{\pi}$ and $\overline{\pi}'$ and rearranging them in the correct order. In case $\overline{M},\overline{M}'$ are the matrix corresponding to $\overline{\pi},\overline{\pi}'$, then naturally $\overline{M} + \overline{M}'$ is the matrix associated to their union.

Throughout the paper we will consider a number of operations on biwords and most of times these will have nice description in the language of matrices.
For instance, if $\overline{M}$ is the matrix corresponding to a weighted biword $\overline{\pi}$, then to its transpose $\overline{M}^T$ it will correspond a biword $\overline{\pi}^{-1}$ obtained from $\overline{\pi}$ swapping the $p$ and the $q$ rows and rearranging the result in prescribed order. This notation is standard and is justified by the fact that when $\overline{\pi}$ is a permutation $\overline{\pi}^{-1}$ is its inverse in the symmetric group.

\subsection{Partitions and Young diagrams} \label{subs:partitions_and_tableaux}
 
A partition $\lambda=(\lambda_1, \lambda_2,\cdots)$ is a weakly decreasing sequence of integers $\lambda_i$ eventually being zero. The number of non-zero parts of $\lambda$ is its \emph{length} and it is denoted by $\ell(\lambda)$. We say that $\lambda$ partitions $k$ if $| \lambda | = \lambda_1 + \lambda_2 + \cdots = k$ and sometimes we write $\lambda \vdash k$. The multiplicative notation $\lambda=1^{m_1(\lambda)} 2^{m_2(\lambda)}\cdots$ is often used and $m_i(\lambda)$ denotes the multiplicity of $i$ in $\lambda$. The set of all partitions is denoted by $\mathbb{Y}$. 
Sometimes we will refer to arrays $\varkappa=(\varkappa_1,\dots,\varkappa_N) \in \mathbb{N}_0^{N}$ of $N$ non-negative integers as \emph{compositions}. Given a composition $\varkappa \in \mathbb{N}_0^{N}$, we denote by $\varkappa^+$ the unique partition that can be generated permuting elements of $\varkappa$.

Partitions are identified by their Young diagrams and we will freely interchange these two notions. Viewing the plane $\mathbb{Z} \times \mathbb{Z}$ with the vertical coordinate increasing downward, the Young diagram of $\lambda$ is the collection of cells $(c,r)$ with
$1 \le r \le \ell(\lambda)$ and $1\le c \le \lambda_r$. 
Reflecting the Young diagram of $\lambda$ with respect to the main diagonal we obtain the transposed partition $\lambda'$ with parts $\lambda_i' = \# \{ j: \lambda_j \ge i \}$. Given two partitions $\mu, \lambda$ we write $\mu \subseteq \lambda$ if $\mu_i \le \lambda_i$ for all $i$ or equivalently if their Young diagrams are encapsulated. When $\mu \subseteq \lambda$ we define the skew Young diagram $\lambda/\mu$ consisting of all all cells in $\lambda$ but not in $\mu$. The number of cells of $\lambda/\mu$ is denoted with $|\lambda/\mu|$.

As hinted in the introduction, for the discussion in this paper, it is essential to allow Young diagrams to have rows at non-positive coordinates. For this we define the upward translation $\mathcal{T}_{-i} : (c,r) \mapsto (c,r-i)$, for any $i\in \mathbb{N}_0$ and the set of \emph{generalized Young diagrams} $\mathbb{Y}_{-i} = \mathcal{T}_{-i} (\mathbb{Y})$. A generalized Young diagram $\lambda \in \mathbb{Y}_{-i}$ is associated to its \emph{generalized partition} $(\lambda_{1-i},\lambda_{2-i},\dots)$. The notion of skew diagrams is defined as always: if $\mu,\lambda \in \mathbb{Y}_{-i}$ then $\lambda/\mu$ is the set of cells in $\lambda$ but not in $\mu$. When drawing generalized Young diagrams we will color cells at non-positive rows in gray, to give a reference. We report an example of a skew Young diagram and an generalized one obtained translating it
\begin{equation}
    \ydiagram{3+3,2+3,1+2,1+2,2,1} 
    \xrightarrow[]{\hspace{.2cm} \mathcal{T}_{-3} \hspace{.2cm}}
    \hspace{.5cm}
    \begin{ytableau}
        \none[{\textcolor{gray}{\scriptstyle -2\hspace{2ex}}}] & \none &\none & \none & *(gray!35) & *(gray!35) &  *(gray!35) 
        \\
        \none[{\textcolor{gray}{\scriptstyle -1\hspace{2ex}}}] & \none & \none & *(gray!35) &  *(gray!35) &  *(gray!35) 
        \\
        \none[{\textcolor{gray}{\scriptstyle 0 \hspace{2ex}}}] & \none & *(gray!35) & *(gray!35) 
        \\
        \none[{\textcolor{gray}{\scriptstyle 1 \hspace{2ex}}}] & \none & &
        \\
        \none[{\textcolor{gray}{\scriptstyle 2 \hspace{2ex}}}] & &
        \\
        \none[{\textcolor{gray}{\scriptstyle 3 \hspace{2ex}}}] &
    \end{ytableau}.
\end{equation}
At times we will need to distinguish generalized Young diagrams from non-generalized ones, that in these circumstances will be called \emph{classical}. Statements and constructions reported in this text often apply the same to classical Young diagrams or to generalized Young diagrams and unless required, we will not stress the difference. Nevertheless we point out that not every operation defined on classical Young diagrams is possible in the generalized case: for instance the notion of transposition $\lambda'$ is only defined if $\lambda \in \mathbb{Y}$.

Given a classical Young diagram $\mu$, its \emph{rectangular decomposition} is given by indeces $0=R_0<R_1<\cdots < R_N=\mu_1$ such that
\begin{equation} \label{eq:rect_dec}
        \mu_{R_{i-1}+1}' = \cdots = \mu_{R_i}' > \mu_{R_i+1}', 
\end{equation}
for $i=1,\dots,N$; see \cref{fig:fundamental_shifts}. When using the notion of rectangular decomposition we will denote by $r_i=R_{i-1}- R_{i}$ the base of each rectangle of $\mu$.
    
\begin{figure}[ht]
     \centering
     \includegraphics{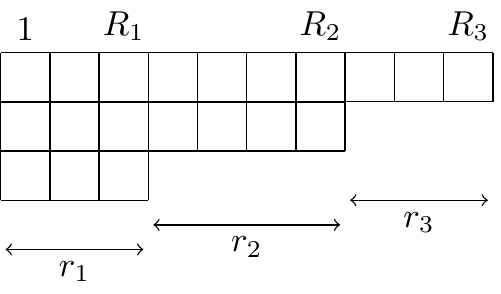}
     \caption{$R_i$'s and $r_i$'s for a partition $\mu$.}
     \label{fig:fundamental_shifts}
\end{figure}

\subsection{Young tableaux}

A Young tableau, or simply a tableau, $T$ is a filling of cells of a Young diagram with natural numbers. The label assigned to a specific cell $(c,r) \in \lambda$ is indicated with $T(c,r)$ and in this case $\lambda$ is called shape of $T$. If a cell has label $\ell$ we call that an $\ell$-cell. The content of a tableau $T$ is recorded by an array $\gamma=(\gamma_1,\gamma_2,\dots)$, where $\gamma_i = \# i$-cells in $T$. We will mainly deal with two types of tableaux. \emph{Semi-standard} tableaux have entries strictly increasing column-wise and weakly increasing row-wise. When their labels range in the alphabet $\mathcal{A}_n$ and the shape $\lambda / \rho$ is fixed we denote their set by $SST(\lambda /\rho ,n)$. In such cases we call $\lambda$ the \emph{external shape} and $\rho$ the \emph{empty shape}. It could happen that the shape $\lambda/\rho$ of $T$ is a generalized skew Young diagram, as in the examples in \cref{subs:examples} and in such cases $T$ is a \emph{generalized semi-standard tableau}. As before, we will not stress the generalized property unless required. A particular class of semi-standard tableaux is that of \emph{standard tableaux}, defined by the property of having content $\gamma_i=1$ for all $i=1,\dots, |\lambda / \rho|$. Their set is denoted by $ST(\lambda/\rho)$.

The other class of tableaux, which will play an important role in this paper, are the \emph{vertically strict tableaux} that have labels strictly increasing column-wise and no additional conditions. These are sometimes called ``column strict fillings" \cite{Lenart_Schilling_charge}, or when entries have no repetitions ``column tabloids" \cite{sagan2001symmetric}. Shapes of vertically strict tableaux will always be straight (i.e. non-skew) classical Young diagrams $\mu \in \mathbb{Y}$ and we denote their set by $VST(\mu,n)$, when entries range in the alphabet $\mathcal{A}_n$. Examples of semi-standard and vertically strict tableaux are 
\begin{equation} \label{eq:examples_sst_and_VST}
    \begin{ytableau}
        \, & & & 2 & 4
        \\
        & 1 & 3 & 3 & 5
        \\
        1 & 2 & 5
    \end{ytableau}
    \qquad
    \text{and}
    \qquad
    \begin{ytableau}
        2 & 4 & 1 & 1 & 3
        \\
        3 & 5 & 3 & 2
        \\
        5
    \end{ytableau},
\end{equation}
with content respetively equal to $(2,2,2,1,2)$ and $(2,2,3,1,2)$.

Given a tableau $T$ we define its \emph{row reading word} $\pi^{\mathrm{row}}_T$ concatenating rows of $T$ starting from the last. Alternatively, the \emph{column reading word} $\pi^{\mathrm{col}}_T$ is formed reading entries of $T$ column by column from the last row up and from left to right. For instance, if $T$ is the semi-standard tableau in \eqref{eq:examples_sst_and_VST} we have
\begin{equation}
    \pi^{\mathrm{row}}_T = 1\,\,2\,\,5\,\,1\,\,3\,\,3\,\,5\,\,2\,\,4
    \qquad
    \textrm{and}
    \qquad
    \pi^{\mathrm{col}}_T = 1\,\,2\,\,1\,\,5\,\,3\,\,3\,\,2\,\,5\,\,4.
\end{equation}

\subsection{Kernels of tableaux} \label{subs:kernels}
We now define a useful statistic of a semi-standard Young tableau of classical shape.
\begin{definition} \label{def:kernel_P}
    Given a classical skew tableau $P\in SST(\lambda/\rho,n)$, its $\emph{kernel}$ is the partition $\varkappa = \ker(P) \in \mathbb{Y}$ such that $\varkappa_j - \varkappa_{j+1}$ is the maximal number of boxes one can shift the first $j$ rows of $P$ to the left without breaking the semi-standard property, for all $j=1,2,\dots$.
\end{definition}

For example if
\begin{equation}
    P =\begin{ytableau}
        \, & & & & & 2 & 4
        \\
        & & 1 & 3 & 3 & 5
        \\
        1 & 2 & 5
    \end{ytableau},
    \qquad
    \text{then}
    \qquad
    \ker(P) = \ydiagram{2,1}.
\end{equation}
In fact shifting the second and first row of $P$  respectively one and two cells to the left we obtain the semi-standard tableau of \eqref{eq:examples_sst_and_VST}, which cannot be ``squeezed" anymore. 

In order to describe more precisely the partition $\ker(P)$ we introduce the notion of \emph{overlap} of two weakly increasing words $A,B$. This is defined as $\ell(B)$ minus the size of the empty shape of the minimal semi-standard tableaux having first and second row content given by $A$ and $B$ respectively. In formulas we have
\begin{equation}
    \overlap(A,B) = \max_{L \in \{0,\dots, \ell(A) \wedge \ell(B) \}}  \left\{ L  : B_{\ell(B)-L+i} > A_{i}, \text{ for all } i=1,\dots,L \right\}.
\end{equation}
For instance for $A=1 \,\, 3 \,\, 3 \,\, 5$ and $B= 1 \,\, 2 \,\, 2 \,\, 3 \,\, 4$ we have $\overlap(A,B) = 2$, since the minimal semi-standard tableaux with first row $A$ and second row $B$ is
\begin{equation}
    \begin{ytableau}
        \, & & & 1 & 3 & 3 & 5
        \\
        1 & 2 & 2 & 3 & 4
    \end{ytableau}
    \,\,.
\end{equation}

Recording the $j$-th row of a tableau $P \in SST(\lambda/\rho,n)$ in a weakly increasing word $p^{(j)}$ of length $\theta_j=\lambda_j-\rho_j$ and setting $\varkappa=\ker(P)$ we have
\begin{equation} \label{eq:nu_PQ}
    \varkappa_j - \varkappa_{j+1} = \rho_j - \lambda_{j+1} +  
    \overlap(p^{(j)}, p^{(j+1)}).
\end{equation} 

Given a pair $(P,Q)$ of semi-standard tableaux with same shape we can define $\ker(P,Q)$ in the same way as in \cref{def:kernel_P}. If $\nu=\ker(P,Q)$ then $\nu_j-\nu_{j+1}$ is the maximal amount of cells we can shift the first $j$ rows of $P$ and $Q$ simultaneously to the left without breaking the semi-standard property. In this case if $\varkappa=\ker(P), \kappa=\ker(Q)$ and $\nu=\ker(P,Q)$. then it is clear that for all $j=1,2\dots$, we have
\begin{equation}
    \nu_{j}-\nu_{j+1} = \min \{ \varkappa_{j} -\varkappa_{j+1} , \kappa_{j} -\kappa_{j+1} \}.
\end{equation}

\subsection{Row coordinate parameterization} \label{subs:row_coordinates}

To any generalized semi-standard tableau $P\in SST (\lambda/\rho, n)$ we can assign its \emph{row coordinate matrix} $\alpha = \rc(P)$, defined by
\begin{equation} \label{eq:alpha_ij_def}
    \alpha_{i,j} = \# \text{ of  $i$-cells at row } j \text{ of } P.
\end{equation}
The set of such infinite matrices is 
\begin{equation}
    \mathbb{M}_{n \times \infty} \coloneqq \{ (\alpha_{i,j}\in \mathbb{N}_0: 1\le i\le n, j\in \mathbb{Z}) :\, \alpha_{i,j}\neq 0 \text{ for finitely many } i,j \}.
\end{equation}
Such encoding of tableaux was defined already in \cite{Danilov_2005}. In case a tableau $P$ is standard we condensate in an array $\mathpzc{a}$ information contained in the row coordinate matrix. Define the \emph{row coordinate array} $\mathpzc{a}$ of a standard tableaux $P$ as
\begin{equation} \label{eq:row_coordinate_array}
    \mathpzc{a}_{\,i} = \text{row with the unique $i$-cell of $P$}.
\end{equation}
We will abuse of the notation and write $\rc(P)=\mathpzc{a}\in \mathbb{Z}^k$ rather than $\rc(P)= \alpha \in \mathbb{M}_{n \times \infty}$, when it is clear from the context that $P$ is standard.

\medskip

Map $\rc:P \mapsto \alpha$ is not bijective since shifting rows of $P$ laterally does not change its row coordinate matrix. It can nevertheless be refined into a bijection recording in a certain way relative positions of rows of $P$.  We do so in the next definition, where, for the sake of a simpler description, we only consider the case of tableaux with classical shape. We will use the notation $\rc(P,Q) = (\rc(P), \rc(Q))$.

\begin{definition}
    Let $\lambda,\rho \in \mathbb{Y}$ and $P,Q \in SST(\lambda /\rho,n)$. The \emph{row coordinate parameterization} of $(P,Q)$ is the triple $(\alpha,\beta;\nu)$ such that $(\alpha, \beta)= \rc(P,Q)$ and $\nu=\ker(P,Q)$. In this case we use the notation $(P,Q) \xleftrightarrow[]{\rc\,}(\alpha,\beta;\nu)$. Choosing $P=Q$ we also define $P \xleftrightarrow[]{\rc} (\alpha;\nu)$ setting $\alpha=\rc(P)$, $\nu=\ker(P)$.
\end{definition}

In the definition above we have assumed that tableaux $P,Q$ have same classical shape $\lambda/\rho$ and that their labels belong to the same alphabet $\mathcal{A}_n$. This forces their row coordinate matrices to belong to set
\begin{equation} \label{eq:M_n_plus}
    \mathcal{M}_n^+ = \left\{ (\alpha,\beta) \in \mathbb{M}_{n \times \infty}^+ \times \mathbb{M}_{n \times \infty}^+ : \sum_{i=1}^n ( \alpha_{i,j} - \beta_{i,j} )= 0 \text{ for all } j\in \mathbb{Z}  \right\},
\end{equation}
where $\mathbb{M}_{n \times \infty}^+$ is the subspace of $\mathbb{M}_{n \times \infty}$ of matrices $\alpha$ such that $\alpha_{i,j} = 0$ if $j \le 0$. 

\begin{proposition} \label{prop:row_coordinate_classical_pair}
    The correspondence $(P,Q) \xleftrightarrow[]{\rc} (\alpha,\beta;\nu)$ is a bijection between the set of pairs $(P,Q)$ of classical semi-standard tableaux with labels in $\mathcal{A}_n$ and $\mathcal{M}_n^+ \times \mathbb{Y}$.
\end{proposition}

\begin{proof}
    We need to construct the inverse map $(\alpha,\beta;\nu) \to (P,Q)$. 
    For this define weakly increasing words $p^{(j)},q^{(j)}$ as 
    \begin{equation} \label{eq:p_j_q_j}
        p^{(j)} = 1^{\alpha_{1,j}} 2^{\alpha_{2,j}} \cdots,
        \qquad
        q^{(j)} = 1^{\beta_{1,j}} 2^{\beta_{2,j}} \cdots.
    \end{equation}
    Since $(\alpha,\beta) \in \mathcal{M}_n^+$,  $p^{(j)},q^{(j)}$ have the same length denoted by $\theta_j$. Define also partition $\eta$ through relations
    \begin{equation} \label{eq:empty_shape_eta}
        \eta_i-\eta_{i+1} = \theta_{j+1} - \min \left\{ 
     \overlap(p^{(j)}, p^{(j+1)}) , \overlap(q^{(j)}, q^{(j+1)} )   \right\}.
    \end{equation}
    Then $P,Q$ are the tableaux of shape $\lambda/\rho$ with $\lambda = \eta + \theta + \nu$ and $\rho= \eta + \nu$ and with $j$-th rows given by words $p^{(j)},q^{(j)}$. It is straightforward to check that maps $(P,Q)\mapsto (\alpha,\beta;\nu)$ and $(\alpha,\beta;\nu) \mapsto (P,Q)$ are mutual inverses.
\end{proof}

\begin{example}\label{ex:rc}
    Consider the pair of semi-standard tableaux
    \begin{equation}
        (P,Q) = \left( 
        \,
        \begin{ytableau}
            \, & & & 2 \\ & & 1 & 3 \\ 1 & 2
        \end{ytableau}
        \,,\,
        \begin{ytableau}
            \, & & & 1 \\ & & 2 & 2 \\ 1 & 3
        \end{ytableau}
        \,
        \right).
    \end{equation}
    Then we have $(P,Q) \xleftrightarrow[]{\rc \,} (\alpha, \beta ; \nu)$, where 
    \begin{equation}
        \alpha = 
        \left( 
        \begin{matrix}
            0 & 1 & 1 & 0 & \cdots
            \\
            1 & 0 & 1 & 0 & \cdots
            \\
            0 & 1 & 0 & 0 & \cdots
        \end{matrix} 
        \right)
        ,
        \qquad
        \beta = 
        \left( 
        \begin{matrix}
            1 & 0 & 1 & 0 & \cdots
            \\
            0 & 2 & 0 & 0 & \cdots
            \\
            0 & 0 & 1 & 0 & \cdots
        \end{matrix} 
        \right),
        \qquad
        \nu = \ydiagram{1,1}
        \,\,.
    \end{equation}
\end{example}

For later use we also introduce the set 
\begin{equation}\label{eq:Z_n_plus}
    \mathcal{Z}^+_n = \left\{ (\mathpzc{a},\mathpzc{b}) \in \mathbb{N}^{n} \times \mathbb{N}^{n} : \mathpzc{b}_{\, i}=\mathpzc{a}_{\,\sigma(i)} \text{ for some } \sigma \in \mathcal{S}_n \right\},
\end{equation}
consisting on all pairs $(\mathpzc{a},\mathpzc{b})$ that are row coordinate arrays of pairs of standard tableaux. At times in the text we will also use sets $\mathcal{M}_n$, $\mathcal{Z}_n$ defined as in \eqref{eq:M_n_plus}, \eqref{eq:Z_n_plus} removing ``+" superscripts and ${}_{\ge 1}$ subscripts.



\subsection{Standardization} \label{subs:standardization}

We define the operation of \emph{standardization} \cite{schutzenberger_RS} of semi-standard tableaux
\begin{equation}
    \std : P \in SST(\lambda / \rho , n) \mapsto P' \in ST(\lambda / \rho).
\end{equation}
Let $\gamma = (\gamma_1 ,\dots, \gamma_n)$ be the content of tableau $P$ and define $\Gamma_i=\gamma_1 + \cdots + \gamma_i$ for $i=1,\dots,n$, where $\Gamma_0=0$ by convention. Then cells of $P'=\std(P)$ are labeled replacing $i$-cells of $P$, from the leftmost to the right with $\Gamma_{i-1} + 1 , \dots , \Gamma_i$. For instance we have
\begin{equation} \label{eq:std_example}
    \begin{ytableau}
        \, & & & 2 \\ & & 1 & 3 \\ 1 & 2
    \end{ytableau}
    \xrightarrow[]{\hspace{.4cm} \std \hspace{.4cm}}
    \begin{ytableau}
        \, & & & 4 \\ & & 2 & 5 \\ 1 & 3
    \end{ytableau}
    \,\,.
\end{equation}
It is clear that, remembering the content $\gamma$ of the original tableaux $P$, one can recover $P$ from its standardization $P'$.

We present the analog of standardization in the language of matrices. Rows of matrices in $\mathbb{M}_{n\times \infty}$ are compactly supported infinite arrays of non-negative integers and we denote their set by
\begin{equation} \label{eq:set_V}
    \mathbb{V}=\bigg\{ (v_j)_{j\in \mathbb{Z}} \subset \mathbb{N}_0^{\infty}: |v| = \sum_{j\in \mathbb{Z}} v_j < +\infty \bigg\}.
\end{equation}
We will write elements $v \in \mathbb{V}$ via the expansion $v= \sum_{k \in \mathbb{Z}} v_k \mathbb{e}_k $ where $\mathbb{e}_k$ is the standard basis of $\mathbb{C}^\infty$. Introducing the Weyl chamber
\begin{equation} \label{eq:Weyl_chamber}
    \mathbb{W}^k = \{ (\mathpzc{a}_{\,1},\dots,\mathpzc{a}_{\,k}) \in \mathbb{Z}^k: \mathpzc{a}_{\,1} \ge \cdots \ge \mathpzc{a}_{\,k} \},
\end{equation}
we define the natural correspondence $\mathbb{W}^k \leftrightarrow \{ v \in \mathbb{V} : |v| = k\},$ by the invertible mapping 
\begin{equation} \label{eq:from_v_to_a}
    \mathpzc{a} \mapsto v(\mathpzc{a})=\sum_{i=1}^n \mathbb{e}_{\mathpzc{a}_{\,i}}.
\end{equation}
Given $\alpha \in \mathbb{M}_{n \times \infty}$ and denoting its rows by $\alpha_1,\alpha_2 ,\dots$, we define the standardization $\mathpzc{a}=\std(\alpha)$ as the array 
\begin{equation}
    \mathpzc{a} = ( \mathpzc{a}^{(1)}_{\,1},\dots, \mathpzc{a}^{(1)}_{\,|\alpha_1|}, \dots , \mathpzc{a}^{(n)}_{\,1},\dots, \mathpzc{a}^{(n)}_{\,|\alpha_n|} ),
\end{equation}
obtained joining smaller arrays $\mathpzc{a}^{(1)}\in \mathbb{W}^{|\alpha_1|}, \dots, \mathpzc{a}^{(n)} \in \mathbb{W}^{|\alpha_n|}$ such that $v(\mathpzc{a}^{(i)}) = \alpha_i$, under correspondence \eqref{eq:from_v_to_a}. One can check that standardization of matrices is compatible with the row coordinate parameterization, or in other words
\begin{equation} 
    \rc \circ \std (P) = \std \circ \rc (P)
\end{equation}
for all semi-standard tableaux $P$. An example of such commutation relation can be observed considering the semi-standard tableau $P$ on the left hand side of \eqref{eq:std_example}, whose row coordinate matrix $\alpha=\rc(P)$ was reported in \cref{ex:rc}. Then we see that both $\std(\alpha)$ and $\rc(\std(P))$ give as a result the array $\mathpzc{a}=(3,2,3,1,2)$.

\section{Skew $\RSK$ map and edge local rules} \label{sec:miscellaneous}

We revisit a combinatorial operation
introduced by Sagan and Stanley in \cite{sagan1990robinson}. In order to fully describe its properties we will present different formulations of this construction. 

\subsection{Skew $\RSK$ map of tableaux} \label{subs:RSK_product_tableaux}
In this subsection we define the skew $\RSK$ map, as the result of consecutive operations on pairs of tableaux $P,Q$. In case $P,Q$ share the same shape this is equivalent to the definition $\RSK=\iota_2^n$ given in the introduction, as proven in \cref{prop:RSK_from_n_int_ins} below.

Let $P$ be a semi-standard tableau of generalized shape $\lambda/ \rho$. A cell $(c,r) \in \lambda/\rho$ is a \emph{corner cell} if $(c-1,r),(c,r-1) \notin \lambda/\rho$. Consider an integer $r$ such that $P$ has a corner cell $(c,r)$ at row $r$. The internal insertion $\mathcal{R}_{[r]}$, first introduced in \cite{sagan1990robinson}, is the operation that constructs tableau $P'=\mathcal{R}_{[r]}(P)$ vacating cell $(c,r)$ of $P$ and inserting value $P(c,r)$ at row $r+1$ following Schensted's bumping algorithm. For this we first find $\overline{c} = \min\{k : P(k,r+1) > P(c,r) \}$. If $\overline{c}$ does not exist, we simply add a $P(c,r)$-cell at the right of row $r+1$ of $P$. Alternatively, if $\overline{c}$ exists, we assign cell $(\overline{c},r+1)$ label $P(c,r)$ and we insert, following the same mechanism $P(\overline{c},r+1)$ at row $r+2$. Eventually this algorithm stops and the result is the tableau $P'$. It could also happen that we try the internal insertion $\mathcal{R}_{[r]}$ at some row $r$ with only empty cells. In that case the result is a tableaux with an extra empty cell at row $r$.

The definition of the skew $\RSK$ map of tableaux is given below. 

\begin{definition}[Skew $\RSK$ map of tableaux] \label{def:RSK_product_tableaux}

Let $P \in SST(\lambda / \rho ,n)$, $Q \in SST(\mu / \rho ,m)$, for some generalized Young diagrams $\lambda,\mu,\rho$. Define the skew $\RSK$ map of $P,Q$
\begin{equation}
    \RSK(P,Q) = (P',Q') \in SST(\varkappa / \mu,n) \times SST(\varkappa / \lambda , m),
\end{equation}
    via the following algorithm. Set $P^{(0)} = P$ and $Q^{(0)} = \lambda/\lambda$, i.e. $Q^{(0)}$ has empty shape and external shape equal to $\lambda$. For $j=1,\dots,m$, let $r^{(j)}_1\ge \cdots \ge r^{(j)}_{k_j}$, be the row coordinates of all $j$-cells of $Q$ and define
    \begin{equation}
        P^{(j)} = \mathcal{R}_{[r^{(j)}_{k_j}]} \circ \cdots \circ \mathcal{R}_{[r^{(j)}_{1}]} (P^{(j-1)}).
    \end{equation}
    Then define $Q^{(j)}$ adding to $Q^{(j-1)}$ $j$-cells so that the external shape of $Q^{(j)}$ matches that of $P^{(j)}$. Finally set $P' = P^{(m)}$ and $Q'=Q^{(m)}$. Sometimes we will consider the skew $\RSK$ map between \emph{standard} tableaux $(P,Q)$ and in such case we may call this operation \emph{skew $\RS$ map}.
\end{definition}

The reader can check the definition of the skew $\RSK$ map of tableaux with the following example, where for simplicity we have taken $P,Q$ of equal shape
\begin{equation} \label{eq:RSK_prod_example}
    \left( 
    \,
    \begin{ytableau}
        \, & & & 2 \\ & & 1 & 3 \\ 1 & 2
    \end{ytableau}
    \,,\,
    \begin{ytableau}
        \, & & & 1 \\ & & 2 & 2 \\ 1 & 3
    \end{ytableau}
    \,
    \right)
    \xrightarrow[]{\RSK}
    \left( 
    \,
    \begin{ytableau}
        \, & & & \\ & & & \\ & & 2 \\ 1 & 1 & 3 \\ 2
    \end{ytableau}
    \,,\,
    \begin{ytableau}
        \, & & & \\ & & & \\ & & 1 \\ 1 & 2 & 2 \\ 3
    \end{ytableau}
    \,
    \right).
\end{equation}
We also report step by step calculations 
\begin{align*}
    \left(P^{(0)}, Q^{(0)} \right) 
    = 
    \left( \,\,
    \begin{ytableau}
        \, & & & 2 \\ & & 1 & 3 \\ 1 & 2
    \end{ytableau}
    \,\, , \,\,
    \begin{ytableau}
        \, & & & \\ & & & \\ &
    \end{ytableau}
    \,\,
    \right)
    \rightsquigarrow 
    \left(P^{(1)}, Q^{(1)} \right) 
    = 
    \left( \,\,
    \begin{ytableau}
        \, & & & \\ & & 1 & 2 \\ & 2 & 3 \\ 1
    \end{ytableau}
    \,\, , \,\, 
    \begin{ytableau}
        \, & & & \\ & & & \\ & & 1 \\ 1
    \end{ytableau}
    \,\, \right)
    \\
    \rightsquigarrow 
    \left(P^{(2)}, Q^{(2)} \right) 
    = 
    \left( \,\, 
    \begin{ytableau}
        \, & & & \\ & & & \\ & 1 & 2 \\ 1 & 2 & 3
    \end{ytableau}
    \,\, , \,\, 
    \begin{ytableau}
        \, & & & \\ & & & \\ & & 1 \\ 1 & 2 & 2
    \end{ytableau}
    \,\,
    \right)
    \rightsquigarrow 
    \left(P^{(3)}, Q^{(3)} \right) 
    = 
    \left( \,\,
    \begin{ytableau}
        \, & & & \\ & & & \\ & & 2 \\ 1 & 1 & 3 \\ 2
    \end{ytableau}
    \,\, , \,\,
    \begin{ytableau}
        \, & & & \\ & & & \\ & & 1 \\ 1 & 2 & 2 \\ 3
    \end{ytableau}
    \,\, \right)
\end{align*}
Additional examples are given in \cref{fig:RS_product} right panel and in \cref{subs:examples}.

\begin{remark}
The skew $\RSK$ map is essentially the same as the \emph{skew Knuth map} in \cite{sagan1990robinson}, with one difference. In \cite{sagan1990robinson} authors allowed the ``external" insertion of new cells, prescribed by a biword $\pi$, in the original pair of \emph{classical} tableaux $(P,Q)$, so that the skew Knuth map had the form $(P,Q;\pi) \mapsto(P',Q')$. In our case we don't consider external insertions, as we imagine that cells of the biword $\pi$ are already contained in generalized tableaux $(P,Q)$, although ``hidden" at non-positive rows. Therefore external insertions in the skew Knuth map correspond, in the skew $\RSK$ map, to cells that, following some internal insertion, move from row 0 to row 1.
\end{remark}

\medskip

Next we present a symmetry of the skew $\RSK$ map of tableaux that was proven in \cite{sagan1990robinson} and that will be useful in a number of cases.

\begin{proposition}[\cite{sagan1990robinson}, Theorem 3.3] \label{prop:RSK_swap_symmetry}
    If $\RSK(P,Q) = (P',Q')$, then $\RSK(Q,P) = (Q',P')$.
\end{proposition}

\Cref{prop:RSK_swap_symmetry} says that, like $P'$ and $P$, also the recording tableau $Q'$ is obtained from $Q$ following a number of internal insertion. This fact is not obvious from \cref{def:RSK_product_tableaux}. 

\medskip

The operation of standardization defined at the end of  \cref{subs:partitions_and_tableaux} is well behaved with respect to the skew $\RSK$ map, as stated in the next proposition. This was already observed in \cite{sagan1990robinson} and such natural reduction simplifies proofs of many statements.

\begin{proposition}
    We have $\std \circ \RSK (P,Q) = \RS \circ \std (P,Q)$.
\end{proposition}
\begin{proof}
    From \cref{def:RSK_product_tableaux} it is clear that if $(P',Q')=\RSK(P,Q)$, then 
    \begin{equation}
        \RSK(P, \std (Q)) = (P',\std(Q')).
    \end{equation}
    Combining this with symmetry of \cref{prop:RSK_swap_symmetry}, we conclude the proof.
\end{proof}

\subsection{Operations $\iota_1,\iota_2$: internal insertion with cycling}

Here we introduce two operations $\iota_1,\iota_2$ on pairs of tableaux sharing the same shape. They are the same as in the Introduction.

\begin{definition}[Internal insertion with cycling] \label{def:iota}
    Let $(P,Q)$ be a pair of semi-standard tableaux with same shape. Let $r_1\ge \cdots \ge r_k$ be the row-coordianates of all $1$-cells of $Q$. Define
    \begin{equation}
        \iota_2(P,Q) = (P', Q')
    \end{equation}
    where $P' = \mathcal{R}_{[r_k]} \cdots \mathcal{R}_{[r_1]}(P)$ and $Q'$ is obtained from $Q$ vacating first all $1$-cells, subtracting 1 from the remaining entries and finally adding $n$-cells so that the final shape equals that of $P'$.  For an example, see \eqref{ex_iota2}.
    We also define $\iota_1=\mathrm{swap} \circ \iota_2 \circ \mathrm{swap}$, where $\mathrm{swap}(x,y)=(y,x)$.
\end{definition}

The next proposition states that both  $\iota_1$ and $\iota_2$ represent refinements 
of the skew $\RSK$ map and that the definition of the skew $\RSK$ map of tableaux given in  \cref{def:RSK_product_tableaux} is consistent 
with the one given in \cref{subs:examples},
under the assumption that $P,Q$ have same shape.

\begin{proposition} \label{prop:RSK_from_n_int_ins}
    For $P,Q\in SST(\lambda / \rho ,n)$, we have
    \begin{equation}
        \iota_1^n(P,Q) = \RSK(P,Q) = \iota_2^n (P,Q).
    \end{equation}
\end{proposition}
\begin{proof}
    Denote $(\tilde{P}^{(j)},\tilde{Q}^{(j)}) = \iota_2^j(P,Q)$ for $1\leq j\leq n$.     
    Comparing \cref{def:RSK_product_tableaux}
    and \cref{def:iota}, it is easy to see that $P^{(j)} = \tilde{P}^{(j)}$ and that, for all $k=1,\dots,j$, $k$-cells of $Q^{(j)}$ correspond to $(k+n-j)$-cells in $\tilde{Q}^{(j)}$.
    Taking $j=n$ 
    proves $\iota_2^n = \RSK$. The complementary statement $\iota_1^n = \RSK$ follows now from the swap symmetry of \cref{prop:RSK_swap_symmetry}.
\end{proof}


\begin{proposition} \label{prop:iota_std}
    Let $P,Q\in SST(\lambda/\rho,n)$. Recalling the content $\gamma$, define $N_\epsilon = \gamma_1(P)$ if $\epsilon=1$ or $N_\epsilon = \gamma_1(Q)$ if $\epsilon=2$. Then $\std \circ \iota_\epsilon (P,Q) = \iota_\epsilon^{N_\epsilon} \circ \std (P,Q)$.
\end{proposition}
\begin{proof}
    This follows from the sequential definition of $\iota_1,\iota_2$ given in \cref{def:iota}.
\end{proof}

\begin{proposition} \label{prop:iota_preserves_ker}
    Let $P,Q\in SST(\lambda/\rho,n)$ for some $\lambda,\rho \in \mathbb{Y}$ and define $(\tilde{P},\tilde{Q})=\iota_\epsilon(P,Q)$ for $\epsilon$ being either $1$ or $2$. Then $\ker(P,Q)=\ker(\tilde{P},\tilde{Q})$.
\end{proposition}

\begin{proof}
    We will only prove our claim for pairs of standard tableaux $P,Q$ and for $\epsilon=2$. This implies the more general case with pairs of semi-standard tableaux by \cref{prop:iota_std}. On the other hand the case $\epsilon=1$ follows by the $\epsilon=2$ case since the kernel of a pair of tableaux is invariant under swap, i.e. $\ker(P,Q)=\ker(Q,P)$.
    
    For any $j$ let $p^{(j)}, q^{(j)}, \tilde{p}^{(j)}, \tilde{q}^{(j)}$ be weakly increasing words recording respectively the $j$-th rows of $P, Q, \tilde{P}, \tilde{Q}$. Assume that the $1$-cell of $Q$ lies at row $r$ and that the $n$-cell of $\tilde{Q}$ lies at row $\bar{r}$, i.e. $1\in p^{(r)}$ and $n\in \tilde{p}^{(\bar{r})}$. Then we have $\tilde{P}=\mathcal{R}_{[r]}(P)$ and during the internal insertion a new cell gets created at row $\bar{r}$. Let $\nu=\ker(P,Q)$ and $\tilde{\nu} = \ker(\tilde{P},\tilde{Q})$. We aim to show that for each $j \ge 1$ we have $\nu_j - \nu_{j+1} = \tilde{\nu}_j - \tilde{\nu}_{j+1}$. For this we use the explicit expression of the kernel of a pair of tableaux discussed in \cref{subs:kernels} and we have
    \begin{equation}
    \begin{split}
        \nu_j-\nu_{j+1} &= \rho_j - \lambda_{j+1} + \overlap(p^{(j)},p^{(j+1)}) \wedge \overlap(q^{(j)},q^{(j+1)}) 
        \\
        \tilde{\nu}_j-\tilde{\nu}_{j+1} &= \tilde{\rho}_j - \tilde{\lambda}_{j+1} + \overlap(\tilde{p}^{(j)},\tilde{p}^{(j+1)}) \wedge \overlap(\tilde{q}^{(j)},\tilde{q}^{(j+1)}),
    \end{split}
    \end{equation}
    where $\tilde{\lambda} / \tilde{\rho}$ is the the shape of $\tilde{P}, \tilde{Q}$. Since
    \begin{equation}
        \tilde{\rho}_j = \rho_j + \delta_{j,r}
        \qquad
        \text{and}
        \qquad
        \tilde{\lambda}_j = \lambda_j + \delta_{j,\bar{r}},
    \end{equation}
    to prove our proposition we need to show that
    \begin{equation} \label{eq:iota_ker_claim_1}
        \overlap(\tilde{p}^{(j)},\tilde{p}^{(j+1)}) \wedge \overlap(\tilde{q}^{(j)},\tilde{q}^{(j+1)} ) = \overlap(p^{(j)},p^{(j+1)}) \wedge \overlap(q^{(j)},q^{(j+1)}) - \delta_{j,r} + \delta_{j,\bar{r}-1}.
    \end{equation}
    We start by comparing overlaps between rows of $Q$ and $\tilde{Q}$. We find that
    \begin{equation} \label{eq:iota_ker_claim_2}
        \overlap(\tilde{q}^{(j)},\tilde{q}^{(j+1)}) = \overlap(q^{(j)},q^{(j+1)}) - \mathbf{1}_{1\in q^{(j)}} + \mathbf{1}_{n \in \tilde{q}^{(j+1)}},
    \end{equation}
    which follows from a simple inspection of cycling of letters in the $Q$ tableau and we only need to take care of rows where a cell is vacated or created. The comparison between $\overlap(p^{(j)}, p^{(j+1)})$ and $\overlap(\tilde{p}^{(j)}, \tilde{p}^{(j+1)})$ is more laborious and we need to check for all different choices of $j$. To simplify our notation we set $a=p^{(j)}$, $b=p^{(j+1)}$, $\tilde{a}=\tilde{p}^{(j)}$, $\tilde{b}=\tilde{p}^{(j+1)}$. If $L=\overlap(a,b)$ then $a_i<b_{\ell(b) - L +i}$ for all $i=1,\dots,L$ and there exists $x$ such that $a_x > b_{\ell(b) - L +x-1}$. Assume that $x$ is the smallest of such indices and call $y=\ell(b) - L +x-1$, as in \cref{fig:two_rows}. We call $(x,y)$ a \emph{blocking pair of depth $L$}. It is clear that the existence of a blocking pair of depth $L$ is equivalent to say that $\overlap(a,b)=L$. Notice that such notation also covers extremal cases when $L=\ell(a)$, and $L=\ell(b)$ where we set respectively $x=\ell(a)+1$ and $y=0$. Let us now confirm \eqref{eq:iota_ker_claim_1} for all cases. 
    \begin{figure}
        \centering
        \includegraphics{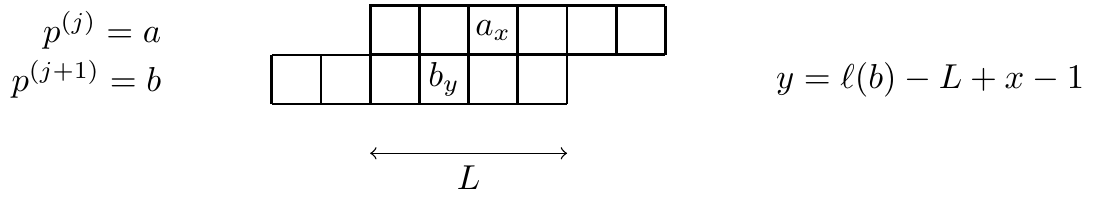}
        \caption{Notation used in the proof of \cref{prop:iota_preserves_ker}. Here $L=\overlap(a,b)$ and $a_x>b_y$ so that $x$ and $y$ form a blocking pair of depth $L$.}
        \label{fig:two_rows}
    \end{figure}
    \begin{enumerate}
        \item[$j<r-1$.] In this case $a=\tilde{a}, b=\tilde{b}$ and by \eqref{eq:iota_ker_claim_2}, \eqref{eq:iota_ker_claim_1} holds.
        \item[$j=r-1$.] Here we have $\tilde{a}=a$ and $\tilde{b}=b_2 \cdots b_{\ell(b)}$. We set $\tilde{x}=x$ and $\tilde{y} = y-1$ and this is a blocking pair of depth $L$ for $\tilde{a},\tilde{b}$, whenever $y \neq 0$. If on the other hand $y=0$, we have $L=\ell(b)$ and since $1 \in q^{(r)}$, then $L>\overlap(q^{(r-1)},q^{(r)}) = \overlap(\tilde{q}^{(r-1)},\tilde{q}^{(r)})$, by \eqref{eq:iota_ker_claim_2}. In both cases \eqref{eq:iota_ker_claim_1} holds.
        \item[$j=r$.] In this case $\tilde{a}=a_2 \dots a_{\ell(a)}$ and $\tilde{b}=b_1 \cdots b_{\bar{k}} \, a_1 \, b_{\bar{k}+1} \cdots b_{\ell(b)}$ for an index $\bar{k} \in \{1,\dots,\ell(b)\}$. If $\bar{k}=\ell(b)$, then and $\overlap(\tilde{a},\tilde{b})=L=0$ and by \eqref{eq:iota_ker_claim_2}, \eqref{eq:iota_ker_claim_1} holds. Assume now that $\bar{k}<\ell(b)$. If $x=1$, then necessarily $\bar{k} = y$ and $\tilde{x}=1$, $\tilde{y}=y+1$ is a blocking pair of depth $L-1$ for $\tilde{a},\tilde{b}$. If on the other hand $x\neq 1$, then $\bar{k} \le y$ and in such case we set $\tilde{x}=x-1, \tilde{y}=y$, which again is a blocking pair of depth $L-1$ for $\tilde{a},\tilde{b}$. Overall we have shown that $\overlap(\tilde{a},\tilde{b}) = \overlap(a,b) -1 +\delta_{\bar{k} , \ell(b)}$, which confirms \eqref{eq:iota_ker_claim_1}.
        \item[$r<j<\bar{r}-1$.] We have $\tilde{a}=a_1 \cdots a_{\bar{m}-1} \, z \, a_{\bar{m}+1} \cdots a_{\ell(a)}$ for some letter $z$ and some index $\bar{m}$. Similarly we have $\tilde{b}=b_1 \cdots b_{\bar{k}} \, a_{\bar{m}} \, b_{\bar{k}+1} \cdots b_{\ell(b)}$ for an index $\bar{k} \in \{1,\dots,\ell(b)\}$. If $\bar{m}>x$ then $\bar{k} \ge y$ and we set $\tilde{x}=x,\tilde{y}=y$. If $\bar{m}=x$, then $\bar{k}=y$ and we set $\tilde{x}=x+1,\tilde{y}=y+1$. Lastly, if $\bar{m}<x$ necessarily $\tilde{b}_y < \tilde{a}_x$ and we set $\tilde{x}=x,\tilde{y}=y$. In all cases $\tilde{x},\tilde{y}$ form a blocking pair of depth $L$ for $\tilde{a},\tilde{b}$, confirming \eqref{eq:iota_ker_claim_1}.
        \item[$r<j=\bar{r}-1$.] Observe that the case $r=j=\bar{r}-1$ was already treated above. We have $\tilde{a}=a_1 \cdots a_{\bar{m}-1} \, z \, a_{\bar{m}+1} \cdots a_{\ell(a)}$ for some letter $z$ and some index $\bar{m}$ and $\tilde{b}=b_1\cdots b_{\ell(b)} \, a_{\bar{m}}$. Since $a_{\bar{m}}>b_k$ for all $k$ we necessarily have $\bar{m}>L$ and $\tilde{x}=x,\tilde{y}=y$ becomes a blocking pair of depth $L+1$ for $\tilde{a},\tilde{b}$. Hence, from \eqref{eq:iota_ker_claim_2}, \eqref{eq:iota_ker_claim_1} holds.
        \item[$j=\bar{r}$.] In this case $\tilde{a}=a_1\cdots a_{\ell(a)} z$ for some letter $z$ greater than all entries of $a$ and $\tilde{b}=b$. If $\ell(a)<L$ then necessarily $\overlap(\tilde{a},\tilde{b}) = \overlap(a,b)$, which implies \eqref{eq:iota_ker_claim_2}. If on the other hand $\ell(a) = L$, then it could happen that $\overlap(\tilde{a},\tilde{b}) =L+1$, if entries of $a$ are small and $z<b_{\ell(b)}$. Nevertheless, since 
        \begin{equation}
            \ell(a) = \ell(q^{(\bar{r})}) \ge \overlap(q^{(\bar{r})}, q^{(\bar{r}+1)}) = \overlap( \tilde{q}^{(\bar{r})}, \tilde{q}^{(\bar{r}+1)}),
        \end{equation}
        we always have $\tilde{q}^{(\bar{r})}, \tilde{q}^{(\bar{r}+1)}) \le L$ and \eqref{eq:iota_ker_claim_1} holds.
        \item[$j>\bar{r}$.] Here $\tilde{a}=a$ and $\tilde{b}= b$ so that $\eqref{eq:iota_ker_claim_1}$ trivially holds.
    \end{enumerate}
    The previous list of checks exhausts all the cases and completes the proof.
\end{proof}

\subsection{The skew RS map of arrays} \label{subs:RS_product}

In this subsection we introduce the skew $\RS$ map of pairs of arrays, following  \cite{viennot_edge_rules}. Due to the correspondence between arrays and standard tableaux, it provides a diagrammatic realization of the skew $\RS$ map for standard tableaux, as will be seen in \cref{cor:RS_rc}.
It represents a reformulation of the shadow line construction by Viennot \cite{Viennot_une_forme_geometrique},
as we see now. 

\begin{definition}[Shadow line construction] \label{def:shadow_line}
Let $\mathpzc{a} \in \mathbb{Z}^n,\mathpzc{b} \in \mathbb{Z}^m$ be a pair of arrays and $\Lambda_{m,n}=\{1,\dots, m\} \times \{ 1,\dots , n\}$ be a finite rectangular lattice.
From the left edge of each cell $(1,i)$ for $i=1,\ldots,n$ (resp. bottom edge of each cell $(j,1)$ for $j=1,\ldots, m$), start drawing a line of color $\mathpzc{a}_{\,i}$ to the right (resp. of color $\mathpzc{b}_{\,j}$ to the top). 
We draw the line configuration until every cell of $\Lambda_{m,n}$ is crossed both vertically and horizontally, using the following rules. Lines of different colors cross each other, while if two lines of the same color $C$ meet, then from the intersection point they will proceed in their rightward and upward run with their color upgraded to $C+1$. The ensemble of lines on $\Lambda_{m,n}$ generated with this procedure is called \emph{shadow line construction}.

\end{definition}

\begin{figure}[ht]
    \centering
    \includegraphics{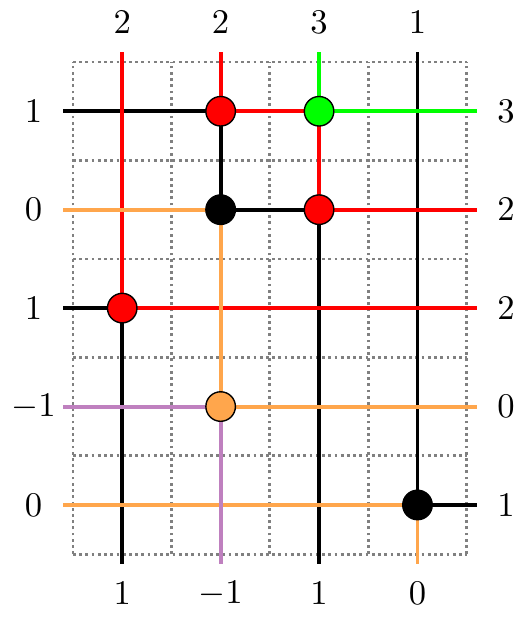}
    \caption{
    The shadow line construction for $\mathpzc{a}=(0,-1,1,0,1)$ and $\mathpzc{b}=(1,-1,1,0)$, equivalent to the skew $\RS$ map in the left panel of \cref{fig:RS_product}
    }
    \label{fig:shadow_lines_example}
\end{figure}

For an example of a shadow line construction, see \cref{fig:shadow_lines_example}. There in correspondence to intersection points of lines with the same color $C$ we drew bullets of color $C+1$.
The rules of assigning colors to lines in this construction can be translated into local rules of configurations on edges of the lattice \cite{Fomin1986,viennot_edge_rules}. 

\begin{definition}[$\mathbb{Z}$-valued edge configurations] \label{def:Z_edge_config}
    On a planar lattice $\Lambda \subseteq \mathbb{Z}\times \mathbb{Z}$, \emph{$\mathbb{Z}$-valued edge configurations} $\mathcal{E}$ are quadruples of functions $(\mathsf{W},\mathsf{S},\mathsf{E},\mathsf{N}):\Lambda \to \mathbb{Z}$ such that $\mathsf{E}(c) = \mathsf{W}(c+\mathbf{e}_1)$ and $\mathsf{N}(c) = \mathsf{S}(c+\mathbf{e}_2)$ for all points $c\in \Lambda$ where these conditions make sense. We say that $\mathcal{E}$ is \emph{admissible} if it satisfies the \emph{local rules}
\begin{equation} \label{eq:local_rules}
    \begin{minipage}{.9\linewidth}
        \begin{enumerate}
            \item $\mathsf{E}(c) = \mathsf{W}(c)$  and  $\mathsf{N}(c) = \mathsf{S}(c)$, if $\mathsf{W}(c) \ne \mathsf{S}(c)$;
            \item $\mathsf{E}(c) = \mathsf{N}(c) = \mathsf{S}(c)+1$, if $\mathsf{S}(c)=\mathsf{W}(c)$. 
        \end{enumerate}
    \end{minipage}
\end{equation}
\end{definition}

In this language we define the skew 
$\RS$ map of arrays as follows. 

\begin{definition}[Skew $\RS$ map of arrays] \label{def:RS}
    Let $\mathpzc{a} \in \mathbb{Z}^n,\mathpzc{b} \in \mathbb{Z}^m$ be a pair of arrays. Consider the unique admissible edge configuration on $\Lambda_{m,n}$ with $\mathsf{W}(1,i)=\mathpzc{a}_{\,i}$, $\mathsf{S}(j,1)=\mathpzc{b}_{\,j}$ for all $i=1,\dots,n$, $j=1,\dots,m$. The skew $\RS$ map  of $\mathpzc{a},\mathpzc{b}$ is the pair $(\mathpzc{a}',\mathpzc{b}')=\RS(\mathpzc{a},\mathpzc{b})$ given by $\mathpzc{a}_{\, i}'=\mathsf{E}(m,i)$ and $\mathpzc{b}_{\, j}'=\mathsf{N}(j,n)$ for $i=1,\dots,n$, $j=1,\dots,m$.
\end{definition}

\begin{figure}[ht]
    \centering
    \includegraphics[scale=.9]{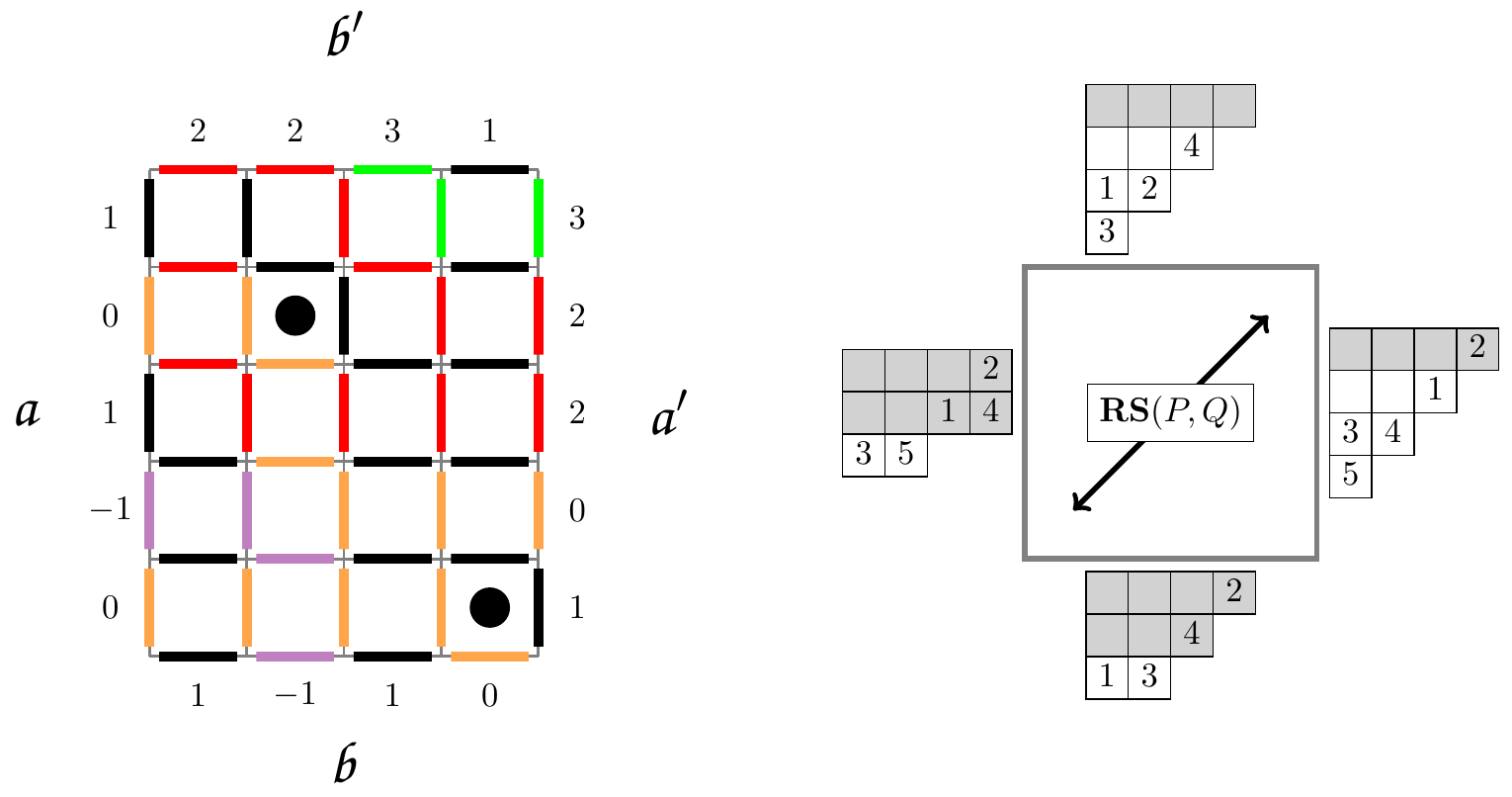}
    \caption{In the left panel we see the graphical representation of the skew $\RS$ map  $(\mathpzc{a}',\mathpzc{b}')=\RS(\mathpzc{a},\mathpzc{b})$ between $\mathpzc{a}=(0,-1,1,0,1) \in \mathbb{Z}^5$, $\mathpzc{b}=(1,-1,1,0) \in \mathbb{Z}^4$. We have colored each edge of the grid based on its value .
    Black bullets denote faces where north and east edges take simultaneously value 1 defining partial permutation $\pi^{(1)}=\left( \begin{smallmatrix} 2 & 4 \\ 4 & 1 \end{smallmatrix} \right)$.
    \\
    \noindent
    In the right panel we reported on the left and bottom sides the tableaux $(P,Q)$ and on the right and top sides tableaux $(P',Q') = \RS(P,Q)$. One can check that $\rc(P,Q)=(\mathpzc{a},\mathpzc{b})$ and $\rc(P',Q')=(\mathpzc{a}',\mathpzc{b}')$. 
    }
    \label{fig:RS_product}
\end{figure}

For an example of edge configurations 
and an evaluation of the skew $\RS$ map of two arrays, corresponding to 
\cref{fig:shadow_lines_example}, 
see \cref{fig:RS_product} left panel.
Note that positions $(j,i)$ of bullets of 
color $k$ are determined by the condition 
$\mathsf{N}(j,i) = \mathsf{W}(j,i) =k$. We will encode the positions $(j_1,i_1),(j_2,i_2),\dots$ of bullets of color $k$ in the partial permutation 
\begin{equation} \label{eq:pi_k}
    \pi^{(k)} = \left( \begin{matrix} j_1 & j_2 & \cdots \\ i_1 & i_2 & \cdots \end{matrix} \right).
\end{equation}

Using edge configurations we also define operators of $\iota_1,\iota_2$ on pairs of arrays.

\begin{definition}[$\iota_1,\iota_2$ on arrays] \label{def:iota_arrays}
    Let $(\mathpzc{a},\mathpzc{b}) \in \mathcal{Z}_n$ be a pair of arrays that are permutations of each other and consider on $\Lambda_{n,n}$ the unique admissible edge configuration with $\mathsf{W}(1,i)=\mathpzc{a}_{\,i},\mathsf{S}(i,1)=\mathpzc{b}_{\,i}$ for all $i=1,\dots,n$. Define  $\iota_2(\mathpzc{a},\mathpzc{b})=(\tilde{\mathpzc{a}}, \tilde{\mathpzc{b}})$ setting
    \begin{equation} \label{eq:iota_2_arrays}
        \tilde{\mathpzc{a}}_{\,i} = \mathsf{W}(2,i) \qquad \text{for } i=1\dots,n
        \qquad 
        \text{and}
        \qquad
        \tilde{\mathpzc{b}}_{\,i} = \begin{cases} \mathsf{S}(i+1,1) \quad &\text{if } i=1,\dots,n-1,
       \\
       \mathsf{N}(1,n) \quad &\text{if } i=n.
       \end{cases}
    \end{equation}
    Analogously define $\iota_1(\mathpzc{a},\mathpzc{b}) = \mathrm{swap}\circ \iota_2 \circ \mathrm{swap} (\mathpzc{a},\mathpzc{b})$.
\end{definition}

Comparing \ref{def:RS} and \ref{def:iota_arrays}, we see 
that $\RS=\iota_1^n=\iota_2^n$
holds as operations on pairs of arrays. 

The next proposition shows that $\iota_1,\iota_2$, both as operations on pairs of standard tableaux and on pairs of arrays, coincide, modulo row coordinate parameterization.

\begin{proposition} \label{prop:iota_rc_commute}
    Let $P,Q \in ST(\lambda/\rho)$ and $(P,Q) \xleftrightarrow[]{\rc \,} (\mathpzc{a},\mathpzc{b};\nu)$. Then $\iota_\epsilon (P,Q) \xleftrightarrow[]{\rc \,} (\iota_\epsilon (\mathpzc{a},\mathpzc{b});\nu)$, for both $\epsilon=1,2$.
\end{proposition}

\begin{proof}
    We prove our statement only for the case $\epsilon=2$, as the remaining case is analogous. The fact that $\nu=\ker(P,Q)$ does not change after the application of $\iota_2$ was shown in \cref{prop:iota_preserves_ker}. Therefore we need to show that if $(\tilde{P},\tilde{Q})=\iota_2(P,Q)$ and $(\tilde{\mathpzc{a}},\tilde{\mathpzc{b}})=\iota_2(\mathpzc{a},\mathpzc{b})$, we have $(\tilde{\mathpzc{a}},\tilde{\mathpzc{b}})=\rc(\tilde{P},\tilde{Q})$. 
    
    Call $r=\mathpzc{b}_{\,1}$, so that $\tilde{P}=\mathcal{R}_{[r]}(P)$. Suppose that during the internal insertion bumping happens $k$ times and cells of $P$ at location $(c_j,r+j)$ move to $(c_{j+1},r+j+1)$ for $j=0,\dots,k-1$. Denoting $p_j=P(c_j,r+j)$, we have $c_0 \ge c_1 \ge \cdots \ge c_k$ and $p_0 < p_1 < \cdots < p_{k-1}$.
    In the row coordinate array $\mathpzc{a}$ this implies that 
    $$\mathpzc{a}_{\, p_j} = r+j$$ 
    and importantly $\mathpzc{a}_{\, p_j + \ell} \neq r+j+1$ for $\ell=1,\dots, p_{j+1} - p_j$. We now draw the edge configuration on $\Lambda_{n,n}$ corresponding to the pair $\mathpzc{a},\mathpzc{b}$ as in \cref{def:iota_arrays} and following the notation of \eqref{eq:iota_2_arrays} we find
    \begin{equation}
        \tilde{\mathpzc{a}}_{\, i} = \begin{cases}
            \mathpzc{a}_{\, i} \qquad & \text{if } i\notin \{p_0,\dots, p_{k-1} \},
            \\
            \mathpzc{a}_{\, i}+1 \qquad & \text{if } i\in \{p_0,\dots, p_{k-1} \},
        \end{cases}
    \end{equation}
    which implies that $\rc(\tilde{P}) = \tilde{\mathpzc{a}}$. 
    
    To show that $\rc(\tilde{Q}) = \tilde{\mathpzc{b}}$ notice that, from the cycling of labels, the $i$-cell of $\tilde{Q}$ lies at row $\mathpzc{b}_{\, i+1} = \tilde{\mathpzc{b}}_{\,i}$ for $i=1,\dots,n-1$. On the other hand the $n$-cell of $\tilde{Q}$ lies at row $r+k$, which, by the computation above is also the value of $\mathsf{N}(1,n)=\tilde{\mathpzc{b}}_{\,n}$. This completes the proof.
\end{proof}

The coincidence extends to the skew $\RS$ maps. 
\begin{corollary} \label{cor:RS_rc}
    Let $P,Q \in ST(\lambda/\rho)$ and $(P,Q) \xleftrightarrow[]{\rc \,} (\mathpzc{a},\mathpzc{b};\nu)$. Then $\RS (P,Q) \xleftrightarrow[]{\rc \,} (\RS (\mathpzc{a},\mathpzc{b});\nu)$.
\end{corollary}

\begin{proof}
    This follows from \cref{prop:iota_rc_commute} and from the fact that $\RS=\iota_\epsilon^n$, with $n=|\lambda / \rho|$ and $\epsilon=1,2$, both as operations on pairs of arrays or on pairs of tableaux.
\end{proof}

\begin{corollary} \label{cor:meaning_matrix_M_k_standard}
    Let $P,Q \in ST(\lambda/\rho)$ and $(P,Q) \xleftrightarrow[]{\rc \,} (\mathpzc{a},\mathpzc{b};\nu)$. Consider the admissible edge configuration on $\Lambda_{n,n}$ corresponding to the pair $\mathpzc{a},\mathpzc{b}$ as in \cref{def:RS}. Fix $i,j\in \{1,\dots,n\}$ and $k\in \mathbb{Z}$. Then, $\mathsf{N}(j,i)=\mathsf{E}(j,i) = k$ if and only if during the evaluation of the skew $\RS$ map $(P,Q) \to (P',Q')$, at the step corresponding to internal insertions of $j$-cells of $Q$, the $i$-cell of the $P$-tableau moves from row $k-1$ to row $k$. 
\end{corollary}

\begin{proof}
    For $j=1$ and $i\in\{1,\dots,n\}$ this follows from the computations reported in the \cref{prop:iota_rc_commute}. Iterating the operations $\iota_2$ yields the general $j$ case.
\end{proof}

We finally report a simple ``restriction property" of 
the shadow line construction. For the next proposition we need the notion of \emph{partial arrays}, which are elements of $(\mathbb{Z} \cup \{\varnothing\})^n$. Given an array $\mathpzc{a}\in\mathbb{Z}^n$ we denote by $\mathpzc{a}^{[\ge k]}$ the partial array with entries $\mathpzc{a}^{[\ge k]}_i=\mathpzc{a}_{\,i}$ if $\mathpzc{a}_{\,i}\ge k$ and $\mathpzc{a}^{[\ge k]}_i=\varnothing$ otherwise. 

\begin{proposition} \label{prop:restriction_standard}
    Let $(\mathpzc{a},\mathpzc{b})\in \mathcal{Z}_n$ and $(\tilde{\mathpzc{a}},\tilde{\mathpzc{b}})=\RS(\mathpzc{a},\mathpzc{b})$. For a fixed $k$ let $\pi^{(k)}$ be the partial permutation defined by \eqref{eq:pi_k}. Then the partial arrays $\tilde{\mathpzc{a}}^{[\ge k]},\tilde{\mathpzc{b}}^{[\ge k]}$ depend uniquely on $\mathpzc{a}^{[\ge k]},\mathpzc{b}^{[\ge k]}$ and $\pi^{(k)}$.
\end{proposition}

\begin{proof}
    This is immediate from the definition of the shadow line construction or equivalently from local rules \eqref{eq:local_rules}. Entries of arrays $\tilde{\mathpzc{a}}^{[\ge k]},\tilde{\mathpzc{b}}^{[\ge k]}$ correspond to shadow lines with colors greater or equal then $k$. Such lines either enter the lattice from west or south and hence are given by $\mathpzc{a}^{[\ge k]},\mathpzc{b}^{[\ge k]}$ or alternatively are created within the grid, in which case are determined by the knowledge of $\pi^{(k)}$.
\end{proof}

\begin{remark}
\label{shrs}
    In case $\mathpzc{a}^{[\ge k]}=\mathpzc{b}^{[\ge k]}=\varnothing$, the partial permutation $\pi^{(k)}$  determines completely $\tilde{\mathpzc{a}}^{[\ge k]},\tilde{\mathpzc{b}}^{[\ge k]}$. Setting $k=1$, this corresponds to the shadow line construction of the classical Robinson-Schensted correspondence \cite{sagan2001symmetric}. 
\end{remark}

\subsection{The skew RSK map of matrices} \label{subs:RSK_prod_matrices}
We extend the operation of skew $\RS$ map to include infinite matrices of integers. 
Thanks to the correspondence between matrices and tableaux of \cref{prop:row_coordinate_classical_pair}, this provides a diagrammatic realization of the skew $\RSK$ map for tableaux. By the standardization most of their properties follow directly from the ones of the skew $\RS$ map. 

We can generalize the shadow line construction of \cref{def:shadow_line} by allowing edges of the lattice $\Lambda_{m,n}$ to be crossed by arbitrary many lines. 
We impose that on each cell for any two lines crossing the same horizontal (resp. vertical) edge the one with higher color stays at the left (resp. below) of the one with lower color; see \cref{fig:shadow_lines_fused}. 
This implies also that for each color bullets at a cell lie strictly to the right of those at cells $c-k \mathbf{e}_2$ and strictly above those at cells $c-k \mathbf{e}_1$ for all $k\in \mathbb{N}_0$. 
We record the list of colored lines crossing a specific edge in an array $v\in\mathbb{V}$ and for any $C\in \mathbb{Z}$ the value $v_C$ will count the number of $C$-colored lines; see \cref{fig:RSK_local_rule}. Such discussion justifies the following definition.

\begin{figure}
    \centering
    \includegraphics[scale=.9]{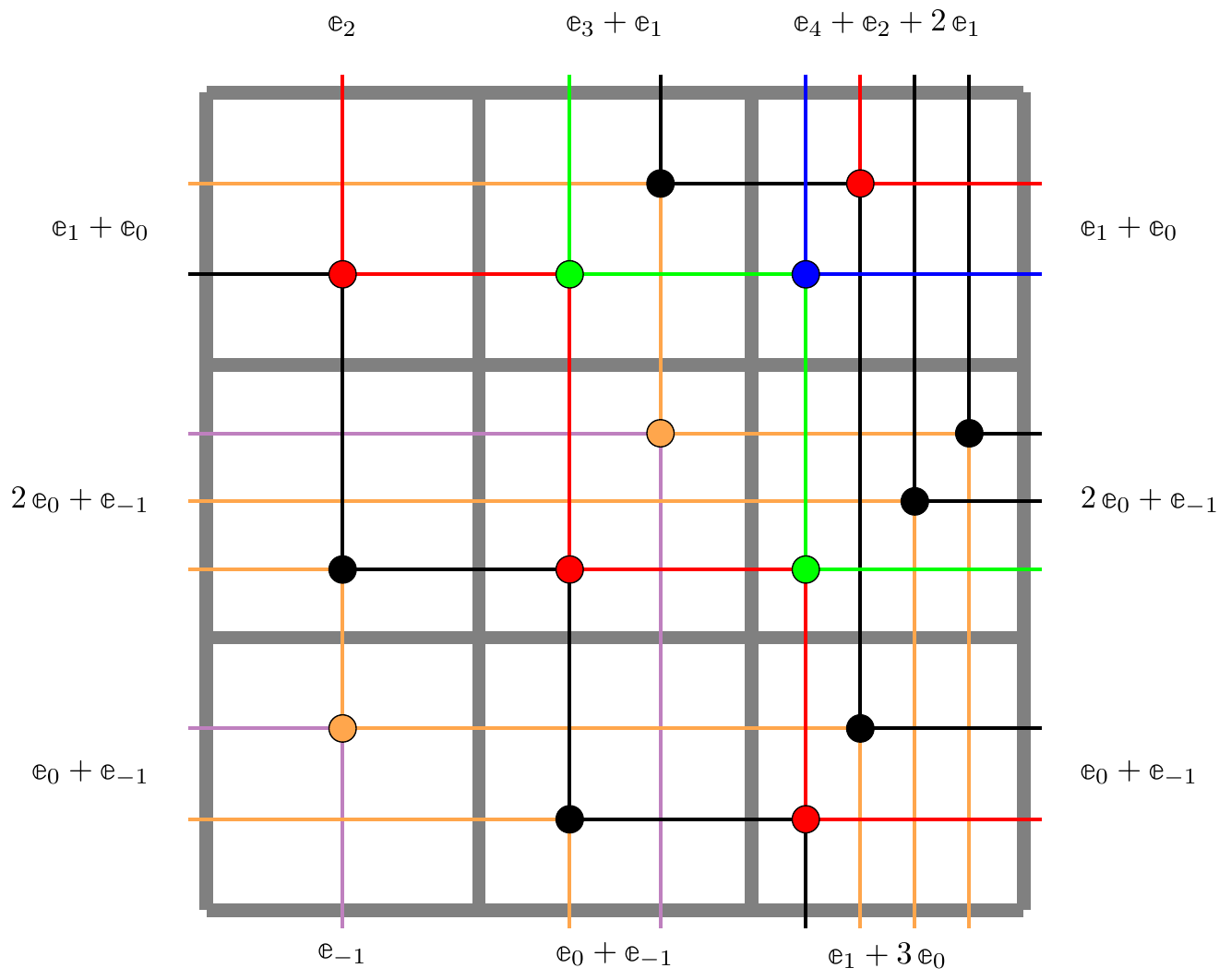}
    \caption{Generalized shadow line construction on the lattice $\Lambda_{3,3}$.}
    \label{fig:shadow_lines_fused}
\end{figure}


\begin{figure}[ht]
    \centering
    \subfloat[]{{\includegraphics[width=.55\linewidth]{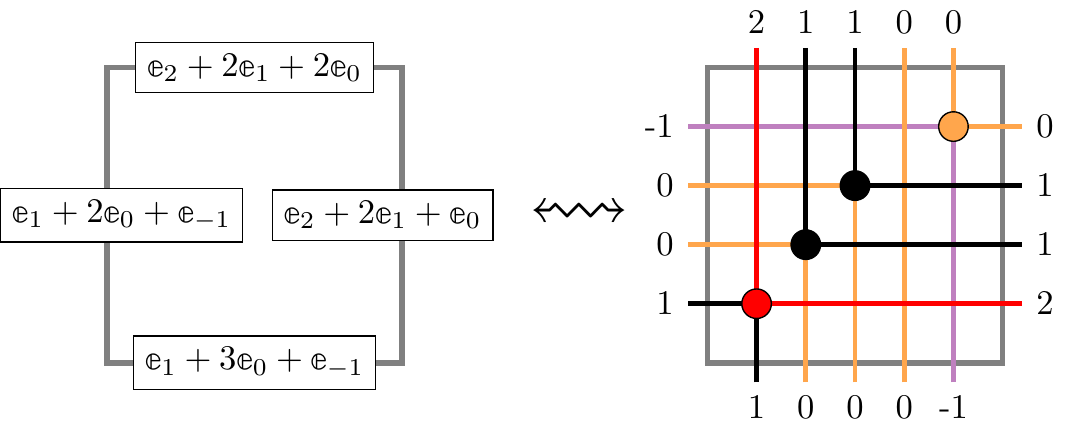}} }%
    \hspace{1cm}
    \subfloat[]{{\includegraphics[width=.35\linewidth]{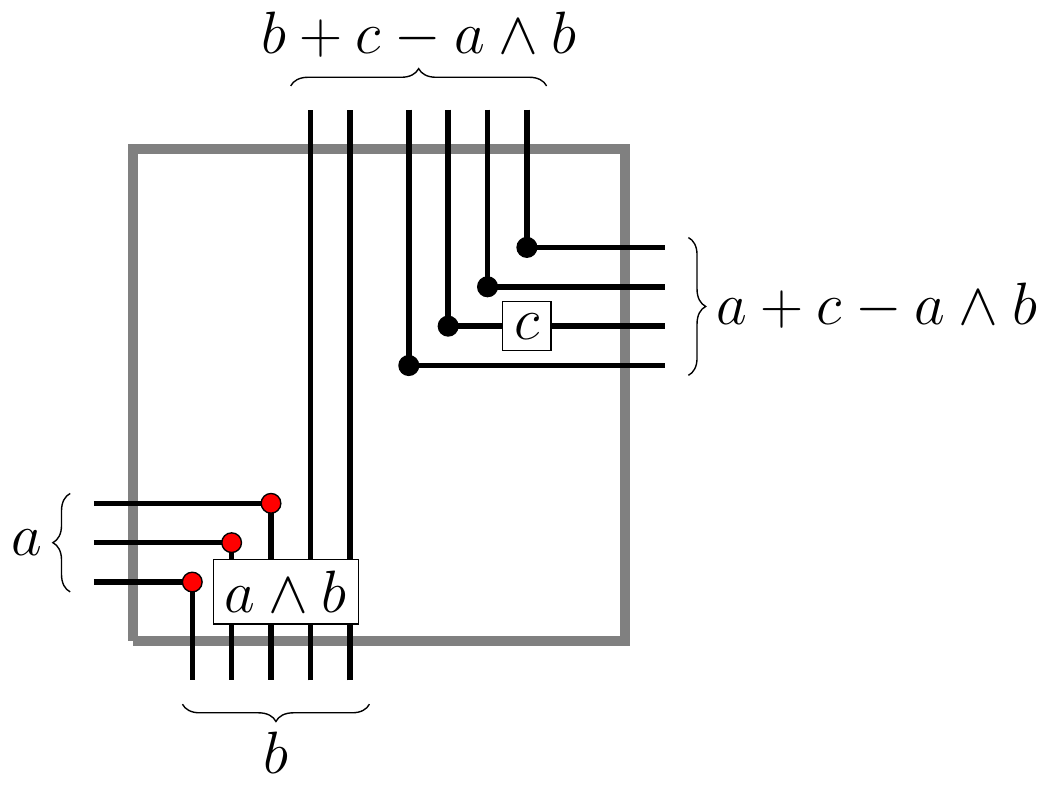} }}%
    \caption{On the left panel we see the equivalence between $\mathbb{V}$-valued edge configurations configuration of colored lines through a face. On the right panel a graphical interpretation of local rules \eqref{eq:RSK_local_rules}}
    \label{fig:RSK_local_rule}
\end{figure}


\begin{definition}[$\mathbb{V}$-valued edge configurations] \label{def:V_edge_conf}
    For a planar lattice $\Lambda\subseteq \mathbb{Z}\times \mathbb{Z}$ define $\mathbb{V}$-valued edge configurations $\mathcal{E}$ as quadruples of functions $(\mathsf{W},\mathsf{S},\mathsf{E},\mathsf{N}):\Lambda \to \mathbb{V}$ such that $\mathsf{E}(c)=\mathsf{W}(c+\mathbf{e}_1)$ and $\mathsf{N}(c)=\mathsf{S}(c+\mathbf{e}_2)$ for all $c\in \Lambda$. An edge configuration is \emph{admissible} if it satisfies the \emph{local rules}
    \begin{equation} \label{eq:RSK_local_rules}
    \begin{minipage}{.9\linewidth}
        \begin{enumerate}
            \item $\mathsf{E}_j(c) = \mathsf{W}_j(c) - \mathsf{S}_j(c) \wedge \mathsf{W}_j(c) + \mathsf{S}_{j-1}(c) \wedge \mathsf{W}_{j-1}(c)$,
            \item $\mathsf{N}_j(c) = \mathsf{S}_j(c) - \mathsf{S}_j(c) \wedge \mathsf{W}_j(c) + \mathsf{S}_{j-1}(c) \wedge \mathsf{W}_{j-1}(c)$,
        \end{enumerate}
    \end{minipage}
\end{equation}
    for all $c\in \Lambda$, $j\in \mathbb{Z}$. Here $a\wedge b=\min(a,b)$.
\end{definition}

Local rules \eqref{eq:RSK_local_rules} describe the arrangement of lines of the each color $j$ around single faces of the lattice. Namely $\mathsf{W}_j(c)$ is the number of $j$-colored lines entering face $c$ from the left and similarly for $\mathsf{S}_j, \mathsf{E}_j, \mathsf{N}_j$. We report this statement in the next proposition and we refer to \cref{fig:RSK_local_rule} for the graphical interpretation of such rules. 

\begin{proposition}  \label{prop:equivalence_local_rules_RS_RSK}
    On the lattice $\Lambda_{1,1}$, consisting of a single face, consider the admissible $\mathbb{V}$-valued edge configuration $\mathsf{W}=a, \mathsf{S}=b, \mathsf{E}=a', \mathsf{N}=b'$. Consider arrays $\mathpzc{a},\mathpzc{a}'\in \mathbb{W}^{|a|}$,$\mathpzc{b},\mathpzc{b}'\in \mathbb{W}^{|b|}$ such that $v(\mathpzc{a})=a,v(\mathpzc{b})=b,v(\mathpzc{a}')=a',v(\mathpzc{b}')=b'$ under the map $v:\mathbb{W}^k \to \mathbb{V}$ defined in \eqref{eq:from_v_to_a}. Then $(\mathpzc{a}',\mathpzc{b}' ) = \RS (\mathpzc{a} , \mathpzc{b})$.
\end{proposition}

\begin{proof}
    This is straightforward after a comparison of local rules \eqref{eq:RSK_local_rules} with the skew $\RS$ map of weakly decreasing arrays $\mathpzc{a},\mathpzc{b}$ corresponding to the generic edge values $a,b$.
\end{proof}

In line with \cref{subs:standardization}, we can define the standardization of an admissible $\mathbb{V}$-valued edge configuration $\mathcal{E}$ on a lattice $\Lambda$. This will be the admissible $\mathbb{Z}$-valued edge configuration $\mathcal{E}'=\std(\mathcal{E})$ on the lattice $\Lambda'$ obtained from $\mathcal{E}$ ``blowing up" faces of $\Lambda$ as prescribed by \cref{prop:equivalence_local_rules_RS_RSK}.

\begin{definition}[Skew $\RSK$ map of matrices] \label{def:RSK_corr}
    Let $\alpha \in \mathbb{M}_{n\times \infty},\beta \in \mathbb{M}_{m\times \infty}$ and consider the unique $\mathbb{V}$-valued admissible edge configuration on $\Lambda_{m,n}$ such that $\mathsf{W}_k(1,i)=\alpha_{i,k}$, $\mathsf{S}_k(j,1)=\beta_{j,k}$, for all $i=1,\dots,n$, $j=1,\dots,m$, $k\in \mathbb{Z}$.
    We define the skew $\RSK$ map 
    \begin{equation*}
        \RSK(\alpha,\beta)=(\alpha',\beta') \in \mathbb{M}_{n\times \infty} \times \mathbb{M}_{m \times \infty},    
    \end{equation*} 
    as the pair of matrices $\alpha_{i,k}'=\mathsf{E}_k(m,i)$, $\beta_{j,k}'=\mathsf{N}_k(j,n)$, for all $i=1,\dots,n$, $j=1,\dots,m$, $k\in \mathbb{Z}$. From the configuration define also the family of matrices $M^{(k)}$ as
    \begin{equation} \label{eq:num_bullets}
        M^{(k)}(j,i) = \mathsf{N}_k(j,i) \wedge \mathsf{E}_k(j,i).
    \end{equation}
\end{definition}


\begin{example} \label{example:RSK_product}
Define matrices
\begin{equation} \label{eq:alpha_beta_example}
    \alpha = 
    \,
    \begin{blockarray}{ccccc}
    \textcolor{gray}{\cdots} & \textcolor{gray}{-1} & \textcolor{gray}{0} & \textcolor{gray}{1} & \textcolor{gray}{\cdots} \\
    \begin{block}{(ccccc)}
        \cdots & 1 & 1 & 0 & \cdots \\
        \cdots & 0 & 2 & 1 & \cdots \\
        \cdots & 0 & 1 & 1 & \cdots \\
    \end{block}
    \end{blockarray}
    \,\,\,
    ,
    \,\,\,\,
    \beta = 
    \,
    \begin{blockarray}{ccccc}
    \textcolor{gray}{\cdots} & \textcolor{gray}{-1} & \textcolor{gray}{0} & \textcolor{gray}{1} & \textcolor{gray}{\cdots} \\
    \begin{block}{(ccccc)}
        \cdots & 1 & 0 & 0 & \cdots \\
        \cdots & 1 & 1 & 1 & \cdots \\
        \cdots & 0 & 3 & 1 & \cdots \\
    \end{block}
    \end{blockarray}
    \,\,\,.
\end{equation}

We evaluate $(\alpha',\beta') = \RSK(\alpha, \beta)$ computing the corresponding $\mathbb{V}$-valued edge configuration on $\Lambda_{3,3}$. The result is reported in \cref{fig:shadow_lines_fused} and we have
\begin{equation} \label{eq:alpha_prime_beta_prime_example}
    \alpha' = 
    \,
    \begin{blockarray}{cccccc}
    \textcolor{gray}{\cdots} & \textcolor{gray}{1} & \textcolor{gray}{2} & \textcolor{gray}{3} & \textcolor{gray}{4} & \textcolor{gray}{\cdots} \\
    \begin{block}{(cccccc)}
        \cdots & 1 & 1 & 0 & 0 & \cdots \\
        \cdots & 2 & 0 & 1 & 0 & \cdots \\
        \cdots & 0 & 1 & 0 & 1 & \cdots \\
    \end{block}
    \end{blockarray}
    \,\,\,
    ,
    \,\,\,\,
    \beta' = 
    \,
    \begin{blockarray}{cccccc}
    \textcolor{gray}{\cdots} & \textcolor{gray}{1} & \textcolor{gray}{2} & \textcolor{gray}{3} & \textcolor{gray}{4} & \textcolor{gray}{\cdots} \\
    \begin{block}{(cccccc)}
        \cdots & 0 & 1 & 0 & 0 & \cdots \\
        \cdots & 1 & 0 & 1 & 0 & \cdots \\
        \cdots & 2 & 1 & 0 & 1 & \cdots \\
    \end{block}
    \end{blockarray}
    \,\,\,.
\end{equation}
To configuration of \cref{fig:shadow_lines_fused} we associate matrices $M^{(k)}$ as described by \eqref{eq:num_bullets}. For instance from the same figure one can check that $M^{(1)}=\left( \begin{smallmatrix} 0 & 1 & 0 \\ 1 & 0 & 2 \\ 0 & 1 & 1 \end{smallmatrix} \right)$ or $M^{(2)}=\left( \begin{smallmatrix} 1 & 0 & 1 \\ 0 & 1 & 0 \\ 0 & 0 & 1 \end{smallmatrix} \right)$.
\end{example}

\begin{definition} \label{def:iota_matrices}
    For a pair of matrices $(\alpha,\beta) \in \mathcal{M}_n$ construct on $\Lambda_{n,n}$ the corresponding $\mathbb{V}$-valued edge configuration as in \cref{def:RSK_corr}. We define $\iota_2(\alpha,\beta) = (\tilde{\alpha}, \tilde{\beta})$, setting for all $k\in\mathbb{Z}$
    \begin{equation} \label{eq:iota_2_matrices}
        \tilde{\alpha}_{i,k} = \mathsf{W}_k(2,i) \qquad \text{for } i=1\dots,n
        \qquad 
        \text{and}
        \qquad
        \tilde{\beta}_{i,k} = \begin{cases} \mathsf{S}_k(i+1,1) \quad &\text{if } i=1,\dots,n-1,
       \\
       \mathsf{N}_k(1,n) \quad &\text{if } i=n.
       \end{cases}
    \end{equation}
    Operator $\iota_1$ is defined by duality $\iota_1(\alpha,\beta)=\mathrm{swap} \circ \iota_2 \circ \mathrm{swap} (\alpha,\beta)$.
\end{definition}

Notice that, in the previous definition, also the pair $(\tilde{\alpha},\tilde{\beta})$ belongs to the set $\mathcal{M}_n$.

\begin{proposition} \label{prop:iota_rc_commute_semi}
    Let $P,Q \in SST(\lambda/\rho,n)$ and $(P,Q) \xleftrightarrow[]{\rc \,} (\alpha,\beta;\nu)$. Then $\iota_\epsilon (P,Q) \xleftrightarrow[]{\rc \,} (\iota_\epsilon (\alpha,\beta);\nu)$, for both $\epsilon=1,2$ and $\RSK (P,Q) \xleftrightarrow[]{\rc \,} (\RSK (\alpha,\beta);\nu)$.
\end{proposition}

\begin{proof}
    This is consequence of the analogous statement for standard tableaux and arrays stated in \cref{prop:iota_rc_commute} and of \cref{prop:equivalence_local_rules_RS_RSK}.
\end{proof}

Geometric interpretation of operators $\iota_1,\iota_2$ provided by \cref{prop:iota_rc_commute_semi} yields a visual proof of the \emph{nontrivial} fact that they commute with each other.

\begin{proposition} \label{prop:iota_commute}
    We have $\iota_1 \circ \iota_2 = \iota_2 \circ \iota_1$, both as operations on pairs of tableaux, or on pairs of matrices.
\end{proposition}

\begin{proof}
    We first prove that $\iota_1$ and $\iota_2$ commute when they act on pairs of matrices. Consider the admissible $\mathbb{V}$-valued edge configuration on the lattice $\Lambda_{n+1,n+1}$ such that for all $k\in \mathbb{Z}$
    \begin{align*}
        \mathsf{W}_k(1,i) = \alpha_{i,k},
        \qquad
        \mathsf{S}_k(1,i) = \beta_{i,k},
        \qquad
        \text{for } i=1,\dots,n
        \\
        \text{and}
        \qquad
        \mathsf{W}_k(1,n+1) = \mathsf{E}_k(n,1)
        \qquad
        \mathsf{S}_k(n+1,1) = \mathsf{N}_k(1,n).
    \end{align*}
    In this lattice $\iota_1$ and $\iota_2$ act respectively as upward and rightward shift, so that defining $\tilde{\alpha}$, $\tilde{\beta}$ as
    \begin{equation}
        \tilde{\alpha}_{i,k} = \mathsf{W}_k(2,i+1),
        \qquad
        \tilde{\beta}_{i,k} = \mathsf{N}_k(i+1,2),
        \qquad
        \text{for } i=1,\dots,n \text{ and } k\in \mathbb{Z}
    \end{equation}
    we easily see that $(\tilde{\alpha} , \tilde{\beta}) = \iota_1 \circ \iota_2 (\alpha,\beta) = \iota_2 \circ \iota_1 (\alpha,\beta)$. This proves the proposition for the case of action on pairs of matrices. By \cref{prop:iota_rc_commute_semi} the same commutation holds when $\iota_1,\iota_2$ act on pair of semi-standard tableaux.
\end{proof}

\begin{proposition} \label{prop:meaning_M_k}
    Let $P,Q \in SST(\lambda/\rho,n)$ and $(P,Q) \xleftrightarrow[]{\rc \,} (\alpha,\beta;\nu)$. Consider the admissible $\mathbb{V}$-valued edge configuration on $\Lambda_{n,n}$ corresponding to the pair $\alpha,\beta$ as in \cref{def:RSK_corr}. Fix $i,j\in \{1,\dots,n\}$ and $k\in \mathbb{Z}$. Then, $\mathsf{N}_k(j,i) \wedge \mathsf{W}_k(j,i) = m$ if and only if during the evaluation of the skew $\RSK$ map $(P,Q) \to (P',Q')$, at the step corresponding to internal insertions of $j$-cells of $Q$, exactly $m$ $i$-cells of the $P$-tableau moves from row $k-1$ to row $k$. 
\end{proposition}
\begin{proof}
    The analogous property for standard tableaux was given in \cref{cor:meaning_matrix_M_k_standard}. Combining this with \cref{prop:equivalence_local_rules_RS_RSK} yields the statement for pairs of semi-standard tableaux $P,Q$.
\end{proof}

The next proposition gives a restriction property analogous to \cref{prop:restriction_standard}. For any matrix $\alpha\in\mathbb{M}_{n,\infty}$ we define its truncation $\delta_{\ge k}(\alpha) = ( \mathbf{1}_{j\ge k} \alpha_{i,j} :i=1,\dots,n,j\in_{\mathbb{Z}})$

\begin{proposition} \label{prop:restriction_matrices}
    Let $(\alpha,\beta)\in \mathcal{M}_n$ and $(\tilde{\alpha},\tilde{\beta})=\RSK(\alpha,\beta)$. For a fixed $k$ let $M^{(k)}$ be the partial permutation defined by \eqref{eq:num_bullets}. Then the truncated matrices $\delta_{\ge k}(\tilde{\alpha}), \delta_{\ge k}(\tilde{\beta})$ depend uniquely on $\delta_{\ge k}(\alpha), \delta_{\ge k}(\beta)$ and $M^{(k)}$. Moreover for all $k$ we have $(\delta_{\ge k}(\tilde{\alpha}), \delta_{\ge k}(\tilde{\beta}) ) \in \mathcal{M}_n$.
\end{proposition}

\begin{proof}
    This is again consequence of \cref{prop:restriction_standard} and \cref{prop:equivalence_local_rules_RS_RSK}. 
\end{proof}

\section{Skew $\RSK$ and Viennot dynamics} \label{sec:iterated_RSK}

In this section we first introduce the skew $\RSK$ dynamics for a pair of tableaux, by iterations of the skew $\RSK$ maps studied in the previous section. A related dynamics, this time on the set of matrices $\overline{\mathbb{M}}_{n \times n}$, which we call the Viennot dynamics, is also defined and it will play an important role when describing conservation laws in \cref{sec:Greene_invariants}.

\subsection{The skew $\RSK$ dynamics} The following definition was sketched in the introductory chapter.

\begin{definition}[Skew $\RSK$ dynamics]
    We define a deterministic dynamics on the space of pairs of generalized tableaux by iterating the skew $\RSK$ map. Fix $P,Q\in SST(\lambda / \rho,n)$ and define the \emph{skew $\RSK$ dynamics} with initial data $(P,Q)$ as
    \begin{equation} \label{eq:RSK_ca}
        \begin{split}
            \begin{cases}
                (P_{t+1},Q_{t+1})=\RSK(P_t,Q_t), 
                \qquad t \in \mathbb{Z},
                \\
                (P_1,Q_1)=(P,Q).
            \end{cases}
        \end{split}
    \end{equation}
    Analogously, we define the skew $\RSK$ dynamics on the space of pairs of matrices. For a fixed initial state $(\alpha,\beta)\in\mathcal{M}_n$ we have
    \begin{equation}
        \begin{split}
            \begin{cases}
                (\alpha^{(t+1)},\beta^{(t+1)})=\RSK(\alpha^{(t)},\beta^{(t)}), 
                \qquad t \in \mathbb{Z},
                \\
                (\alpha^{(1)},\beta^{(1)})=(\alpha,\beta).
            \end{cases}
        \end{split}
    \end{equation}
\end{definition}

Since the skew $\RSK$ map is invertible, we see that each realization of the dynamics is uniquely characterized by its initial state. This observation is very powerful as it implies that specific properties of tableaux $(P,Q)$ can be deduced observing their state at an arbitrary time of the skew $\RSK$ dynamics. As already pointed out in the \cref{subs:examples}, dynamics \eqref{eq:RSK_ca} presents conservation laws, that we will characterize in \cref{sec:Greene_invariants} and \cref{sec:scattering_rules} below. Moreover, in \cref{sec:linearization} we will devise a linearization technique of the dynamics, which one can regard as a combinatorial variant of the inverse scattering method for classical integrable systems. The study of asymptotic states of dynamics will be exceptionally revealing.

\subsection{Edge configurations on the twisted cylinder} \label{subs:edge_config_on_cyl}

In the \cref{sec:miscellaneous} we have seen the equivalence between the two versions of the skew $\RSK$ maps on tableaux and on matrices through edge local rules on a finite rectangular lattice. Here we introduce an infinite lattice with certain periodicity and edge configurations compatible with it. This will lead us to define a dynamics on the same lattice, which is closely related to the skew $\RSK$ dynamics. 

\begin{definition}[Twisted cylinder] \label{def:twisted_cylinder}
     The \emph{twisted cylinder} is the periodic lattice $\mathscr{C}_n= \mathbb{Z}^2 / \sim_n$, where $(j,i)\sim_n(j', i')$ if $(j',i') = (j+kn,j-kn)$ for some $k\in\mathbb{Z}$. Natural representations of $\mathscr{C}_n$ we will use are (see \cref{fig:twisted_cylinder}):
     \begin{equation} \label{eq:twisted_cyl_vertical}
         \begin{minipage}{.9\linewidth}
             the infinite vertical strip $\{ 1, \dots , n\} \times \mathbb{Z}$, where we impose faces $(n,i)$ and $(1,i+n)$ to be adjacent for all $j$.
         \end{minipage}
     \end{equation}
     \begin{equation}\label{eq:twisted_cyl_horizontal}
         \begin{minipage}{.9\linewidth}
             the infinite horizontal strip $\mathbb{Z} \times \{ 1, \dots , n\}$ where we impose faces $(j,n)$ and $(j+n,1)$ to be adjacent for all $j$. 
         \end{minipage}
     \end{equation}
\end{definition}

\begin{figure}[t]
    \centering
    \includegraphics[scale=.8]{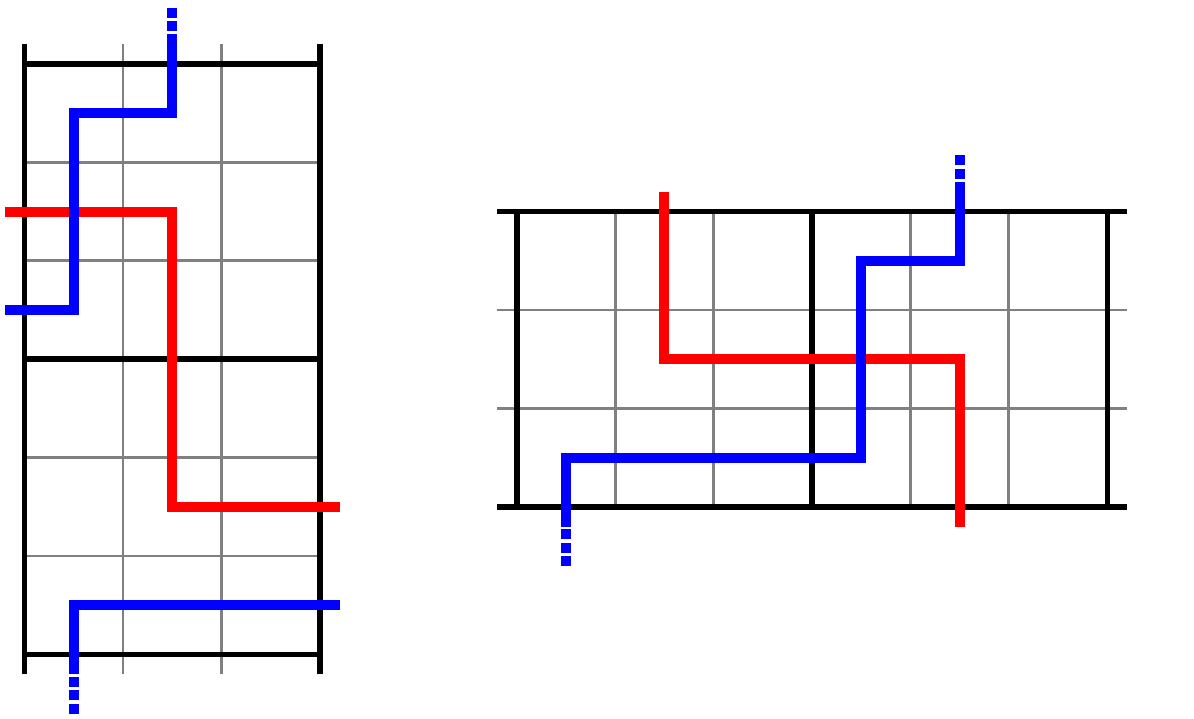}
    \caption{Two graphical representation of the twisted cylinder $\mathscr{C}_3$. Blue line represents an up-right path, while red line a down-right loop.}
    \label{fig:twisted_cylinder}
\end{figure}

A \emph{down-right loop} $\xi$ on $\mathscr{C}_n$ is a sequence $(\xi_k: k\in \{1,\dots,2n\}) \subset \mathscr{C}_n$, such that 
\begin{equation}
    \xi_{k+1} \sim_n \xi_{k} + \mathbf{e}_1,
    \qquad
    \text{or}
    \qquad
    \xi_{k+1} \sim_n \xi_{k} - \mathbf{e}_2,
\end{equation}
for $k=1,\dots, 2n$, where indices are taken $\mod 2n$. An example of a down-right loop is the red path drawn in \cref{fig:twisted_cylinder}, having the form
$$
c \to c+\mathbf{e}_1 \to \cdots \to c+n\mathbf{e}_1 \to c+n\mathbf{e}_1-\mathbf{e}_2 \to \cdots \to c+n\mathbf{e}_1-(n-1)\mathbf{e}_2 \to c.
$$
for some $c\in \mathscr{C}_n$. Assigning edge values along a down right loop $\xi$ automatically determines the full configuration $\mathcal{E}$ on $\mathscr{C}_n$ as a result of local rules \eqref{eq:RSK_local_rules}. We use this to visualize the skew $\RSK$ dynamics on the set of matrices as an edge configuration on $\mathscr{C}_n$. To any pair
$(\alpha, \beta) \in \mathcal{M}_n$ we associate the edge configuration $\mathcal{E}$ on $\mathscr{C}_n$ identified by the assignment
    \begin{equation} \label{eq:map_matrices_config}
        (\alpha , \beta) \mapsto \mathcal{E}
        \qquad
        \colon
        \qquad
        \mathsf{E}_k(n, i) = \alpha_{i,k},
        \qquad
        \mathsf{N}_k (i,n) = \beta_{i,k},
        \qquad
        \text{for } i=1,\dots,n, \,\,k\in \mathbb{Z}. 
    \end{equation}
The subclass of admissible edge configurations accessible through mapping \eqref{eq:map_matrices_config} is defined next.

\begin{definition} \label{def:complete_edge_conf}
    Let $\mathcal{E}$ be a $\mathbb{V}$-valued admissible edge configurations on $\mathscr{C}_n$ and for any fixed $c\in \mathscr{C}_n$ define the pair $(\alpha, \beta)_c$ as $\alpha_{i,k} = \mathsf{W}_k(c+ (i-1) \mathbf{e}_2)$ and $\beta_{i,k} = \mathsf{S}_k(c+ (i-1) \mathbf{e}_1)$ for $i=1,\dots,n$, $k\in \mathbb{Z}$. We define the set $\mathfrak{E}_n$ consisting of all configurations $\mathcal{E}$ such that $(\alpha,\beta)_c \in \mathcal{M}_n$ for all $c\in \mathscr{C}_n$.
\end{definition}

\begin{proposition} \label{prop:local_completeness}
    The sets $\mathcal{M}_n$ and $\mathfrak{E}_n$ are in bijection.
\end{proposition}
\begin{proof}
    We only need to show that configuration $\mathcal{E}$ defined by \eqref{eq:map_matrices_config} belongs to $\mathfrak{E}_n$. Notice first that, in the notation of \cref{def:complete_edge_conf} we have $(\alpha,\beta) = (\alpha,\beta)_{(1,n+1)}$.
    For fixed $N_1,N_2 \in \mathbb{Z}$, let $(\widetilde{\alpha},\widetilde{\beta})=\iota_1^{N_1} \circ \iota_2^{N_2} (\alpha,\beta)$. Then, by \cref{def:iota_matrices} and taking into account periodicity of $\mathscr{C}_n$ we have $(\tilde{\alpha},\tilde{\beta})=(\alpha,\beta)_{(1+N_2, n+N_1+1)}$. Since $(\tilde{\alpha},\tilde{\beta}) \in \mathcal{M}_n$ for all $N_1,N_2$, by \cref{prop:iota_rc_commute_semi}, then $\mathcal{E} \in \mathfrak{E}_n$.
\end{proof}

\begin{corollary}
    Let $(\alpha^{(t)},\beta^{(t)})$ the skew $\RSK$ dynamics with initial data $(\alpha^{(1)},\beta^{(1)}) =(\alpha,\beta) \in \mathcal{M}_n$ and let $\mathcal{E}$ be the configuration associated to $(\alpha,\beta)$ by \eqref{eq:map_matrices_config}. Then, for all $i=1,\dots,n$ and $t \in \mathbb{Z}$ we have $(\alpha^{(t)},\beta^{(t)}) = (\alpha,\beta)_{(1,tn+1)}$;
    see \cref{fig:twisted_cyl_matrices}.
\end{corollary}

\begin{figure}[t]
    \centering
    \includegraphics{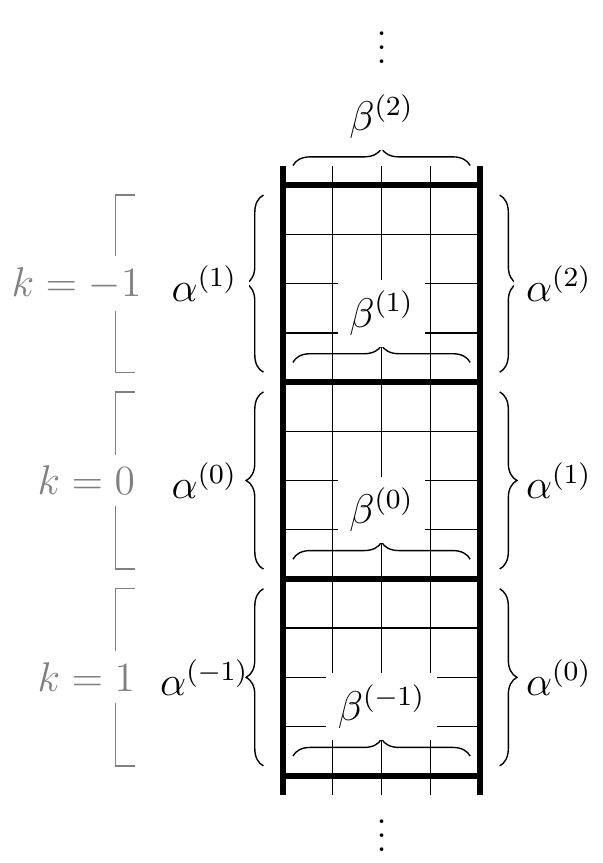}
    \caption{A visualization of the skew $\RSK$ dynamics of matrices $(\alpha^{(t)},\beta^{(t)})$ as edges of a configuration $\mathcal{E}$. Faces of $\mathscr{C}_n$ have coordinates $(j,i-kn)$.}
    \label{fig:twisted_cyl_matrices}
\end{figure}

\subsection{Periodic shadow line construction and Viennot dynamics} \label{subs:periodic_shadow_line}
Edge configurations $\mathcal{E} \in \mathfrak{E}_n$, in analogy with the finite case, identify families of compactly supported maps $\overline{M}^{(t)}:\mathscr{C}_n \to \mathbb{N}_0$ assigning
\begin{equation}\label{eq:from_matrix_to_conf}
    \overline{M}^{(t)}(c) = \mathsf{N}_t(c) \wedge \mathsf{E}_t(c).
\end{equation}
On the other hand, as proven in \cref{thm:matrices_configurations} below and for any map $\overline{M} \in \overline{\mathbb{M}}_{n\times n}$ we can construct the family $\overline{M}^{(t)}$, with the convention that $\overline{M}^{(1)}=\overline{M}$ and the configuration $\mathcal{E}$. Such procedure constitutes a periodic variant of Viennot shadow line construction \cite{Viennot_une_forme_geometrique} or of Fulton's matrix ball construction \cite{fulton1997young}.

\begin{definition}[Viennot map] \label{def:Viennot_map}
    
    Let $\overline{M}\in \overline{\mathbb{M}}_{n\times n}$. 
    At each face $c \in \mathscr{C}_n$ allocate $\overline{M}(c)$ black bullets and apply the shadow line construction explained at the beginning of section \ref{subs:RSK_prod_matrices}, letting each bullet emanate two black rays in the north and east direction.
    By periodicity of $\mathscr{C}_n$ and the fact that there are only a finite number of black bullets, each ray terminates somewhere intersecting with another. The collection of such mutual intersections of rays determines a new generation of red bullets, see \cref{fig:Viennot_map} right panel, and we define
    \begin{equation}
        \overline{M}'(c) = \#  \text{ of new generation bullets at cell $c$ }.
    \end{equation}
    The map $\mathbf{V} :\overline{M} \mapsto \overline{M}'$ takes the name of \emph{Viennot map}.
    Using correspondence \eqref{eq:matrix_weighted_biword} we define the action of Viennot map also on weighted biwords $\mathbf{V}(\overline{\pi}) = \overline{\pi}'$, imposing $\overline{M}(\overline{\pi}')= \mathbf{V}(\overline{M}(\overline{\pi}))$.
\end{definition}

An example of evaluation of map $\mathbf{V}$ is reported in \cref{fig:Viennot_map}. There, we see the transition
\begin{equation}\label{eq:Viennot_map_example_2}
    \overline{\pi} = 
    \left(
    \begin{matrix}
        1 & 1 & 2 & 2 & 2 & 2 & 3
        \\
        1 & 2 & 1 & 1 & 1 & 1 & 3
        \\
        0 & 1 & 1 & 1 & 0 & 0 & 1
    \end{matrix}
    \right)
    \xrightarrow[]{\hspace{.6cm}}
    \V(\overline{\pi}) = 
    \left(
        \begin{matrix}
        1 & 1 & 2 & 2 & 2 & 2 & 3
        \\
        1 & 3 & 1 & 1 & 1 & 2 & 1
        \\
        0 & 0 & 1 & -1 & -1 & 0 & 0
    \end{matrix}
    \right).
\end{equation}
The shadow lines
produced by the computation of $\overline{\pi} \mapsto \mathbf{V}(\overline{\pi})$ consist in a sequence of connected broken lines as we see in \cref{fig:Viennot_map}. Recording the cells of $\mathscr{C}_n$ visited by each of these broken lines we naturally define a sequence of down right loops $\xi^{(1)},\xi^{(2)},\dots$. We will use these loops later in \cref{sec:Greene_invariants}.For instance in \cref{fig:Viennot_map} (b) the two topmost loops have the same form
\begin{equation}
    (1,4) \to (2,4) \to (2,3) \to (2,2) \to (2,1) \to (3,1) \to (1,4).
\end{equation}

Map $\mathbf{V}$ can be thought as a generator of a deterministic dynamics on the space $\overline{\mathbb{M}}_{n\times n}$.

\begin{definition}[Viennot dynamics]
  \label{def:Vdyn}
    Fix $\overline{M} \in \overline{\mathbb{M}}_{n\times n}$ and define the \emph{Viennot dynamics} with initial data $\overline{M}$ as
    \begin{equation} \label{eq:viennot_ca}
        \begin{cases}
            \overline{M}^{(t+1)} = \mathbf{V}(\overline{M}^{(t)}),\qquad t\in \mathbb{Z}
            \\
            \overline{M}^{(1)} = \overline{M}.
        \end{cases}
    \end{equation}
    Analogously, for $\overline{\pi}\in \overline{\mathbb{A}}_{n,n}$, one defines the Viennot dynamics $\overline{\pi}^{(t+1)}=\mathbf{V}(\overline{\pi}^{(t)})$ with initial data $\overline{\pi}^{(1)}=\overline{\pi}$.
\end{definition}

\begin{remark}
    This may be considered a generalization of shadow line construction on $\Lambda_{n,n}$ of classical RS correspondence, see Remark 
    \ref{shrs}. 
\end{remark}

\begin{figure}[ht]
        \centering
        \subfloat[]{{\includegraphics[scale=1.4]{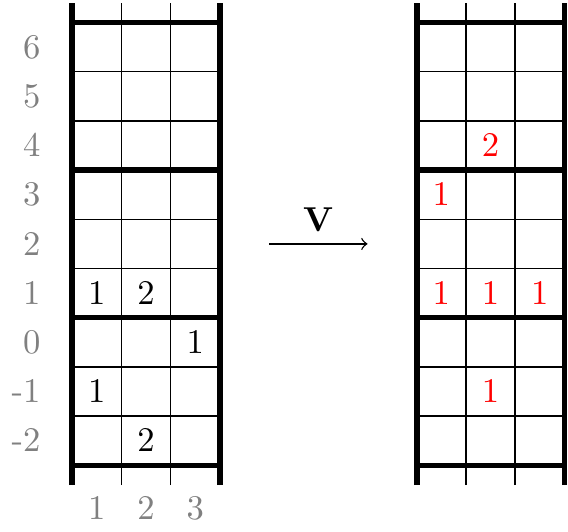} }}%
        \hspace{2cm}
        \subfloat[]{{\includegraphics[scale=0.7]{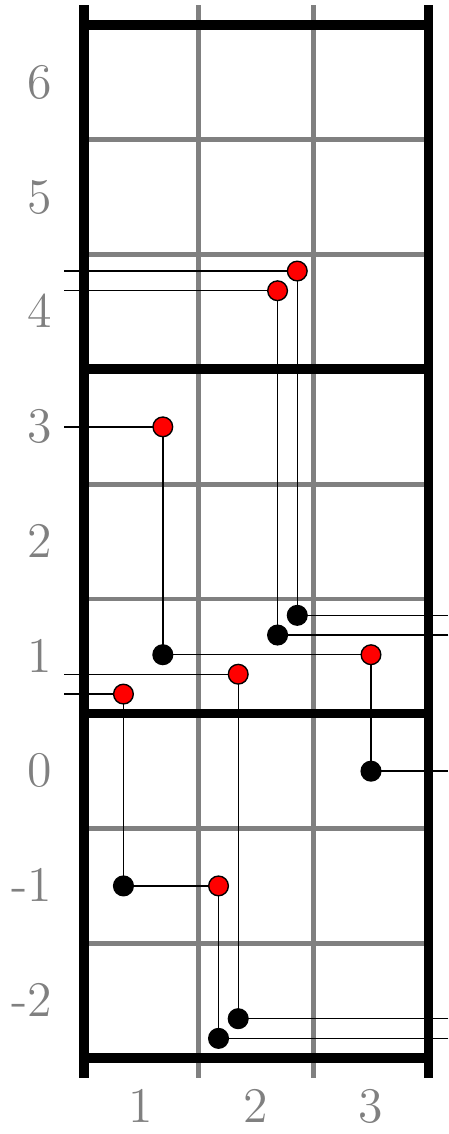}}}%
        \caption{In the left panel we see the evaluation of Viennot map transforming $\overline{\pi}$ of \eqref{eq:Viennot_map_example_2}. This is done through the periodic shadow line construction explained in 
        \cref{def:Viennot_map} and
        represented in the right panel. We made cells of $\mathscr{C}_n$ fattier in order to allocate 
        multiple bullets while letting them keep the correct relative positions.}
        \label{fig:Viennot_map}
    \end{figure}

\begin{proposition} \label{thm:matrices_configurations}
    The sets $\mathfrak{E}_n$, $\mathcal{M}_n$ and $\overline{\mathbb{M}}_{n
        \times n}$ are in bijection
\end{proposition}

\begin{proof}
    The bijection between $\mathfrak{E}_n$ and $\mathcal{M}_n$ is proven in \cref{prop:local_completeness}. We then need to prove that the assignment 
    \begin{equation}
        \overline{M}(c) = \mathsf{N}_1(c) \wedge \mathsf{E}_1(c)
    \end{equation}
    defines a bijection $\overline{M} \mapsto \mathcal{E}$ between $\overline{\mathbb{M}}_{n \times n}$ and $\mathfrak{E}_n$. 
    Let $\overline{M}^{(t)}$ be the Viennot dynamics with initial data $\overline{M}$. Then for any $c\in \mathscr{C}_n$ we set $\mathsf{W}_t(c)$ as the number of lines of the construction $\overline{M}^{(t)} \to \overline{M}^{(t+1)}$ entering the face $c$ from west and similarly for $\mathsf{S}_t,\mathsf{E}_t,\mathsf{N}_t$. From such edge configuration define $(\alpha^{(t)},\beta^{(t)}) = (\alpha,\beta)_{(1,tn+1)}$, following the notation of \cref{def:complete_edge_conf} and recall the truncation operator $\delta_{\ge1}$ from \cref{prop:restriction_matrices}.
    Since $\overline{M}$ is compactly supported there exists $k^* > 0$ such that $\overline{M}_{i,j}(k)=0$ for all $i,j\in\{1,\dots,n\}$ and $|k|>k^*$. This implies that $\delta_{\ge 1}(\alpha^{(-k^*)}) = \delta_{\ge 1}(\beta^{(-k^*)}) = 0$ and $\delta_{\ge 1}(\alpha^{(k^*)}) = \alpha^{(k^*)}$,  $\delta_{\ge 1}(\beta^{(k^*)}) = \beta^{(k*)}$. By recursive application of \cref{prop:restriction_matrices} we find that $(\alpha^{(k^*)},\beta^{(k^*)}) \in \mathcal{M}_n$ and they are uniquely determined. This implies that $\mathcal{E}\in\mathfrak{E}_n$ completing the proof. 
    Now that we have constructed the correspondence $\overline{M} \leftrightarrow \mathcal{E}$ we can associate to the matrix $\overline{M}$ the pair $(\alpha,\beta)$ as in \eqref{eq:map_matrices_config}.
    
\end{proof}

\begin{remark}
    The notation $\overline{M}^{(t)}$ 
    appeared already in \eqref{eq:from_matrix_to_conf}
    before \cref{def:Vdyn} but they
    are consistent because, 
    if a matrix $\overline{M}$ and a configuration $\mathcal{E}$ are in correspondence and $\overline{M}^{(t)}$ is the Viennot dynamics with initial data $\overline{M}$, then \eqref{eq:from_matrix_to_conf} holds.
\end{remark}

If a pair $(\alpha,\beta) \in \mathcal{M}_n$ and a matrix $\overline{M} \in \overline{\mathbb{M}}_{n \times n}$ are in correspondence through the bijection described in \cref{thm:matrices_configurations} we will use the notation $(\alpha,\beta) \xleftrightarrow[]{\skwRSK\,} \overline{M}$. Moreover, composing such correspondence with the row-coordinate parameterization $(P,Q) \xrightarrow[]{\rc} (\alpha,\beta)$ defines a projection denoted by
\begin{equation} \label{eq:map_PQ_M}
    (P,Q) \xrightarrow[]{\skwRSK \,} \overline{M}.
\end{equation}
An analogous projection $(P,Q) \xrightarrow[]{\skwRSK \,} \overline{\pi} \in \overline{\mathbb{A}}_{n,n}$ is defined taking advantage of mapping \eqref{eq:matrix_weighted_biword}.

By the same arguments as in proof of  \cref{thm:matrices_configurations}, 
$\mathcal{M}_n^+$ and $\overline{\mathbb{M}}_{n \times n}^+$
are also in bijection. Combining this with the bijection between $\mathcal{M}_n^+\times\mathbb{Y}$ and pairs $(P,Q)$ of classical semi-standard tableaux reported in \cref{prop:row_coordinate_classical_pair} gives the Sagan-Stanley correspondence stated below. 

\begin{theorem}[\cite{sagan1990robinson}, Theorem 6.6] \label{thm:SS}
    There exists a canonical bijection
    \begin{equation}
        \overline{\mathbb{M}}_{n \times n}^+ \times \mathbb{Y}
        \xleftrightarrow[]{\hspace{.4cm}\skwRSK\hspace{.4cm}} \bigcup_{\rho,\lambda \in \mathbb{Y}} SST(\lambda / \rho ,n) \times SST(\lambda / \rho ,n),
    \end{equation}
    which we denote by $(\overline{M},\nu) \xleftrightarrow[]{\skwRSK\,}(P,Q)$. Moreover the property
    \begin{equation} \label{eq:SS_weight_preserving}
        |\rho| = \wt(\overline{M}) + |\nu|,
    \end{equation}
    holds.
\end{theorem}

\subsection{Relations between skew $\RSK$ and Viennot dynamics}
    
The Viennot dynamics enjoys a very simple relations with the skew $\RSK$ dynamics which we describe in the two propositions below.

\begin{proposition} \label{prop:Viennot_and_RSK_ca}
    Let $\overline{M}^{(t)}, (P_t,Q_t)$ be respectively the Viennot and the skew $\RSK$ dynamics with initial data $\overline{M}^{(1)}=\overline{M}$ and $(P_1,Q_1)=(P,Q)$. Additionally assume that $(P,Q) \xrightarrow[]{\skwRSK \,} \overline{M}$. Then $\overline{M}_{i,j}^{(t)}(k)=m$ if and only if, during the update $(P_{-k},Q_{-k}) \to (P_{-k+1},Q_{-k+1})$, exactly $m$ $i$-cells of $P_{-k}$ move from row $t-1$ to row $t$ during the internal insertion of $j$-cells of $Q_{-k}$.
\end{proposition}

\begin{proof}
    This is an immediate consequence of \cref{prop:meaning_M_k}.
\end{proof}

\begin{proposition}\label{prop:RSK_ca_and_Viennot_ca}
    Let $P,Q\in SST(\lambda/\rho,n)$ and define $(P,Q) \xrightarrow[]{\skwRSK\,} \overline{M}$. Construct tableaux $P',Q'$ shifting up by one unit each cell of $P,Q$; i.e. $P'(c,r)=P(c,r+1)$ and same for $Q,Q'$. Then $(P',Q')\xrightarrow[]{\skwRSK\,} \overline{M}' = \mathbf{V}(\overline{M})$.
\end{proposition}

\begin{proof}
    Let $(\alpha,\beta)=\rc(P,Q)$ and $(\alpha',\beta')=\rc(P',Q')$. Then, by definition, we have
    \begin{equation} \label{eq:relation_alpha_beta_and_primed}
        \alpha_{i,j}' = \alpha_{i,j+1},
        \qquad
        \beta_{i,j}' = \beta_{i,j+1}.
    \end{equation}
    We now construct edge configurations $(\alpha,\beta) \mapsto \mathcal{E}$, $(\alpha',\beta') \mapsto \mathcal{E}'$ as in \eqref{eq:map_matrices_config}. Thanks to the relation \eqref{eq:relation_alpha_beta_and_primed}, it is clear that all edge values of $\mathcal{E}'$ will differ by those of $\mathcal{E}$ by the same shift, or more precisely
    $$
    (\mathsf{W}_j',\mathsf{S}_j',\mathsf{E}_j',\mathsf{N}_j')(c) = (\mathsf{W}_{j+1},\mathsf{S}_{j+1},\mathsf{E}_{j+1},\mathsf{N}_{j+1})(c),
    \qquad
    \text{for all } c \in \mathscr{C}_n, j\in\mathbb{Z}.
    $$
    By \eqref{eq:from_matrix_to_conf} this implies that $\overline{M}' = \mathbf{V}(\overline{M})$.
\end{proof}

\subsection{Asymptotic states of skew $\RSK$ dynamics} 

We describe pairs of tableaux $\RSK^t(P,Q)$ when $t$ becomes large. Contents discussed in this subsection were introduced, along with examples in \cref{subs:examples}. 

\begin{definition}[Asymptotic increments]
    Let $P,Q\in SST(\lambda/\rho,n)$ and consider the skew $\RSK$ dynamics $(P_t,Q_t)$ with initial data $(P,Q)$. We define the \emph{asymptotic increment} $\mu(P,Q)$, through its transpose $\mu'$ as
    \begin{equation} \label{eq:asymptotic_increment}
        \mu_j' = \lim_{t \to \infty} (\lambda^{t})'_j/t
    \end{equation}
    where $\lambda^{t}/\rho^{t}$ is the shape of $P_t,Q_t$.
\end{definition}

\begin{proposition}
    Numbers $\mu_j'$ defined by \eqref{eq:asymptotic_increment} form a weakly decreasing sequence of integers and define a partition. 
\end{proposition}

\begin{proof}
    Assume that tableaux $P,Q$ are standard and follow the evolution of the $i$-cell in $P_t$, for $t=1,2,\dots$, which has coordinate $(c_t,r_t)$. From the bumping algorithm it follows that $r_1<r_2<\cdots$ and also $c_1 \ge c_2 \ge \cdots >0$. Such weak monotonicity of column coordinates implies that from a certain $t$ onward $c_t=c_{t+1}=\cdots$ and proves that $\mu_j'$ are integers. Moreover $\mu_j'$ must also form a weakly decreasing sequence since longer columns ``travel faster" under the skew $\RSK$ dynamics.
\end{proof}

\begin{definition}[Stable states] \label{def:stable_states}
    Consider a pair of semi-standard tableaux $(P,Q)$ with shape $\lambda / \rho$. For $t \ge 0$, let $(P_t,Q_t) = \RSK^t (P,Q)$ and denote their shape by $\lambda^{(t)} / \rho^{(t)} $ and by $\mu=\mu(P,Q)$ the asymptotic increment. We say that the pair $(P,Q)$ is $\RSK$-\emph{stable} if for all $t\ge 0$ we have
    \begin{equation}
        (\lambda^{(t)})'= \lambda ' + t \times \mu' 
        \qquad
        \text{and}
        \qquad
        (\rho^{(t)})' = \rho' + t \times \mu'.
    \end{equation}
\end{definition}

Reading off columns of pairs of $\RSK$-stable tableaux we can associate pairs of vertically strict tableaux.

\begin{definition}[Asymptotic vertically strict tableaux] \label{def:asymptotic_VST}
    Let $P,Q \in SST(\lambda/\rho,n)$ and consider the skew $\RSK$ dynamics $(P_t,Q_t)$ with initial data $(P,Q)$. Denote by $\mu=\mu(P,Q)$ the asymptotic increment. The \emph{asymptotic vertically strict tableaux} $V,W \in VST(\mu,n)$ associated to $(P,Q)$ have $j$-th column entries given by
    \begin{gather}
        V(j,i) = \lim_{t\to \infty} P_t( j, \rho^{(t) \prime} +i ), 
        \\
        W(j,i) = \lim_{t\to \infty} Q_t( j, \rho^{(t) \prime} +i ),
    \end{gather}
    where $\lambda^{(t)}/\rho^{(t)}$ denotes the shape of $P_t,Q_t$.
\end{definition} 

\begin{definition}
    The projection map 
    \begin{equation} \label{eq:Phi}
        \Phi : \bigcup_{\rho,\lambda} SST(\lambda/\rho,n) \times SST(\lambda/\rho,n) \to \bigcup_{\mu \in \mathbb{Y}} VST(\mu,n) \times VST(\mu,n),
    \end{equation}
    assigns to a pair of (generalized) skew tableaux their asymptotic vertically strict tableaux $\Phi:(P,Q) \mapsto (V,W)$.
\end{definition}

\begin{remark}
    Composing map $\Phi$ with the Sagan-Stanley correspondence would generate a projection $\tilde{\Phi}:\overline{\pi} \mapsto (V,W)$ resembling Pak's asymptotic construction of Shi's affine Robinson-Schensted's correspondence \cite{Shi_affine_RS}. There, pairs of tabloids, along with an array of weights are put in correspondence with \emph{periodic permutations}, which we can see as weighted permutations $\overline{\pi}$ with total weight $\wt(\overline{\pi})=0$.
    In recent works \cite{CPY_affine_matrix_ball,C_Bl_P_monodromy,chmutovEtAl_affine_evacuation} authors studied symmetries of the affine RS correspondence which include Knuth relations and crystals. It would be interesting to clarify similarities between projection $\tilde{\Phi}$, or rather bijection $\tilde{\Upsilon}$ we will introduce in \cref{sec:bijection}, and Shi's affine RS correspondence. We leave this investigation for a future work. 
\end{remark}

\subsection{Asymptotic states of Viennot dynamics} \label{subs:Viennot_asymptotics}
For any fixed weighted biword $\overline{\pi}$, we aim now to characterize $\mathbf{V}^t(\overline{\pi})$ for large $t$. To describe the limiting form of such biwords we need the following definitions.

\begin{definition}[Strict down-right loops]
    A \emph{strict down-right loop} $\varsigma$ is a sequence of points $\varsigma = (\varsigma_j: j = 1,\dots ,J) \subset \mathscr{C}_n$ such that $\varsigma \subset \xi$ for some down-right loop $\xi$ and
    \begin{equation}
        \varsigma_{k+1} \sim_n \varsigma_k +a_k \mathbf{e}_1 - b_k \mathbf{e}_2,
    \end{equation}
    for some numbers $a_k,b_k \in \{1,\dots,n\}$ and $k=1,\dots,J$. Indices here are taken $\mod J$ and $\varsigma_{J+1}=\varsigma_1$. The length of the loop $\varsigma$ is $\ell(\varsigma)=J$. Notice that we necessarily have $J\le n$ and that strict down-right loops are ``localized" in the sense that $\varsigma$ is always contained in a band $\{1,\dots,n\} \times \{ j,\dots,j+n \} \subset \mathscr{C}_n$ for some $j\in \mathbb{Z}$.     
\end{definition}

\begin{figure}[t]
    \centering
    \includegraphics{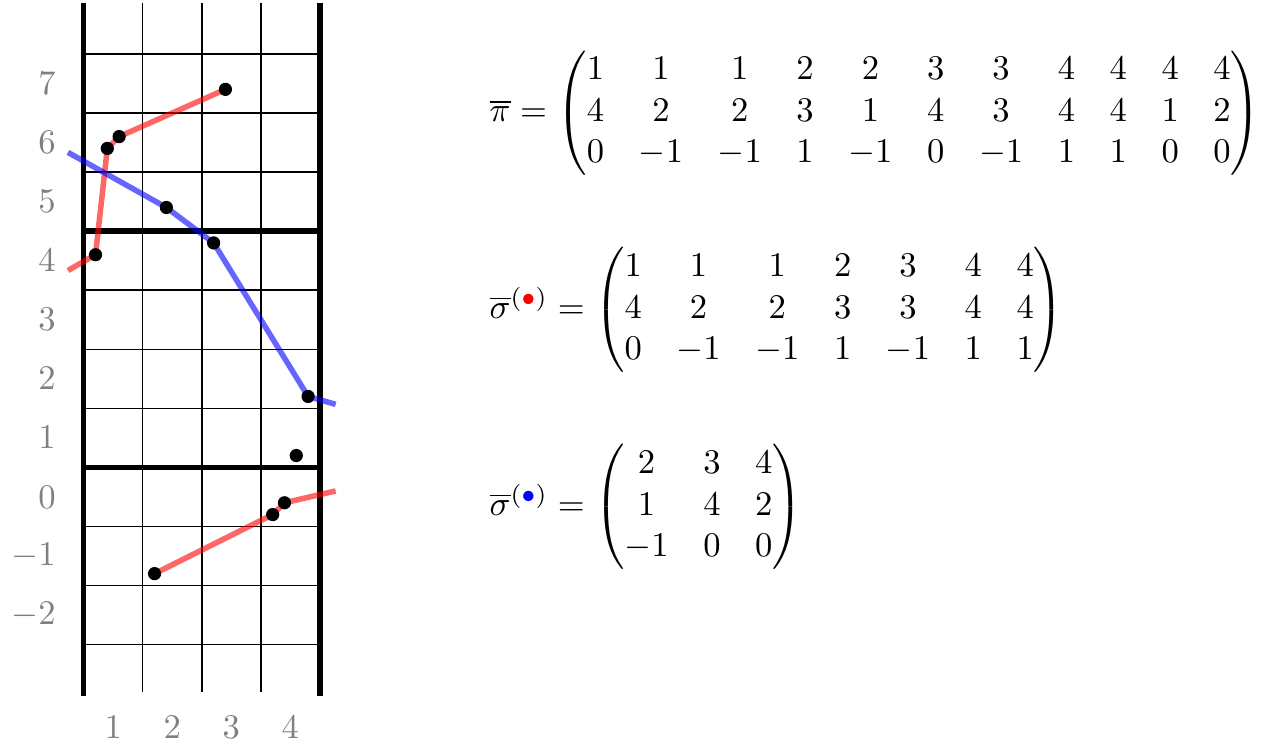}
    \caption{A weighted biword $\overline{\pi}$ viewed as a configuration of points on $\mathscr{C}_n$. Localized decreasing subsequences, as $\overline{\sigma}^{(\textcolor{blue}{\bullet})}$ form down right loops around $\mathscr{C}_n$.
    Increasing subsequences as $\overline{\sigma}^{(\textcolor{red}{\bullet})}$ form up-right path winding around the cylinder.}
    \label{fig:example_incr_decr_subs}
\end{figure}

Given a biword $\overline{\pi}$, written in the usual notation \eqref{eq:weighted_biword}, we will associate to its columns $\overline{\pi}_i$ points on $\mathscr{C}_n$ as
\begin{equation} \label{eq:points_on_cyl}
    [\overline{\pi}_i] = (q_i, p_i - n w_i) \in \mathscr{C}_n.
\end{equation}
We also denote by $[\overline{\pi}]$ the collections of points $[\overline{\pi}_i]$ for $i=1,\dots , \ell(\overline{\pi})$. We will confuse at times points $c \in \mathscr{C}_n$ and elements of a weighted biword and write $c \in \overline{\pi}$ if $[\overline{\pi}_i] \sim_n c$ for some $i$.

\begin{definition}[Localized decreasing sequences] \label{def:lds}
    A weighted biword $\overline{\pi} \in \overline{\mathbb{A}}_{n,n}$ is a \emph{localized decreasing sequence}, LDS for short, if the set of points $([\overline{\pi}_i]:i=1,\dots,\ell(\overline{\pi}))$
    forms a strict down right loop on $\mathscr{C}_n$.
\end{definition}

For the sake of future discussion we also define increasing sequences on $\mathscr{C}_n$. They will be used at length in \cref{sec:Greene_invariants}.

\begin{definition}[Increasing sequences]
    A weighted biword $\overline{\pi} \in \overline{\mathbb{A}}_{n,n}$ is an \emph{increasing sequence}, if the set of points $([\overline{\pi}_i]:i=1,\dots,\ell(\overline{\pi}))$
    is contained in an up-right path of $\mathscr{C}_n$.
\end{definition}

If two weighted biwords are such that $\overline{\sigma} \subseteq \overline{\pi}$ we will say that $\overline{\sigma}$ is a subsequence of $\overline{\pi}$. Analogously we refer to weighted biwords as sequences. An example of a localized decreasing subsequence of a weighted biword is given by $\overline{\sigma}^{(\textcolor{blue}{\bullet})}$ in \cref{fig:example_incr_decr_subs}. In the same figure $\overline{\sigma}^{(\textcolor{red}{\bullet})}$ is an increasing subsequence of $\overline{\pi}$.

Let us now characterize asymptotic states of the Viennot dynamics.

\begin{proposition} \label{prop:Viennot_asymptotic}
    Let $\overline{\pi} \in \overline{\mathbb{A}}_{n,n}$. Then there exist $t^* \in \mathbb{N}_0$ and localized decreasing sequences $\overline{\xi}^{(1)},\dots,\overline{\xi}^{(k)}$, such that, for all $t>0$ we have
    \begin{equation}
        \mathbf{V}^{t^*+t}(\overline{\pi}) = \mathbf{V}^{t}(\overline{\xi}^{(1)}) \cupdot \cdots \cupdot \mathbf{V}^{t}(\overline{\xi}^{(k)}).
    \end{equation}
    Moreover, if $\mu=\mu(P,Q)$ is the asymptotic increment of any pair such that $(P,Q) \xrightarrow[]{\skwRSK\,} \overline{\pi}$, then, listing the $\overline{\xi}^{(i)}$'s  decreasingly in length, we have $\ell( \overline{\xi}^{(i)} ) = \mu_i'$ for $i=1,\dots,\mu_1$ and in particular $k=\mu_1$.
\end{proposition}

\begin{proof}
    We are going to use the notion of asymptotic increments for skew tableaux and the relation between skew $\RSK$ dynamics and Viennot dynamics presented in \cref{prop:Viennot_and_RSK_ca}. With no loss of generality we assume that $\overline{\pi}$ is a weighted permutation, since such choice simplifies the notation. The general case $\overline{\pi} \in \overline{\mathbb{A}}_{n,n}$ is, as usual, recovered by standardization. 
    
    Let $(P,Q)$ be a pair of standard tableaux such that $(P,Q)\xrightarrow[]{\skwRSK\,} \overline{\pi}$ and consider $(P_t,Q_t)$ the skew $\RSK$ dynamics with initial data $(P,Q)$. Let $T\in \mathbb{N}$ be large enough so that the pair $(P_T,Q_T)$ is $\RSK$-stable and call $\mu=\mu(P,Q)$ its asymptotic increment. Let $0=R_0, R_1,R_2,\dots $ define a rectangular decomposition of $\mu$ as in \eqref{eq:rect_dec}
    and set $r_i=R_i - R_{i-1}$. Then, when $t \ge T$, during the update $(P_t,Q_t) \to (P_{t+1},Q_{t+1})$ columns $R_{i-1}+1,\dots ,R_{i}$ are shifted down by $\mu_{R_i}'$ cells. If $\lambda/ \rho$ is the skew shape of $(P_T,Q_T)$, the vertical displacement between labeled boxes at columns $R_i$ and $R_i+1$ can be assumed to be arbitrarily large, or in other words $\rho_{R_i} ' \gg \lambda_{R_i+1}'$, choosing $T$ sufficiently large.
        
    Consider an integer $K > \lambda_1'$. Then by choosing $K$ large enough there exist times $T_1 > T_2 >\cdots > T$ such that in the skew shape of $(P_{T_i},Q_{T_i})$ columns $1, \dots, R_i$ have only cells with row coordinate larger than $K$, while cells at columns $R_i+1,R_i+2,\dots$ have row coordinate smaller than $K$. By choosing $K$ large enough, we can assume that the differences $T_{i+1}-T_{i}$ are arbitrarily large. 
    
    Consider now the Viennot dynamics $\overline{\pi}^{(t)}$ with initial data $\overline{\pi}$. Recall that, by \cref{prop:Viennot_and_RSK_ca}, the weighted permutation $\overline{\pi}^{(K)}$ encodes the times at which specific entries of $(P,Q)$ reach the $K$-th row during the skew $\RSK$ dynamics. By the discussion above, taking $T_1,T_2,\dots$ sufficiently large and spread apart we find that $\overline{\pi}^{(K)}$ can be written as $\overline{\pi}^{(K)} = \overline{\sigma}^{(1)} \cupdot \overline{\sigma}^{(2)} \cupdot \cdots$, where $\overline{\sigma}^{(j)}$'s are weighted biwords encoding information about columns of length $\mu_{R_j}'$ of $(P_T,Q_T)$. More in detail, denoting
    \begin{equation}
        \overline{\sigma}^{(j)} = \left( \begin{matrix}
        q^{(j)}_1 & \cdots & q^{(j)}_{I_j}
        \vspace{.1cm}
        \\
        p^{(j)}_1 & \cdots & p^{(j)}_{I_j}
        \vspace{.1cm}
        \\
        w^{(j)}_1 & \cdots & w^{(j)}_{I_j}
        \end{matrix} \right),
    \end{equation}
    we have $I_j = \mu_{R_j}' r_j$ and $\max_{1\le i \le I_j} w_i^{(j)} \ll \min_{1\le i \le I_{j-1}} w_i^{(j-1)}$. By stability of $(P_T,Q_T)$ we can also conclude that under the action of Viennot map, $\overline{\sigma}^{(1)}, \overline{\sigma}^{(2)}, \dots$ evolve independently from each other, or in other words
    \begin{equation} \label{eq:sigma_j_evolve_independently}
        \mathbf{V}^{s}\left(\overline{\pi}^{(K)} \right) = \mathbf{V}^s \left(\overline{\sigma}^{(1)} \right) \cupdot \mathbf{V}^s \left(\overline{\sigma}^{(2)} \right) \cupdot \cdots,
    \end{equation}
    for all $s \ge 0$. To prove our proposition we need to show that there exist LDS's $\overline{\xi}^{(j,1)} ,\dots, \overline{\xi}^{(j,r_j)}$ such that  $\overline{\sigma}^{(j)} = \overline{\xi}^{(j,1)} \cupdot \cdots \cupdot \overline{\xi}^{(j,r_j)}$, having length $\ell( \overline{\xi}^{(j,k)} ) = \mu_{R_j}'$ for $1\leq k\leq r_j$ and that evolve autonomously under Viennot dynamics
    \begin{equation}
        \mathbf{V}^s \left(\overline{\sigma}^{(j)} \right) = \mathbf{V}^s \left(\overline{\xi}^{(j,1)} \right) \cupdot \cdots \cupdot \mathbf{V}^s \left(\overline{\xi}^{(j,r_j)} \right).
    \end{equation}

    We use \cref{prop:Viennot_and_RSK_ca}. For any $j$ such that $r_j>0$ let $c\in\{R_{j-1}+1 ,\dots, R_j \}$ and denote $\mu_{R_j}'=m$. Let $p_1<\cdots <p_m$ and $q_1 <\dots < q_m$ be entries of $c$-th columns of $P_t,Q_t$ for $t$ large enough. Let $\widetilde{t}\in \{ T_{j-1}+1,\dots,T_j \}$ be the first time $p_m$ reaches a row greater than $K$. Since columns evolve autonomously in stable states, the only possibility is that $p_m$ reached row $K$ as a result of the internal insertion at cell corresponding to the $q_s$-cell in $Q_{\widetilde{t}-1}$ for some $s\in\{1,\dots,m\}$.  Moreover, internal insertion corresponding to $q_{s+1},\dots, q_m$ during the update will result in $p_2,\dots,p_{m-s+1}$ reaching row $K$. This implies that $(q_S,p_{m+s-S} + n \widetilde{t} \,) \in [\overline{\sigma}^{(j)}]$ for all $S=s,\dots, m$. During the update $(P_{\, \widetilde{t}},Q_{ \, \widetilde{t}}) \to (P_{\, \widetilde{t}+1},Q_{\, \widetilde{t}+1})$ the remaining cells $p_1,\dots,p_{s-1}$ will reach row $K$ and they will do so in correspondence of internal insertion of $q_1,\dots, q_{s-1}$-cells in $Q_{\widetilde{t}}$. Therefore also $(q_S,p_S + n (\widetilde{t}+1) \,) \in [\overline{\sigma}^{(j)}]$ for all $S=1,\dots, s-1$ and this implies that
    \begin{multline*}
        [\overline{\xi}^{(j,1)}] = (q_1,p_1 + n (\widetilde{t}+1)) \to \cdots \to (q_{s-1},p_{s-1} + n (\widetilde{t}+1)) 
        \\
        \to (q_s,p_1 + n \widetilde{t}) \to \cdots \to (q_m , p_{s} + n \widetilde{t})
    \end{multline*}
    is a strict down right loop identifying a subsequence of $\overline{\sigma}^{(j)}$. Repeating the argument for all columns $c=R_{j-1}+1,\dots, R_k$ we find disjoint localized decreasing subsequences of $\overline{\sigma}^{(j)}$, which we might denote by $\overline{\xi}^{(j,1)},\dots, \overline{\xi}^{(j,r_j)}$, all having length equal to $\mu_{R_j}'$. One also easily sees that since columns of tableaux $(P_t,Q_t)$ evolve autonomously, then also the corresponding LDS's $\xi^{(j,r)}$ evolve independently under Viennot map and this completes the proof.
\end{proof}

\section{Affine crystal structures} \label{sec:Knuth_and_crystals}
In this section we first review basic notions in the theory of Kashiwara crystals, focusing only on the type $A^{(1)}_{n-1}$ case. Many of the objects encountered in the previous sections possess affine bicrystal graph structures, such as the set of pairs $(V,W)$ of vertically strict tableaux, the set of pairs $(P,Q)$ of semi-standard tableaux, or the set of matrices $\overline{\mathbb{M}}_{n\times n}$. Vertically strict tableaux are a standard combinatorial model of an affine (bi)crystal, whereas the affine bicrystal structure on pairs $(P,Q)$ or on matrices $\overline{M}$ described in \cref{subs:pairs_of_tab_affine_bicrystals,subs:matrices_affine_bicrystal} is new. The main result of this section is given by \cref{thm:symmetries_RSK}, which asserts the $\widehat{\mathfrak{sl}}_n$ bicrystal symmetry of the skew $\RSK$ map. This symmetry will allow to characterize the skew $\RSK$ dynamics completely in later sections, 
see \cref{rem:symmetries_RSK}. 

\subsection{Crystals and bicrystals} \label{subs:crystals_and_bi_crystals}

Most of the material in this subsection is contained in 
\cite{Hong_Kang_book_crystals,BumpSchilling_crystal_book}.
For a short introductory account on the subject the reader may consult \cite{Shimozono_dummies}.

\medskip

An $\widehat{\mathfrak{sl}}_n$-\emph{crystal graph}\footnote{also called affine crystal graph of type $A^{(1)}_{n-1}$ in the Dynkin diagram language}, or equivalently for us an \emph{affine crystal graph}, is a set of vertices $B$, equipped with a function $\gamma : B \to \mathbb{N}_0^n$ (commonly referred to as \emph{weight}, but here called \emph{content}), and colored directed edges $b \xrightarrow[]{i} b'$, with colors $i$ ranging in $\{ 0,\dots,n-1 \}$ satisfying the following two properties.
\begin{enumerate}
    \item There are no multiple edges. In case $b \xrightarrow[]{i} b'$ we write 
    \begin{equation*}
        b' = \F{i}(b),
        \qquad
        \text{or}
        \qquad
        b=\E{i}(b'),
    \end{equation*}
    where $\F{i},\E{i}$ are respectively the $i$-th \emph{lowering and raising Kashiwara operators}. When $\F{i}$ is not defined for an element $b$ we will write $\F{i}(b)=\varnothing$ and similarly for $\E{i}$. Kashiwara operators define numbers $\varphi_i,\varepsilon_i : B \to \mathbb{N}_0$ as
    \begin{gather*}
        \varphi_i(b) = \max\{m : \F{i}^m(b) \neq \varnothing \},
        \\
        \varepsilon_i(b) = \max\{m : \E{i}^{\,m}(b) \neq \varnothing \}.
    \end{gather*}
    \item Let $h_0=\mathbf{e}_n - \mathbf{e}_{1}$ and $h_i = \mathbf{e}_i - \mathbf{e}_{i+1}$, $i=1,\dots,n-1$. Then, for all $b\in B$ we have
    \begin{equation*}
        \langle h_i , \gamma(b) \rangle = \varphi_i(b) - \varepsilon_i (b)
    \end{equation*}
    and, whenever $\F{i}(b)\neq \varnothing$, we have
    \begin{equation*}
        \gamma (\F{i}(b)) = \gamma(b) - h_i.
    \end{equation*}
    Here $\mathbf{e}_i$ and $\langle \cdot , \cdot \rangle$ are respectively the standard basis and the standard scalar product of $\mathbb{C}^n$.
\end{enumerate}

The set $B$ in this case is called \emph{crystal}, but 
unless necessary we will not distinguish between the notion of crystal and its graph. An example of an affine crystal graph is reported in \cref{fig:demazure_subgraph} below. We also define $\mathfrak{sl}_n$ crystals graphs, which we refer to as \emph{classical crystals graphs}, removing from the description above all statements concerning $0$-edges. For instance, in \cref{fig:demazure_subgraph} the classical crystal graph is given erasing red and blue edges.

Lowering and raising operators $\E{i},\F{i}$ are partial mutual inverses; i.e. if $\E{i}(b) \neq \varnothing$, then $\F{i}\circ \E{i}(b)=b$ and same for the opposite case. For this reason, in case $h$ is an operator written as
\begin{equation*}
    h=(\E{i_1})^{N_1} \circ (\F{i_2})^{N_2} \circ \cdots \circ (\E{i_{k-1}})^{N_{k-1}} \circ (\F{i_k})^{N_k},
\end{equation*}
then we will denote by $h^{-1}$ the operator
\begin{equation*}
    h^{-1} = (\E{i_k})^{N_k} \circ (\F{i_{k-1}})^{N_{k-1}} \circ \cdots \circ (\E{i_2})^{N_2} \circ (\F{i_1})^{N_1}.
\end{equation*}
Clearly, if $h(b)\neq \varnothing$, then $h^{-1}\circ h(b) =b$. 

\medskip

An $\widehat{\mathfrak{sl}}_n$ \emph{bicrystal graph} is a set of vertices $B$ possessing two commuting  $\widehat{\mathfrak{sl}}_n$ crystal graph structures, i.e. two commuting families of Kashiwara operators. We will denote the two sets of Kashiwara operators for bicrystals with the notation $\widetilde{E}^{(1)}_i, \widetilde{F}^{(1)}_i$ and $\widetilde{E}^{(2)}_i, \widetilde{F}^{(2)}_i$, $i=0,\dots,n-1$. For instance, if $B$ is an $\widehat{\mathfrak{sl}}_n$ crystal, then the cartesian product $B \times B$ is an $\widehat{\mathfrak{sl}}_n$ bicrystal setting
    \begin{equation}\label{eq:bicrystal_operators}
        \widetilde{E}^{(1)}_i = \E{i} \times \mathbf{1},
        \qquad
        \widetilde{E}^{(2)}_i = \mathbf{1} \times \E{i},
        \qquad
        \widetilde{F}^{(1)}_i = \F{i} \times \mathbf{1},
        \qquad
        \widetilde{F}^{(2)}_i = \mathbf{1} \times \F{i},
    \end{equation}
and letting content function $\gamma$ act independently on single components. We will introduce below a more elaborate example of bicrystal.
Also in the case of bicrystals we will adopt the same convention as above for inverse operators. If $h$ is a combinations of Kashiwara operators $\widetilde{E}^{(\epsilon)}_i, \widetilde{F}^{(\epsilon)}_i$, then $h^{-1}$ will be the operator obtained reading the expansion of $h$ backward and substituting each $\widetilde{E}^{(\epsilon)}_i$ with $\widetilde{F}^{(\epsilon)}_i$ and viceversa.

\begin{definition} \label{def:morphism_of_crystal_graphs}
    Let $\mathfrak{g}$ be either $\mathfrak{sl}_n$ or $\widehat{\mathfrak{sl}}_n$ and $B,B'$ be $\mathfrak{g}$ crystals. A map $\phi : B \to B'$ is a \emph{morphism of crystals} if
    \begin{equation}
        \phi \circ \E{i} = \E{i} \circ \phi,
        \qquad
        \phi \circ \F{i} = \F{i} \circ \phi
        \qquad
        \text{for all }i=0,\dots,n-1
    \end{equation}
    and $\gamma(\phi(b)) = \gamma(b)$ for all $b\in B$. An \emph{isomorphism} of crystals is a bijective morphism of crystals whose inverse is also a morphism of crystals. We use the convention that $\phi(\varnothing)=\varnothing$.
    
    Analogously we define a morphism of affine bicrystals $B,B'$ as a map $\phi:B\to B'$ that is a morphism of crystal for both crystal graphs structures. If $\phi$ is invertible and its inverse is a morphism of bicrystals then $\phi$ is an isomorphism of bicrystals.
\end{definition}

\subsection{Classical Kashiwara operators} \label{subs:classical_crystals}

On the set of words $\mathcal{A}_n^*$ we can define the action of Kashiwara operators $\E{i}, \F{i}$ for $i=1,\dots , n-1$. The raising operator $\E{i}$ acts replacing an entry $i+1$ with $i$ following the so-called \emph{signature rule}. It goes as follows
    \begin{equation} \label{alg:e_i}
        \begin{minipage}{.9\linewidth}
            \begin{enumerate}
                \item Replace every $i$ in $\pi$ with the $``)"$ symbol and every $i+1$ with $``("$;.
                \item Sequentially match all pairs of consecutive symbols $``("$, $``)"$. At the end of this procedure the subword made of unmatched parenteses will have the form $) \cdots ) ( \cdots ($.
                \item Replace the leftmost unmatched $``("$ parenthesis with $``)"$.
                \item Substitute back $``)"$ with $i$'s and $``("$ with $i+1$. 
            \end{enumerate}
        \end{minipage}
    \end{equation}
Sometimes this operation is impossible (eg. when $\pi\in\mathcal{A}_n^*$ has no $(i+1)$'s) and in that case we impose $\E{i}(\pi) = \varnothing$.

The lowering operator $\F{i}$ is defined analogously to $\E{i}$ with the difference that the third step of \eqref{alg:e_i} becomes
    \begin{equation}
        \begin{minipage}{.9\linewidth}
            \begin{enumerate}
                \item[(3')]  Replace the rightmost unmatched $``)"$ parenthesis with $``("$.
            \end{enumerate}
        \end{minipage}
    \end{equation}
As before when the procedure is impossible we set $\F{i}(\pi)=\varnothing$. It is easy to see that $\E{i}$ and $\F{i}$ are mutual inverses, when both their operations are defined. We illustrate the action of Kashiwara operators with an example. Here $i=2$:
    \begin{equation} \label{eq:example_Kashiwara}
        \begin{matrix}
        \pi = & 4 & 2 & 3 & 2 & 1 & 2 & 3 & 1 & 4 & 3 & 3 & 2 & 1 & 2 & 4 & 1 & 2 & 3 & 3
        \\
         && \boldsymbol{)} & \boldsymbol{(} & \boldsymbol{)} &  & \boldsymbol{)} & \boldsymbol{(} &  &  & \boldsymbol{(} & \boldsymbol{(} & \boldsymbol{)} &  & \boldsymbol{)} &  &  & \boldsymbol{)} & \boldsymbol{(} & \boldsymbol{(}
         \\
         && \boldsymbol{)} & \textcolor{gray!60}{(} & \textcolor{gray!60}{)} &  & \boldsymbol{)} & \textcolor{gray!60}{(} &  &  & \textcolor{gray!60}{(} & \textcolor{gray!60}{(} & \textcolor{gray!60}{)} &  & \textcolor{gray!60}{)} &  &  & \textcolor{gray!60}{)} & \boldsymbol{(} & \boldsymbol{(}
         \\
         \E{2} (\pi) = & 4 & 2 & 3 & 2 & 1 & 2 & 3 & 1 & 4 & 3 & 3 & 2 & 1 & 2 & 4 & 1 & 2 & \textcolor{red}{2} & 3
         \\
         \F{2} (\pi) = & 4 & 2 & 3 & 2 & 1 & \textcolor{red}{3} & 3 & 1 & 4 & 3 & 3 & 2 & 1 & 2 & 4 & 1 & 2 & 3 & 3
        \end{matrix}
    \end{equation}
    In the third line matched parentheses were drawn in light gray, while in the last two lines we highlighted in red the letters of $\pi$ that were changed by $\E{2}$ and $\F{2}$.

    The action of Kashiwara operators endows the set $\mathcal{A}_n^*$ with an $\mathfrak{sl}_n$ crystal graph structure, the content $\gamma$ being defined as in \cref{subs:biwords}. For any word $\pi$, $\varphi_i(\pi)$ counts the number of unmatched $``)"$ parentheses we find while computing the action of $i$-th Kashiwara operators on $\pi$ and $\varepsilon_i(\pi)$ is the number of unmatched $``("$ parenteses. For instance in \eqref{eq:example_Kashiwara}, $\varphi_2(\pi) = \varepsilon_2(\pi) =2$.
    
    \medskip
    
The action of Kashiwara operators extends to the set of semi-standard tableaux. If $P\in SST(\lambda/ \mu, n)$, denoting $\pi_P$ its column reading word, we define $\E{i}(P)$ as the tableaux with shape $\lambda/\mu$ and reading word $\E{i}(\pi_P)$. The action of $\F{i}$ is defined equivalently. For example we have
\begin{equation}
    \begin{ytableau} 
    \, & & & 1 & 2 \\ 
     & 1 & 2 & 2 \\
     2 & 3 & 3
    \end{ytableau}
    \xrightarrow[]{ \hspace{1cm} {\LARGE \E{1}} \hspace{1cm} }
    \begin{ytableau} 
    \, & & & 1 & 2 \\ 
     & 1 & \textcolor{red}{1} & 2 \\
     2 & 3 & 3
    \end{ytableau},
\end{equation}
where we highlighted in red the cell that changed its label.
    
\medskip
    
The study of $\mathfrak{sl}_n$ crystal structure of the set of semi-standard tableaux has a long history rooted in the seminal paper \cite{robinson1938representations}. In fact, as recalled in \cite{Shimozono_dummies}, the first instance of the RSK correspondence, was realized by Robinson imposing that map $\pi \mapsto P$ would commute with the action of $\E{i},\F{i}$. 
In general, Kashiwara operators are \emph{coplactic} transformations, which means that they commute with \emph{jeu de taquin} moves; see \cref{subs:jdt}. This known result is stated in the following proposition. In order to make the text self-contained all ingredients necessary to its proof, along with a short review of theory of Knuth relations, are reported in \cref{app:Knuth_rel}.

\begin{proposition} \label{prop:commutation_jdt_kashiwara}
    Let $h$ be anyone between $\E{i},\F{i},i=1,\dots,n$ and $P$ be a semi-standard tableau. Then for any transformation $J$ that performs a sequence of jeu de taquin slides we have $h(J(P)) = J(h(P))$.
\end{proposition}
\begin{proof}
    Combining \cref{thm:haiman_dual_eq} and \cref{thm:kash_op_equiv} we find that $h(J(P)) \stackrel{*}{\simeq} J(h(P))$. On the other hand thanks to \cref{thm:Knuth_equiv_jdt} and \cref{thm:kash_op_equiv} we have $h(J(P)) \simeq J(h(P))$. Therefore $h(J(P)) = J(h(P))$ by \cref{thm:uniqueness_dual_eq_non_dual_eq}.
\end{proof}

We will utilize result of \cref{prop:commutation_jdt_kashiwara} in a slightly weaker form reported next. Recall that $\mathcal{R}_{[r]}$ denotes the internal insertion.

\begin{corollary} \label{prop:commutation_kashiwara_int_ins}
    Let $h$ be anyone between $\E{i},\F{i},i=1,\dots,n-1$ and $P$ be a semi-standard tableau. Then for all $r\in \mathbb{Z}$ we have $h \circ \mathcal{R}_{[r]}(P) = \mathcal{R}_{[r]} \circ h (P)$.
\end{corollary}
\begin{proof}
    Internal insertion transformation $P\mapsto \mathcal{R}_{[r]}(P)$ can be realized through jeu de taquin moves, as a consequence of \cref{prop:Knuth_eq_int_ins}, \cref{thm:Knuth_equiv_jdt} and hence it commutes with $\E{i},\F{i}, i=1,\dots, n-1$ by \cref{prop:commutation_jdt_kashiwara}.
\end{proof}

\subsection{Vertically strict tableaux as affine crystals} \label{subs:VST_affine_crystals}

Denote by $B^{r,1}$ the set of semi-standard Young tableaux of single column shape $1^r$ and entries from $\mathcal{A}_n$. Following \cite{shimozono_affine}, on $B^{r,1}$ we define the $0$-th Kashiwara operators $\E{0}, \F{0}$ as
\begin{equation} \label{eq:zero_Kashiwara_op}
    \E{0} = \pr^{-1} \circ \E{1} \circ \pr
    \qquad
    \text{and}
    \qquad
    \F{0} = \pr^{-1} \circ \F{1} \circ \pr,
\end{equation}
where $\pr$ is the \emph{promotion operator}. For any $b \in B^{r,1}$, its promotion $b'=\pr(b) \in  B^{r,1}$ is the only tableau with content $\gamma_i(b') = \gamma_{i-1}(b)$, where indices $i$ are taken $\mod n$. In words $\E{0}(b)$ is obtained replacing the 1-cell of $b$ with an $n$-cell and reordering the result. For example, if $n=6$, we have
\begin{equation}
    \begin{ytableau}
        1 \\ 3 \\ 4 \\ 5
    \end{ytableau}
    \xrightarrow[]{\hspace{.4cm}\E{0} \hspace{.4cm}}
    \begin{ytableau}
        3 \\ 4 \\ 5 \\ 6
    \end{ytableau}
    .
\end{equation}
Such operation is impossible if $b$ has an $n$-cell or if it does not have a $1$-cell, in which cases we set $\E{0}(b)=\varnothing$. An analogous description may be given for $\F{0}$.
By convention we define $B^{0,1} = \{ \mathbf{0} \}$ and we assume $\E{i},\F{i}:\mathbf{0}\to \varnothing$, for all $i=0,\dots,n-1$ and $\pr(\mathbf{0}) = \mathbf{0}$. Naturally the column word of $\mathbf{0}$ is the empty word.

For any composition $\varkappa=(\varkappa_1, \dots , \varkappa_N)$ define $B^\varkappa = B^{\varkappa_1,1} \otimes \cdots \otimes B^{\varkappa_N,1}$. Classical Kashiwara operators $\E{i},\F{i}, i=1,\dots ,n-1$ are defined on any element $b \in B^{\varkappa}$ by their action on the column word of $b$. The action of $\E{0},\F{0}$ is also well posed forcing $\pr(b_1 \otimes \cdots \otimes b_n) = \pr(b_1) \otimes \cdots \otimes \pr(b_N)$ and the same for $\pr^{-1}$. Notice that we always have
\begin{equation}
    \E{i}(b_1\otimes \cdots \otimes b_N) = b_1\otimes \cdots \otimes b_{k-1} \otimes \E{i}(b_k) \otimes b_{k+1} \otimes \cdots \otimes b_N,
\end{equation}
for some $k$, which is prescribed by the signature rule. Naturally the same holds for operators $\F{i}$. The content function is given by $\gamma(b) = (\gamma_1,\dots,\gamma_n)$, where $\gamma_i$ counts the total number of $i$-cells in the different tensor factors of $b$. We define the \emph{affine crystal graph} $\widehat{B}(\varkappa)$ as the graph having set of vertices $B^\varkappa$ and edges defined by operators $\E{i},\F{i}, i=0,\dots,n-1$. Denote with $B(\varkappa)$ the subgraph of $\widehat{B}(\varkappa)$ obtained erasing all edges generated by $\E{0},\F{0}$. We refer to $B(\varkappa)$ as the \emph{classical crystal subgraph} and its connected components are called \emph{classical connected components}. When $\varkappa$ is a partition we identify $B^\varkappa$ with $VST(\varkappa\,',n)$ and so the set of vertically strict tableaux possesses an affine crystal graph structure. Moreover, for any partition $\mu$, we endow the set $VST(\mu, n) \times VST (\mu,n)$ of an $\widehat{\mathfrak{sl}}_n$ bicrystal structure defining families of Kasiwara operators as in \eqref{eq:bicrystal_operators}.

\medskip

A remarkable property of the affine crystal graph $\widehat{B}(\varkappa)$ is that it is connected. Such result holds in much broader generality and for Kirillov-Reschetikhin crystals of type $A^{(1)}$ was proven in \cite{Akasaka_Kashiwara}. An algorithmic proof of this statement can be found, for instance, in \cite{Fourier_Schilling_Shimozono_demazure,schilling_tingley}.

\begin{proposition}[\cite{Akasaka_Kashiwara}] \label{prop:connectedness_crystal_graph}
    For any composition $\varkappa$, the affine crystal graph $\widehat{B}(\varkappa)$ is connected. 
\end{proposition}

Proposition \ref{prop:connectedness_crystal_graph} will be very important to us as it allows to prove general statements about affine crystal graphs by simply checking that special properties hold for particular elements. With this purpose we introduce the \emph{leading vector}, or \emph{dominant extremal vector} \cite{kashiwara_on_level_zero}, $\varkappa^{\mathrm{lv}} \in B^\varkappa$ as
\begin{equation}
    \varkappa^{\mathrm{lv}} = \varkappa_1^{\mathrm{lv}} \otimes \cdots \otimes \varkappa_N^{\mathrm{lv}},
    \qquad
    \text{with}
    \qquad
    k^{\mathrm{lv}}
    =
    \begin{ytableau}
        1
        \\
        2
        \\
        \smallvdots
        \\
        k
    \end{ytableau}
    \in
    B^{k,1}.
\end{equation}
We observe that $\varkappa^{\mathrm{lv}}$ is the unique element of the crystal $B^\varkappa$ with content equal to $\varkappa^+$.

An immediate consequence of connectedness of affine crystal graphs is that, if two affine crystals $\widehat{B}(\varkappa),\widehat{B}(\eta)$ are isomorphic, then such isomorphism $\phi$ is unique and moreover $\eta^+=\varkappa^+$.

\begin{proposition} \label{prop:unique_crystal_graph_isomorphism}
    For two compositions $\varkappa$ and $\eta$, let
    $\phi:B^\varkappa \to B^\eta$ be an isomorphism between the crystal graphs,
    $B^\varkappa$ and $B^\eta$. Then $\phi$ is unique and can be expressed as
    \begin{equation} \label{eq:unique_crystal_graph_isomorphism}
        \phi(b) = h_b^{-1}( \eta^{\mathrm{lv}} )
    \end{equation}
    where $h_b$ is any composition of Kashiwara operators such that $h_b:b \mapsto \varkappa^{\mathrm{lv}}$.
\end{proposition}
\begin{proof}
    The image of the leading vector $\varkappa^{\mathrm{lv}}$ must be $\eta^{\mathrm{lv}}$ by content considerations and this uniquely determines $\phi(b)$ for all $b\in B^\varkappa$.
    Let $b'=\phi(b)$ and consider a map $h_b$. Since $\phi$ is a morphism we have 
    \begin{equation}
        b' = \phi(b) = h_b^{-1} \circ h_b \circ \phi(b) = h_b^{-1} \circ \phi \circ h_b(b) = h_b^{-1} \circ \phi ( \varkappa^{\mathrm{lv}} ) =  h_b^{-1}( \eta^{\mathrm{lv}} ).
    \end{equation}
    This proves the proposition.
\end{proof}

\subsection{Pairs of tableaux as affine bicrystals} \label{subs:pairs_of_tab_affine_bicrystals}
In this subsection, we present a novel realization of an affine bicrystal structure 
on the set of pairs of (generalized) semi-standard Young tableaux. 
Let us define the action of two families of Kashiwara operators $\widetilde{E}_i^{(\epsilon)}, \widetilde{F}_i^{(\epsilon)}$, $\epsilon=1,2$
on a pair of semi-standard Young tableaux $(P,Q)$ as follows. For $i=1,\dots, n-1$, 
they are given by \eqref{eq:bicrystal_operators}, whereas for $i=0$ we set
\begin{equation} \label{eq:0_th_operators}
    \begin{split}
        &\widetilde{E}_0^{(1)} = \iota_1 \circ ( \E{1} \times \mathbf{1} ) \circ \iota_1^{-1},
        \qquad
        \widetilde{F}_0^{(1)} = \iota_1 \circ ( \F{1} \times \mathbf{1} ) \circ \iota_1^{-1},
        \\
        &\widetilde{E}_0^{(2)} = \iota_2 \circ ( \mathbf{1} \times \E{1} ) \circ \iota_2^{-1},
        \qquad
        \widetilde{F}_0^{(2)} = \iota_2 \circ ( \mathbf{1} \times \F{1} ) \circ \iota_2^{-1} .
    \end{split}
\end{equation}
Compare (\ref{eq:0_th_operators}) with (\ref{eq:zero_Kashiwara_op}). Below in Corollary 
\ref{cor:Phi_is_morphism} we will show a consistency of these under the 
projection $\Phi$ \eqref{eq:Phi}.

\begin{proposition}
      The two families of Kashiwara operators defined above equip the set
    \begin{equation}
        \bigcup_{\rho , \lambda} SST( \lambda / \rho , n ) \times SST( \lambda / \rho , n ),
    \end{equation}
    with an $\widehat{\mathfrak{sl}}_n$ bicrystal structure.  
\end{proposition}

\begin{proof}
    It is straightforward to verify that each of the families $\widetilde{E}^{(\epsilon)}_i,\widetilde{F}^{(\epsilon)}_i$ satisfy hypothesis listed in \cref{subs:crystals_and_bi_crystals}, so that they both endow the set of pairs $(P,Q)$ of an affine crystal structure. It remains to show that these two families are commuting. Clearly, for all $i,j=1,\dots,n-1$, we have
    \begin{equation} \label{eq:commutation_kashiwara_bi_crystal}
        \widetilde{E}^{(1)}_i \circ \widetilde{E}^{(2)}_j = \widetilde{E}^{(2)}_j \circ \widetilde{E}^{(1)}_i, 
    \end{equation}
    and similarly for other relations involving also $\widetilde{F}^{(1)}_i, \widetilde{F}^{(2)}_j$. Let us now show that \eqref{eq:commutation_kashiwara_bi_crystal} holds for $j=0$ and $i=0,1,\dots,n-1$. Following \cref{prop:iota_commute}, $\iota_1,\iota_2$ commute, so we need to show that $\iota_2$ commutes with $\widetilde{E}^{(1)}_i$ for $i=1,\dots,n-1$. This last statement is a consequence of \cref{prop:commutation_kashiwara_int_ins} yielding the proof.
\end{proof}

The following theorem gives a characterization of symmetries of the skew $\RSK$ map.

\begin{theorem} \label{thm:symmetries_RSK}
    The skew $\RSK$ map is an isomorphism of $\widehat{\mathfrak{sl}}_n$ bicrystals.
\end{theorem}
\begin{proof}
    The skew $\RSK$ map clearly commutes with $\iota_1,\iota_2$ as a result of \cref{prop:RSK_from_n_int_ins}. If $P,Q\in SST(\lambda / \rho ,n)$ and $(P',Q')=\RSK(P,Q)$, again by \cref{prop:RSK_from_n_int_ins}, $P',Q'$ are obtained respectively from $P,Q$ after a sequence of internal insertions. By \cref{prop:commutation_kashiwara_int_ins}, this implies that the skew $\RSK$ map commutes with classical operators $\widetilde{E}^{(\epsilon)}_i, \widetilde{F}^{(\epsilon)}_i$ for $i=1,\dots,n-1$, $\epsilon=1,2$. This shows that $\RSK$ is a morphism of bicrystals. Analogously one can show that the inverse $\RSK^{-1}$, which is always defined, is a morphism of bicrystals concluding the proof.
\end{proof}

In the example reported in \cref{fig:symmetries_RSK} the statement of \cref{thm:symmetries_RSK} is expressed in the form of a commutative diagram, that the reader can easily check.

\begin{remark} \label{rem:symmetries_RSK}
    We will show in \cref{subs:linearization} that, modulo symmetries prescribed by \cref{thm:symmetries_RSK}, the skew $\RSK$ map is in fact a linear transformation. This implies that the affine bicrystal structure completely characterizes the skew $\RSK$ map and hence the Sagan-Stanley correspondence of \cref{thm:SS}.
\end{remark}

\begin{figure}[ht]
    \centering
    \includegraphics[scale=.9]{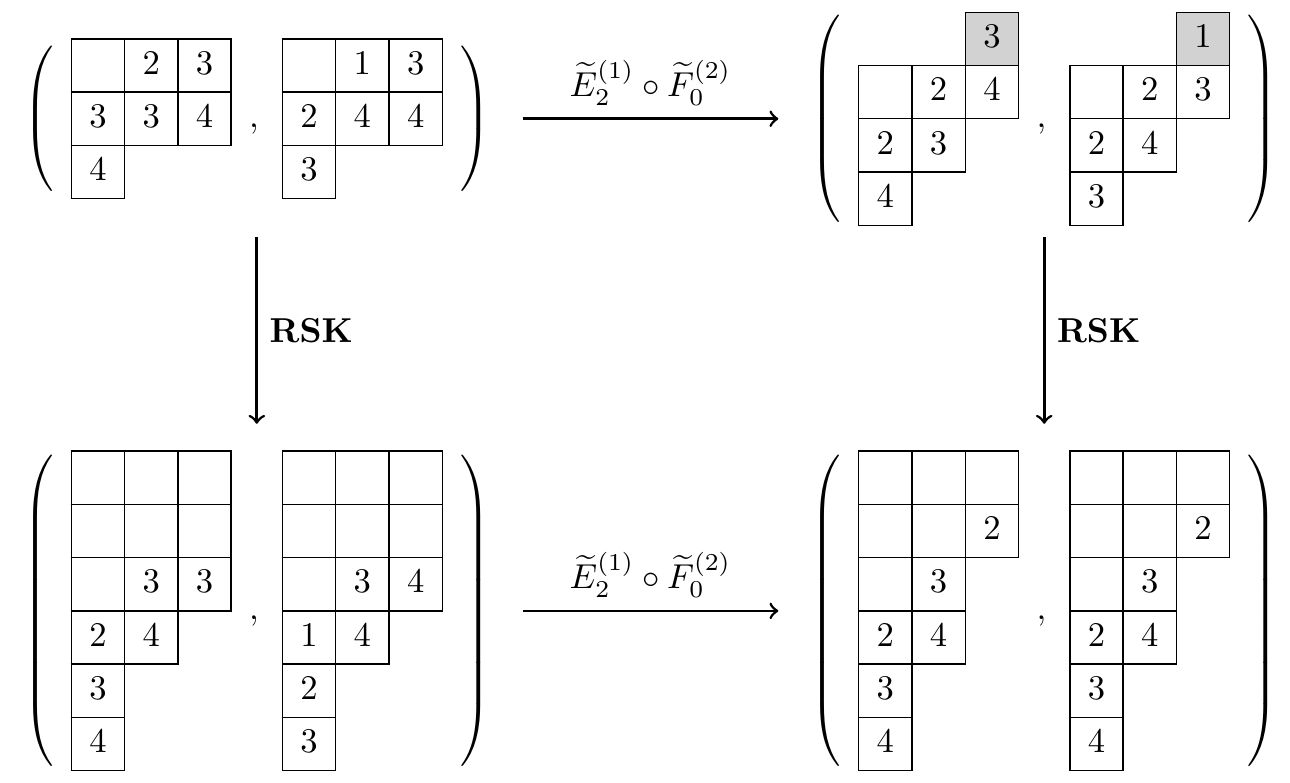}.
    \caption{The skew $\RSK$ map commutes with the two families of Kashiwara operators $\widetilde{E}^{(\epsilon)}_i, \widetilde{F}^{(\epsilon)}_i, i=0,\dots,n-1$, $\epsilon=1,2$.}
    \label{fig:symmetries_RSK}
\end{figure}

\begin{corollary} \label{cor:Phi_is_morphism}
    Projection $\Phi:(P,Q) \mapsto (V,W)$ defined by \eqref{eq:Phi} is a morphism of affine crystal graphs.
\end{corollary}

\begin{proof}
    Composition of morphisms of crystals is clearly a morphism of crystals and so is $\RSK^t$ for any $t$. This implies that $\widetilde{E}^{(\epsilon)}_i \circ \RSK^t(P,Q) = \RSK^t \circ \widetilde{E}^{(\epsilon)}_i (P,Q)$ and hence 
    \begin{equation} \label{eq:phi_kashiwara}
        \widetilde{E}^{(\epsilon)}_i \circ \Phi (P,Q) = \Phi \circ \widetilde{E}^{(\epsilon)}_i(P,Q),    
    \end{equation} 
    for all $i=1,\dots,n-1$, $\epsilon=1,2$, since classical operators $\widetilde{E}^{(\epsilon)}_i$ are defined in the same way for semi-standard tableaux and vertically strict tableaux. To prove that \eqref{eq:phi_kashiwara} holds also for $i=0$, we compare the action of $\iota_1,\iota_2$ and the promotion operator. We call $(P_t,Q_t)=\RSK^{t-1}(P,Q)$, and $(\tilde{P}_t,\tilde{Q}_t)=\iota_2^{-1}(P_t,Q_t)$. Also denote by $\mu$ the asymptotic increment of $(P,Q)$. Then, when $t$ is very large we see that 
    \begin{enumerate}
        \item The contents at column $j$ of $P_t$ and $\tilde{P}_t$ are equal;
        \item Defining $b_j,\tilde{b}_j \in B^{\mu_j',1}$ as $j$-th columns respectively of $Q_t,\tilde{Q}_t$, we have $\tilde{b}_j = \pr (b_j)$, for all $j=1,2,\dots$. 
    \end{enumerate}
    Therefore assuming $\Phi(P,Q)=(V,W)$, we have $\Phi \circ \iota_2^{-1}(P,Q) = (V,\pr (W))$ and by a similar argument $\Phi \circ \iota_1^{-1}(P,Q) = (\pr (V), W)$. Comparing the definition of the 0-th Kashiwara operators for pairs of semi-standard tableaux and vertically strict tableaux we can now conclude that \eqref{eq:phi_kashiwara} holds also for $i=0$ and $\Phi$ is a morphism of $\widehat{\mathfrak{sl}}_n$ bicrystals.
\end{proof}

Result of \cref{cor:Phi_is_morphism} establishes consistency between bicrystal structure for pairs of vertically strict tableaux and that of pairs of semi-standard tableaux. This consideration justifies the following definition.

\begin{definition}
    Let $(V,W) = \Phi (P,Q)$ and consider an operator
    \begin{equation}
        h=\left( \widetilde{E}^{(\epsilon_1)}_{i_1}\right)^{N_1} \circ \left(\widetilde{F}^{(\epsilon_2)}_{i_2}\right)^{N_2} \circ \cdots,
    \end{equation}
    that is an arbitrary composition of Kashiwara operators such that $h(V,W) \neq \varnothing$. Then the action of the operator $h$ is defined also on the pair $(P,Q)$, replacing Kashiwara operators for vertically strict tableaux by the corresponding operators for skew tableaux. Under these assumption $h(P,Q) \neq \varnothing$ and we call such map the \emph{$\Phi$-pullback} of $h$. For simplicity we will not introduce a special notation to denote pullback maps.    
\end{definition}

\begin{remark}
    The $\widehat{\mathfrak{sl}}_n$ bicrystal structure on the set of pairs $(P,Q)$ defines an $\widehat{\mathfrak{sl}}_n$ crystal structure on the set of single semi-standard Young tableaux of generalized shape. This is done defining Kashiwara operators $\widetilde{E}_i(P) = P'$, if $\widetilde{E}_i^{(1)} \circ \widetilde{E}_i^{(2)}  (P,P) = (P',P')$. This yields an affine crystal structure different than the one described in \cite{shimozono_affine}, whose 0-operators were given by \eqref{eq:zero_Kashiwara_op}, where $\mathrm{pr}$ becomes the Lascoux-Sch{\"u}tzenberger promotion operator. For instance we can check that 
    \begin{equation}
        \widetilde{E}_0 : \begin{ytableau}
            \, & & 1 \\ 1 & 2 \\ 2 & 3
        \end{ytableau}
        \mapsto
        \begin{ytableau}
            \, & & 1 \\ & 2 \\ & 3 \\ 2 \\ 3
        \end{ytableau}.
    \end{equation}
\end{remark}

In general it is not simple to describe concretely the effect of the action of $0$-th Kashiwara operators on the shape of skew tableaux $(P,Q)$. The next proposition does this in the case of pairs that are $\RSK$-stable.

\begin{proposition}\label{prop:cells_and_local_energy}
    Let $P,Q \in SST(\lambda/\rho ,n)$ such that $\iota_1^{-1}(P,Q)$ is an $\RSK$-stable pair of tableaux. Define also $\Phi(P,Q)=(V,W)$. Identify $V$ with the tensor product of its columns $v_1 \otimes \cdots \otimes v_N$ and assume that 
    \begin{equation}
        \F{0}(V) = v_1 \otimes \cdots \otimes \F{0}(v_k) \otimes \cdots \otimes v_N, 
    \end{equation}
    for some $k$. Define $(\widetilde{P},\widetilde{Q})=\widetilde{F}_0^{(1)}(P,Q)$ and let $\widetilde{\lambda} / \widetilde{\rho}$ be the shape of $\widetilde{P},\widetilde{Q}$. Then
    \begin{equation} \label{eq:action_0_op_on_shape}
        \widetilde{\lambda}'_j = \lambda'_j + \mathbf{1}_{j=k}
        \qquad
        \text{and}
        \qquad
        \widetilde{\rho}'_j = \rho'_j + \mathbf{1}_{j=k}.
    \end{equation}
    Analogous statements hold for $\widetilde{F}_0^{(2)}, \widetilde{E}^{(1)}_0, \widetilde{E}^{(2)}_0$.
\end{proposition}
\begin{proof}
    Let $(\widehat{P},\widehat{Q})=\iota_1^{-1}(P,Q)$ and call $\widehat{\lambda}/ \widehat{\rho}$ the shape of the pair. Define $\theta^{(i)}(P)=(\theta_1^{(i)},\dots \theta_{\lambda_1}^{(i)})$ as 
    \begin{equation}
        \theta_c^{(i)}(P) = \mathbf{1}_{\{ \textrm{there is an $i$-cell at column $c$ of $P$} \} }.
    \end{equation}
    Then, by definition of $\iota_1^{-1}$ and by the property of $\RSK$-stability of $(\widehat{P},\widehat{Q})=\iota_1^{-1}(P,Q)$, we have
    \begin{equation}
        \widehat{\lambda}'=\lambda' - \theta^{(n)}(P)
        \qquad
        \text{and}
        \qquad
        \widehat{\rho}'=\rho' - \theta^{(n)}(P).
    \end{equation}
    Clearly if $(\widehat{P},\widehat{Q})$ is $\RSK$-stable so is $(\F{i}(\widehat{P}),\widehat{Q})$ whenever $i=1,\dots,n-1$ and $\F{i}(\widehat{P})\neq \varnothing$. Then as above we have 
    \begin{equation}
        \widetilde{\lambda}' = \widehat{\lambda}' + \theta^{(1)}(\F{i}(\widehat{P})) = \lambda' - \theta^{(n)}(P) + \theta^{(1)}(\F{i}(\widehat{P}))
    \end{equation}
    and similarly for $\widetilde{\rho}$, which is exactly the claim \eqref{eq:action_0_op_on_shape}.
\end{proof}
    
\subsection{Matrices $\overline{\mathbb{M}}_{n\times n}$ as affine bicrystals} \label{subs:matrices_affine_bicrystal}

We equip the set of matrices $\overline{\mathbb{M}}_{n\times n}$ of an $\widehat{\mathfrak{sl}}_n$ bicrystal structure, transporting, via the Sagan-Stanley correspondence, the structure on the set of pairs of tableaux discussed in \cref{subs:pairs_of_tab_affine_bicrystals}. Similar investigations were recently pursued in \cite{Gerber_Lecouvey_bicrystals}, where authors considered the set of infinite binary matrices as affine bicrystals. For the case of integral matrices $\mathbb{M}_{n\times n}$ classical bicrystal structure had been discussed in \cite{Danilov_Kolshevoy_bicrystals,vanLeeuwen_crystal_matrices}. 

\medskip

We start defining the action of two families of classical Kashiwara operators on the set of matrices $\overline{\mathbb{M}}_{n\times n}$. For a given $\overline{M}\in \overline{\mathbb{M}}_{n \times n}$ and $i=1,\dots,n-1$, we set
\begin{equation}
    \widetilde{E}_i^{(1)} ( \overline{M} ) = \overline{M}',
\end{equation}
where $\overline{M}'$, as a map on $\mathscr{C}_n$, is 
\begin{equation} \label{eq:signature_rule_matrix}
    \overline{M}' (c) = 
    \begin{cases}
        \overline{M} (\widehat{k} , i+1) - 1 \qquad & \text{if } c \sim_n (\widehat{k},i+1),
        \\
        \overline{M} (\widehat{k} , i) + 1 \qquad & \text{if } c \sim_n (\widehat{k},i),
        \\
        \overline{M} (c) \qquad & \text{else }.
    \end{cases}
\end{equation}
The value of $\widehat{k}$ is determined by the signature rule, that in this case reads
\begin{equation} \label{eq:signature_rule_matrix_2}
    \widehat{k} = \min \left\{ k : \sum_{j = k}^{s} \overline{M} (j,i+1) > \sum_{j = k+1}^{s+1} \overline{M} (j,i), \,\, \forall s \ge k \right\}.
\end{equation}
See \cref{fig:Kashiwara_op_matrices} for an example.
\begin{figure}
    \centering
    \includegraphics{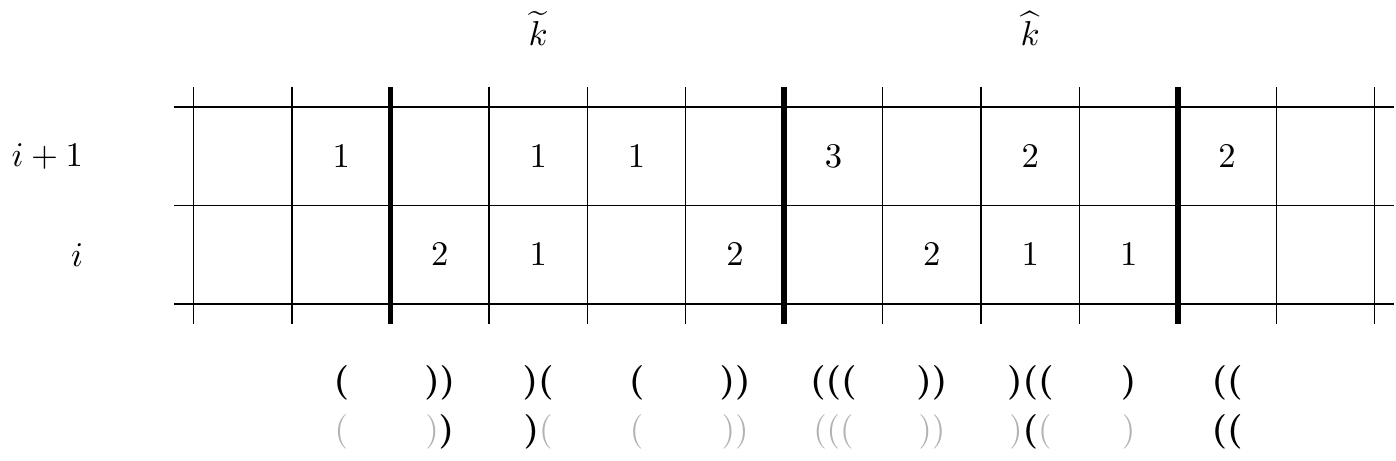}
    \caption{An example of the signature rule determining $\widehat{k}$ and $\widetilde{k}$ as in \eqref{eq:signature_rule_matrix_2}, \eqref{eq:signature_rule_matrix_3}.}
    \label{fig:Kashiwara_op_matrices}
\end{figure}
Analogously we define $\widetilde{F}^{(1)}_i$ as the (partial) inverse of $\widetilde{E}^{(1)}_i$, i.e., 
\begin{equation}
    \widetilde{F}_i^{(1)} ( \overline{M} ) = \overline{M}',
    \qquad
    \overline{M}' (c) = 
    \begin{cases}
        \overline{M} (\widetilde{k} , i+1) + 1 \qquad & \text{if } c \sim_n (\widetilde{k},i+1),
        \\
        \overline{M} (\widetilde{k} , i) - 1 \qquad & \text{if } c \sim_n (\widetilde{k},i),
        \\
        \overline{M} (c) \qquad & \text{else },
    \end{cases}
\end{equation}
where this time 
\begin{equation} \label{eq:signature_rule_matrix_3}
    \widetilde{k} = \max \left\{ k : \sum_{j =s}^k \overline{M} (j,i) > \sum_{j = s-1}^{k-1} \overline{M} (j,i+1), \,\, \forall s \le k \right\}.
\end{equation}
In order to define the $0$-th Kashiwara operators we introduce the shifts
\begin{equation} \label{eq:shift_T}
    T_\epsilon(f)(c) = f(c-\mathbf{e}_\epsilon),
\end{equation}
for $\epsilon=1,2$, acting on any map $f$ from the twisted cylinder $\mathscr{C}_n$. We set
\begin{equation}
    \widetilde{E}^{(1)}_0 = T_2 \circ \widetilde{E}^{(1)}_1 \circ T_2^{-1},
    \qquad
    \widetilde{F}^{(1)}_0 = T_2 \circ \widetilde{F}^{(1)}_1 \circ T_2^{-1},
\end{equation}
The second family of Kashiwara operators is defined by duality,
\begin{equation} \label{eq:Kashiwara_op_matrices_second}
    \widetilde{E}^{(2)}_i(\overline{M}^T) = \widetilde{E}^{(1)}_i(\overline{M})^T,
    \qquad
    \widetilde{F}^{(2)}_i(\overline{M}^T) = \widetilde{F}^{(1)}_i(\overline{M})^T.
\end{equation}

\begin{remark} \label{rem:weighted_biwords_bi_crystals}
    We can translate the definitions for matrices above to those for weighted biwords through identification \eqref{eq:matrix_weighted_biword}.  
    For instance classical Kashiwara operators become
    \begin{gather}
        \widetilde{E}^{(1)}_i (\overline{\pi}) = \overline{\sigma},
        \qquad
        \colon
        \qquad
        q(\overline{\sigma}) = q(\overline{\pi}), 
        \qquad
        w(\overline{\sigma}) = w(\overline{\pi}),
        \qquad
        p(\overline{\sigma}^{\na}) = \E{i} (p(\overline{\pi}^{\na} ) )
    \end{gather}
    and analogously for the $\widetilde{F}^{(1)}_i$ operator for $i=1,\dots, n-1$. The second family $\widetilde{E}^{(2)}_i,\widetilde{F}^{(2)}_i$, again for $i=1,\dots,n-1$, is defined by duality 
    \begin{equation} \label{eq:Kashiwara_op_weighted_biwords_second}
        \widetilde{E}^{(2)}_i (\overline{\pi}^{-1}) = \widetilde{E}^{(1)}_i (\overline{\pi})^{-1},
    \end{equation}
    and similarly for the $F^{(2)}_i$ operators.
\end{remark}

\begin{proposition}
    The two families $\{\widetilde{E}^{(\epsilon)}_i, \widetilde{F}^{(\epsilon)}_i : i=0,\dots,n-1 \}$, $\epsilon=1,2$ defined above equip the set $\overline{\mathbb{M}}_{n \times n}$ of an $\widehat{\mathfrak{sl}}_n$ bicrystal structure.
\end{proposition}

\begin{proof}
    Defining the content functions $\gamma^{(1)},\gamma^{(2)}$ as 
    \begin{equation}
        \gamma^{(1)}_i (\overline{M}) = \sum_{j\in \mathbb{Z}} \overline{M}(j,i),
        \qquad
        \gamma^{(2)}_i (\overline{M}) = \sum_{j\in \mathbb{Z}} \overline{M}(i,j),
    \end{equation}
    it is straightforward to verify that $\widetilde{E}^{(1)}_i, \widetilde{F}^{(1)}_i$ and $\widetilde{E}^{(2)}_i, \widetilde{F}^{(2)}_i$ fulfill hypothesis enumerated in \cref{subs:crystals_and_bi_crystals}. Commutativity of the two families of Kashiwara operators can also be checked directly. This was done for finite integral matrices in \cite{Danilov_Kolshevoy_bicrystals,vanLeeuwen_crystal_matrices}.
\end{proof}

The affine bicrystal structure we impose on the set of matrices $\overline{\mathbb{M}}_{n \times n}$ is by design compatible with the bicrystal structure defined on set of pairs of tableaux. 
    
\begin{proposition} \label{prop:PQ_M_morphism}
    The map \eqref{eq:map_PQ_M} $(P,Q) \xrightarrow[]{\skwRSK\,} \overline{M}$ is a morphism of $\widehat{\mathfrak{sl}}_n$ bicrystals.
\end{proposition}

\begin{proof}
We first show that $(P,Q) \xrightarrow[]{\skwRSK\,} \overline{M}$ is a morphism of $\mathfrak{sl}_n$ bicrystals. For this it is convenient to use the formalism of weighted biwords $\overline{\pi}(\overline{M})$, whose bicrystal structure was given in \cref{rem:weighted_biwords_bi_crystals}, rather than matrices. If $(P,Q) \xrightarrow[]{\skwRSK\,} \overline{\pi}$, call $\pi_P$ the row reading word of $P$ and $p=p(\overline{\pi}^{\na})$. By \cref{prop:jdt_of_PQ} we have $\pi_P \simeq p$, where $\simeq$ denotes the Knuth equivalence. Notice that, the statement in \cref{prop:jdt_of_PQ} requires that tableau $P$ is of classical shape for simplicity, but it is straightforward to understand this assumption is not needed. Since $\pi_P \simeq p$, the map $P \mapsto p$ can be realized as a sequence of jeu de taquin moves identifying the word $p$ with its anti diagonal strip tableau
\begin{equation}
    \begin{ytableau}
        \, & & & p_k
        \\
        & & \dots
        \\
        & p_2
        \\
        p_1
    \end{ytableau}
    .
\end{equation}
This implies, by \cref{prop:commutation_jdt_kashiwara} that $\E{i}(P) \mapsto \E{i}(p)$ and hence
\begin{equation}
    \widetilde{E}^{(1)}(P,Q) = \widetilde{E}^{(1)}(\overline{\pi}),
\end{equation}
for $i=1,\dots,n-1$. The same can be said for family $\widetilde{E}^{(2)}_i, i=1,\dots ,n-1$ and hence $(P,Q) \xrightarrow[]{\skwRSK\,} \overline{\pi}$ is a morphism of $\mathfrak{sl}_n$ bicrystals and so is \eqref{eq:map_PQ_M}. To prove that $(P,Q) \xrightarrow[]{\skwRSK\,} \overline{M}$ is a morphism of affine bicrystals we use \cref{prop:local_completeness}, which along with \cref{thm:matrices_configurations} implies that $\iota_1(P,Q) \xrightarrow[]{\skwRSK\,} T_2 (\overline{M}), \iota_2(P,Q) \xrightarrow[]{\skwRSK\,} T_1 (\overline{M})$. This proves that for both $\epsilon=1,2$,  $\widetilde{E}^{(\epsilon)}_i(P,Q) \mapsto \widetilde{E}^{(\epsilon)}_i(\overline{M})$ also for $i=0$.
\end{proof}

\begin{theorem} \label{thm:symmetries_Viennot}
    The Viennot map $\mathbf{V}$ is an isomorphism of $\widehat{\mathfrak{sl}}_n$ bicrystals.
\end{theorem}

\begin{proof}
    This is a consequence of \cref{prop:PQ_M_morphism}. For any matrix $\overline{M}$ there always exists a pair of tableaux $(P,Q)$ such that, $(P,Q) \xrightarrow[]{\skwRSK \,} \overline{M}$ under map \eqref{eq:map_PQ_M}. This is a consequence of \cref{thm:matrices_configurations} and of construction reported (only for tableaux of classical shape for brevity) in \cref{prop:row_coordinate_classical_pair}. Let $(P',Q')$ be tableaux obtained rigidly shifting $P,Q$ one row up. Then, by \cref{prop:RSK_ca_and_Viennot_ca} we have $(P',Q') \xrightarrow[]{\skwRSK \,} \mathbf{V}(\overline{M})$. Since the transformation $(P,Q) \mapsto (P',Q')$ is realized through a sequence of jeu de taquin moves we have $\widetilde{E}^{(\epsilon)}_i(P,Q) \mapsto \widetilde{E}^{(\epsilon)}_i(P',Q')$ and using \cref{prop:PQ_M_morphism}, $\widetilde{E}^{(\epsilon)}_i \circ \mathbf{V}(\overline{M}) = \mathbf{V} \circ \widetilde{E}^{(\epsilon)}_i(\overline{M})$.
\end{proof}

In \cref{fig:example_symmetries_Viennot} we report an example of commutation relations prescribed by \cref{thm:symmetries_Viennot}.

\begin{figure}[ht]
    \centering
    \includegraphics[scale=.9]{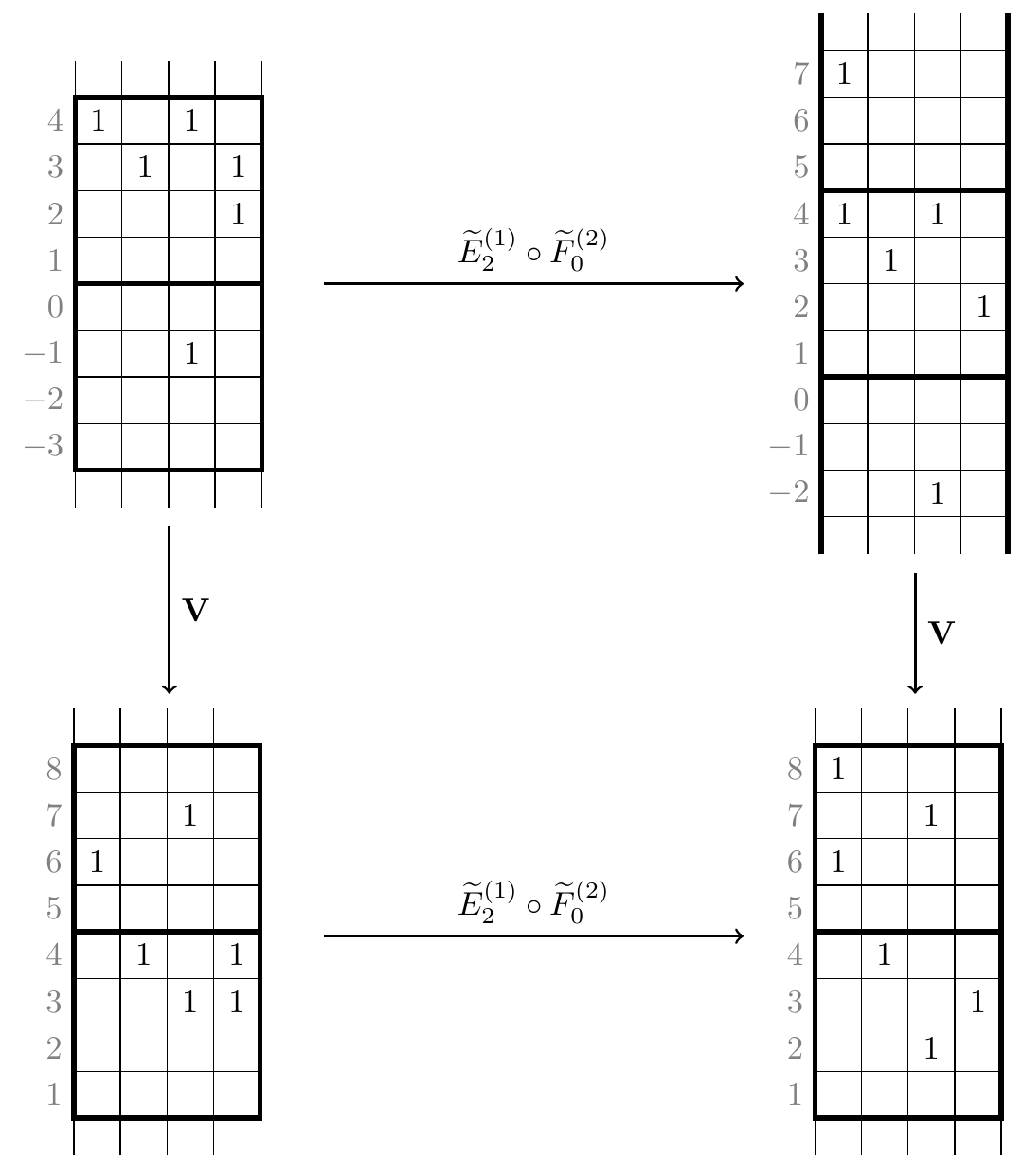}
    \caption{The Viennot map commutes with the two families of Kashiwara operators $\widetilde{E}^{(\epsilon)}_i, \widetilde{F}^{(\epsilon)}_i, i=0,\dots,n-1$, $\epsilon=1,2$.}
    \label{fig:example_symmetries_Viennot}
\end{figure}

\medskip

We could extend the description of crystal operators to matrices of integers $(\alpha, \beta)$. Such considerations will not play important role in this paper and therefore we do not discuss them here. The interested reader can consult \cite{Danilov_2005}, where similar ideas were investigated by the authors.

\section{Generalized Greene invariants} \label{sec:Greene_invariants}

In this section we study increasing subsequences and localized decreasing subsequences of weighted biwords, or of infinite matrices. These were defined in \cref{subs:edge_config_on_cyl} and represent generalizations of the classical increasing and decreasing subsequences which in the RSK correspondence capture the shape of the tableaux.
We show that the maximal increasing and localized decreasing subsequences are invariant under the action of Kashiwara operators in \cref{thm:subsequences_crystals} and that they are preserved by the Viennot map in \cref{thm:subsequences_Viennot}. Their interpretation in the language of tableaux is given in \cref{thm:asymptotic_shape_RSK} and they describe the asymptotic increment of $(P,Q)$ under skew $\RSK$ dynamics. From this last fact we deduce a generalization of Schensted's theorem describing of the first row of the shape of a pair of skew tableaux $(P,Q) \leftrightarrow (\overline{\pi} ; \nu)$ in terms of the longest increasing subsequence of $\overline{\pi}$ and of $\nu_1$.

\subsection{Passage times and subsequences} \label{subs:subsequences}
In \cref{subs:Viennot_asymptotics} we defined increasing and localized decreasing subsequences of a weighted biword $\overline{\pi}$. We now extend these definitions considering decompositions of $\overline{\pi}$ into multiple subsequences.

\begin{definition}
A weighted biword $\overline{\pi} \in \overline{\mathbb{A}}_{n,n}$ is $k$-\emph{increasing} if it can be written as a disjoint union of $k$ increasing weighted biwords $\overline{\pi} = \overline{\pi}^{(1)} \cupdot \cdots \cupdot \overline{\pi}^{(k)}$.
Analogously, $\overline{\pi}$ is \emph{$k$-localized decreasing} if it can be written as a disjoint union of $k$ localized decreasing weighted biwords $\overline{\pi} = \overline{\pi}^{(1)} \cupdot \cdots \cupdot \overline{\pi}^{(k)}$.
\end{definition}

\begin{definition}[Greene invariants] \label{def:Greene_invariants}
     For a weighted biword $\overline{\pi}\in \overline{\mathbb{A}}_{n,n}$ we define statistics
    \begin{align*}
        I_k (\overline{\pi}) & \coloneqq \text{length of the longest $k$-increasing subsequence of $\overline{\pi}$},
        \\
        D_k (\overline{\pi}) & \coloneqq \text{length of the longest $k$-localized decreasing subsequence of $\overline{\pi}$}.    
    \end{align*}
    If $\overline{M} \in \overline{\mathbb{M}}_{n \times n}$ is the matrix corresponding to $\overline{\pi}$ we will denote $I_k (\overline{M})=I_k (\overline{\pi})$ and same for $D_k (\overline{M}) = D_k(\overline{\pi})$.
\end{definition}

It is straightforward to notice that the notion of $I_k(\overline{M})$, defined in \cref{subs:results} in terms of last passage times of up-right paths is equal to the one given in \cref{def:Greene_invariants}.

\begin{remark}
    In case $\overline{\pi}$ is such that $w_i(\overline{\pi}) = 0$ for all $i$ notions of increasing and localized decreasing become respectively the usual notions of increasing and decreasing for words \cite[Chapter 3.3]{sagan2001symmetric}.
\end{remark}

Statistics $I_k(\overline{\pi}), D_k(\overline{\pi})$, as we will prove in \cref{thm:subsequences_crystals} are invariants under the action of Kashiwara operators. Moreover, in \cref{subs:generalized_Knuth_rel} we will define a generalized notion of Knuth relations and we will prove in \cref{thm:Greene_invariants_gen_knuth_rel}, that also with respect to these transformations $I_k,D_k$ are invariants. Therefore, they represent generalization in skew setting of Greene invariants, which should justify the terminology used.

\subsection{Greene invariants and crystal operators}

In classical setting Greene invariants of a biword $\pi$, or equivalently of an integral matrix $M$, are known to be invariant under the action of classical Kashiwara operators. 
This can be proven either using the fact that the RSK algorithm $M \mapsto (P,Q)$ is a morphism of classical crystals and leveraging Greene's theorem \cite{BumpSchilling_crystal_book}, or through a direct check \cite{Danilov_Kolshevoy_bicrystals, vanLeeuwen_crystal_matrices}. In skew/affine setting we find that analogous invariances hold, as stated in the next theorem.

\begin{theorem} \label{thm:subsequences_crystals}
    Let $\overline{\pi} \in \overline{\mathbb{A}}_{n,n}$. Let $h=\widetilde{E}^{(\epsilon)}_i$ or $h=\widetilde{F}^{(\epsilon)}_i$ for some $\epsilon=1,2$, $i=0,1,\dots,n-1$ and assume $h(\overline{\pi}) \ne \varnothing$. Then, for all $k$, we have
    \begin{equation} \label{eq:I_k_D_k_invariants}
        I_k(h(\overline{\pi})) = I_k(\overline{\pi})
        \qquad
        \text{and}
        \qquad
        D_k(h(\overline{\pi})) = D_k(\overline{\pi}).
    \end{equation}
\end{theorem}

Proof of \cref{thm:subsequences_crystals} is reported in \cref{app:proof_thm}. Arguments we use are rather straightforward and consist in direct checks of conservation laws \eqref{eq:I_k_D_k_invariants}. Similar strategies were elaborated in \cite{Danilov_Kolshevoy_bicrystals,vanLeeuwen_crystal_matrices} in classical setting. Compared to these previous works, our approach is conceptually equivalent, although technically more involved.

\subsection{Greene invariants, Viennot map and skew $\RSK$ dynamics}

Here we describe two main results of this section. The first, given in \cref{thm:subsequences_Viennot}, illustrates fundamental conservation laws of the Viennot map $\mathbf{V}$. The second, presented in \cref{thm:asymptotic_shape_RSK}, characterizes the asymptotic shape $\mu(P,Q)$ of a pair of skew tableaux $(P,Q)$ in terms of the Greene invariants $I_k,D_k$. 
This second result gives a generalization of Greene's theorem \cite{Greene1974} in skew setting. Proofs of \cref{thm:subsequences_Viennot} and \cref{thm:asymptotic_shape_RSK} are reported in \cref{subs:proofs_of_two_thm} below.

\begin{theorem} \label{thm:subsequences_Viennot}
    Let $\overline{\pi} \in \overline{\mathbb{A}}_{n,n}$. Then for all $k$ we have
    \begin{equation} 
        I_k(\mathbf{V}(\overline{\pi})) = I_k(\overline{\pi})
        \qquad
        \text{and}
        \qquad
        D_k(\mathbf{V}(\overline{\pi})) = D_k(\overline{\pi}).
    \end{equation}
\end{theorem}

For our next statement associate to each weighted biword $\overline{\pi}$ two a priori different partitions $\mu(\overline{\pi})$ and $\widetilde{\mu}(\overline{\pi})$. They are defined through the Greene invariants of $\overline{\pi}$ as
\begin{gather}
    \mu_1' + \cdots + \mu_k' = D_k(\overline{\pi}), \label{eq:mu_Greene}
    \\
    \widetilde{\mu}_1 + \cdots + \widetilde{\mu}_k = I_k(\overline{\pi}). \label{eq:mu_Greene_2}
\end{gather} 
    
\begin{theorem} \label{thm:asymptotic_shape_RSK}
    Let $P,Q \in SST (\lambda/\rho , n)$ and consider the projection $(P,Q) \xrightarrow[]{\skwRSK \,} \overline{\pi}$. Denote by $\mu(P,Q)$ the asymptotic increment of $(P,Q)$ under skew $\RSK$ dynamics. Then $\mu(P,Q) = \mu(\overline{\pi}) = \widetilde{\mu}(\overline{\pi})$.
\end{theorem}

We will see, in \cref{lemma:Viennot_preserves_Dk} below, that establishing invariance of statistics $D_k$ is relatively straightforward and it follows from an intuitive graphical argument. Proving invariance of the length of longest $k$-increasing subsequences $I_k$ from the shadow line construction seems to be less elementary. Therefore we will prove the slightly less direct fact that partitions $\mu = \widetilde{\mu}$. For this we will take advantage of symmetries with respect to crystal operators, provided by \cref{thm:subsequences_crystals} and we will also evoke the connectedness property of the affine crystal graph $\widehat{B}(\varkappa)$ recalled in \cref{prop:connectedness_crystal_graph}.  

\subsection{An extension of Schensted's theorem}

We present a generalization of Schensted's theorem \cite{Schensted1961}. In classical setting this relates the length of the first row of a pair of straight standard tableaux $(P,Q)$ with the longest increasing subsequence of the corresponding permutation $\pi$.

\begin{theorem}\label{thm:Schensted}
    Let $\overline{\pi} \in \overline{\mathbb{A}}_{n,n}^+$ be a weighted biword and $\nu$ a partition. Let $\lambda/\rho$ be the skew shape of tableaux $(P,Q) \xleftrightarrow[]{\skwRSK \,}(\overline{\pi};\nu)$. Then $\lambda_1 = \nu_1 + I_1(\overline{\pi})$.
\end{theorem}
\begin{proof}
    This is a simple corollary of \cref{thm:asymptotic_shape_RSK}. By \cref{prop:iota_rc_commute_semi} the application of skew $\RSK$ map does not change partition $\nu=\ker(P,Q)$. Let $(P,Q)\xleftrightarrow[]{\rc\,} (\alpha,\beta;\nu)$, then if $(\widetilde{P},\widetilde{Q})=\RSK^t(P,Q)$ and $(\widetilde{\alpha}, \widetilde{\beta})=\RSK^t(\alpha, \beta)$, we have $(\widetilde{P},\widetilde{Q})\xleftrightarrow[]{\rc\,} (\widetilde{\alpha}, \widetilde{\beta};\nu)$. When $t$ is large enough the pair $(\widetilde{P},\widetilde{Q})$ becomes $\RSK$-stable and calling $\widetilde{\lambda} / \widetilde{\rho}$ its shape and $\mu$ the asymptotic increment, we have $\widetilde{\lambda}_1 = \nu_1 + \mu_1$. Since the skew $\RSK$ map does not modify the length of the first row of tableaux $P,Q$ we conclude that $\lambda_1 = \nu_1+\mu_1$ and hence the claim of the theorem by \cref{thm:asymptotic_shape_RSK}.
\end{proof}

\subsection{Proofs of \cref{thm:subsequences_Viennot} and of \cref{thm:asymptotic_shape_RSK}} \label{subs:proofs_of_two_thm}

We start by proving that the Viennot map preserves the length of the longest localized decreasing subsequences.

\begin{lemma} \label{lemma:Viennot_preserves_Dk}
    Let $\overline{\pi} \in \overline{\mathbb{A}}_{n,n}$. Then for all $k$ we have $D_k(\mathbf{V}(\overline{\pi})) = D_k(\overline{\pi}).$
\end{lemma}
\begin{figure}
    \centering
    \includegraphics{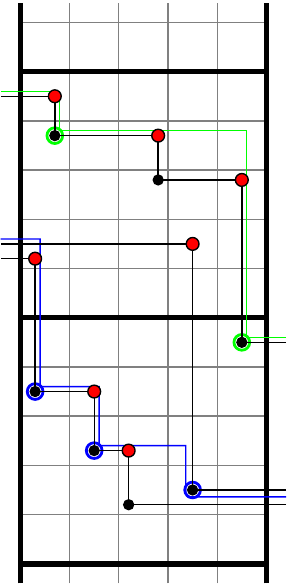}
    \caption{An example of the construction described in the proof of \cref{lemma:Viennot_preserves_Dk}. Green circled black bullets correspond to LDS $\overline{\sigma}^{(1)}$, while blue circled correspond to $\overline{\sigma}^{(2)}$. Red bullets falling on green and blue broken lines determine LDS's $\overline{\xi}^{(1)}$ and $\overline{\xi}^{(2)}$ of $\mathbf{V}(\overline{\pi})$.}
    \label{fig:LDS_and_Viennot_Map}
\end{figure}
\begin{proof}
    Let $\overline{\sigma}=\overline{\sigma}^{(1)}\cupdot \cdots \cupdot \overline{\sigma}^{(k)}$ be a $k$-LDS of $\overline{\pi}$. Consider the shadow line construction of $\overline{\sigma}$, which by \cref{lemma:shadow_lines_number_loops} consists of at most $k$ down right loops $\overline{\varsigma}^{(1)}, \dots, \overline{\varsigma}^{(k)}$
    (note the last few of them could be empty). With no loss of generality we can assume $\overline{\sigma}^{(j)}$ consists of points in $\overline{\sigma} \cap \overline{\varsigma}^{(j)}$ without multiplicity. For instance in \cref{fig:LDS_and_Viennot_Map} black bullets correspond to $\overline{\pi}$ and  $\overline{\sigma}=\overline{\sigma}^{(1)}\cupdot \overline{\sigma}^{(2)}$, where $\overline{\sigma}^{(1)}$ (resp. $\overline{\sigma}^{(2)}$) corresponds to green (resp. blue) circled bullets, while $\overline{\varsigma}^{(1)}$ (resp. $\overline{\varsigma}^{(2)}$) identifies the green (resp. blue) broken line. We now compute the shadow line construction of $\overline{\pi}$ to determine $\mathbf{V}(\overline{\pi})$ and we define $\overline{\xi}^{(j)} = \mathbf{V}(\overline{\pi}) \cap \overline{\varsigma}^{(j)}$, again without multiplicity. If $[\overline{\sigma}^{(j)}_m]=(a_1,a_2), \overline{\sigma}^{(j)}_{m+1}=(b_1,b_2)$ are two consecutive points of $\overline{\sigma}^{(j)}$, then the union of two segments $[\overline{\sigma}^{(j)}_m] \to (b_1,a_2) \to [\overline{\sigma}^{(j)}_{m+1}]$ necessarily hosts at least one point of $\mathbf{V}(\overline{\pi})$. Since there are $\ell ( \overline{\sigma}^{(j)} )$ such pairs (the last and the first point are consecutive by periodicity) we conclude that $\ell ( \overline{\sigma}^{(j)} ) \le \ell ( \overline{\xi}^{(j)} )$. See the example of \cref{fig:LDS_and_Viennot_Map} where $\overline{\xi}^{(1)}$ and $\overline{\xi}^{(2)}$ are given by red bullets lying respectively on the green and blue broken lines. Therefore defining $\overline{\xi}=\overline{\xi}^{(1)} \cupdot \cdots \cupdot \overline{\xi}^{(k)}$ we have $\ell ( \overline{\sigma} ) \le \ell ( \overline{\xi} )$ and in general $D_k(\overline{\pi}) \le D_k(\mathbf{V}(\overline{\pi}))$ for all $\overline{\pi}$.
    An analogous argument shows that the same monotonicity property holds for the map $\mathbf{V}^{-1}$, which is realized through a shadow line construction inverse to that of $\mathbf{V}$. Therefore we have $D_k(\overline{\pi}') \le D_k(\mathbf{V}^{-1}(\overline{\pi}'))$ for all $\overline{\pi}'$. Combining the two inequalities we find $D_k(\overline{\pi}) \le D_k(\mathbf{V}(\overline{\pi})) \le D_k(\overline{\pi})$, which completes the proof.
\end{proof}

\begin{proposition} \label{prop:mu_equals_mu}
    Adopting the notation of \cref{thm:asymptotic_shape_RSK}, we have $\mu(P,Q)=\mu(\overline{\pi})$.
\end{proposition}

\begin{proof}
    We will match $\mu(P,Q)$ with $\mu(\mathbf{V}^t(\overline{\pi}))$ a large enough $t$. This will prove our claim since by \cref{lemma:Viennot_preserves_Dk} we have $\mu(\overline{\pi}) = \mu(\mathbf{V}^t(\overline{\pi}))$ for all $t$. Let $\overline{\pi}^{(t)}$ be the Viennot dynamics with initial data $\overline{\pi}$ and analogously let $(P_t,Q_t)$ be the skew $\RSK$ dynamics with initial data $(P,Q)$. In \cref{prop:Viennot_asymptotic} we have proven that for a large enough $t^*$, there exists weighted biwords $\overline{\sigma}^{(1)},\overline{\sigma}^{(2)},\dots$ such that
    \begin{equation}
        \mathbf{V}^{s}\left(\overline{\pi}^{(t^*)} \right) = \mathbf{V}^s(\overline{\sigma}^{(1)}) \cupdot \mathbf{V}^s(\overline{\sigma}^{(2)}) \cupdot \cdots,
    \end{equation}
    for all $s \ge 0$. Moreover $\overline{\sigma}^{(j)}$'s can be further decomposed as $\overline{\sigma}^{(j)}=\overline{\xi}^{(j,1)}\cupdot \overline{\xi}^{(j,2)} \cupdot \cdots$, where $\overline{\xi}^{(j,r)}$'s are localized decreasing sequences of length $\ell ( \overline{\xi}^{(j,r)} ) =\mu_{R_j}'$ and evolve idependently under Viennot dynamics. To prove our theorem we need to show that the longest LDS of $\overline{\sigma}^{(j)}$ is no longer than $\mu_{R_j}'$, implying that $D_s(\overline{\sigma}^{(j)}) = s \times \mu_{R_j}'$, for $s=1,\dots,r_j$.
    Since point configurations corresponding to each $\overline{\sigma}^{(j)}$ are far apart in $\mathscr{C}_n$, each LDS of $\overline{\pi}$ is necessarily contained in one of the $\overline{\sigma}^{(j)}$. This implies that
    \begin{equation}
        D_k(\overline{\pi}) = \sum_{i=1}^j D_{r_i}(\overline{\sigma}^{(i)})+D_{k-R_j}(\overline{\sigma}^{(j)})
    \end{equation}
    when $R_j < k \leq R_{j+1}$, which proves that $\mu(P,Q) = \mu(\overline{\pi})$. 
        
    To prove the claimed bound on the lenght of LDS's of $\overline{\sigma}^{(j)}$ we utilize an argument similar to the one presented in the proof of \cref{lemma:Viennot_preserves_Dk}. The guiding principle here is that in the Viennot dynamics longer LDS's are ``slower" than shorter ones.
    Define the upward translation in $\mathscr{C}_n$ as 
    \begin{equation}
        \mathcal{T}:(a,b) \to (a, b+n). 
    \end{equation}
    By the fact that columns of tableaux $P_t,Q_t$, for $t$ large enough evolve autonomously and during any update their $c$-th columns receive a downward shift of $\mu_c'$ cells we have
    \begin{equation} \label{eq:free_ev_sigma_j}
        \mathbf{V}^{\mu_{R_j}'} (\overline{\sigma}^{(j)}) = \mathcal{T}(\overline{\sigma}^{(j)}).
    \end{equation}
    Assume that there exists an LDS $\overline{\eta}$ of $\overline{\sigma}^{(j)}$ of length $\ell ( \overline{\eta} ) = L > \mu_{R_j}'$. Then, by \cref{lemma:Viennot_preserves_Dk} there will exist an LDS $\overline{\eta}^{(s)} \subset \mathbf{V}^{s}(\overline{\sigma}^{(j)})$ such that $\ell ( \overline{\eta}^{(s)} ) = L$, for all $s \ge 1$. Additionally elements of $\overline{\eta}^{(1)}$ can be assumed to lie on the only down right loop $\varsigma$ resulting from the shadow line construction of points of $\overline{\eta}$. In particular $\overline{\eta}^{(1)}$ lies weakly ``below'' $\mathbf{V}(\overline{\eta})$ in the sense that if
    \begin{equation}
        \overline{\eta}^{(1)} = (a_1',b_1') \to \cdots \to (a_L',b_L')
        \qquad
        \text{and}
        \qquad
        \overline{\eta} = (a_1,b_1) \to \cdots \to (a_L,b_L),
    \end{equation}
    then $b_1'\le b_1, \dots, b_L'\le b_L$. Inductively one can show that for any $s$, $\overline{\eta}^{(s)}$ lies ``below" $\mathbf{V}^{s}(\overline{\eta})$ in the same sense. We can now compare the evolution of $\overline{\sigma}^{(j)}$ with that of $\overline{\eta}$ under iteration of the Viennot map. Since $\overline{\eta}$ is a localized decreasing subsequence it is true that
    \begin{equation}
        \mathbf{V}^L(\overline{\eta}) = \mathcal{T}(\overline{\eta}).
    \end{equation}
    Combining this last equality with \eqref{eq:free_ev_sigma_j} we have
    \begin{equation}
        \mathbf{V}^{N L \mu_{R_j'}} (\overline{\sigma}^{(j)}) = \mathcal{T}^{N L} (\overline{\sigma}^{(j)}),
        \qquad
        \mathbf{V}^{N L \mu_{R_j'}} (\overline{\eta}) = \mathcal{T}^{N \mu_{R_j'}} (\overline{\eta}),
    \end{equation}
    which for $N$ large enough implies that $\mathcal{T}^{N L} (\overline{\sigma}^{(j)})$ and $\mathcal{T}^{N \mu_{R_j'}} (\overline{\eta})$ lie far apart from each other, since $L> \mu_{R_j}'$. This is a contradiction since elements of $\overline{\eta}^{(NL \mu_{R_j'})}$ should lie below $\mathcal{T}^{N \mu_{R_j'}} (\overline{\eta})$. Therefore there cannot exist any LDS of $\overline{\sigma}^{(j)}$ strictly longer than $\mu_{R_j}'$. This shows that $D_s(\overline{\sigma}^{(j)}) = s \times \mu_{R_j}'$ and hence completes the proof of our proposition.
\end{proof}

In the following proposition we prove that partitions $\mu, \widetilde{\mu}$ defined through Greene invariants $D_k, I_k$ are in fact equal.

\begin{proposition} \label{prop:mu_equals_mu_tilde}
    Let $\mu, \widetilde{\mu}$ be as in \eqref{eq:mu_Greene}, \eqref{eq:mu_Greene_2}.
    Then, for all $\overline{\pi} \in \mathbb{A}_{n,n}$, we have $\mu(\overline{\pi}) = \widetilde{\mu}(\overline{\pi})$.
\end{proposition}

For the sake of the proof of \cref{prop:mu_equals_mu_tilde}, we introduce now statistics of weighted biwords $\overline{\pi}$ which are ``dual'' to the Greene invariants $D_k$. Let $\mathfrak{D}(\overline{\pi})$ denote the set of decompositions of $\overline{\pi}$ into localised decreasing subsequences
\begin{equation}
    \mathfrak{D}(\overline{\pi}) = \{ \mathfrak{d}=(\overline{\sigma}^{(1)} ,\overline{\sigma}^{(2)}, \dots ):\, \overline{\pi}=\overline{\sigma}^{(1)} \cupdot \overline{\sigma}^{(2)} \cupdot \cdots \text{ and } \overline{\sigma}^{(j)} \text{ is LDS for all } j   \}.
\end{equation}
Given $\mathfrak{d} = (\overline{\sigma}^{(1)} ,\overline{\sigma}^{(2)}, \dots ) \in \mathfrak{D}(\overline{\pi})$  define
\begin{equation}
    g_k (\mathfrak{d}) = \sum_{i\ge 1} \min \left\{ k , \ell ( \overline{\sigma}^{(i)} ) \right\}
\end{equation}
and
\begin{equation}
    G_k(\overline{\pi}) = \min_{\mathfrak{d} \in \mathfrak{D}(\overline{\pi}) } g_k (\mathfrak{d}).
\end{equation}
In words $g_k$ tells us how ``spread out" the decomposition $\mathfrak{d}$ is, as in the summation localized decreasing subsequences longer than $k$ contribute with a penalized weight. On the other hand statistics $G_k$ record how likely it is to decompose $\overline{\pi}$ in the least number of localized decreasing subsequences. We have the following.

\begin{lemma} \label{lemma:equivalence_partition_mu_and_dual}
    Let $\overline{\pi} \in \overline{\mathbb{A}}_{n,n}$, take $\mu$ as in \eqref{eq:mu_Greene} and define $\varkappa$ as
    \begin{gather}
        \varkappa_1 + \cdots + \varkappa_k =  G_k(\overline{\pi}).
    \end{gather}
    Then $\mu = \varkappa$.
\end{lemma}

\begin{proof}
    Let $\overline{\sigma}^{(1)} \cupdot \cdots \cupdot \overline{\sigma}^{(k)}$ be a maximising $k$-LDS of $\overline{\pi}$, or in other words assume that $\ell ( \overline{\sigma}^{(1)} ) + \cdots + \ell ( \overline{\sigma}^{(k)} ) = \mu_1' + \cdots + \mu_k'$. It is clear that $\ell ( \overline{\sigma}^{(j)} ) \ge \mu_k'$ for all $j$. Otherwise, say $\ell ( \overline{\sigma}^{(k)} ) \le \mu_k'-1$, then this would imply that $\ell ( \overline{\sigma}^{(1)} ) + \cdots + \ell ( \overline{\sigma}^{(k-1)} ) \ge \mu_1' + \cdots + \mu_{k-1}' +1$, which contradicts the definition of $\mu$. Let $\mathfrak{d} = (\overline{\sigma}^{(1)},\dots,\overline{\sigma}^{(k)},\overline{\eta}^{(k+1)},\dots)
    \in\mathfrak{D}(\overline{\pi})$ where $\overline{\eta}^{(l)}, l>k$ are LDS's 
    formed with elements of $\overline{\pi} \setminus \overline{\sigma}$ and note
    $\ell ( \mathfrak{d} ) = |\mu|$.
    Similarly as above each $\overline{\eta}^{(l)}$ has length which is no longer than $\mu_k'$. Then
    \begin{equation}
    \begin{split}
        g_{\mu_k'}(\mathfrak{d}) 
        &= \sum_{i=1}^{k} \min\{ \mu_k' , \ell ( \overline{\sigma}^{(i)} ) \}  + \sum_{i>k} \min\{ \mu_k' , \ell ( \overline{\eta}^{(i)} ) \}
        \\
        &=  k \mu_k' + \mu_{k+1}' + \mu_{k+2}' + \cdots
        \\
        &
        =
        \mu_1 + \cdots + \mu_{\mu_k'}, 
    \end{split}
    \end{equation}
    which implies $\varkappa_1 + \cdots \varkappa_{\mu_k'} \le \mu_1 + \cdots \mu_{\mu_k'}$. Assume that this last inequality can be made strict. This means that we can find $\mathfrak{d}' = (\overline{\xi}^{(1)},\overline{\xi}^{(2)},\dots)$, with LDS's arranged decreasingly in length, such that
    \begin{equation}
        g_{\mu_k'}(\mathfrak{d}') = m \mu_k' + \ell ( \overline{\xi}^{(m + 1)} ) + \ell ( \overline{\xi}^{(m +2)} ) +\cdots <  k \mu_k' + \mu_{k+1}' + \mu_{k+2}' + \cdots,
    \end{equation}
    where $m$ is the number of LDS's $\overline{\xi}^{(j)}$ with length greater or equal to $\mu_{k}'$. Since $\ell ( \overline{\xi}^{(1)} ) + \cdots + \ell ( \overline{\xi}^{(m)} ) \le \mu_1' + \cdots +\mu_m' $ we have
    \begin{equation}
        \mu_{m +1 }' + \mu_{m+2}' +\cdots \le \ell ( \overline{\xi}^{(m +1)} ) + \ell ( \overline{\xi}^{(m +2)} ) +\cdots,
    \end{equation}
    which implies the inequality
    \begin{equation} \label{eq:inequality_mu_varkappa}
        (\mu_{m +1 }' + \mu_{m+2}' +\cdots ) - (\mu_{k +1 }' + \mu_{k+2}' +\cdots ) < (k -m) \mu_k'.
    \end{equation}
    In case $m<k$, \eqref{eq:inequality_mu_varkappa} becomes $\mu_{m + 1 }'+ \cdots + \mu_{k}' < (k - m) \mu_k'$, which is impossible since terms of $\mu$ are weakly decreasing. Alternatively in case $m \ge k$, \eqref{eq:inequality_mu_varkappa} becomes $(m - k) \mu_k' < \mu_{k+1}' +\cdots + \mu_m'$, which is also impossible for the same reason. Therefore $\varkappa_1 + \cdots + \varkappa_{\mu_k'} = \mu_1 + \cdots + \mu_{\mu_k'}$ for all $k$ and this completes the proof.
\end{proof}

\begin{lemma} \label{lemma:mu_tilde_less_mu}
    Let $\overline{\pi} \in \overline{\mathbb{A}}_{n,n}$ and consider partitions $\mu, \widetilde{\mu}$ as in \eqref{eq:mu_Greene}, \eqref{eq:mu_Greene_2}. Then $\widetilde{\mu} \trianglelefteq \mu$, where $``\trianglelefteq"$ is the dominance order $\widetilde{\mu}_1+\cdots + \widetilde{\mu}_k \le \mu_1+\cdots + \mu_k$ for all $k$.
\end{lemma}

\begin{proof}
    Let $\mathfrak{d}=(\overline{\sigma}^{(1)},\overline{\sigma}^{(2)}, \dots ) \in \mathfrak{D}(\overline{\pi})$ and consider $k$ disjoint increasing subsequences $\overline{\eta}^{(1)} ,\dots, \overline{\eta}^{(k)}$ of $\overline{\pi}$. Any intersection $\overline{\eta}^{(i)} \cap \overline{\sigma}^{(j)}$ has at most one element, therefore we have
    \begin{equation}
        \ell ( \overline{\eta}^{(1)} ) + \cdots + \ell ( \overline{\eta}^{(k)} ) \le \sum_{i \ge 1} \min\{ k, \ell ( \overline{\sigma}_i ) \} = g_k(\mathfrak{d}).
    \end{equation}
    Minimizing the right hand side over $\mathfrak{d}$ we obtain that $\ell ( \overline{\eta}^{(1)} ) + \cdots + \ell ( \overline{\eta}^{(k)} ) \le G_k (\overline{\pi})$ and hence $I_k(\overline{\pi}) \le G_k (\overline{\pi})$ maximizing over the choice of increasing subsequences $\overline{\eta}^{(i)}$. We can now use \cref{lemma:equivalence_partition_mu_and_dual} to identify $G_k(\overline{\pi})$ with $\mu_1 + \cdots + \mu_k$, completing the proof.
\end{proof}

\begin{lemma} \label{lemma:gamma_less_mu_tilde}
    Consider $\overline{\pi} \in \overline{\mathbb{A}}_{n,n}$ such that the content of its $p$-word $\gamma=\gamma(p(\overline{\pi}))$. Then, in the dominance order, $\gamma^+ \trianglelefteq \widetilde{\mu}$, where $\widetilde{\mu}$ is given by \eqref{eq:mu_Greene_2}. 
\end{lemma}

\begin{proof}
    Consider the subwords $\overline{\sigma}^{(i)} \subseteq \overline{\pi}$ formed by all elements of $\overline{\pi}$ of the form $\left( \begin{smallmatrix} q_j \\ i \\ w_j \end{smallmatrix} \right)$. Then $\overline{\sigma}^{(i)}$ are increasing subsequences as each cell $[\overline{\sigma}^{(i)}_j]$ is contained in the up-right path $\varpi^{(i)}=\mathbb{Z} \times \{i\}$. Since $\ell ( \overline{\sigma}^{(i)} ) = \gamma_i$ we have, for all $k$
    \begin{equation}
        \gamma_1^+ + \cdots + \gamma_k^+ \le I_k(\overline{\pi})
    \end{equation}
    and hence $\gamma^+ \trianglelefteq \widetilde{\mu}$.
\end{proof}

\begin{lemma} \label{lemma:gamma_equal_mu}
    Consider $\overline{\pi} \in \overline{\mathbb{A}}_{n,n}$ and let $\mu=\mu(\overline{\pi})$ be as in \eqref{eq:mu_Greene}. Then there exists a transformation $h$, which is composition of Kashiwara operators $\widetilde{E}^{(1)}_i,\widetilde{F}^{(1)}_i$ for $i=0,\dots ,n-1$ such that, denoting $\overline{\pi'}= h(\overline{\pi})$, we have $\gamma(p(\overline{\pi}')) = \mu$.
\end{lemma}
\begin{proof}
    Let $(P,Q)$ be a pair of tableaux such that $(P,Q) \xrightarrow[]{\skwRSK \,} \overline{\pi}$ and denote by $\mu$ the asymptotic increment. By \cref{prop:mu_equals_mu} we have $\mu=\mu(\pi)$. The projection $\Phi$ acts on such pair as $\Phi(P,Q) = (V,W)$ with $V,W \in VST(\mu,n)$. Viewed as an affine crystal graph, $CTS(\mu,n)$ is connected, by \cref{prop:connectedness_crystal_graph} and hence there exists a map $h_V = \E{i_1}^{N_1} \circ \F{i_2}^{N_2} \circ \cdots$ such that $h_V(V) = \mu^{\mathrm{lv}}$, where as in \cref{subs:VST_affine_crystals}, $\mu^{\mathrm{lv}}$ denotes the unique vertically strict tableau of shape $\mu$ and content $\mu$. We can lift the action of the map $h_V$ to the set of pairs $(V',W')$ defining $h=h_V \times \mathbf{1} : (V',W') \mapsto (h_V(V'),W')$. Further, as in \cref{subs:pairs_of_tab_affine_bicrystals} we consider the $\Phi$-pullback map of $h$ that acts on pairs of skew tableaux as $h=\left(\widetilde{E}^{(1)}_{i_1}\right)^{N_1} \circ \left(\widetilde{F}^{(1)}_{i_2}\right)^{N_2} \circ \cdots$. Then, by \cref{cor:Phi_is_morphism}, we have $h(P,Q)=(P',Q')$ with $\gamma(P') = \mu$. By the fact that projection $(P,Q)\xrightarrow[]{\skwRSK\,} \overline{\pi}$ is a morphism of bicrystals we define $(P',Q') \xrightarrow[]{\skwRSK\,} \overline{\pi}'$ and $\gamma(p(\overline{\pi}')) = \mu$.
\end{proof}

This leads up to the proof of \cref{prop:mu_equals_mu_tilde}.

\begin{proof}[Proof of \cref{prop:mu_equals_mu_tilde}]
    By \cref{lemma:mu_tilde_less_mu} we have $\widetilde{\mu}(\overline{\pi}) \trianglelefteq \mu (\overline{\pi})$. Moreover, by \cref{lemma:gamma_equal_mu}, we can always find a composition of Kashiwara operators $h$ such that $h(\overline{\pi})=\overline{\pi}'$ and $\gamma(p(\overline{\pi}')) = \mu(\overline{\pi})$. By \cref{thm:subsequences_crystals} we have $\mu(\overline{\pi}') = \mu(\overline{\pi})$, so that by \cref{lemma:gamma_less_mu_tilde} we can conclude that also $\mu(\overline{\pi}) \trianglelefteq \widetilde{\mu} (\overline{\pi})$. Therefore $\mu(\overline{\pi}) = \widetilde{\mu} (\overline{\pi})$.
\end{proof}

We can finally prove our main results of the section.

\begin{proof}[Proof of \cref{thm:subsequences_Viennot}]
    The fact that the Viennot map preserves statistics $D_k$ is the result of \cref{lemma:Viennot_preserves_Dk}. This also implies that the partition $\mu(\overline{\pi})$ defined as in \eqref{eq:mu_Greene} is invariant. The fact that also $I_k$'s are invariant under $\mathbf{V}$ follows from \cref{prop:mu_equals_mu_tilde}, which implies the chain of equalities
    \begin{equation}
        \widetilde{\mu}(\mathbf{V}(\overline{\pi})) = \mu(\mathbf{V}(\overline{\pi})) = \mu (\overline{\pi}) = \widetilde{\mu}(\overline{\pi}).
    \end{equation}
    This concludes the proof.
\end{proof}

\begin{proof}[Proof of \cref{thm:asymptotic_shape_RSK}]
    By \cref{prop:mu_equals_mu} asymptotic increment $\mu(P,Q)$ is always equal to the partition $\mu(\overline{\pi})$, whenever $(P,Q) \xrightarrow[]{\skwRSK\,} \overline{\pi}$. \Cref{prop:mu_equals_mu_tilde} then provides the equalities $\mu(P,Q) = \mu(\overline{\pi}) = \widetilde{\mu}(\overline{\pi})$.
\end{proof}

\section{Energy function, Demazure crystals and linearization of dynamics} \label{sec:linearization}

\subsection{Combinatorial $R$ matrix and energy function} \label{subs:combinatorial_R}

For any $r_1,r_2 \in \mathbb{N}$, crystal graphs $\widehat{B}(r_1,r_2)$ and $\widehat{B}(r_2,r_1)$ are isomorphic, via a unique isomorphism of crystal graphs called \emph{combinatorial $R$-matrix} 
\begin{equation}
    R:B^{r_1,1} \otimes B^{r_2,1} \to B^{r_2,1} \otimes B^{r_1,1}.
\end{equation}
There exists a number of equivalent definitions of $R$ (see \cite{Kang_Kashiwara_et_al,Nakayashiki_Yamada,shimozono_affine}) and the one reported below is a reformulation of the original algorithm by Nakayashiki and Yamada as in Rule 3.10 of \cite{Nakayashiki_Yamada}. Consider $b_i \in B^{r_i,1}$ for $i=1,2$ and we want to find $\tilde{b}_i \in B^{r_i,1}$ such that 
\begin{equation}
    R(b_1 \otimes b_2) = \tilde{b}_2 \otimes \tilde{b}_1,
\end{equation}
by shifting cells from one column to the other. The procedure goes as follows.

\begin{equation} \label{eq:combinatorial_R_def}
    \begin{minipage}{.9\linewidth}
        \begin{enumerate}
            \item Prepare the word $w = 1^{m_1(b_1)} 1^{m_1(b_2)} 2^{m_2(b_1)} 2^{m_2(b_2)} \cdots $ writing in increasing order all entries, with multiplicities, appearing in $b_1$ and $b_2$. Associate to each entry of $b_1$ an opening parenthesis ``(" and to each entry of $b_2$ a closing parenthesis ``)". In case a letter $i$ appears twice in $w$ we follow the convention that the leftmost one belongs to $b_1$.
            \item Sequentially match all pairs of consecutive symbols ``(", ``)". At the end of this process the subword made of unmatched parentheses will have the form $)\cdots ) ( \cdots ($. 
            \item Assuming periodic boundary conditions sequentially match all leftmost unmatched symbols ``)" with rightmost unmatched ``(". Call these \emph{winding pairs} of $b_1 \otimes b_2$.
            \item After matching winding pairs, the list of unmatched parentheses will form a sequence of $r_1 - r_2$ symbols ``(" or $r_2 - r_1$ symbols ``)", depending on whether $r_1 \ge r_2$ or viceversa. Either way swap the orientation of all remaining unmatched parentheses.
            \item Construct $\tilde{b}_1$ from letters of $w$ associated with ``(" symbols and $\tilde{b}_2$ from those associated with ``)".
        \end{enumerate}
    \end{minipage}
\end{equation}

We also define the \emph{energy function} 
\begin{equation}
    H(b_1 \otimes b_2) = \text{number of winding pairs of }b_1 \otimes b_2.
\end{equation}

To have a better understanding of the algorithm for $R$ consider the following example.
\begin{example}\label{example:combinatorial_R}
    In the evaluation below we have $r_1=3$ and $r_2=4$:
\begin{equation} \label{eq:example_combinatorial_R}
    \begin{ytableau}
        2\\3\\6
    \end{ytableau} 
    \,\,
    \otimes
    \,\,
    \begin{ytableau}
        1\\2\\4\\5
    \end{ytableau}
    \xrightarrow[]{\hspace{.3cm} R \hspace{.3cm}}
    \begin{ytableau}
        2\\3\\5\\6
    \end{ytableau} 
    \,\,
    \otimes
    \,\,
    \begin{ytableau}
        1\\2\\4
    \end{ytableau}
    \,\,\,,
\end{equation}
This follows from the construction
    \begin{equation} \label{eq:combinatorial_R_example}
        \begin{matrix}
         1 & 2 & 2 & 3 & 4 & 5 & 6
        \\
         \boldsymbol{)} & \boldsymbol{(} & \boldsymbol{)} & 
         \boldsymbol{(} & \boldsymbol{)} & \boldsymbol{)} &
         \boldsymbol{(}
         \\
         \boldsymbol{)} & \textcolor{gray}{(} & \textcolor{gray}{)} & 
         \textcolor{gray}{(} & \textcolor{gray}{)} & \boldsymbol{)} &
         \boldsymbol{(}
         \\
         \textcolor{gray}{)} & \textcolor{gray}{(} & \textcolor{gray}{)} & 
         \textcolor{gray}{(} & \textcolor{gray}{)} & \boldsymbol{)} &
         \textcolor{gray}{(}
         \\
         \textcolor{gray}{)} & \textcolor{gray}{(} & \textcolor{gray}{)} & 
         \textcolor{gray}{(} & \textcolor{gray}{)} & \boldsymbol{(} &
         \textcolor{gray}{(}
         \\
         \boldsymbol{)} & \boldsymbol{(} & \boldsymbol{)} & 
         \boldsymbol{(} & \boldsymbol{)} & \boldsymbol{(} &
         \boldsymbol{(}
        \end{matrix}
        \,\,\,.
    \end{equation}
Notice that in the forth line we matched the only winding pair of parenthesis as for (3) of \eqref{eq:combinatorial_R_def}. This in particular implies that in this case we have $H(b_1\otimes b_2)=1$.
\end{example}

\begin{remark}
The combinatorial $R$-matrix, as the name suggests, satisfies the Yang-Baxter equation
\begin{equation}
    (R \otimes \mathbf{1}) (\mathbf{1} \otimes R) (R \otimes \mathbf{1}) = (\mathbf{1} \otimes R) (R \otimes \mathbf{1}) (\mathbf{1} \otimes R)
\end{equation}
and can be also defined as the $q \to 0$ limit of the fused $R$-matrix of $\mathcal{U}_q(\widehat{sl}_{n})$. We will not make use of this fact in this paper, but the interested reader can consult \cite{Kang_Kashiwara_et_al,shimozono_affine}.
\end{remark}

We use the notation $R_i$ to denote the $R$-matrix acting only on the $i$-th and $(i+1)$-th component of a tensor product $b_1\otimes \cdots \otimes b_N$ 
\begin{equation} \label{eq:R_matrix_local}
    R_i = \mathbf{1}^{\otimes (i-1)}  \otimes R \otimes \mathbf{1}^{\otimes (N-i)}.
\end{equation}

\begin{proposition} \label{prop:unique_isomorphism}
    Consider a permutation $\sigma = \sigma_{i_1} \cdots \sigma_{i_M}$ written as a product of elementary transpositions $\sigma_i$ that exchange $i$ and $i+1$.
    Fix compositions $\varkappa=(\varkappa_1,\dots,\varkappa_N)$, $\eta=(\varkappa_{\sigma(1)},\dots,\varkappa_{\sigma(N)})$. Then, 
    \begin{equation}
        R_\sigma = R_{i_1} \cdots R_{i_M} : B^\varkappa \to B^\eta
    \end{equation}
    is the unique isomorphism of crystal graphs $B^\varkappa \to B^\eta$.
\end{proposition}

\begin{proof}
    Composition of isomorphisms of crystal graphs is still an isomorphism of crystal graphs. This shows that $R_\sigma:B^\varkappa \to B^\eta$ is an isomorphism of crystal graphs and it is the only one, by \cref{prop:unique_crystal_graph_isomorphism}.
\end{proof}

\medskip

A theorem in \cite{Kang_Kashiwara_et_al} describes how the energy function changes under the action of crystal operators. Assuming that $R(b_1 \otimes b_2)= \tilde{b}_2 \otimes \tilde{b}_1$, we have
\begin{equation} \label{eq:energy_and_kashiwara}
    H(\E{i}(b_1 \otimes b_2))
    =
    \begin{cases}
         H(b_1 \otimes b_2) - 1
        \qquad 
        & \text{if } i=0, \,\varphi_0(b_1) \ge \varepsilon_0(b_2)  \text{ and } \varphi_0(\tilde{b}_2) \ge \varepsilon_0(\tilde{b}_1),
        \\
        H(b_1 \otimes b_2) + 1
        \qquad 
        & \text{if } i=0, \,\varphi_0(b_1) < \varepsilon_0(b_2)  \text{ and } \varphi_0(\tilde{b}_2) < \varepsilon_0(\tilde{b}_1),
        \\
        H(b_1 \otimes b_2) 
        & \text{else}.
    \end{cases}
\end{equation}

Just as we associate an energy function to a tensor product of two elements $b_1 \otimes b_2$, there exists a canonical way of defining an energy on arbitrary finite products $b_1 \otimes \cdots \otimes b_N$. 
\begin{definition}[Intrinsic energy] \label{def:energy}
    Consider a composition $\varkappa=(\varkappa_1,\dots, \varkappa_N)$. For any $b=b_1 \otimes \cdots \otimes b_N \in B^\varkappa$, the \emph{local energies} $\mathscr{H}_i$ and the \emph{intrinsic energy} $\mathscr{H}$ are the functions
    \begin{gather}
        \mathscr{H}_i (b) = \sum_{j=i+1}^N H (b_i^{(j-1)} \otimes b_j),
        \\
        \mathscr{H}(b) = \sum_{i=1}^{N-1} \mathscr{H}_i (b),
    \end{gather}
    where $b^{(i)}_i=b_i$ and $b_i^{(j-1)}$ is defined recursively by 
    \begin{equation}
        R(b_i^{(j-2)} \otimes b_{j-2}) = \tilde{b}_{j-2} \otimes b_i^{(j-1)}.
    \end{equation}
\end{definition}

The intrinsic energy $\mathscr{H}$ is, as all the local energies $\mathscr{H}_i$ are, constant on classical connected components of $\widehat{B}(\varkappa)$. This is a consequence of \eqref{eq:energy_and_kashiwara} and of the fact that the combinatorial $R$-matrix commutes with Kashiwara operators. 

\subsection{Demazure subgraph} \label{subs:demazure_subgraph}

For an affine crystal graph $\widehat{B}(\varkappa)$ we define its \emph{Demazure subgraph} $\widetilde{B}(\varkappa)$. Its set of vertices is $B^\varkappa$, while its edges, called \emph{Demazure arrows}, are defined next.

\begin{definition}[Demazure arrows]
     Let $b \in B^{\varkappa}$. We say that $b \rightarrow \F{i}(b)$ is a Demazure arrow if $i=1,\dots, n-1$, or if $i=0$ and $\varepsilon_0(b) > 0$. Equivalently, $b\rightarrow \E{i}(b)$ is a Demazure arrow if $i=1,\dots,n-1$, or if $i=0$ and $\varepsilon_0(b)>1$.
\end{definition}

\begin{figure}
    \centering
    \includegraphics[width=\linewidth]{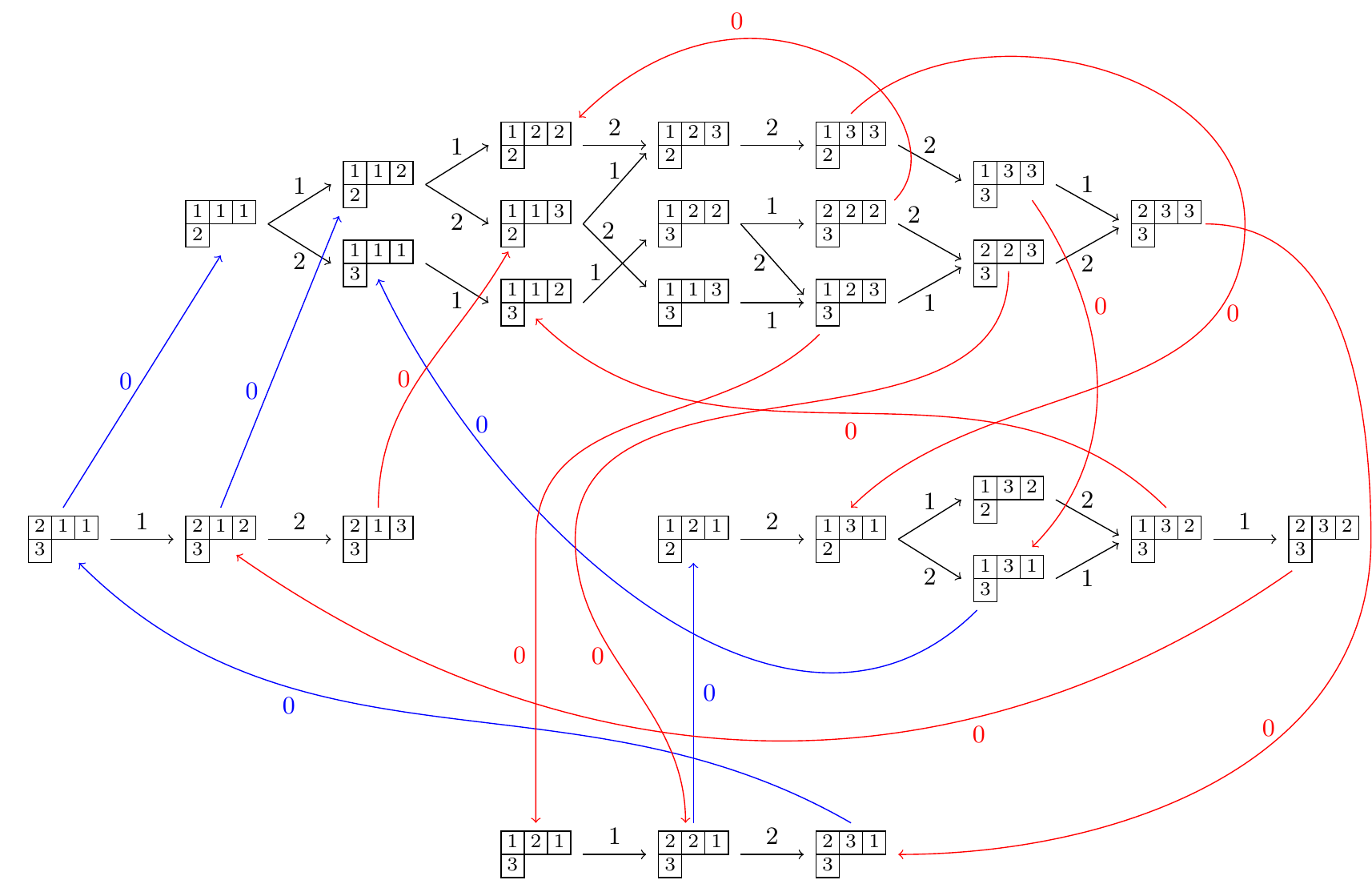}
    \caption{The affine crystal graph $\widehat{B}(\varkappa)$, for $\varkappa=(2,1,1)$. Edges $\xrightarrow[]{i}$ are defined by the action of $\F{i}$. Black arrows define the classical crystal graph $B(\varkappa)$. Blue arrows denote 0-Demazure arrows, so that the Demazure subgraph $\widetilde{B}(\varkappa)$ consists in black and blue edges. Red arrows are 0-arrows that are not Demazure arrows. Notice the defining property of 0-Demazure arrows, that always originate from vertices $b$ that are endpoints of $0$-arrows.}
    \label{fig:demazure_subgraph}
\end{figure}

An example of a Demazure subgraph is shown in \cref{fig:demazure_subgraph}. Notice that the Demazure subgraph is not a crystal graph as in the definition given in \cref{subs:crystals_and_bi_crystals}. In fact $\widetilde{B}(\varkappa)$ is the image under a canonical isomorphsm, denoted by $j$ in \cite{schilling_tingley}, of the corresponding \emph{Demazure crystal} \cite[Chapter 13]{BumpSchilling_crystal_book}. For this reason, $\widetilde{B}(\varkappa)$, just as the affine crystal graph $\widehat{B}(\varkappa)$, is connected, as stated next.

\begin{proposition} \label{prop:Demazure_connected} 
    For any $\varkappa$, the Demazure subgraph $\widetilde{B}(\varkappa)$ is connected.
\end{proposition}

Result of \cref{prop:Demazure_connected} follows from the general result \cite[Theorem 4.4]{Fourier_Schilling_Shimozono_demazure}, \cite[Theorem 6.1]{schilling_tingley}. The structure of subgraph $\widetilde{B}(\varkappa)$ defines a grading function $D:B^{\varkappa}\to \mathbb{Z}$ that associates to any element $b\in B^\varkappa$ the difference between 0-Demazure arrows $\#\F{0} - \#\E{0}$ found in any map $h$ such that $h$ is composition of only Demazure arrows and $h(b)=\varkappa^{\mathrm{lv}}$.
It was found in \cite{schilling_tingley} that such grading function $D(B)$ equals, at least in $A^{(1)}_{n}$ type, the intrinsic energy function $\mathscr{H}(b)$. This is a consequence of the following proposition, which again is a particular case of \cite[Lemma 7.3]{schilling_tingley}

\begin{proposition} \label{prop:Demazure_energy}
    Let $b=b_1\otimes \cdots \otimes b_N \in B^\varkappa$ and 
    \begin{equation}
        b'=\F{0}(b) = b_1 \otimes \cdots \otimes \F{0}(b_k) \otimes \cdots \otimes b_N,
    \end{equation}
   be such that $b\to b'$ is a Demazure arrow. 
   Then, for the local energies $\mathscr{H}_i$, we have
    \begin{equation}
        \mathscr{H}_i(b') = 
        \begin{cases}
            \mathscr{H}_i(b), \qquad &\text{for }i \neq k,
            \\
            \mathscr{H}_k(b)-1, &\text{for }i = k.
        \end{cases}
    \end{equation}
    In particular $\mathscr{H}(b')=\mathscr{H}(b)-1$.
\end{proposition}

We give the following natural definition.

\begin{definition}[Leading map] \label{def:leading_map_VST}
    For any $b\in B^\varkappa$, a \emph{leading map} $\mathcal{L}_b$ is a composition of Demazure arrows $\F{i}, \E{i}, i=0,\dots , n-1$ such that $\mathcal{L}_b (b) = \varkappa^{\, \mathrm{lv}}$.
\end{definition}

\noindent
The fact that for any element $b \in B^\varkappa$ a leading map $\mathcal{L}_b$ exists is a consequence of connectedness of Demazure subgraph stated in \cref{prop:Demazure_connected}.
Leading maps can be visualized as walks on the Demazure subgraph $\widetilde{B}(\varkappa)$ starting at $b$ and terminating at $\varkappa^{\mathrm{lv}}$. For instance, from \cref{fig:demazure_subgraph} we see that a leading map for $b=\begin{ytableau} 2 & 1 & 3 \\ 3 \end{ytableau}$ is given by
\begin{equation} \label{eq:leading_map_example}
    \mathcal{L}_b = \F{0} \circ \E{1} \circ \E{2}.
\end{equation}

\begin{remark} \label{rem:non_uniqueness_leading_map}
     When there are multiple walks from $b$ to $\varkappa^{\mathrm{lv}}$ on $\widetilde{B}(\varkappa)$, the leading map $\mathcal{L}_b$ does not admit a unique expansion in terms of Kashiwara operators.
     However, following the convention on inverse maps established in \cref{subs:crystals_and_bi_crystals}, given $b \in B^{\varkappa}$ and a leading map $\mathcal{L}_b$ it is always true that $b=\mathcal{L}_b^{-1} (\varkappa^{\mathrm{lv}})$.
\end{remark} 

Result of \cref{prop:Demazure_energy} implies that $\mathscr{H}(b)$ is the difference between the number of $\E{0}$ and $\F{0}$ in any leading map $\mathcal{L}_b$. A more precise version of this statement is given by the next proposition, for which we need to prepare some notation. For any $b \in B^\varkappa$, consider a leading map $\mathcal{L}_b = h_m \circ \cdots \circ h_1$, where $h_j$ are Demazure arrows. Denote by $b^{(j)}=h_j \circ \cdots \circ h_1 (b)$ the partial evaluations of $\mathcal{L}_b$ for $j=1,\dots ,m$. Let $u_k(\mathcal{L}_b)$ be the number of 0-Demazure arrows $h_j=\F{0}$ in $\mathcal{L}_b$ such that
    \begin{equation}
        h_j:b^{(j-1)} \mapsto b_1^{(j-1)} \otimes \cdots \otimes \F{0}(b_k^{(j-1)}) \otimes \cdots \otimes b_N^{(j-1)}.
    \end{equation}
    Analogously define $d_k(\mathcal{L}_b)$ as the number of 0-Demazure arrows $h_j=\E{0}$ such that
    \begin{equation}
        h_j:b^{(j-1)} \mapsto b_1^{(j-1)} \otimes \cdots \otimes \E{0}(b_k^{(j-1)}) \otimes \cdots \otimes b_N^{(j-1)}.
    \end{equation}
In other words $u_k(\mathcal{L}_b)$ (resp. $d_k(\mathcal{L}_b)$) counts the number of $\F{0}$ Demazure arrows in $\mathcal{L}_b$ that during the evaluation of $\mathcal{L}_b$ act as $\F{0}$ (resp $\E{0}$) on the $k$-th tensor factor of the argument. The next proposition relates $u_k$ and $d_k$ with the local energy $\mathscr{H}_k$. 
\begin{proposition} \label{prop:local_energies_demazure}
    In the notation introduced above, we have
    \begin{equation}
        \mathscr{H}_k(b) = u_k(\mathcal{L}_b) - d_k(\mathcal{L}_b).
    \end{equation}
\end{proposition}
\begin{proof}
    This follows from \cref{prop:Demazure_energy} and from the fact that $\mathscr{H}_k(\varkappa^{\mathrm{lv}})=0$ for all $k$.
\end{proof}

One can verify \cref{prop:local_energies_demazure} looking at \cref{fig:demazure_subgraph}. Setting $b=\begin{ytableau} 2 & 2 & 1 \\ 3 \end{ytableau}$, local energies can be computed as
\begin{equation}
    \mathscr{H}_1(b) = 1,
    \qquad
    \mathscr{H}_2(b) = 1,
    \qquad
    \mathscr{H}_3(b) = 0,
\end{equation}
either following \cref{def:energy} or checking the action of 0-Demazure arrows throughout any path on the Demazure subgraph connecting $b$ to $\begin{ytableau} 1 & 1 & 1 \\ 2 \end{ytableau}$.

\begin{remark}
    It is not true that for any $b\in B^\varkappa$ there always exist a choice of a leading map $\mathcal{L}_b$ that does not contain $\E{0}$ Demazure arrows. For example, taking
    \begin{equation}
        b=
        \begin{ytableau}
            1
            \\
            2
        \end{ytableau}
        \otimes
        \begin{ytableau}
            1
            \\
            3
        \end{ytableau}
        \otimes
        \begin{ytableau}
            1
            \\
            2
        \end{ytableau}
        \otimes
        \begin{ytableau}
            1
            \\
            2
        \end{ytableau}
        \,
        ,
    \end{equation}
    one can verify that in the classical connected component of $b$ there does not exist $b'$ such that $\varepsilon_0(b')\neq 0$ and $\varphi_0(b') \neq 0$.  Hence in this case $\mathcal{L}_b$ must contain at least one $\E{0}$ operator.
\end{remark}

Given a pair of vertically strict tableaux $(V,W)$ we define the leading map of the pair $\mathcal{L}_{V,W}$ as
\begin{equation} \label{eq:leading_map_pair_VST}
    \mathcal{L}_{V,W} : (V',W') \mapsto (\mathcal{L}_V(V') , \mathcal{L}_{W}(W')),
\end{equation}
whenever the operation is defined. Clearly maps $\mathcal{L}_V$ and $\mathcal{L}_W$ are defined through the usual identification of vertically strict tableaux with crystals.

\subsection{Leading map for pairs of skew tableaux} \label{subs:leading_map_for_tab}

On the affine bicrystal graph of pairs of skew semi-standard tableaux of generalized shape
    \begin{equation}
        \bigcup_{\rho,\lambda} SST(\lambda/\rho,n) \times SST(\lambda/\rho,n),
    \end{equation}
we define the Demazure subgraph similarly as in \cref{subs:demazure_subgraph}. For either $\epsilon=1,2$ we say that $(P,Q)\to \widetilde{F}^{(\epsilon)}_i(P,Q)$ is a Demazure arrow if $i=1,\dots,n-1$ or if $i=0$ and $\widetilde{E}^{(\epsilon)}_0(P,Q) \neq \varnothing$. Analogously $(P,Q)\to \widetilde{E}^{(\epsilon)}_i(P,Q)$ is a Demazure arrow if $i=1,\dots,n-1$ or if $i=0$ and $\left(\widetilde{E}^{(\epsilon)}_0\right)^2(P,Q) \neq \varnothing$. 

We extend the notion of leading map presented for single vertically strict tableaux in \cref{def:leading_map_VST} to pairs of semi-standard tableaux. This produces a new transformation which, as the other notions related to the affine bicrystal structure of pairs of skew tableaux, is new in this paper.

\begin{definition}[Leading map for skew tableaux] \label{def_map_G_PQ}
    Let $P,Q\in SST(\lambda / \rho ,n)$ and consider the projection $(V,W) = \Phi(P,Q)$. A \emph{leading map for the pair} $(P,Q)$, denoted by $\mathcal{L}_{P,Q}$, is defined as the $\Phi$-pullback of a leading map $\mathcal{L}_{V,W}$.
\end{definition}

We report an example of a leading map for a simple pair of skew tableaux.
\begin{example}
    Consider the pair of skew tableaux and the projection 
    \begin{equation}
        (P,Q) = \left( \,\, \begin{ytableau} \, & 1 & 3 \\ & 2 \\ \\ 3 \end{ytableau} \,\, , \,\, \begin{ytableau} \, & 1 & 1 \\ & 3 \\ \\ 3 \end{ytableau} \,\, \right),
        \qquad
        \Phi(P,Q) = (V,W) = \left( \,\, \begin{ytableau} 2 & 1 & 3 \\ 3 \end{ytableau} \,\, , \,\, \begin{ytableau} 1 & 3 & 1 \\ 3 \end{ytableau} \,\,  \right),
    \end{equation}
    which can be easily computed. A possible leading map for $V$ was computed in \eqref{eq:leading_map_example}, whereas a leading map for $W$ is given by $\E{2} \circ \F{0}$ as it can be seen from \cref{fig:demazure_subgraph}. This defines the leading map $\mathcal{L}_{P,Q}$ as
    \begin{equation}
        \mathcal{L}_{P,Q} = \widetilde{E}^{(2)}_2 \circ \widetilde{F}^{(2)}_0 \circ \widetilde{F}^{(1)}_0 \circ \widetilde{E}^{(1)}_1 \circ \widetilde{E}^{(1)}_2,
    \end{equation}
    so that 
    \begin{equation}
        \mathcal{L}_{P,Q} (P,Q) = \left( \,\, \begin{ytableau} \, & 1 & 1 \\ 1 & 2 \end{ytableau} \,\, , \,\, \begin{ytableau} \, & 1 & 1 \\ 1 & 2 \end{ytableau} \,\, \right).
    \end{equation}
\end{example}

\begin{remark} \label{rem:non_uniqueness_leading_map_pair}
    As pointed out in \cref{rem:non_uniqueness_leading_map} a leading map $\mathcal{L}_{P,Q}$ for a pair $(P,Q)$ always exists, but its expression in terms of Kashiwara operators is not unique.
    The non-uniqueness of leading map for a pair $(P,Q)$ might allow, in principle, that the evaluation of two different leading maps $\mathcal{L}_{P,Q}(P,Q)$ and $\mathcal{L}_{P,Q}'(P,Q)$ could give different results. This is indeed not the case as proven in \cref{thm:leading_map_and_leading_tableaux} below. As a result of this fact, the image $(T,T)$ of any leading map $\mathcal{L}_{P,Q}$ does not depend on the realization of $\mathcal{L}_{P,Q}$ and moreover $(P,Q) = \mathcal{L}_{P,Q}^{-1}(T,T)$. We recall that the convention on inverse maps was discussed in \cref{subs:crystals_and_bi_crystals}.
\end{remark}

\subsection{Leading tableaux} \label{subs:leading_tableaux}
In this subsection we aim to characterize the image of a pair $(P,Q)$ of skew tableaux under a leading map $\mathcal{L}_{P,Q}$. This result is reported in \cref{thm:leading_map_and_leading_tableaux}, while its proof is given later in \cref{subs:linearization} as it uses the concept of linearization of the skew $\RSK$ map discussed in the same section. For our description we define the following class of tableaux. 

\begin{definition}[Leading tableaux] \label{def:leading_tableaux}
    A semi-standard tableau $T$ is \emph{leading} if, whenever $T$ has $k$ $i$-cells at row $r$, then it has at least $k$ $(i-1)$-cells at row $r-1$ for all $r$ and $i = 2,3,\dots$. The content of a leading tableau is hence a partition. We denote the set of leading tableau with classical skew shape and with fixed content $\mu$ as $\mathrm{LdT}(\mu)$.
\end{definition}

An example of a leading tableaux is given below in \eqref{eq:leading_tableaux_example}.

\begin{remark}
    To keep the notation simple, below we will focus only on the case where $P,Q$ are tableaux of classical skew shape, leaving the case when their shape is a generalized skew Young diagram as an easy exercise.
\end{remark}

The notion of leading can be translated to matrices, recalling \eqref{eq:alpha_ij_def}. 
A matrix $\alpha\in\mathbb{M}_{n,+\infty}$ is called {\it leading} when it satisfies 
\begin{equation} \label{eq:leading_tab_row_coordinate}
        \alpha_{1,j} \ge \alpha_{2,j+1} \ge \alpha_{3,j+2} \ge \cdots
        \qquad 
        \text{for all } j \in \mathbb{N}.
\end{equation}
The set of leading matrices is denoted by $\mathcal{M}^{\rm Ld}$. To get a more explicit description of leading tableaux, we introduce the following notion.

\begin{definition} \label{def:fundamental_shifts}
    For any partition $\mu$ define the set
    \begin{equation}
        \mathcal{K}(\mu) = \{ (\kappa_1,\dots,\kappa_{\mu_1}) \in \mathbb{N}_0^{\mu_1} : \kappa_i \ge \kappa_{i+1} \text{ if } \mu_i \ge \mu_{i+1} \}.
    \end{equation}
    If $0=R_0,R_1,R_2,\dots$ is a rectangular decomposition of $\mu$ as defined around \cref{eq:rect_dec}, then we will sometimes write elements of $\mathcal{K}(\mu)$ as lists of subarrays $(\kappa^{(1)},\kappa^{(2)},\dots)$, gathering together the weakly decreasing components $\kappa^{(i)}=(\kappa_{R_{i-1}+1}, \dots \kappa_{R_i})$. 
    As usual $\kappa^+$ will denote the unique partition that can be formed sorting elements of $\kappa$ and $|\kappa|=\sum_i \kappa_i$.
\end{definition}

In the following, we will construct a bijection between $\mathcal{K}(\mu) \times \mathbb{Y}$ and 
$\mathrm{LdT}(\mu)$. 
This is most conveniently done through row-coordinate matrices. 
For numbers $k \in \mathbb{N}_0$, $m\in \{1,\dots ,n\}$, define matrices $A(m,k) \in \mathbb{M}_{n,+\infty}$ as 
    \begin{equation}
        A(m,k)_{i,j} = \delta_{i,j-k} \delta_{i\le m}.
    \end{equation}
That is, $A(m,k)$'s only non-zero values are the first $m$ entries in the $(k+1)$-th upper diagonal $A(m,k)_{1,k+1} = \cdots = A(m,k)_{m,k+m} = 1$.
For a given $\kappa \in \mathcal{K}(\mu)$, construct a matrix 
\begin{equation}
\label{eqn:alphaA}
        \alpha = \alpha_{\mu}(\kappa)=\sum_{i=1}^{\mu_1} A(\mu_i', \kappa_i).
\end{equation}        
This is obviously leading, but in fact the opposite is also true as stated in the next proposition. For a given $\mu$, let us denote by $\mathcal{M}^{\rm Ld}(\mu)$ the set of leading matrices $\alpha$ such that $\alpha_{i,1}+\alpha_{i,2}+\cdots=\mu_i$ for all $i$. In other words $\mathcal{M}^{\rm Ld}(\mu)$ represents the set of row coordinate matrices of leading tableaux $T$ with content $\gamma(T)=\mu$. Then we have the following. 

\begin{proposition} \label{lem:alphaA}
   For a given $\mu$, the map $\alpha_{\mu}$ defined in (\ref{eqn:alphaA}) is a bijection between $\mathcal{K}(\mu)$ and $\mathcal{M}^{\rm Ld}(\mu)$. 
\end{proposition}

\begin{proof}
We will indeed establish the bijection between the set 
$\{ (\mu,\kappa),\mu\in\mathbb{Y}, \kappa\in\mathcal{K}(\mu)\}$
and $\mathcal{M}^{\rm Ld}$. Restriction to a fixed $\mu$ gives the bijection in the statement of the lemma. 
We only have to show that any leading matrix $\alpha\in\mathbb{M}_{n,+\infty}$ can be uniquely written in the form of (\ref{eqn:alphaA}). 
We peel off matrix $\alpha$ with the help of the $A(m,k)$'s removing maximal diagonals of non-zero entries. 
Define numbers $k_1,m_1$ as
    \begin{equation}
        k_1=\min \{ k : \alpha_{1,k+1} > 0 \}
        \qquad
        \text{and}
        \qquad
        m_1 = \max\{ m : \alpha_{m,k_1 + m} >0 \}.
    \end{equation}
    and let 
    \begin{equation}
        \alpha^{(1)} = \alpha - A(m_1,k_1).
    \end{equation}
    Then, by \eqref{eq:leading_tab_row_coordinate} and by the fact that $m_1$ is maximal, also $\alpha^{(1)}$ is a leading matrix. We can now recursively construct 
    \begin{equation}
        k_j = \min \{ k : \alpha^{(j-1)}_{1,k+1} > 0 \},
        \qquad
        m_j = \max\{ m : \alpha^{(j-1)}_{m,k_j + m} >0 \}
    \end{equation}
    and $\alpha^{(j)} = \alpha^{(j-1)} - A(m_j,k_j)$, until for some $j'$ we exhaust all positive entries and $\alpha^{(j')}=0$. This proves that there exist $k_1,k_2,\dots $ and $m_1,m_2,\dots$ such that
    \begin{equation}
        \alpha = A(m_1,k_1) + A(m_2,k_2) + \cdots.
    \end{equation}
    It is clear that $\mu_i' = m_{j_i}$ for some $j_1, j_2, \cdots$. To avoid ambiguity we choose $j_i$ such that $k_{j_i} > k_{j_{i+1}}$ whenever $\mu_i' = \mu_{i+1}'$. Defining the new sequence $\tilde{k}_1 = k_{j_1}, \tilde{k}_2 = k_{j_2} , \dots $, we can finally identify $\kappa =(\kappa^{(1)},\kappa^{(2)}, \dots )$ as
    \begin{equation}
        \kappa^{(i)} = ( \tilde{k}_{R_{i-1}+1}, \dots , \tilde{k}_{R_i} ).
    \end{equation}
\end{proof}
As we mentioned below \cref{ex:rc}, there is a bijection $(\alpha; \nu) \xleftrightarrow[]{\rc \,} P$
between a pair $(\alpha,\nu)$ of a row-coordinate matrix $\alpha$ and a partition $\nu$, and a classical tableau $P$.
Restriction to leading ones gives a bijection between $\mathcal{M}^{\rm Ld}(\mu)\times \mathbb{Y}$ and $\mathrm{LdT}(\mu)$. 
Combining this with the bijection between $\mathcal{K}(\mu)$ and $\mathcal{M}^{\rm Ld}(\mu)$ in 
\cref{lem:alphaA}, we get the desired characterization of the set of leading tableaux.   

\begin{proposition} \label{prop:bijection_leading_tab_kappa}
    For a given $\mu\in\mathbb{Y}$, the map 
    \begin{equation} \label{eq:map:kappa_to_T}
        T(\mu, \cdot; \cdot) :\mathcal{K}(\mu) \times \mathbb{Y} \longrightarrow \mathrm{LdT}(\mu),
    \end{equation}
    defined by 
    \begin{equation}
    T(\mu,\kappa ; \nu)\xleftrightarrow{\rc\,}(\alpha;\nu) \quad {\rm with} 
    \quad \alpha = \alpha_{\mu}(\kappa)= \sum_{i=1}^{\mu_1} A(\mu_i', \kappa_i),    
    \end{equation}
    where $\kappa\in\mathcal{K}(\mu), \nu\in\mathbb{Y}$, 
    is a bijection. Moreover, if $\lambda/\rho$ is the shape of $T(\mu,\kappa;\nu)$, then $\rho = (\kappa^+)' + \nu$.
\end{proposition}

\begin{proof}
The first part has already been shown in the arguments above. We are left to check that $(\kappa,\nu) \mapsto T(\mu,\kappa;\nu)$ satisfies relation $\rho=(\kappa^+)' + \nu$. By the fact that $\nu = \ker (T(\mu,\kappa;\nu)) $ we only need to show such property holds for $\nu=\varnothing$. For this notice that if $p^{(1)},p^{(2)},\dots$ are the $1$st, $2$nd, ... row words of $T=T(\mu,\kappa;\varnothing)$, then, by the fact that $T$ is leading we have
    \begin{equation}
        \overlap(p^{(j+1)},p^{(j)}) = \ell( p^{(j+1)} ) - m_1(p^{(j+1)}), 
    \end{equation}
    where $\ell, m_1$ denote respectively the length and the multiplicity of letter $1$ in the word $p^{(j+1)}$. Calling $\eta$ the empty shape of $T$, we have, by \eqref{eq:empty_shape_eta}
    \begin{equation}
        \eta_j - \eta_{j+1} = m_1 (p^{j+1}).
    \end{equation}
    This implies that $\eta = (\kappa^+)'$.
\end{proof}

\begin{example}
    First let us construct a leading tableaux for a given $\mu,\kappa,\nu$. 
    Consider the case $\mu=(4, 2 , 2, 1), \kappa = ((1),(3),(2,1)), \nu = (1,1)$ and for the sake of a better visualization we present these quantities in a colored form as
    \begin{equation}
        \mu= 
        \begin{ytableau}
            *(red!35) & *(blue!35) & *(green!35) & *(orange!35)
            \\
            *(red!35) & *(blue!35)
            \\
            *(red!35) & *(blue!35)
            \\
            *(red!35)
        \end{ytableau}
        \qquad 
        \kappa = ((\textcolor{red}{1}),(\textcolor{blue}{3}),(\textcolor{green}{2},\textcolor{orange}{1})),
        \qquad
        \nu=\begin{ytableau}
            *(gray!35) 
            \\
            *(gray!35)
        \end{ytableau}.
    \end{equation}
    Then we determine, using \eqref{eq:map:kappa_to_T}, the row-coordinate matrix $\alpha$ as
    \begin{equation}
        \alpha = 
        \left( 
        \begin{matrix}
         0 & \textcolor{red}{1} + \textcolor{orange}{1} & \textcolor{green}{1} & \textcolor{blue}{1} & 0 & 0 & 0 & \cdots
         \\
         0 & 0 & \textcolor{red}{1} & 0 & \textcolor{blue}{1} & 0 & 0 & \cdots
         \\
         0 & 0 & 0 & \textcolor{red}{1} & 0 & \textcolor{blue}{1} & 0 & \cdots
         \\
         0 & 0 & 0 & 0 & \textcolor{red}{1} & 0 & 0 & \cdots
        \end{matrix}
        \right)
        =
        \left( 
        \begin{matrix}
         0 & 2 & 1 & 1 & 0 & 0 & 0 & \cdots
         \\
         0 & 0 & 1 & 0 & 1 & 0 & 0 & \cdots
         \\
         0 & 0 & 0 & 1 & 0 & 1 & 0 & \cdots
         \\
         0 & 0 & 0 & 0 & 1 & 0 & 0 & \cdots
        \end{matrix}
        \right),
    \end{equation}
    resulting in the tableau $T \xleftrightarrow[]{\rc\,} (\alpha;\nu)$
    \begin{equation} \label{eq:leading_tableaux_example}
        T(\mu,\kappa;\nu) =
        \begin{ytableau}
            *(blue!35) & *(green!35) & *(gray!35) & *(orange!35) & *(red!35)
            \\
            *(blue!35) & *(green!35) & *(gray!35) & \textcolor{orange}{1} & \textcolor{red}{1}
            \\
            *(blue!35) & \textcolor{green}{1} & \textcolor{red}{2}
            \\
            \textcolor{blue}{1} & \textcolor{red}{3}
            \\
            \textcolor{blue}{2} & \textcolor{red}{4}
            \\
            \textcolor{blue}{3}
        \end{ytableau}
        =
        \begin{ytableau}
            \, &  &  &  &
            \\
            &  &  & 1 & 1
            \\
            & 1 & 2
            \\
            1 & 3
            \\
            2 & 4
            \\
            3
        \end{ytableau}
        .
    \end{equation}
    One can also check that the procedure can be reversed to recover $\mu,\kappa,\nu$
    from the leading tableau above. 
\end{example}

Before ending the subsection, we state a result for the image of $(P,Q)$ under a leading map 
$\mathcal{L}$. 
\begin{theorem} \label{thm:leading_map_and_leading_tableaux}
    Let $P,Q \in SST(\lambda / \rho ,n)$, for $\lambda,\rho \in \mathbb{Y}$. Then for any leading map $\mathcal{L}_{P,Q}$ we have $(T,T)=\mathcal{L}_{P,Q}(P,Q)$, where $T$ is a leading tableau independent of the particular choice of $\mathcal{L}_{P,Q}$. Moreover $T\in \mathrm{LdT}(\mu)$, where $\mu$ is the Greene invariant corresponding to $P,Q$.
\end{theorem}
\noindent
A proof of this theorem will be given in the next subsection because it uses a linearization of 
the skew RSK map for leading tableaux which is explained below.

\subsection{Linearization} \label{subs:linearization}
The property of being leading for a tableau is preserved by the skew $\RSK$ map. Moreover the action of the skew $\RSK$ map on leading tableaux reduces to a simple shift of $\kappa$ by $\mu'$. This is the main result we state and prove in this subsection.

\begin{theorem}[Skew $\RSK$ map on leading tableaux] \label{thm:leading_tab_have_no_phase_shift}
    Let $T=T(\mu,\kappa;\nu)\in \mathrm{LdT}(\mu)$ and $(T',T')=\RSK(T,T)$. Then $T'$ is also a leading tableau and we have $T'=T(\mu,\kappa+ \mu';\nu)$.
\end{theorem}

The proof of \cref{thm:leading_tab_have_no_phase_shift} is based on a direct inspection of the skew $\RSK$ map on a pair $(T,T)$, for a leading tableau $T$. We will need the following preliminary result.

\begin{lemma} \label{lemma:RS_increasing_arrays}
    Fix a weakly decreasing array $\mathpzc{a}=(\mathpzc{a}_1,\dots,\mathpzc{a}_N) \in \mathbb{W}^N$. Let $\mathpzc{b}=(\mathpzc{b}_1,\dots ,\mathpzc{b}_M)$ be a sub-array of $\mathpzc{a}$ with $\mathpzc{b}_i=\mathpzc{a}_{j_i}$ for $j_1 > \cdots > j_M$. Consider the skew $\RS$ map $(\mathpzc{a}',\mathpzc{b}') = \RS (\mathpzc{a},\mathpzc{b})$. Then we have
    \begin{equation}
        \mathpzc{b}'=(\mathpzc{b}_1+1 ,\dots , \mathpzc{b}_M+1 )
    \end{equation}
    and there exists indices $J'=\{j_1', \dots ,j_M'\}$ such that
    \begin{equation}
        \mathpzc{a}' = (\mathpzc{a}_1',\dots,\mathpzc{a}_N')
        \qquad
        \text{with} 
        \qquad
        \mathpzc{a}_k' = \begin{cases}
            \mathpzc{a}_k \qquad & \text{if } k\notin J'
            \\
            \mathpzc{a}_{j_i}+1 \qquad & \text{if } k = j_i'.
        \end{cases}
    \end{equation}
\end{lemma}

\begin{proof}
    Define index $j_1'$ as the position of the leftmost element of $\mathpzc{a}$ equal to $\mathpzc{b}_1$ and then, sequentially define $j_{i+1}' = \min \{ j > j_i : \mathpzc{a}_j = \mathpzc{b}_{i+1} \}$. Since $\mathpzc{b}$ is a sub-array of $\mathpzc{a}$ and $\mathpzc{a}$ is weakly decreasing indices $j_1',\dots, j_M'$ exist. To compute the skew $\RS$ map of $\mathpzc{a}$ and $\mathpzc{b}$ we produce the edge configuration in the rectangle $\{1,\dots ,N \} \times \{1 ,\dots, M \}$ with boundary data 
    \begin{equation}
        \mathsf{W}(1,j) = \mathpzc{b}_j \qquad \text{and} \qquad \mathsf{S}(i,1) = \mathpzc{a}_i.
    \end{equation}
    Consider the edge configuration along the bottom row of the rectangle. From local rules \eqref{eq:local_rules}, and from the fact that $\mathpzc{a}$ is weakly decreasing, it is clear that 
    \begin{gather*}
        \mathsf{S}(1,j) = \mathsf{N}(1,j) = \mathpzc{a}_j,
        \qquad
        \mathsf{E}(1,j) = \mathsf{W}(1,j) = \mathpzc{b}_1 \qquad \text{for } j=1,\dots, j_1'-1,
        \\
        \mathsf{S}(1,j_1') = \mathsf{N}(1,j_1') = \mathpzc{a}_{j_1'} +1,
        \qquad
        \mathsf{E}(1,j_1') = \mathsf{W}(1,j_1') +1  = \mathpzc{b}_1 + 1 ,
        \\
        \mathsf{S}(1,j) = \mathsf{N}(1,j) = \mathpzc{a}_j,
        \qquad
        \mathsf{E}(1,j) = \mathsf{W}(1,j) = \mathpzc{b}_1 +1 \qquad \text{for } j=j_1' + 1,\dots, N,
    \end{gather*}
    where in the second line we used the fact that $\mathpzc{b}_1=\mathpzc{a}_{j_1'}$; see \cref{fig:RS_one_row}. Moreover, by definition of index $j_1'$ the array $(\mathpzc{a}_1,\dots, \mathpzc{a}_{j_1'-1},\mathpzc{a}_{j_1'}+1, \mathpzc{a}_{j_1'+1},\dots \mathpzc{a}_N)$ is still weakly decreasing. We now move on with the evaluation of the edge configuration along the second row of the rectangle. Just as for the first row case we have
    \begin{gather*}
        \mathsf{S}(2,j) = \mathsf{N}(2,j) = \mathpzc{a}_j,
        \qquad
        \mathsf{E}(2,j) = \mathsf{W}(2,j) = \mathpzc{b}_2 \qquad \text{for } j=1,\dots, j_2'-1,
        \\
        \mathsf{S}(2,j_2') = \mathsf{N}(2,j_2') = \mathpzc{a}_{j_2'} +1,
        \qquad
        \mathsf{E}(2,j_2') = \mathsf{W}(2,j_2') +1  = \mathpzc{b}_2 + 1 ,
        \\
        \mathsf{S}(2,j) = \mathsf{N}(2,j) = \mathpzc{a}_j,
        \qquad
        \mathsf{E}(2,j) = \mathsf{W}(2,j) = \mathpzc{b}_2 +1 \qquad \text{for } j=j_2' + 1,\dots, N.
    \end{gather*}
    Repeating the same argument $M$ times yields the proof.
    \begin{figure}[t]
        \centering
        \includegraphics{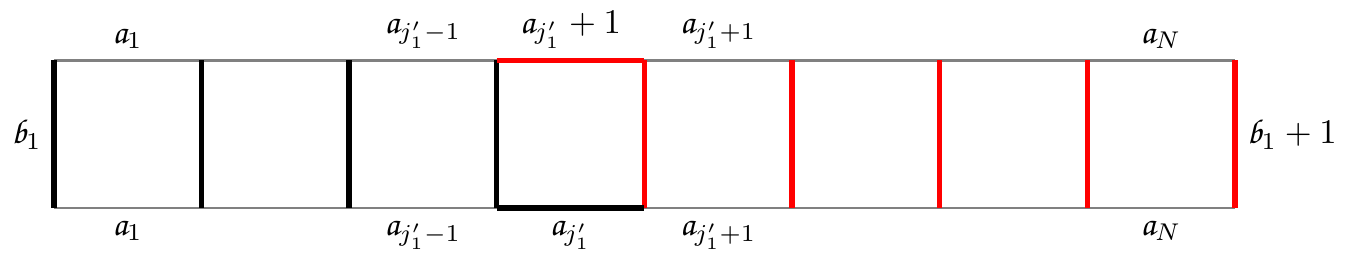}
        \caption{The skew $\RS$ map of a weakly decreasing array $\mathpzc{a}$ with an element $\mathpzc{b}_1 = \mathpzc{a}_{j_1'}$.}
        \label{fig:RS_one_row}
    \end{figure}
\end{proof}

\begin{proof}[Proof of \cref{thm:leading_tab_have_no_phase_shift}]
    We know, by \cref{prop:iota_preserves_ker}, that if $\nu=\ker(T)=\ker(T,T)$, then $\nu=\ker(T')=\ker(T',T')$, so we reduce to prove our statement in the case $\nu = \varnothing$. 
    Let $\alpha=\rc(T)$. To prove our claim we need to compute $(\alpha',\alpha')=\RSK(\alpha,\alpha)$. We use, as usual, standardization and we encode matrix $\alpha$, in an array
    \begin{equation}
        \mathpzc{a}=( \mathpzc{a}^{(1)} ,\dots ,\mathpzc{a}^{(n)} ),
        \qquad
        \mathpzc{a}^{(i)} = ( \mathpzc{a}_{M_{i-1}+1} , \dots , \mathpzc{a}_{M_i}),
    \end{equation}
    with $M_i= \mu_1 + \dots + \mu_i$. Sub-arrays $\mathpzc{a}^{(i)}$ record row-coordinates of $i$-cells of $T$ and are weakly decreasing $\mathpzc{a}_{M_{i-1}+1} \ge  \cdots \ge \mathpzc{a}_{M_i}$. The leading property of $T$ implies that 
    \begin{equation}
        m_r(\mathpzc{a}^{(i)}) \ge m_{r+1}(\mathpzc{a}^{(i+1)}),
    \end{equation}
    for $i,r \ge 1$, where again $m_r$ is the multiplicity of letter $r$ in a word. Moreover, by \eqref{eq:map:kappa_to_T} we have that, as sets
    \begin{equation}
    \begin{aligned}
        &\mathpzc{a}^{(1)} =  \{ \kappa_1 + 1, \dots, \kappa_{\mu_1} + 1\},
        \\
        &\mathpzc{a}^{(2)} =  \{ \kappa_1 + 2, \dots, \kappa_{\mu_2} + 2\},
        \\
        &\dots
        \\
        &\mathpzc{a}^{(n)} =  \{ \kappa_1 + n, \dots, \kappa_{\mu_n} + n\}.
    \end{aligned}
    \end{equation} 
    Let now $\mathpzc{b}= ( \mathpzc{b}^{(1)},\dots ,\mathpzc{b}^{(n)} )$ be the array encoding row-coordinates of cells of $T'$. Or in other words, let $(\mathpzc{b},\mathpzc{b}) = \RS (\mathpzc{a}, \mathpzc{a})$. Then statement of \cref{thm:leading_tab_have_no_phase_shift} reduces to show that, as sets
    \begin{equation} \label{eq:array_b_union_of_d}
        \mathpzc{b}^{(i)} = \{ \kappa_1 + \mu_1 ' + i, \dots, \kappa_{\mu_i} + \mu_{\mu_i}' + i\}
        \qquad
        \text{for } i=1,\dots,n.
    \end{equation}
    We prove \eqref{eq:array_b_union_of_d} by an induction argument over $n$. When $n=1$ we have $\mathpzc{a}=\mathpzc{a}^{(1)}$ and by \cref{lemma:RS_increasing_arrays} $\mathpzc{b}=\mathpzc{b}^{(1)}$ is obtained by adding one to each entry of $\mathpzc{a}$.
    
    Consider now the array $\mathpzc{a}=(\mathpzc{a}^{(1)} ,\dots ,\mathpzc{a}^{(n)})$ and assume, by induction that our theorem holds for any array of $(n-1)$ maximal non-increasing components. In particular define
    \begin{equation}
        \tilde{\mathpzc{a}} = (\mathpzc{a}^{(1)} ,\dots ,\mathpzc{a}^{(n-1)}),
        \qquad
        \tilde{\mathpzc{b}} = (\tilde{\mathpzc{b}}^{(1)} ,\dots ,\tilde{\mathpzc{b}}^{(n-1)}),
        \qquad
        \text{where}
        \qquad
        (\tilde{\mathpzc{b}},\tilde{\mathpzc{b}}) = \RS(\tilde{\mathpzc{a}},\tilde{\mathpzc{a}}).
    \end{equation}
    By inductive hypothesis we have that, as sets
    \begin{equation} 
        \tilde{\mathpzc{b}}^{(i)} = \{ \kappa_1 + \min(\mu_1 ',n-1) + i , \dots, \kappa_{\mu_i} + \min(\mu_{\mu_i}',n-1) + i\}
        \qquad
        \text{for } i=1,\dots,n-1.
    \end{equation}
    
    \begin{figure}[t]
        \centering
        \includegraphics[]{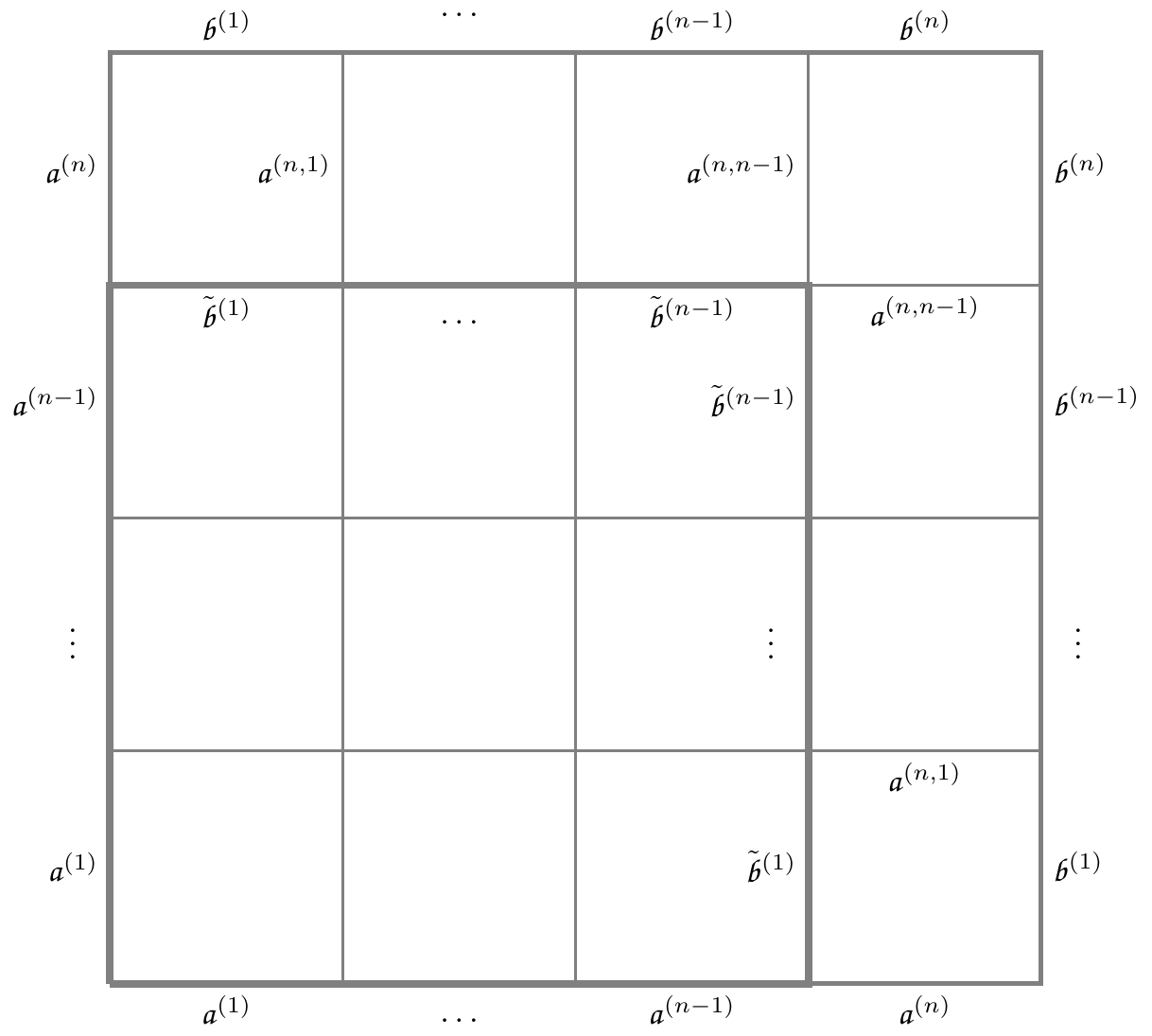}
        \caption{Computation of $(\mathpzc{b}, \mathpzc{b})=\RS(\mathpzc{a},\mathpzc{a})$ by induction over sub-arrays $\mathpzc{a}^{(i)}.$}
        \label{fig:RS_prod_recursive}
    \end{figure}
    
    By definition of skew $\RS$ map, depicted in  \cref{fig:RS_prod_recursive}, it is clear that
    \begin{equation}
    \begin{aligned}
        &\RS(\tilde{\mathpzc{b}}^{(1)},\mathpzc{a}^{(n)}) = (\mathpzc{b}^{(1)},\mathpzc{a}^{(n,1)}),
        \\        &\RS(\tilde{\mathpzc{b}}^{(2)},\mathpzc{a}^{(n,1)}) = (\mathpzc{b}^{(2)},\mathpzc{a}^{(n,2)})
        ,
        \\
        &\dots
        \\
        &\RS(\tilde{\mathpzc{b}}^{(n-1)},\mathpzc{a}^{(n,n-2)}) = (\mathpzc{b}^{(n-1)},\mathpzc{a}^{(n,n-1)})
        \\
        &\RS(\mathpzc{a}^{(n,n-1)},\mathpzc{a}^{(n,n-1)}) = (\mathpzc{b}^{(n)},\mathpzc{b}^{(n)})
        ,
    \end{aligned}
    \end{equation}
    for arrays $\mathpzc{a}^{(n,1)},\dots,\mathpzc{a}^{(n,n-1)}$ and we use these relations to evaluate $\mathpzc{b}^{(1)},\mathpzc{b}^{(2)},\dots$. We start with $\mathpzc{b}^{(1)}$. Since 
    \begin{equation}
        \mathpzc{a}^{(n)} = \{ \kappa_1 + n ,\dots ,\kappa_{\mu_n} + n \}
    \end{equation}
    and
    \begin{equation}
        \tilde{\mathpzc{b}}^{(1)} = \{ \kappa_1 + n ,\dots, \kappa_{\mu_n} + n, \kappa_{\mu_n +1} + \mu_{\mu_n+1}' + 1, \dots , \kappa_{\mu_1} + \mu_{\mu_1}' +1   \},
    \end{equation}
    we can apply \cref{lemma:RS_increasing_arrays} to discover that
    \begin{equation}
        \mathpzc{a}^{(n,1)} = \{ \kappa_1 + n + 1 ,\dots ,\kappa_{\mu_n} + n + 1 \}
    \end{equation}
    and
    \begin{equation}
        \mathpzc{b}^{(1)}=\{ \kappa_1 + \mu_1 ' + 1, \dots, \kappa_{\mu_1} + \mu_{\mu_1}' + 1 \},
    \end{equation}
    which confirms \eqref{eq:array_b_union_of_d} for $i=1$. We can then move to the computation of $\mathpzc{b}^{(2)}$. We have
    \begin{equation}
        \tilde{\mathpzc{b}}^{(2)} = \{ \kappa_1 + n +1,\dots, \kappa_{\mu_n} + n+1, \kappa_{\mu_n +1} + \mu_{\mu_n+1}' + 2, \dots , \kappa_{\mu_1} + \mu_{\mu_1}' +2   \},
    \end{equation}
    so that again, by \cref{lemma:RS_increasing_arrays}, taking the skew $\RS$ map of $\tilde{\mathpzc{b}}^{(2)}$ and $\mathpzc{a}^{(n,1)}$ we find that
    \begin{equation}
        \mathpzc{a}^{(n,2)} = \{ \kappa_1 + n + 2 ,\dots ,\kappa_{\mu_n} + n + 2 \}
    \end{equation}
    and
     \begin{equation}
        \mathpzc{b}^{(2)}=\{ \kappa_1 + \mu_1 ' + 2, \dots, \kappa_{\mu_1} + \mu_{\mu_1}' + 2 \},
    \end{equation}
    confirming \eqref{eq:array_b_union_of_d} for $i=2$. It is clear that we can iterate the same argument for any $i=1,\dots,n-1$ repeatedly using \cref{lemma:RS_increasing_arrays} and obtaining 
     \begin{equation}
        \mathpzc{a}^{(n,i)} = \{ \kappa_1 + n + i ,\dots ,\kappa_{\mu_n} + n + i \},
    \end{equation}
    confirming \eqref{eq:array_b_union_of_d} for all cases except $i=n$. To verify this final case we can easily see, either by direct inspection or through \cref{lemma:RS_increasing_arrays}, that
    \begin{equation}
        \RS(\mathpzc{a}^{(n,n-1)},\mathpzc{a}^{(n,n-1)}) = (\mathpzc{b}^{(n)} , \mathpzc{b}^{(n)})
    \end{equation}
    yields the predicted result \eqref{eq:array_b_union_of_d}. This completes the proof.
\end{proof}

An application of \cref{thm:leading_tab_have_no_phase_shift} gives a proof of \cref{thm:leading_map_and_leading_tableaux}, presented below. We need the following technical lemma stating that 0-Demazure arrows fix the space of pairs of skew tableau with classical shape.

\begin{lemma} \label{lemma:0_Demazure_arrows_preserve_tableaux}
    Let $(P,Q)$ be a pair of skew tableaux with same shape $\lambda / \rho$ with $\lambda,\rho \in \mathbb{Y}$. Define $(\widetilde{P},\widetilde{Q})=h(P,Q)\neq \varnothing$ for $h\in\{\widetilde{E}_0^{(1)},\widetilde{E}_0^{(2)},\widetilde{F}_0^{(1)},\widetilde{F}_0^{(2)} \}$, such that $h$ is a 0-Demazure arrow. Then all cells of $\widetilde{P},\widetilde{Q}$ lie at positive rows and, denoting by $\widetilde{\lambda} / \widetilde{\rho}$ their shape, we have
    \begin{equation} \label{eq:0_Demazure_arrow_last_col}
        \lambda_1= \widetilde{\lambda}_1,
        \qquad
        \rho'_{\lambda_1} = \widetilde{\rho}'_{\lambda_1}.
    \end{equation}
\end{lemma}

\begin{proof}
    We prove our statement only for $h$ being a 0-Demazure arrow $\widetilde{E}_0^{(1)}$ or $\widetilde{F}_0^{(1)}$, as the complementary cases are analogous. We focus first on the case $h=\widetilde{F}^{(1)}_0$. Let $(\widehat{P},\widehat{Q})=\iota_1^{-1}(P,Q)$ and 
    let $\widehat{\pi}$ be the column reading word of $\widehat{P}$. Then, by definition of 0-Demazure arrow we have $\E{1}(\widehat{\pi}),\F{1}(\widehat{\pi})\neq \varnothing$ and we want to understand implications of this fact. As in the signature rule \eqref{alg:e_i} we assign parentheses ``)" and ``(" respectively to each occurrence of a 1 and of a 2 letter in $\widehat{\pi}$. Then, matching consecutive pairs of opening and closing unmatched parentheses we reach a reduced word $)^{\varphi_1(\widehat{\pi})} (^{\varepsilon_1(\widehat{\pi})}$ with $\varphi_1(\widehat{\pi}), \varepsilon_1(\widehat{\pi}) >0$. Each of these umatched ``)" parenteses identifies a different 1 letter in $\widehat{\pi}$, which unambiguously identify a 1-cell in $\widehat{P}$. Call $\theta_1$ the set of $\varphi_1(\widehat{\pi})$ such 1-cells in $\widehat{P}$. Analogously call $\theta_2$ the set of 2-cells of $\widehat{P}$ corresponding to the $\varepsilon_1(\widehat{\pi})$ unmatched ``(" parentheses in the reduced word generated from $\widehat{\pi}$. By the definition of the matching procedure it is clear that each cell of $\theta_1$ lies at a column strictly to the left of any cell of $\theta_2$. In particular no cell of $\theta_1$ occupies the rightmost column of $\widehat{P}$. Since the application of $\F{1}:\widehat{P} \mapsto \F{1}(\widehat{P})$ changes the label of a cell of $\theta_1$ from $1\mapsto 2$ it is clear that $\widetilde{F}^{(1)}_0(P,Q)$ cannot modify the rightmost column of $P,Q$ and hence \eqref{eq:0_Demazure_arrow_last_col} holds.
    
    When $h=\widetilde{E}_0^{(1)}$, we have that both $\E{1}(\widehat{P}),\E{1}^{\,2}(\widehat{P})\neq \varnothing$ and also in this case, using signature rule as explained above, operator $\E{1}$ cannot transform any 2-cell lying at the rightmost column of $\widehat{P}$. This confirms \eqref{eq:0_Demazure_arrow_last_col} and completes the proof.
\end{proof}

\begin{proof}[Proof of \cref{thm:leading_map_and_leading_tableaux}]
    
    By \cref{thm:symmetries_RSK} any $\mathcal{L}_{P,Q}$ commutes with the skew $\RSK$ map. We can therefore write, for any $t$,
    \begin{equation}
    \begin{split}
        (T,T) &= \RSK^{-t} \circ \RSK^t \circ \mathcal{L}_{P,Q}(P,Q) 
        \\
        &= \RSK^{-t} \circ \mathcal{L}_{P,Q} \circ \RSK^t (P,Q).
    \end{split}
    \end{equation}
    Let $(\widetilde{P},\widetilde{Q}) = \RSK^t (P,Q)$. 
    It is clear that $\mathcal{L}_{P,Q}$ is a leading map also for $(\widetilde{P},\widetilde{Q})$, since $\Phi(P,Q)=\Phi(\widetilde{P},\widetilde{Q})$. 
    When $t$ is large, the pair $(\widetilde{P},\widetilde{Q})$ becomes $\RSK$-stable. 
    In such cases the action of the leading map $\mathcal{L}_{P,Q}$ deforms the original shape of $\widetilde{P},\widetilde{Q}$ as prescribed by \cref{prop:cells_and_local_energy}, but it does not changes the number of labeled cells lying at each column. Therefore 
    $(\widetilde{T},\widetilde{T}) = \mathcal{L}_{P,Q}(\widetilde{P},\widetilde{Q})$ clearly defines a leading tableau $\widetilde{T}$. This is because labeled cells at the $i$-th column of $\widetilde{T}$ are exactly $(\mu_i')^{\mathrm{lv}}$. By \cref{prop:bijection_leading_tab_kappa} we write $\widetilde{T}=T(\mu, \widetilde{\kappa};\nu)$, for some uniquely determined $\widetilde{\kappa},\nu$. We want to show that $\widetilde{\kappa},\nu$ are independent of the choice of leading map $\mathcal{L}_{P,Q}$ and subsequently derive that $T$ is a leading tableau with $T=T(\mu,\kappa;\nu)$ where, by \cref{thm:leading_tab_have_no_phase_shift}, we have
    \begin{equation} \label{eq:kappa_and_kappa_tilde}
        \kappa = \widetilde{\kappa} - t \times \mu'.
    \end{equation}
    This will imply that tableaux $T$ is independent of the particular choice of $\mathcal{L}_{P,Q}$ yielding the proof.

    We first observe that partition $\nu$ is independent of $\mathcal{L}_{P,Q}$, since  $\nu=\ker(\widetilde{T})=\ker(\widetilde{P},\widetilde{Q})$. The second equality follows from the general fact that Kashiwara operators $\widetilde{E}^{(\epsilon)}_i,\widetilde{F}^{(\epsilon)}_i$ preserve the kernel of any pair of tableaux as implied by \cref{prop:iota_preserves_ker}. To prove independence of $\widetilde{\kappa}$ from $\mathcal{L}_{P,Q}$ we combine \cref{prop:cells_and_local_energy} and \cref{prop:local_energies_demazure}. Let $\widetilde{\lambda}/\widetilde{\rho}$ be the skew shape of $(\widetilde{P},\widetilde{Q})$ and denote by $\widehat{\lambda}/ \widehat{\rho}$ the skew shape of $\widetilde{T}$. As observed above any leading map $\mathcal{L}_{P,Q}$ in addition to modifying the content of $\widetilde{P},\widetilde{Q}$ can only shift columns rigidly upward or downward. By \cref{prop:local_energies_demazure} we can quantify by how many cells each column gets displaced. Calling $(V,W)=\Phi(P,Q)$ and assuming $\mathcal{L}_{P,Q}$ is the $\Phi$-pullback of the leading map $\mathcal{L}_{V,W}$, we have
    \begin{equation}
        \widehat{\rho}_k' = \widetilde{\rho}_k' - \mathscr{H}_k(V) - \mathscr{H}_k(W).
    \end{equation}
    Such expression for $\widehat{\rho}'$ is independent of $\mathcal{L}_{P,Q}$. By \cref{prop:bijection_leading_tab_kappa} $\widetilde{\kappa}$ is also independent of $\mathcal{L}_{P,Q}$, since it is determined by $\widehat{\rho} = (\widetilde{\kappa}^+)'+\nu$.
    
    In order to complete the proof we want to check that $\kappa$, defined by \eqref{eq:kappa_and_kappa_tilde}, does in fact belong to $\mathcal{K}(\mu)$. Since $\kappa^{(i)}_k = \widetilde{\kappa}^{(i)}_k - t \mu_k'$ it is clear that $\kappa^{(i)}_1 \ge \kappa^{(i)}_2 \ge \cdots $ for all $i=1,2,\dots$, so we only need to verify that $\kappa^{(i)}_k\ge 0$ for all $i$ and $k$. The second statement holds if and only if tableau $T$ does not have cells at non positive rows, condition that is guaranteed by \cref{lemma:0_Demazure_arrows_preserve_tableaux}.
\end{proof}

\section{A new bijection} \label{sec:bijection}

In this section we establish a bijection between a pair of skew tableaux $(P,Q)$ and a quadruple $(V,W;\kappa;\nu)$ consisting of vertically strict tableaux, an array of weights and a partition.  
Combining with the Sagan-Stanley correspondence in \cref{thm:SS}, we also get an RSK type bijection between triples $(V,W;\kappa) \in VST(\mu) \times VST(\mu) \times \mathcal{K}(\mu)$ and weighted permutations $\overline{\pi} \in \overline{\mathbb{A}}_{n,n}^+$, or equivalently matrices $\overline{M}\in \overline{\mathbb{M}}^+_{n \times n}$, with fixed Greene invariant $\mu(\overline{\pi}) = \mu$.

\subsection{The bijection $\Upsilon$} \label{subs:bijection}

We first construct the map $\Upsilon:(P,Q)\mapsto (V,W;\kappa;\nu)$. Subsequently we present the inverse map and in \cref{thm:new_bijection} we prove that the construction is well posed and defines a bijection.

\medskip

\textbf{Map} $\Upsilon : (P,Q) \to (V, W; \kappa;\nu)$
\begin{equation} \label{eq:bijection_from_pi_to_V1_V2_kappa}
    \begin{minipage}{.9\linewidth}
        \begin{enumerate}
            \item Let $\mu$ be the Greene invariant of $(P,Q)$ defined by \cref{def:asymptotic_VST} or equivalently by  \cref{thm:asymptotic_shape_RSK}. Determine vertically strict tableaux $V,W \in VST(\mu)$ iterating the skew $\RSK$ map of $(P,Q)$. In other words set $(V,W) = \Phi (P,Q)$, where projection $\Phi$ was defined in \cref{eq:Phi}.
            \item Let $\mathcal{L}_{P,Q}$ the leading map of the pair $(P,Q)$ as per \cref{def_map_G_PQ} and compute its action
            \begin{equation*}
                 (T,T)=\mathcal{L}_{P,Q}(P,Q).
            \end{equation*}
            \item By \cref{thm:leading_map_and_leading_tableaux}, $T$ is a leading tableau so that $\kappa$ and partition $\nu$ are defined by
            \begin{equation*}
                T=T(\mu, \kappa;\nu),
            \end{equation*}
            following correspondence of \cref{prop:bijection_leading_tab_kappa}.
        \end{enumerate}
    \end{minipage}
\end{equation}

\medskip

\textbf{Map} $\Upsilon^{-1} : (V, W; \kappa;\nu) \to (P,Q)$
\begin{equation} \label{eq:bijection_from_V1_V2_kappa_to_pi}
    \begin{minipage}{.9\linewidth}
        \begin{enumerate}
            \item From $V,W \in VST(\mu)$ define a leading map $\mathcal{L}_{V,W}$ of the pair $(V,W)$, as in \cref{def:leading_map_VST}.
            \item Through correspondence of \cref{prop:bijection_leading_tab_kappa}, from $\kappa, \mu,\nu$ prepare the leading tableau $T=T(\mu,\kappa;\nu)$.
            \item Denoting by $\mathcal{L}$ the $\Phi$-pullback of map $\mathcal{L}_{V,W}$, define skew tableaux $(P,Q)$ as 
            \begin{equation*}
                (P,Q) = \mathcal{L}^{-1} (T,T),
            \end{equation*}
            where the convention on inverse map was discussed in \cref{subs:crystals_and_bi_crystals}.
        \end{enumerate}
    \end{minipage}
\end{equation}

\begin{theorem} \label{thm:new_bijection}
    The map $\Upsilon$ defined by \eqref{eq:bijection_from_pi_to_V1_V2_kappa}, \eqref{eq:bijection_from_V1_V2_kappa_to_pi} is a bijection
    $$\bigcup_{\rho,\lambda \in \mathbb{Y}} SST(\lambda/\rho,n) \times SST(\lambda/ \rho,n)
    \xleftrightarrow[]{\hspace{.4cm} \Upsilon \hspace{.4cm} } \bigcup_{\mu \in \mathbb{Y}} VST(\mu) \times VST(\mu) \times \mathcal{K}(\mu)\times \mathbb{Y}.$$ 
    In particular, if $(P,Q)\xleftrightarrow[]{\Upsilon\,} (V,W;\kappa;\nu)$, and $\rho$ is the empty shape of $P,Q$, we have
    \begin{equation} \label{eq:weight_preserving}
        |\rho| = \mathscr{H}(V) + \mathscr{H}(W) + |\kappa| + |\nu|.
    \end{equation}
\end{theorem}

\begin{proof}
    This theorem is consequence of \cref{thm:leading_map_and_leading_tableaux}, which itself follows from \cref{thm:symmetries_RSK,thm:leading_tab_have_no_phase_shift}. Let us show that $\Upsilon$ is well posed and injective analyzing the three steps in \eqref{eq:bijection_from_pi_to_V1_V2_kappa}. Given a pair $(P,Q)$, the corresponding partition $\mu$ and the asymptotic vertically strict tableaux $(V,W)=\Phi(P,Q)$ are unambiguously defined. Leading map $\mathcal{L}_{P,Q}$ is determined composing leading maps $\mathcal{L}_V$ and $\mathcal{L}_W$ as in \cref{def_map_G_PQ}. As pointed out in \cref{rem:non_uniqueness_leading_map,rem:non_uniqueness_leading_map_pair} the expression of $\mathcal{L}_{P,Q}$ as a combination of Kashiwara operators is not unique. Nevertheless, thanks to \cref{thm:leading_map_and_leading_tableaux}, the tableaux $T$ such that $\mathcal{L}_{P,Q}(P,Q)=(T,T)$ is independent of the particular realization of the leading map and it is a leading tableaux that uniquely identifies the remaining data $\kappa \in \mathcal{K}(\mu)$ and $\nu \in \mathbb{Y}$. This shows that $\Upsilon:(P,Q) \mapsto (V,W;\kappa,\nu)$ is injective.
    
    On the other hand given $(V,W;\kappa,\nu)$ and constructed the leading tableau $T=T(\mu,\kappa;\nu)$ we know, again from \cref{thm:leading_map_and_leading_tableaux}, that the action of the map $\mathcal{L}^{-1}$, defined by (3) of \eqref{eq:bijection_from_V1_V2_kappa_to_pi}, is independent of the particular realization of leading maps $\mathcal{L}_V, \mathcal{L}_W$. This implies that $(P,Q)$ are uniquely determined by the data $(V,W;\kappa,\nu)$ and one can easily see that this operation if the inverse of $\Upsilon$.
\end{proof}

Restrincting bijection $(P,Q) \xleftrightarrow[]{\Upsilon \,} (V,W;\kappa;\nu)$ to the case $\nu = \ker(P,Q)=\varnothing$ and composing with projection induced by the Sagan-Stanley correspondence $(P,Q) \xrightarrow[]{\skwRSK\,} \overline{\pi}$ yields a map $\overline{\pi} \xleftrightarrow[]{\tilde{\Upsilon} \,} (V,W;\kappa)$ more in the spirit of the RSK correspondence. 

\begin{corollary} \label{cor:bijection_pi_V1_V2_kappa}
    Map defined by \eqref{eq:bijection_from_pi_to_V1_V2_kappa}, \eqref{eq:bijection_from_V1_V2_kappa_to_pi} naturally restricts to a content preserving bijection
    \begin{equation*}
        \overline{\mathbb{A}}_{n,n}^+ \xleftrightarrow[]{ \hspace{.4cm} \tilde{\Upsilon} \hspace{.4cm}} 
        \bigcup_{\mu \in \mathbb{Y} : \ell (\mu) \le n} VST(\mu) \times VST(\mu) \times \mathcal{K}(\mu).
    \end{equation*}
    In case $\overline{\pi} \xleftrightarrow[]{\tilde{\Upsilon} \,} (V,W;\kappa)$, we have
    \begin{equation}
        \wt(\overline{\pi}) = \mathscr{H}(V) + \mathscr{H}(W) + |\kappa|.
    \end{equation}
\end{corollary}

\begin{proof}
    One only needs to notice that if $(P,Q) \xleftrightarrow[]{\skwRSK \,} (\overline{\pi};\nu)$ then $(P,Q) \xleftrightarrow[]{\Upsilon \,} (V,W;\kappa,\nu)$ where the partition $\nu$ is equal for both cases. Factoring out information about $\nu$ we are left with the desired bijection.
\end{proof}

\noindent
Clearly this also induces a bijection $\overline{M} \xleftrightarrow[]{\tilde{\Upsilon}} (V,W;\kappa)$, which we denote by the same notation, where  $\overline{M}\in\overline{\mathbb{M}}^+_{n\times n}$.

\subsection{A worked out example} \label{subs:worked_out_example}

In this subsection we present an example of bijection $\Upsilon$ defined in \eqref{eq:bijection_from_pi_to_V1_V2_kappa}. We also take this as an opportunity to review various constructions introduced throughout the text. Let
    \begin{equation}
        P = 
        \begin{ytableau}
            \, & & & 1
            \\
            & 1 & 3
            \\
            2 & 4
        \end{ytableau},
        \qquad
        Q = 
        \begin{ytableau}
            \, & & & 1
            \\
            & 2 & 2
            \\
            1 & 3
        \end{ytableau}
        .
    \end{equation}
    A single iteration of the skew $\RSK$ map yields the pair $(P',Q')=\RSK(P,Q)$ as
    \begin{equation}
        P'
        =
        \begin{ytableau}
            \, & & & 
            \\
            & &
            \\
            & & 1
            \\
            1 & 3
            \\
            2
            \\
            4
        \end{ytableau}
        ,
        \qquad
        Q'
        =
        \begin{ytableau}
            \, & & & 
            \\
            & &
            \\
            & & 2
            \\
            1 & 1
            \\
            2
            \\
            3
        \end{ytableau}
        ,
    \end{equation}
    which is $\RSK$-stable, so that vertically strict tableaux $V,W$ are
    \begin{equation}
        V = 
        \begin{ytableau}
            1 & 3 & 1
            \\
            2
            \\
            4
        \end{ytableau},
        \qquad
        W = 
        \begin{ytableau}
            1 & 1 & 2
            \\
            2
            \\
            3
        \end{ytableau}
        .
    \end{equation}
    A possible leading map for $V$ is 
    \begin{equation}
        \mathcal{L}_{V} = \E{3} \circ \E{2} \circ \F{0} \circ \F{3} \circ \F{2},
    \end{equation}
    since
    \begin{equation} \label{eq:example_leading_map}
        \begin{ytableau}
            1 & 3 & 1
            \\
            2
            \\
            4
        \end{ytableau}
        \xrightarrow[]{\hspace{.2cm}\F{2}\hspace{.2cm}}
        \begin{ytableau}
            1 & 3 & 1
            \\
            3
            \\
            4
        \end{ytableau}
        \xrightarrow[]{\hspace{.2cm}\F{3}\hspace{.2cm}}
        \begin{ytableau}
            1 & 4 & 1
            \\
            3
            \\
            4
        \end{ytableau}
        \xrightarrow[]{\hspace{.2cm}\F{0}\hspace{.2cm}}
        \begin{ytableau}
            1 & 1 & 1
            \\
            3
            \\
            4
        \end{ytableau}
        \xrightarrow[]{\hspace{.2cm}\E{2}\hspace{.2cm}}
        \begin{ytableau}
            1 & 1 & 1
            \\
            2
            \\
            4
        \end{ytableau}
        \xrightarrow[]{\hspace{.2cm}\E{3}\hspace{.2cm}}
        \begin{ytableau}
            1 & 1 & 1
            \\
            2
            \\
            3
        \end{ytableau}
        .
    \end{equation}
    Notice that the $0$-th Kashiwara operators $\F{0}$ is, in this particular case, a Demazure arrow. Leading map for $W$ is even simpler and we can take $\mathcal{L}_{W} = \E{1}$. Combining $\mathcal{L}_{V}$ and $\mathcal{L}_{W}$ we produce the leading map for the pair $(P,Q)$,
    \begin{equation}
        \mathcal{L}_{P,Q} = \widetilde{E}^{(1)}_3 \circ \widetilde{E}^{(1)}_2 \circ \widetilde{F}^{(1)}_0 \circ \widetilde{F}^{(1)}_3 \circ \widetilde{F}^{(1)}_2 \circ \widetilde{E}^{(2)}_1,
    \end{equation}
    whose action can be computed as
    \begin{equation}
        \begin{split}
            &
            \left(
            \begin{ytableau}
            \, & & & 1
            \\
            & 1 & 3
            \\
            2 & 4
            \end{ytableau}
            ,
            \begin{ytableau}
            \, & & & 1
            \\
            & 2 & 2
            \\
            1 & 3
            \end{ytableau}
            \right)
            \xrightarrow[]{\widetilde{F}^{(1)}_2 \circ \widetilde{E}^{(2)}_1}
            \left(
            \begin{ytableau}
            \, & & & 1
            \\
            & 1 & 3
            \\
            3 & 4
            \end{ytableau}
            ,
            \begin{ytableau}
            \, & & & 1
            \\
            & 1 & 2
            \\
            1 & 3
            \end{ytableau}
            \right)
            \xrightarrow[]{ \widetilde{F}^{(1)}_3 }
            \left(
            \begin{ytableau}
            \, & & & 1
            \\
            & 1 & 3
            \\
            3 & 4
            \end{ytableau}
            ,
            \begin{ytableau}
            \, & & & 1
            \\
            & 1 & 2
            \\
            1 & 3
            \end{ytableau}
            \right)
            \\
            &
            \xrightarrow[]{\widetilde{F}^{(1)}_0}
            \left(
            \begin{ytableau}
            \, & & & 1
            \\
            1 & 1 & 3
            \\
            4
            \end{ytableau}
            ,
            \begin{ytableau}
            \, & & & 1
            \\
            1 & 1 & 2
            \\
            3
            \end{ytableau}
            \right)
            \xrightarrow[]{ \widetilde{E}^{(1)}_2 }
            \left(
            \begin{ytableau}
            \, & & & 1
            \\
            1 & 1 & 2
            \\
            4
            \end{ytableau}
            ,
            \begin{ytableau}
            \, & & & 1
            \\
            1 & 1 & 2
            \\
            3
            \end{ytableau}
            \right)
            \xrightarrow[]{ \widetilde{E}^{(1)}_3 }
            \left(
            \begin{ytableau}
            \, & & & 1
            \\
            1 & 1 & 2
            \\
            3
            \end{ytableau}
            ,
            \begin{ytableau}
            \, & & & 1
            \\
            1 & 1 & 2
            \\
            3
            \end{ytableau}
            \right).
        \end{split}
    \end{equation}
    In the right hand side we obtained $(T,T)=\mathcal{L}_{P,Q}(P,Q)$, where $T$ is a leading tableau as it can be checked (more in general this is implied by \cref{thm:leading_map_and_leading_tableaux}). Using correspondence of  \cref{prop:bijection_leading_tab_kappa} we write $T$ as
    \begin{equation}
        \begin{ytableau}
            \, & & & 1
            \\
            1 & 1 & 2
            \\
            3
            \end{ytableau} = T(\mu, \kappa;\nu),
            \qquad
            \text{with}
            \qquad
            \mu = \ydiagram{3,1,1},
            \qquad
            \kappa=(0,1,1),
            \qquad
            \nu=\ydiagram{1} \, .
    \end{equation}
    Therefore correspondence \eqref{eq:bijection_from_pi_to_V1_V2_kappa}, in this case, yields
    \begin{equation}
        \left(
            \begin{ytableau}
            \, & & & 1
            \\
            & 1 & 3
            \\
            2 & 4
            \end{ytableau}
            ,
            \begin{ytableau}
            \, & & & 1
            \\
            & 2 & 2
            \\
            1 & 3
            \end{ytableau}
        \right)
        \xleftrightarrow[]{\hspace{.4cm} \Upsilon \hspace{.4cm}}
        \left(
            \begin{ytableau}
            1 & 3 & 1
            \\
            2
            \\
            4
        \end{ytableau},
        \begin{ytableau}
            1 & 1 & 2
            \\
            2
            \\
            3
        \end{ytableau}
        ;
        (0,1,1)
        ;
        \ydiagram{1}
        \right).
    \end{equation}
    We can finally verify that the relation between empty shape, energies, $\kappa$ and $\nu$ holds, since 
    \begin{equation}
        \mathscr{H}(V)=1,
        \qquad
        \mathscr{H}(W)=0
    \end{equation}
    and
    \begin{equation}
        |\rho| = \mathscr{H}(V) + \mathscr{H}(W) + |\kappa| + |\nu|
        \,\,\,
        \rightsquigarrow
        \,\,\,
        4 = 1+ 0 + 2 + 1.
    \end{equation}
    Clearly, reading backward the example we just presented, gives a realization of the inverse map $\Upsilon^{-1}$.

\subsection{Extensions} \label{subs:extensions}
    Arguments and constructions described throughout this paper admit a few natural extensions. We will outline some of these in the next few paragraphs, although, to keep the exposition concise, we will not enter the details of any of the cases we present. 
    
    \medskip
    
    In order to establish \cref{thm:new_bijection} we have leveraged properties of the skew $\RSK$ map and of theory of affine crystals. In particular our skew $\RSK$ map was defined in \cref{sec:miscellaneous} through a sequence of internal \emph{row} insertions. A natural twist to this story would come from a replacement of row insertions by \emph{column insertions} as described in \cite[Chapter 3.2]{sagan2001symmetric}. Call \emph{skew $\RSK^{\mathrm{col}}$ map} the map defined by switching row and column insertions in \cref{def:RSK_product_tableaux}. Then one could define a \emph{skew $\RSK^{\mathrm{col}}$ dynamics}, which conversely from the skew $\RSK$ dynamics, would evolve the shape of tableaux $(P,Q)$ ``rightward" rather than ``downward". It is natural to expect that repeating arguments developed in this paper one could produce an additional new bijection
    \begin{equation}
        \Upsilon^{\mathrm{col}}:(P,Q) \mapsto (V,W;\kappa,\nu),
    \end{equation}
    analogous to that of \cref{thm:new_bijection}. Here $(P,Q)$ is a pair of semi-standard skew tableaux, while this time $V,W$ are \emph{horizontally weak} tableaux (sometimes called \emph{tabloids} \cite[Chapter 2.1]{sagan2001symmetric}) of same shape $\mu$, $\kappa$ is a suitable adaptation of \cref{def:fundamental_shifts} and $\nu$ is a partition. This would yield a correspondence similar to Shi's affine Robinson-Schensted correspondence \cite{Shi_affine_RS} and a comparison between the two constructions would be of much interest.
    
    \medskip
    
    Another natural extension of our theory comes from replacing the skew $\RSK$ map with its \emph{dual} variant which could be defined following \cite[Section 7]{sagan1990robinson}. Sagan and Stanley used this idea to put in correspondence pairs $(P,Q)$ of semi-standard tableaux with conjugate shape with pairs $(\overline{M};\nu)$ consisting of a \emph{binary} infinite matrix $\overline{M}$ and a partition $\nu$.
    Calling \emph{skew $\RSK^{\vee}$ map} such dual map, we can define a \emph{skew $\RSK^\vee$ dynamics} in which the $P$ tableau evolves in the ``downward" direction, while the $Q$ tableau evolves ``rightward" as a result of the fact that their shapes are one the transpose of the other. In this case it is natural to expect that a reformulation of our arguments would lead to another new bijection
    \begin{equation}
        \Upsilon^\vee : (P,Q) \mapsto (V,W;\kappa,\nu).
    \end{equation}
    In this case $(P,Q)\in SST(\lambda/\rho ,n) \times SST(\lambda' / \rho',n)$, while $V$ and $W$ are respectively a vertically strict and a horizontally weak tableaux (i.e. cells are weakly increasing along rows and no condition is set on columns) of conjugate shapes $\mu,\mu'$, $\kappa$ is again a suitable adaptation of the \cref{def:fundamental_shifts} depending on $\mu$ and $\nu$ is a partition.
    
    We shall consider the two extensions discussed above more precisely in future works.
    
\section{Scattering rules} \label{sec:scattering_rules}

In this section we analyse the skew $\RSK$ dynamics from the viewpoint of discrete classical integrable systems, as outlined in \cref{subs:examples}. Conservation laws stated in \cref{thm:asymptotic_shape_RSK} and symmetries of \cref{thm:symmetries_RSK} reveal analogies with the renowned Box-Ball Systems (BBS) introduced in \cite{Takahashi_Satsuma}; for a review see \cite{Inoue_2012}.

\subsection{Setup}
We have defined a pair $(P,Q)$ to be $\RSK$-stable, if the action of arbitrary many iteration of the skew $\RSK$ map on $(P,Q)$ has the only effect of shifting columns vertically. Analogously we define a pair $(P,Q)$ to be $\RSK^{-1}$-stable if $\RSK^{-t}(P,Q)$ differs from the original pair $(P,Q)$ by vertical shifts of the shape, with no changes in the column content. The natural question we address in this section is the following.

\begin{question} \label{question:scattering_problem}
    Consider a pair $(P,Q)$ of skew semi-standard tableaux and assume that such pair is $\RSK^{-1}$-stable. We know that for $t$ large enough $(\widetilde{P}, \widetilde{Q}) = \RSK^t(P,Q)$ becomes $\RSK$-stable. One can think, for instance of pairs of tableaux depicted in the left and right hand side of \cref{fig:scattering}. Can we precisely describe $(\widetilde{P}, \widetilde{Q})$ purely in terms of $(P,Q)$ and $t$? 
    More specifically:
    \begin{itemize}
        \item Can we predict the content of columns of $(\widetilde{P}, \widetilde{Q})$?
        \item Can we predict the shape of tableaux $\widetilde{P}, \widetilde{Q}$?
    \end{itemize}
\end{question}

It turns out that Question \ref{question:scattering_problem} admits a precise answer, that we present in the two main theorems of this section. In \cref{thm:affine_evacuation} we describe how labeled cells of tableaux $P,Q$ rearrange following the scattering produced by the dynamics of the skew $\RSK$ dynamics. Leveraging on linearization techniques of the skew $\RSK$ map elaborated in \cref{sec:linearization}, the task of describing the shape of $(\widetilde{P},\widetilde{Q})$ becomes then a simple exercise. We present it in \cref{thm:phase_shift} in \cref{subs:phase_shift}.

\begin{figure}[ht]
    \centering
    \includegraphics[scale=.9]{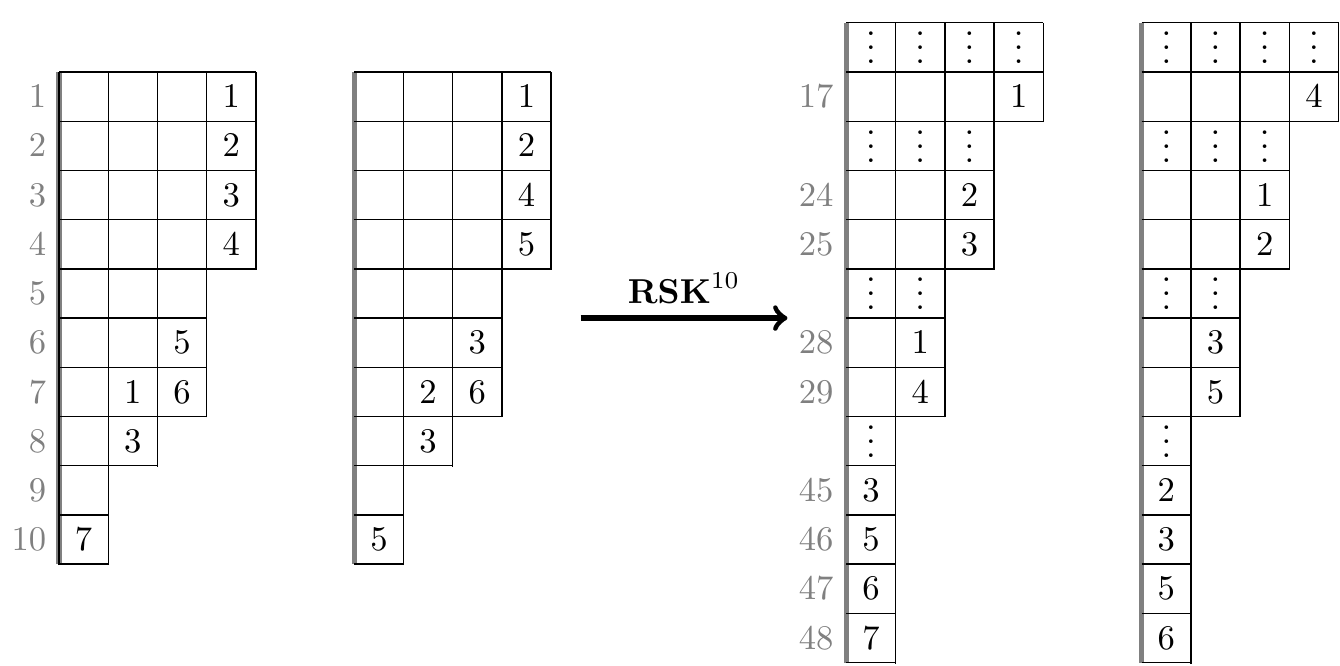}
    \caption{A full scattering from an $\RSK^{-1}$-stable pair on the left to an $\RSK$-stable pair on the right.}
    \label{fig:scattering}
\end{figure}

\subsection{Scattering in the skew $\RSK$ dynamics} \label{subs:scattering_in_the_RSK}

Fix $P,Q\in SST(\lambda/ \rho,n)$ and consider the skew $\RSK$ dynamics $(P_t,Q_t)$ with initial data $(P,Q)$. For any $t$ define $v^{(t)}=v^{(t)}_1 \otimes \cdots \otimes v^{(t)}_{\lambda_1}$ to be the element of the crystal $B^{\varkappa(t)}$ formed by the tensor product of columns of $P_t$. Analogously define $u^{(t)}=u^{(t)}_1 \otimes \cdots \otimes u^{(t)}_{\lambda_1}$ from columns of $Q_t$. Composition $\varkappa(t)$ records then the number of labeled elements at each column of $P_t,Q_t$.  We have seen in \cref{sec:Greene_invariants} that, when $t$ becomes large, $\varkappa(t) \xrightarrow[t\to \infty]{} \mu'$, where $\mu$ is the Greene invariant of $\overline{\pi}$ related to $(P,Q)$ by the Sagan-Stanley correspondence $(P,Q) \xrightarrow[]{\skwRSK \,} \overline{\pi}$. Analogously we can consider $\varkappa(t)$ for $t\to -\infty$. In the following theorem we denote with $\overleftarrow{\eta}=(\eta_N,\dots,\eta_1)$ the reverse ordering of a partition $\eta = (\eta_1,\dots,\eta_N)$.

\begin{theorem} \label{thm:soliton_conservation}
    In the notation introduced above we have $\varkappa(-t) = \overleftarrow{\mu}'$ eventually for $t$ large enough.
\end{theorem}

\begin{proof}
    In \cref{thm:asymptotic_shape_RSK} we have related the asymptotic increment $\varkappa(t)$ for $t \gg 0$ with conserved quantities of the Viennot map $\mathbf{V}$, yielding the equality $\varkappa(t)=\mu'$. Analogously we can relate backward asymptotic increments $\varkappa(-t)$ for $t \gg 0$ with conserved quantities of the inverse Viennot map $\mathbf{V}^{-1}$, which are the same as $\mathbf{V}$. Through such argument one can easily complete the proof.
\end{proof}

\begin{remark}
    In view of \cref{thm:soliton_conservation}, columns of tableaux in the skew $\RSK$ dynamics may be seen as solitons. Thanks to conservation laws, they survive after collisions with others and eventually propagate at their own characteristic speeds, similarly to the ones in  BBS.
    In the case of the BBS,  conservation of solitons can be proven in several different ways. These include commutation of transfer matrices \cite{fukuda_okado_yamada_BBS_energy} or bijection with rigged configuration through the Kerov-Kirillov-Reschetikhin (KKR) correspondence \cite{Kuniba_et_al_KKR_BBS}. Similarities between the BBS and the skew $\RSK$ dynamics provided by \cref{thm:soliton_conservation,thm:phase_shift,thm:affine_evacuation} suggest that the framework developed in this paper might provide an alternative, more combinatorial, route to study the BBS. It would be interesting to understand relations between our results and KKR correspondence or even to understand extension of such correspondence in types other than $A^{(1)}_{n-1}$. We plan to pursue these directions in future publications.
\end{remark}

Carrying on with the notation introduced at the beginning of the present subsection we see elements $v^{(t)},u^{(t)}\in B^{\varkappa(t)}$ become, for large $t$ equivalent to the asymptotic vertically strict tableaux of \cref{def:asymptotic_VST}. Analogously, $v^{(-t)},u^{(-t)}$, for large $t$, eventually stabilize and we define the limits
\begin{equation}
    V^-=\lim_{t \to \infty} v^{(-t)},
    \qquad
    W^-=\lim_{t \to \infty} u^{(-t)}.
\end{equation}
By \cref{thm:soliton_conservation} we have $V^-, W^- \in B^{\overleftarrow{\mu}'}$ and we define the \email{backward projection}
\begin{equation}
    \Phi^-:(P,Q) \mapsto (V^-, W^-),
\end{equation}
as the negative time counterpart of $\Phi$ given in \cref{eq:Phi}.

The relation between the backward and forward asymptotic states $(V^-,W^-)$ and $(V,W)$ gives the 
scattering rules of solitons in the skew $\RSK$ dynamics. A typical feature of integrable systems 
is that effects of multi-body scattering are fully determined by the knowledge of the two-body scattering. 
In the case of BBSs, two body scattering is given by the combinatorial $R$-matrix. The following theorem 
claims that the scattering rules relating $(V^-,W^-)$ and $(V,W)$ in our skew $\RSK$ dynamics are also 
described by consecutive applications of the same combinatorial $R$-matrices corresponding to the 
change of orders of solitons. 
In order to give a precise statement we define, following \cref{prop:unique_isomorphism}, the unique isomorphism of crystal graphs $B^\mu \to  B^{\overleftarrow{\mu}'}$. It is expressed, naming $\delta(N)$ the permutation $(N \,\, \, N-1 \,\, \, \cdots \,\,\, 1)$, as
\begin{equation} \label{eq:R_shape_reversing}
    R_{\delta(\mu_1)} \coloneqq 
    R_1 \cdot (R_2  R_1) \cdot (R_3  R_2  R_1) \cdots (R_{\mu_1-1} \cdots R_1) : B^{\mu'} \to B^{\overleftarrow{\mu}'}.
\end{equation}
For an example of the action of $R_{\delta(\mu_1)}$, see \eqref{eq:R_shape_reversing} below.

\begin{theorem}\label{thm:affine_evacuation}
    Let $(P,Q)$ be a pair of tableaux, call $\mu$ the respective Greene invariant and consider the projections
    \begin{equation}
        \Phi(P,Q) = (V, W),
        \qquad
        \Phi^-(P,Q) = (V^-, W^-).
    \end{equation}
    Then the map $\Psi:(V,W) \to (V^-, W^-)$ is well defined, does not depend on the choice of $(P,Q)$ and it is given by
    \begin{equation}
        V^-= R_{\delta(\mu_1)} (V),
        \qquad
        W^-= R_{\delta(\mu_1)} (W).
    \end{equation}
\end{theorem}

\begin{proof}
    Define $\psi=R_{\delta(\mu_1)}$. We will only show that $V^- = \psi (V)$, since the same relation for $V,W^-$ can be proven in analogous fashion.
    This defines maps 
    \begin{equation}
        \phi^+:(P,Q) \mapsto V
        \qquad
        \text{and}
        \qquad
        \phi^-:(P,Q) \mapsto V^-.
    \end{equation}
    The proof \cref{thm:affine_evacuation} reduces to characterizeing the map $V^- \mapsto V$ and to prove that it is given by $\psi^{-1}$. Notice that in principle such map could be not well defined since there might exist pairs $(P,Q),(P',Q')$ such that, for instance, $\phi^-(P,Q)=\phi^-(P',Q')$, but $\phi^+(P,Q) \neq \phi^+(P',Q')$. We show that this is indeed not the case.
    
    Fix an element $V^- \in B^{\overleftarrow{\mu}'}$ and consider a pair $(P,Q)$ such that $\phi^-(P,Q) = V^-$. Notice first that if $V^-$ is the leading vector $V^-=(\overleftarrow{\mu}')^{\mathrm{lv}}$, then necessarily $\gamma(P)=\mu$ and $V=(\mu')^{\mathrm{lv}}$. For more general $V^-$, we can always connect it to the leading vector $(\overleftarrow{\mu}')^{\mathrm{lv}}$ through a leading map $\mathcal{L}_{V^-}$, as in \cref{def:leading_map_VST}. The $\Phi$-pullback of $\mathcal{L}_{V^-}$ defines a map on pairs of tableaux $\mathcal{L}_{V^-}:(P,Q) \to (P',Q')$, which commutes with the skew $\RSK$ map and hence with projection $\Phi$. 
    This shows that fixed $V^-$, we always have
    \begin{equation} \label{eq:b_inverse_path_b_tilde}
        V=\mathcal{L}_{V^-}^{-1}((\mu')^{\mathrm{lv}}).
    \end{equation}
    Comparing \eqref{eq:b_inverse_path_b_tilde} with result of \cref{prop:unique_crystal_graph_isomorphism}, we can conclude that map $V\mapsto V^-$ exists and it is the unique isomorphism of crystals $B^{\overleftarrow{\mu}'}$ and by \cref{prop:unique_isomorphism} we have $V=\psi^{-1}(V^-)$. Notice that this is independently of the choice of $(P,Q)$, as long as $\phi^-(P,Q)=V^-$.
\end{proof}
    
\begin{example}
    We can verify statement of \cref{thm:affine_evacuation} in the example reported in \cref{fig:scattering}. The transformation $V \mapsto V^-$ step by step, reads
    \begin{equation} \label{eq:R_matrix_shape_reverse}
        \begin{split}
            V =\,
            &
            \begin{ytableau}
            3 & 1 & 2 & 1
            \\
            5 & 4 & 3
            \\
            6
            \\
            7
        \end{ytableau}
        \xrightarrow[]{\hspace{.2cm} R_1 \hspace{.2cm}}
        \begin{ytableau}
            3 & 1 & 2 & 1
            \\
            7 & 4 & 3
            \\
            \none & 5
            \\
            \none & 6
        \end{ytableau}
        \xrightarrow[]{\hspace{.2cm} R_2 \hspace{.2cm}}
        \begin{ytableau}
            3 & 1 & 2 & 1
            \\
            7 & 6 & 3
            \\
            \none & \none & 4
            \\
            \none & \none & 5
        \end{ytableau}
        \xrightarrow[]{\hspace{.2cm} R_3 \hspace{.2cm}}
        \begin{ytableau}
            3 & 1 & 5 & 1
            \\
            7 & 6 & \none & 2
            \\
            \none & \none & \none & 3
            \\
            \none & \none & \none & 4
        \end{ytableau}
        \\
        &
        \hspace{3cm}
        \xrightarrow[]{\hspace{.2cm} R_1 \hspace{.2cm}}
        \begin{ytableau}
            3 & 1 & 5 & 1
            \\
            7 & 6 & \none & 2
            \\
            \none & \none & \none & 3
            \\
            \none & \none & \none & 4
        \end{ytableau}
        \xrightarrow[]{\hspace{.2cm} R_2 \hspace{.2cm}}
        \begin{ytableau}
            3 & 1 & 5 & 1
            \\
            7 & \none & 6 & 2
            \\
            \none & \none & \none & 3
            \\
            \none & \none & \none & 4
        \end{ytableau}
        \xrightarrow[]{\hspace{.2cm} R_1 \hspace{.2cm}}
        \begin{ytableau}
            7 & 1 & 5 & 1
            \\
            \none & 3 & 6 & 2
            \\
            \none & \none & \none & 3
            \\
            \none & \none & \none & 4
        \end{ytableau}
        = V^-,
        \end{split}
    \end{equation}
    where we have suppressed symbol $\otimes$ between different columns.
\end{example}

\begin{remark}
    In \cite{chmutovEtAl_affine_evacuation} authors described an analogous phenomenon as that of \cref{thm:affine_evacuation} in the context of Affine Matrix Ball Construction. In that paper the operation called \emph{affine evacuation} $V\mapsto \evac(V)$ corresponds to transforming element $V^-$ in a vertically strict tableaux of shape equal to $V$, rotating $V^-$ by $180^\circ$ and replacing each entry $i$ by $n-i+1$.
\end{remark}

\begin{remark}
    Note that similarities between exchange of degrees of freedom for solitons in the BBS and vertically strict tableaux in the skew $\RSK$ dynamics do not imply that all properties of the skew $\RSK$ dynamics can be studied by standard techniques of the BBS. For instance it is not clear if one could express the skew $\RSK$ map as a transfer matrix, i.e. as a product of local $R$-matrices. To understand fully relations between the two models, more elaborate considerations are needed.
\end{remark}

\subsection{Phase Shift} \label{subs:phase_shift}

Leveraging on results disussed in \cref{subs:linearization}, here we describe the phase shift in the scattering between columns of different length in the skew $\RSK$ dynamics. As we explained, each soliton has its own speed, not changing after collisions. But its phase, which 
determines the position of the linear trajectory of a soliton, 
may shift during a collision with another one. This is the phase shift. It is an important notion in soliton theory because together with the description of exchanges of degrees of freedom it completely characterizes the whole scattering process of solitons. 

The phase shift in the skew $\RSK$ dynamics may be well explained by the example in \cref{fig:scattering}. There, we see that, in the left hand side the column hosting four labeled cells, which we call 4-soliton, ``starts" at the first row. After 10 iterations of the skew $\RSK$ map, when all collisions are completed, we see that, in the tableaux on the right hand side, the 4-soliton in the first column occupies rows 45 to 48. This means that interactions with other columns have accelerated the motion of the 4-soliton, which otherwise would have traveled only $4\times 10 = 40$ cells. In this case the phase shift was equal to 4. The same phenomenon can be observed also tracking locations of other columns of the tableaux.

The next theorem gives a precise description of phase shifts in the skew $\RSK$ dynamics and gives a
a full answer to the second bullet of Question \ref{question:scattering_problem}.

\begin{theorem}\label{thm:phase_shift}
    Let $(P,Q)$ be an $\RSK^{-1}$-stable pair of skew tableaux of shape $\lambda/\rho$ and assume $\ker(P,Q)=\varnothing$. Denote the Greene invariant of $(P,Q)$ by $\mu$ and define its rectangular decomposition by indices $0=R_0,R_1,\dots,R_N=\mu_1$
    and $r_i$ as discussed around \eqref{eq:rect_dec}. Moreover set $\widetilde{R}_i = R_N - R_{i+1}$.
    Let $t$ be large enough so that $(\widetilde{P},\widetilde{Q})=\RSK^t (P,Q)$ is $\RSK$-stable and denote by $\widetilde{\lambda} / \widetilde{\rho}$ the shape of $(\widetilde{P},\widetilde{Q})$. Then the phase shift of the $j$-th column of lenght $\mu_{R_{i+1}}'$ is given by
    \begin{equation} \label{eq:phase_shift}
        \widetilde{\rho}_{R_i + j}^{\, \prime} - \rho_{\widetilde{R}_i + j}' - t \times \mu_{R_{i+1}}' =
        \mathscr{H}_{R_i + j}(V)+\mathscr{H}_{R_i + j}(W) - \mathscr{H}_{\widetilde{R}_i + j}(V^-) - \mathscr{H}_{\widetilde{R}_i + j}(W^-),
    \end{equation}
    for all $i=0,\dots,N-1, j=1,\dots,r_{i+1}$ and
    where $(V^-,W^-) = \Phi^-(P,Q)$, while $(V, W) = \Phi(P,Q)$.
\end{theorem}

\begin{proof}
    Let $\mathcal{L}_{P,Q}$ be the leading map of the pair $(P,Q)$ and let $T=T(\mu,\kappa;\varnothing)$ be the leading tableau such that $(T,T)=\mathcal{L}_{P,Q}(P,Q)$. In order to prove our claim we make use of commutation relation
    \begin{equation}
        (\widetilde{P},\widetilde{Q}) = \mathcal{L}_{P,Q}^{-1} \circ \RSK^{t} \circ \mathcal{L}_{P,Q} (P,Q),
    \end{equation}
    which follows from \cref{thm:symmetries_RSK}. Since $(P,Q)$ is $\RSK^{-1}$-stable so is $(T,T)$ and hence $\kappa=(\kappa^{(1)},\dots, \kappa^{(N)})$ is of the form
    \begin{equation} \label{eq:kappa_inverse_RSK_stable}
        \kappa^{(N)}_1 \ge \cdots \ge \kappa^{(N)}_{r_N} > \kappa^{(N-1)}_1 \ge \cdots \kappa^{(N-1)}_{r_{N-1}}> \cdots \ge \kappa^{(1)}_1 > \cdots > \kappa^{(1)}_{r_1}.
    \end{equation}
    From \cref{prop:bijection_leading_tab_kappa}, the empty shape of $T$ is $(\kappa^+)'$ and by \cref{prop:cells_and_local_energy} we have
    \begin{equation} \label{eq:empty_shape_inverse_RSK_stable}
        \rho_i' = \kappa^+_i + \mathscr{H}_i(V^-) + \mathscr{H}_i(W^-),
    \end{equation}
    where $\kappa^+$ is explicitly determined by \eqref{eq:kappa_inverse_RSK_stable}.
    Applying $t$ times the skew $\RSK$ map to pair $(T,T)$ we obtain, by \cref{thm:leading_tab_have_no_phase_shift},
    \begin{equation}
        (\widetilde{T},\widetilde{T}) = \RSK^t(T,T),
        \qquad
        \text{where}
        \qquad
        \widetilde{T}=T(\mu,\widetilde{\kappa};\varnothing),
        \qquad
        \widetilde{\kappa}=\kappa + t \times \mu'.
    \end{equation}
    When $t$ is large, as in the hypothesis of the theorem, we have that $\widetilde{\kappa}$ is itself a partition and hence the empty shape of $\widetilde{T}$ is given by $\widetilde{\kappa}'$. The action of the map $\mathcal{L}^{-1}_{P,Q}$ on $(\widetilde{T}, \widetilde{T})$ has the effect of growing the empty shape as prescribed by \cref{prop:cells_and_local_energy}, implying
    \begin{equation} \label{eq:shape_rho_t}
    \begin{split}
            \widetilde{\rho}^{\, \prime}_i 
            &= \widetilde{\kappa}_i + \mathscr{H}_i(V) + \mathscr{H}_i(W)
            \\
            &=
            \kappa_i + t \times \mu_i' + \mathscr{H}_i(V) + \mathscr{H}_i(W).
        \end{split}
    \end{equation}
    Expressing the term $\kappa_i$ in terms of $\rho'$ and of local energies of elements $V^-,W^-$ yields \eqref{eq:phase_shift}.
\end{proof}

\begin{example}
    We can confirm the validity of \cref{thm:phase_shift} computing \eqref{eq:phase_shift} for tableaux presented in \cref{fig:scattering}. For that case, in the notation of \cref{thm:phase_shift}, we have
    \begin{equation}
        \ytableausetup{smalltableaux}
        V^-=
        \begin{ytableau}
            7 & 1 & 5 & 1
            \\
            \none & 3 & 6 & 2
            \\
            \none & \none & \none & 3
            \\
            \none & \none & \none & 4
        \end{ytableau}
        \, ,
        \qquad
        W^-=
        \begin{ytableau}
            5 & 2 & 3 & 1
            \\
            \none & 3 & 6 & 2
            \\
            \none & \none & \none & 4
            \\
            \none & \none & \none & 5
        \end{ytableau}
        \, ,
        \qquad
        V =
        \begin{ytableau}
            3 & 1 & 2 & 1
            \\
            5 & 4 & 3
            \\
            6
            \\
            7
        \end{ytableau}
        ,
        \qquad
        W =
        \begin{ytableau}
            2 & 3 & 1 & 4
            \\
            3 & 5 & 2
            \\
            5
            \\
            6
        \end{ytableau},
    \end{equation}
    so that we can compute the local energies
    \begin{gather}
        (\mathscr{H}_i(V^-))_{i=1,\dots,4} = (2,2,2,0),
        \qquad
        (\mathscr{H}_i(W^-))_{i=1,\dots,4} = (1,1,1,0),
        \\
        (\mathscr{H}_i(V))_{i=1,\dots,4} = (3,2,1,0),
        \qquad
        (\mathscr{H}_i(W))_{i=1,\dots,4} = (1,2,0,0).
    \end{gather}
    Then, \eqref{eq:phase_shift} reduce to
    \begin{gather*}
        \widetilde{\rho}_{1}^{\, \prime} = \rho_{4}' + t \times \mu_{1}' +\mathscr{H}_{1}(V)+\mathscr{H}_{1}(W) - \mathscr{H}_{4}(V^-) - \mathscr{H}_{4}(W^-)
        \\
        \hspace{7cm}
        \rightsquigarrow 44 = 0 + 40 + 3 + 1 - 0 - 0
        ,
        \\
        \widetilde{\rho}_{2}^{\, \prime} = \rho_{2}' + t \times \mu_{2}' +\mathscr{H}_{2}(V)+\mathscr{H}_{2}(W) - \mathscr{H}_{2}(V^-) - \mathscr{H}_{2}(W^-)
        \\
        \hspace{7cm}
        \rightsquigarrow 27 = 6 + 20 + 2 + 2 - 2 - 1
        ,
        \\
        \widetilde{\rho}_{3}^{\, \prime} = \rho_{3}' + t \times \mu_{3}' +\mathscr{H}_{3}(V)+\mathscr{H}_{3}(W) - \mathscr{H}_{3}(V^-) - \mathscr{H}_{3}(W^-)
        \\
        \hspace{7cm}
        \rightsquigarrow 23 = 5 + 20 + 1 + 0 - 2 - 1
        ,
        \\
        \widetilde{\rho}_{4}^{\, \prime} = \rho_{1}' + t \times \mu_{4}' + \mathscr{H}_{4}(V)+\mathscr{H}_{4}(W) - \mathscr{H}_{1}(V^-) - \mathscr{H}_{1}(W^-)
        \\
        \hspace{7cm}
        \rightsquigarrow 16 = 9 + 10 + 0 + 0 - 2 - 1.
    \end{gather*}
    Notice the nontrivial pairing in the previous equalities between columns of $\widetilde{\rho}$ and $\rho$. In particular indeces are not simply reversed, but they are given by numbers $R_i,\widetilde{R}_i$ as in \cref{thm:phase_shift}.
\end{example}

\begin{remark}
    Formulas similar to those in \cref{thm:phase_shift} for the phase shift have been found for the BBS \cite{fukuda_okado_yamada_BBS_energy}. In particular, when in the skew $\RSK$ dynamics we consider initial conditions with $P=Q$, which forces $P_t=Q_t$ for all $t$, then the phase shift of solitons in multi species BBS is exactly the same as the one resulting from \eqref{eq:phase_shift}. Under these assumptions on the initial conditions the skew $\RSK$ dynamics and the BBS become very similar models and they both possess $\widehat{\mathfrak{sl}}_n$ symmetry. On the other hand for general initial data $P,Q$ the skew $\RSK$ dynamics possesses a larger set of symmetries (i.e. $\widehat{\mathfrak{sl}}_n \times \widehat{\mathfrak{sl}}_n$) and they are no longer equivalent. We shall examine precisely analogies between the BBS and skew $\RSK$ dynamics in a future work.
\end{remark}

\section{Summation identities and bijective proofs} \label{sec:summation_identities}

We explore the consequences of bijection $\Upsilon$ proving a number of summation identities for $q$-Whittaker polynomials. These are known Cauchy and Littlewood identities, presented in \cref{subs:summations_qW}
along with new identities between summations of $q$-Whittaker and skew Schur polynomials, which are presented in \cref{subs:summations_qW_Schur}.

\subsection{Summation identities for $q$-Whittaker polynomials} \label{subs:summations_qW}

The bijection discussed in \cref{subs:bijection} reveals a number of combinatorial properties of $q$-Whittaker polynomials $\mathscr{P}_{\mu}(x;q)$, that are Macdonald polynomials $\mathscr{P}_\mu(x;q,t)$ specialized at $t=0$ \cite[Chapter VI]{Macdonald1995}. These have several different representations. For our purposes the most useful one is given by the combinatorial formula reported in the next proposition as a sum over vertically strict tableaux \cite[Corollary 9.5]{schilling_tingley}; see also \cite{Nakayashiki_Yamada}. The meaning of such expression is that $q$-Whittaker polynomials are characters of certain Demazure modules of $\widehat{\mathfrak{sl}}_n$ \cite{Sanderson_Demazure_q_Whittaker}, whose grading is given by the intrinsic energy function \cite{schilling_tingley}. 

\begin{proposition}
    For all partitions $\mu$, we have
    \begin{equation} \label{eq:q_Whittaker}
        \mathscr{P}_\mu(x;q) = \sum_{V\in VST(\mu)} q^{\mathscr{H}(V)} x^V,
    \end{equation}
    where $x^V=x_1^{m_1(V)} x_2^{m_2(V)} \cdots  $ and $m_i(V)$ counts the number of $i$-cells in $V$.
\end{proposition}

When $q=0$, \eqref{eq:q_Whittaker} reduces to the well known combinatorial formula for the Schur polynomial reported below in \eqref{eq:skew_Schur_combinatorial}, where one should set $\rho=\varnothing$.

Symmetric polynomials $\mathscr{P}_\mu$ enjoy Cauchy identities, as reported in \cite{Macdonald1995}, for the case of general Macdonald polynomials. Leveraging correspondence reported in \cref{cor:bijection_pi_V1_V2_kappa}, here we present a bijective proof. We use the notion of $q$-Pochhammer symbol
\begin{equation}
    (z;q)_k = \prod_{i=0}^{k-1} (1-q^i z) 
    \qquad
    \text{and}
    \qquad
    (z;q)_\infty = \prod_{i=0}^\infty (1-q^i z),
\end{equation}
where the second expression holds for $|q|<1$.
\begin{theorem}
    Fix $|q|<1$. Consider variables $x=(x_1,\dots,x_n)$ and $y=(y_1,\dots,y_n)$ with $|x_i y_j|<1$ for all $i,j$. Then
    \begin{equation} \label{eq:Cauchy_id}
        \sum_{\mu \in \mathbb{Y}} \mathdutchcal{b}_\mu(q) \mathscr{P}_\mu (x;q) \mathscr{P}_\mu (y;q) = \prod_{i,j=1}^n \frac{1}{(x_i y_j;q)_{\infty}},
    \end{equation}
    where
    \begin{equation} \label{eq:b_mu}
        \mathdutchcal{b}_\mu(q) = \prod_{i\ge 1} \frac{1}{(q;q)_{\mu_i - \mu_{i+1}}}.
    \end{equation}
\end{theorem}

\begin{proof}
    We start by noticing that,  
    \begin{equation}
        \mathdutchcal{b}_\mu(q) = \sum_{\kappa \in \mathcal{K}(\mu)} q^{|\kappa|},
        \label{bmu_kappa}
    \end{equation}
    which follows from the summation identity
    \begin{equation}
        \sum_{\substack{\nu \in \mathbb{Y} \\ \nu_1 \le N}} q^{|\nu|} = \frac{1}{(q;q)_N}.
    \end{equation}
    Then, using \eqref{eq:q_Whittaker} and \cref{cor:bijection_pi_V1_V2_kappa} we deduce the following equalities
    \begin{equation}
    \begin{split}
        \sum_{\mu} \mathdutchcal{b}_\mu(q) \mathscr{P}_\mu (x;q) \mathscr{P}_\mu (y;q)  &= \sum_\mu \sum_{V,W \in VST(\mu)} \sum_\kappa q^{|\kappa| + H(V) + H(W)} x^{V} y^{W}
        \\
        &=
        \sum_{\overline{\pi} \in \overline{\mathbb{A}}_{n,n}^+} q^{\wt(\overline{\pi})} x^{p(\overline{\pi})} y^{q(\overline{\pi})}
        =
        \prod_{i,j=1}^n \prod_{k \ge 0} \frac{1}{1-q^k x_i y_j}.
    \end{split}
    \end{equation}
\end{proof}

Taking summations over the set of symmetric weighted biwords $\overline{\pi} = \overline{\pi}^{-1}$ yields identities involving single polynomials $\mathscr{P_\mu}$. To state our result we define
\begin{equation} \label{eq:b_q_z}
    \mathdutchcal{b}_\mu(q;z) = \prod_{i = 2,4,6\dots} \frac{[qz^2 + 1]_{q^2}^{\mu_i - \mu_{i+1}}}{(q^2;q^2)_{\mu_i - \mu_{i+1}}} \prod_{i=1,3,5,\dots} \frac{z^{\mathbf{1}_{\mu_i > \mu_{i+1}}}}{(q;q)_{\mu_i - \mu_{i+1}}},
\end{equation}
where
\begin{equation}
    [A+B]_p^k = \sum_{j=0}^k A^j B^{k-j} \Qbinomial{k}{j}{p} 
\end{equation}
and
\begin{equation}
    \Qbinomial{k}{j}{p} = \frac{(p;p)_k}{(p;p)_j (p;p)_{k-j}}
\end{equation}
is the Gaussian binomial coefficient. In literature the function $h_n(x;p) = [x+1]_p^n$ is commonly known as Rogers-Szeg{\"o} polynomial \cite{andrews_1984}.

\begin{theorem} \label{thm:Littlewood_like_id}
    Fix $|q|<1$. Consider variables $z$ and $x=(x_1,\dots,x_n)$ with $|z x_i|<1$ for all $i,j$. Then we have
    \begin{equation} \label{eq:Littlewood_like_id}
        \sum_{\mu} \mathdutchcal{b}_\mu (q;z) \mathscr{P}_\mu(x;q^2) = \prod_{i=1}^n \frac{1}{(zx_i ;q)_\infty} \prod_{1\le i < j \le n} \frac{1}{(x_i x_j;q^2)_\infty}.
    \end{equation}
\end{theorem}

Notice that setting $z=0$ in \eqref{eq:Littlewood_like_id} and using the convention $0^0=1$, since $\mathdutchcal{b}_\mu(q,0)=\mathdutchcal{b}_\mu(q^2) 
\prod_{i=1,3,5,\dots} \mathbf{1}_{\mu_i=\mu_{i+1}}$, we obtain the Littlewood identity for $q$-Whittaker polynomials
\begin{equation}
    \sum_{\mu:\mu' \text{ is even}} \mathdutchcal{b}_\mu(q^2) \mathscr{P}_\mu (x;q^2) = \prod_{1\le i < j < \le n} \frac{1}{ (x_i x_j ; q^2)_\infty},
\end{equation}
which becomes (i) of example 4 in chapter VI,7 of \cite{Macdonald1995}, after rescaling $q^2 \to q$. On the other hand, taking $z=1$, we observe that $\mathdutchcal{b}_\mu(q;1)=\mathdutchcal{b}_\mu(q)$ as a result of the known identity $[q+1]_{q^2}^k = (-q;q)_k$ for Rogers-Szeg{\"o} polynomials, see Example 5 in Chapter 3 of \cite{andrews_1984}. Then \eqref{eq:Littlewood_like_id} becomes
    \begin{equation}
        \sum_\mu \mathdutchcal{b}_\mu (q) \mathscr{P}_\mu(x;q^2) = \prod_{i=1}^n \frac{1}{ (  x_i ; q )_\infty} \prod_{1 \le i<j\le n} \frac{1}{ ( x_i x_j ; q^2 )_\infty},
    \end{equation}
which is a special case of an identity for Macdonald polynomials conjectured by Kawanaka \cite{Kawanaka_1999} and proven in \cite{Langer_Schlosser_Warnaar_2009}. When parameter $z$ is general, identity \eqref{eq:Littlewood_like_id} is equivalent, after plethystic substitution, to a Littlewood identity proven by Warnaar in \cite{Warnaar_Rogers_Szego}; see \cref{rem:from_qW_to_HL}. Additional Littlewood identities are presented in \cite{Warnaar_Rogers_Szego,rains_warnaar2021bounded}, although it is not clear if their bijective proof is accessible through the theory developed in this paper.

In order to show \eqref{eq:Littlewood_like_id} we have to relate the left hand side with a summation over symmetric weighted biwords $\overline{\pi}$, where the variable $z$ weights the number of fixed points of $\overline{\pi}$ (i.e. elements $\overline{\pi}_i = \left( \begin{smallmatrix} j \\ j \\ k \end{smallmatrix} \right)$ for some $j\in \mathcal{A}_n$, $k\in \mathbb{Z}$). We need a few preliminary results. In the following lemmas we denote by $\mathrm{odd}(\eta)$ the number of odd elements of an integer sequence $\eta$. For instance if $\lambda$ is a partition $\mathrm{odd}(\lambda')$ is the number of its odd length columns. For a weighted biword $\overline{\pi}$ we also define
\begin{equation}
    \mathrm{fixed}(\overline{\pi}) = \mathrm{tr}(\overline{M}) = \sum_{j=1}^n \sum_{k\in \mathbb{Z}} \overline{M}_{j,j}(k),
\end{equation}
where as usual $\overline{\pi}$ and $\overline{M}$ are related by \cref{eq:matrix_weighted_biword}.

\begin{lemma}[\cite{sagan1990robinson} Corollary 4.6] \label{lemma:odd_lamda_rho}
    Let $P$ be a semi-standard skew tableau of shape $\lambda/ \rho$ and let $(P,P) \xleftrightarrow[]{\skwRSK \,}  (\overline{\pi}; \nu)$. Then
    \begin{equation}
        \mathrm{odd}(\lambda') + \mathrm{odd}(\rho') = \mathrm{fixed}(\overline{\pi}) + 2 \, \mathrm{odd}(\nu').
    \end{equation}
\end{lemma}

\begin{lemma} \label{lemma:fixed_pi}
    Let $\overline{\pi} \xleftrightarrow[]{\tilde{\Upsilon}} (V,V;\kappa)$ with $V \in VST(\mu,n)$. Then
    \begin{equation}
        \mathrm{fixed}(\overline{\pi}) = \mathrm{odd}(\kappa) + \mathrm{odd}(\kappa + \mu').
    \end{equation}
\end{lemma}

\begin{proof}
    Let $P$ be such that $(P,P) \xleftrightarrow[]{\skwRSK\,}(\overline{\pi};\varnothing)$. Consider now the skew $\RSK$ dynamics $(P_t,P_t)$ with initial data $(P,P)$ and let $\lambda^{(t)}/\rho^{(t)}$ be the shape of $P_t$. Then, by \eqref{eq:shape_rho_t},  
    we have, for $t$ large enough
    \begin{gather}
        (\rho^{(t)})_i' = 2 \, \mathscr{H}_i (V) + \kappa_i + t \times \mu_i',
        \\
        (\lambda^{(t)})_i' = 2 \, \mathscr{H}_i (V) + \kappa_i + (t+1) \times \mu_i'.
    \end{gather}
    On the other hand $\mathrm{odd}((\lambda^{(1)})') + \mathrm{odd}((\rho^{(1)})') = \mathrm{odd}((\lambda^{(t)})') + \mathrm{odd}((\rho^{(t)})')$ as a consequence of \cref{lemma:odd_lamda_rho}.
    This is because if $(P_t,P_t) \xleftrightarrow[]{\skwRSK\,} (\overline{\pi}',\varnothing)$, then $\overline{\pi}$ and $\overline{\pi}'$ have the same $q$ and $p$ words and their weights differ only by a constant shift, i.e. $w(\overline{\pi}')_i = w(\overline{\pi})_i+t-1$ for all $i$, implying $\mathrm{fixed}(\overline{\pi}) = \mathrm{fixed}(\overline{\pi}')$.
    Combining these observations we find
    \begin{equation}
        \mathrm{odd}((\lambda^{(1)})') + \mathrm{odd}((\rho^{(1)})') = \mathrm{odd}(\kappa) + \mathrm{odd}(\kappa + \mu'),
    \end{equation}
    where the expression in the right hand side is a result of checking parities of $\kappa_i,\mu_i'$ and $t$ in all cases. 
    
\end{proof}

We now define functions
\begin{gather}
    \mathdutchcal{g}_k(z,q) = \sum_{\nu:\nu_1 = k} z^{2 \mathrm{odd}(\nu')} q^{|\nu|},
    \\
    \tilde{\mathdutchcal{g}}_k(z,q) = \sum_{\nu:\nu_1 \le k} z^{2 \mathrm{odd}(\nu')} q^{|\nu|} = g_0(z,q) + g_1(z,q) + \cdots g_k(z,q). 
\end{gather}

\begin{lemma}
    For $k\ge 0$, we have
    \begin{equation} \label{eq:g_k}
        \mathdutchcal{g}_k(z,q) = \frac{[q z^2 + q^2]^k_{q^2}}{(q^2;q^2)_k}
    \end{equation}
    and
    \begin{equation}\label{eq:G_k}
        \tilde{\mathdutchcal{g}}_k(z;q) = \frac{[qz^2 + 1]_{q^2}^k}{(q^2;q^2)_k}.
    \end{equation}
\end{lemma}

\begin{proof}
    Any partition $\nu$ with first row of length $k$ can be written as $\nu'=\tilde{\nu}' + \eta(\varepsilon;k)'$ where $\widetilde{\nu}$ has all even columns and first row $\widetilde{\nu}_1\le k$ and $\eta(\varepsilon;k)$ is the partition defined by
    \begin{equation}
        \eta(\varepsilon;k)'_i-\eta(\varepsilon;k)'_{i+1} = |\varepsilon_i - \varepsilon_{i+1}|,
        \qquad
        \text{for $i=1,\dots,k-1$} 
        \qquad
        \text{and}
        \qquad
        \eta(\varepsilon;k)_k'= 2- \varepsilon_k,
    \end{equation}
    for $\varepsilon\in \{ 0,1\}^k$. The binary sequence $\varepsilon$ encodes location of odd columns of $\nu$. An example, for $k=6$ can be
    \begin{equation}
        \nu=\ydiagram{6,5,5,3,2,2,2,1},
        \qquad
        \widetilde{\nu} = \ydiagram{5,5,2,2},
        \qquad
        \eta(\varepsilon,k) = \ydiagram{6,3,2,1},
    \end{equation}
    with $\varepsilon=(0,1,0,1,1,1)$.
    Since by construction $\eta(\varepsilon;k)_1=k$, we can further decompose $\eta(\varepsilon;k)$ taking away one box from odd length columns and two boxes from even length columns, as 
    $$
    \eta(\varepsilon;k)' = 2 \widetilde{\eta}(\varepsilon;k)' + \varepsilon + 2(1-\varepsilon).
    $$ 
    Notice that for fixed $j=|\varepsilon|=\varepsilon_1+\cdots+\varepsilon_k$, we always have $\widetilde{\eta}(\varepsilon;k)'_1\le j$ and $\widetilde{\eta}(\varepsilon;k)_1\le k$. Moreover for any partition $\lambda$ such that $\lambda_1\le k$ and $\lambda_1' \le j$, there always exists a choice of $\varepsilon$ such that $\widetilde{\eta}(\varepsilon;k)=\lambda$. Consider the generating function of Young diagrams $\eta(\varepsilon;k)$ 
    \begin{equation}
        \mathcal{Z}(\zeta,k) = \sum_{\varepsilon \in \{ 0,1 \}^k } \zeta^{ \, \mathrm{odd}(\eta(\varepsilon;k)')} q^{|\eta(\varepsilon;k)|} = \sum_{\varepsilon \in \{ 0,1 \}^k } (q \zeta)^{|\varepsilon|} q^{2(k-|\varepsilon|) + 2 \, |\widetilde{\eta}(\varepsilon;k)|}.
    \end{equation}
    By a notable combinatorial property of the Gaussian binomial coefficient \cite[Section 10]{AndrewsAskeyRoy2000}, the right hand side becomes, summing over fixed $|\varepsilon|$,
    \begin{equation}
        \mathcal{Z}(\zeta,k) = \sum_{j=0}^k (q \zeta)^j q^{2(k-j)} \Qbinomial{k}{j}{q^2}  = [q \zeta + q^2]_{q^2}^k.
    \end{equation}
    Then the function $\mathdutchcal{g}_k$, becomes, 
    \begin{equation}
        \mathdutchcal{g}_k(z,q) = \sum_{\substack{\widetilde{\nu} : \widetilde{\nu}' \text{ is even } \\ \widetilde{\nu}_1 \le k}} q^{|\widetilde{\nu}|} \mathcal{Z}(z^2,k),
    \end{equation}
    proving \eqref{eq:g_k}. Exact formula \eqref{eq:G_k} easily follows from \eqref{eq:g_k} by induction.
\end{proof}

\begin{lemma} \label{lemma:b_mu_z_q}
    For all $\mu$, we have
    \begin{equation}
        \sum_{\kappa \in \mathcal{K}(\mu) } q^{|\kappa|} z^{\mathrm{odd}(\kappa) + \mathrm{odd}(\mu' + \kappa) } = \mathdutchcal{b}_\mu (q;z).
    \end{equation}
\end{lemma}
\begin{proof}
    Summing over all different components of $\kappa = (\kappa^{(1)},\kappa^{(2)},\dots)$ and utilizing \eqref{eq:G_k} we obtain the claimed result.
\end{proof}

We finally come to the proof of \eqref{eq:Littlewood_like_id}.

\begin{proof}[Proof of \cref{thm:Littlewood_like_id}]
    By making use of computation reported in \cref{lemma:b_mu_z_q} we obtain
    \begin{equation}
    \begin{split}
        \sum_{\mu} \mathdutchcal{b}_\mu (q;z) \mathscr{P}_\mu(x;q^2)
        &= \sum_{\mu\in \mathbb{Y}} \sum_{V \in VST(\mu)} \sum_{\kappa \in \mathcal{K}(\mu)} q^{|\kappa| + 2 H(V)} z^{\mathrm{odd}(\kappa) + \mathrm{odd}(\mu' + \kappa) } x^V
        \\
        &=
        \sum_{\overline{\pi} \in \overline{\mathbb{A}}_{n,n}^+ : \overline{\pi} = \overline{\pi}^{-1}} z^{\mathrm{fixed}(\overline{\pi})}
        q^{\wt(\overline{\pi})}
        x^{p(\overline{\pi})}
        \\
        &=
        \prod_{k \ge 0}
        \prod_{i=1}^n \frac{1}{1-q^k z x_i} \prod_{1 \le i<j\le n} \frac{1}{1-q^{2k} x_i x_j}.
    \end{split}
    \end{equation}
\end{proof}

\begin{remark} \label{remark:general_specializations}
    Identities \eqref{eq:Cauchy_id}, \eqref{eq:Littlewood_like_id} hold both numerically and formally in the algebra of symmetric functions. In this second case variables $x$ can be thought as generic algebra homomorphisms defined on the (algebraic) basis of power sum symmetric functions $\{p_n;n\in \mathbb{N}_0 \}$ as
    \begin{equation}
        x : p_n \mapsto x (p_n).
    \end{equation}
\end{remark}
    
\begin{remark} \label{rem:from_qW_to_HL}
    It is known \cite{Macdonald1995} that the algebra homomorphism
    \begin{equation}
        \omega_{u,v} : p_r \mapsto (-1)^{r-1} \frac{1 - u^r}{1-v^r} p_r,
    \end{equation}
    acts on Macdonald polynomials $\mathscr{P}_\mu(x;q,t),\mathscr{Q}_\mu(x;q,t)$ as
    \begin{equation}\label{eq:symmetry_Macdonald}
        \omega_{q,t} \mathscr{P}_{\mu} (x;q,t) = \mathscr{Q}_{\mu'}(x, t, q),
    \end{equation}
    Then, applying $\omega_{q^2,0}$ to both sides of \eqref{eq:Littlewood_like_id} and renaming parameters $q \mapsto t$ yields the identity for Hall-Littlewood polynomials $\mathscr{Q}_\mu(x;q=0,t)$
    \begin{equation}
        \sum_\mu \mathdutchcal{b}_{\mu'}(t;z) \mathscr{Q}_\mu (x;0,t^2) = \prod_{i=1}^n \frac{(1+ z x_i)(1+t z x_i)}{1-x^2_i} \prod_{1\le i < j \le n} \frac{1-t^2 x_i x_j}{ 1- x_i x_j}.
    \end{equation}
    This identity is a particular case of \cite[Theorem 1.1]{Warnaar_Rogers_Szego}, which in turn interpolates between one of Macdonald's Littlewood identities \cite{Macdonald1995} and Kawanaka's Littlewood identity \cite{Kawanaka_1991}.
\end{remark}

\begin{remark}
    In this paper we have focused our attention on $q$-Whittaker polynomials, which naturally arise as generating functions of vertically strict tableaux. Following recipe outlined in \cref{subs:extensions} it should be possible to study bijectively summation identities involving \emph{modified Hall-Littlewood polynomials} $\mathscr{Q}'_{\mu}(x;q)$; see \cite{Desarmenien_Hall_Littlewood} for a review. They can be defined as generating function of row weak tableaux of fixed shape and weighted by a suitable adaptation of the intrinsic energy function \cite{Nakayashiki_Yamada}.
\end{remark}

\subsection{Identities between summations of $q$-Whittaker and skew Schur functions} \label{subs:summations_qW_Schur}

Bijection presented in \cref{thm:new_bijection}, along with generalization of Schensted's theorem of \cref{thm:Schensted}, reveal correspondences between certain summations of $q$-Whittaker polynomials and skew Schur polynomials. For partitions $\rho \subseteq \lambda$ define the skew Schur polynomial \cite{Macdonald1995} in $n$ variables $x=(x_1,\dots, x_n)$ as
\begin{equation} \label{eq:skew_Schur_combinatorial}
    s_{\lambda/\rho} (x) = \sum_{P \in SST(\lambda/\rho,n)} x^P.
\end{equation}

The following theorem was first proved in \cite{IMS_matching} using methods coming from integrable probability. Here we give its bijective proof.

\begin{theorem} \label{thm:qW_and_Schur_1}
    Fix $|q|<1$ and set of variables $x=(x_1,\dots,x_n)$, $y=(y_1,\dots,y_n)$. Then, for all $k=0,1,2,\dots$, we have
    \begin{equation} \label{eq:qW_and_Schur_1}
        \sum_{\ell=0}^k \frac{q^\ell}{(q;q)_\ell} \sum_{ \mu: \mu_1 = k - \ell} \mathdutchcal{b}_\mu(q) \mathscr{P}_\mu (x;q) \mathscr{P}_\mu (y;q) = \sum_{\lambda,\rho : \lambda_1= k} q^{|\rho|} s_{\lambda / \rho}(x) s_{\lambda / \rho}(y).
    \end{equation}
\end{theorem}

\begin{proof}
    Right hand side of \eqref{eq:qW_and_Schur_1} can be written, by means of bijection of \cref{thm:new_bijection} and \cref{thm:Schensted}, as
    \begin{equation}
        \begin{split}
            \sum_{\lambda,\rho : \lambda_1= k} q^{|\rho|} s_{\lambda / \rho}(x) s_{\lambda / \rho}(y)
            &=
            \sum_{\rho, \lambda: \lambda_1=k}\sum_{P,Q\in:SST(\lambda/\rho,n) } q^{|\rho|} x^P y^Q
            \\
            &=\sum_{\substack{\nu,\mu \\ \nu_1+\mu_1=k}} \sum_{\substack{\overline{\pi} \in \overline{\mathbb{A}}_{n,n}^+ \\ \mu(\overline{\pi}) = \mu } }
            q^{|\nu| + \wt(\overline{\pi})} x^{p(\overline{\pi})} y^{q(\overline{\pi})}
            \\
            &=
            \sum_{\ell=0}^k \sum_{\mu:\mu_1=k-\ell} \sum_{\nu: \nu_1=\ell} q^{|\nu|} \sum_{\kappa \in \mathcal{K}(\mu)} q^{|\kappa|} \sum_{V,W\in VST(\mu)}  q^{\mathscr{H}(V) + \mathscr{H}(W)} x^{V} y^{W},
        \end{split}
    \end{equation}
    which reduce to the left hand side after putting all summations in closed form.
\end{proof}

Imposing a symmetry to our bijection, we can easily prove the following additional identity.

\begin{theorem} \label{thm:qW_and_Schur_2}
    Fix $|q|<1$ and set of variables $x=(x_1,\dots,x_n)$. Then, recalling notation \eqref{eq:b_q_z}, \eqref{eq:g_k}, we have
    \begin{equation} \label{eq:qW_and_Schur_2}
        \sum_{\ell=0}^k \mathdutchcal{g}_\ell(z,q) \sum_{ \mu: \mu_1 = k - \ell} \mathdutchcal{b}_\mu(q;z) \mathscr{P}_\mu (x;q^2) = \sum_{\lambda,\rho : \lambda_1= k} z^{\mathrm{odd}(\lambda') + \mathrm{odd}(\rho')} q^{|\rho|} s_{\lambda / \rho}(x)
    \end{equation}
     for all $k=0,1,2,\dots$
\end{theorem}

\begin{proof}
    Using \cref{lemma:odd_lamda_rho}, \cref{lemma:fixed_pi} and bijection of \cref{thm:new_bijection}, right hand side of \eqref{eq:qW_and_Schur_2} can be written as
    \begin{equation*}
        \begin{split}
            \sum_{\lambda,\rho : \lambda_1= k} z^{\mathrm{odd}(\lambda') + \mathrm{odd}(\rho')} q^{|\rho|} s_{\lambda / \rho}(x)
            &= 
            \sum_{\rho,\lambda:\lambda_1=k} \sum_{P\in SST(\lambda/\rho,n)} 
            z^{\mathrm{odd}(\lambda') + \mathrm{odd}(\rho')} q^{|\rho|}
            x^P
            \\
            &=\sum_{\substack{\nu,\mu \\ \nu_1+\mu_1=k}} \sum_{\substack{\overline{\pi} \in \overline{\mathbb{A}}_{n,n}^+: \overline{\pi} = \overline{\pi}^{-1} \\ \mu(\overline{\pi}) = \mu } }
            q^{|\nu| + \wt(\overline{\pi})}
            z^{\mathrm{fixed}(\overline{\pi}) + 2\mathrm{odd}(\nu')}
            x^{p(\overline{\pi})} 
            \\
            &=
            \sum_{\ell=0}^k \sum_{\nu: \nu_1=\ell} q^{|\nu|} z^{2\mathrm{odd}(\nu')} 
            \\
            &
            \qquad
            \times 
            \sum_{\mu:\mu_1=k-\ell} \sum_{\kappa \in \mathcal{K}(\mu)} q^{|\kappa|} z^{\mathrm{odd}(\kappa)+\mathrm{odd}(\kappa + \mu')} \sum_{V\in VST(\mu)}  q^{2\mathscr{H}(V)} x^{V},
        \end{split}
    \end{equation*}
    which reduces to the left hand side after using \eqref{eq:g_k}, \cref{lemma:b_mu_z_q}.
\end{proof}

\begin{remark}
    Identities stated in \cref{thm:qW_and_Schur_1,thm:qW_and_Schur_2} can be further refined taking advantage of homogeneity of $q$-Whittaker and skew Schur polynomials. For instance, for any fixed $k,N=0,1,2,\dots$ we have
    \begin{equation}
        \sum_{\ell=0}^k \frac{q^\ell}{(q;q)_\ell} \sum_{\substack{ \mu: \mu_1 = k - \ell \\ |\mu|=N}} \mathdutchcal{b}_\mu(q) \mathscr{P}_\mu (x;q) \mathscr{P}_\mu (y;q) = \sum_{\substack{\lambda,\rho : \lambda_1= k \\ |\lambda / \rho |=N}} q^{|\rho|} s_{\lambda / \rho}(x) s_{\lambda / \rho}(y).
    \end{equation}
    A similar refinement can be given for \eqref{eq:qW_and_Schur_2}, fixing the degree $N$ of polynomials in left and right hand side.
\end{remark}

\begin{remark}
    Just as discussed in \cref{remark:general_specializations}, also identities \eqref{eq:qW_and_Schur_1}, \eqref{eq:qW_and_Schur_2} hold both numerically and formally in the algebra of symmetric function. They are therefore still true if variables $x$ are replaced by algebra homomorphisms $x:p_n\mapsto x(p_n)$. An application of this fact is that, through the action of $\omega_{q^2,0}$, \eqref{eq:qW_and_Schur_1}, \eqref{eq:qW_and_Schur_2} turn into summation identities relating Hall-Littlewood symmetric polynomials $\mathscr{P}_\mu(x,q=0,t)$ and Schur functions. This fact has deep consequences in the context of stochastic solvable models related to $q$-Whittaker and Hall-Littlewood symmetric polynomials (see \cite{BorodinCorwin2011Macdonald,vuletic2009generalization,BorodinBufetovWheeler2016,barraquand_half_space_mac,barraquand2018} and we will investigate these aspects in a forthcoming paper \cite{IMS_KPZ_free_fermions}).
\end{remark}

\appendix

\section{Knuth relations and generalizations} \label{app:Knuth_rel}

\subsection{Knuth equivalence and Jeu de taquin} \label{subs:jdt} In this subsection we cover some prerequisites on the theories of Knuth relations and jeu de taquin and on their interplay. The material presented here is standard and for a more detailed expositions on the topic we suggest the interested reader to consult textbooks as \cite{sagan2001symmetric,lothaire_2002}. 

Following \cite{Knuth1970}, on the set of words $\mathcal{A}_n^*$ we define the Knuth relation $\pi \simeq \pi'$ as the equivalence relation generated by the transformations,
\begin{equation} \label{eq:Knuth_rel}
    \alpha \, x \, z \, y \, \beta \rightleftharpoons \alpha \,z \, x \, y \, \beta
    \hspace{.4cm}
    \text{if } x\le y < z,
    \hspace{.4cm}
    \text{and}
    \hspace{.4cm}
    \alpha \, y \, z \, x \, \beta \rightleftharpoons \alpha \, y \, z \,x \, \beta,
    \hspace{.4cm}
    \text{if } x < y \le z,
\end{equation}
where $\alpha,\beta$ are generic words. These are often called elementary transformations or Knuth moves. In practice they represent a realization in the language of words of the Schensted's insertion of a letter in a row of a tableau. To explain this analogy take a word $w=w_1\cdots w_k$ with letters in weakly increasing order $w_1 \le \cdots \le w_k$, which we can interpret as a word formed reading a row of a semi-standard tableau. Then, for any $x < w_k$ we have, applying \eqref{eq:Knuth_rel} repeatedly,
\begin{equation}
    wx \simeq x^* w^*,
\end{equation}
where $x^*$ is the smallest $w_i$ to be strictly bigger than $x$ and $w^*$ is the word obtained from $w$ substituting $x^*$ with $x$. On the other hand if $x \ge w_k$, then no Knuth moves can be applied to transform the word $w^*=wx$. We see that in both cases $w^*$ is the row word after the insertion of $x$ and when $x<w_k$, the letter $x^*$ will be the one inserted in the following row. This idea motivates the characterization of Knuth equivalence classes.

\begin{theorem}[\cite{sagan2001symmetric} Theorem 3.4.3] \label{thm:Knuth_eq_P}
    Two words $\pi,\pi'$ are equivalent if and only if their $P$-tableaux under RSK correspondence is equal $P(\pi)=P(\pi')$.  
\end{theorem}

The notion of Knuth equivalence extends also at the level of skew shaped semi-standard tableaux. We say that two tableaux $P$ and $P'$ are \emph{Knuth equivalent} or simply equivalent if their row reading words $w_P$ and $w_{P'}$ are, in which case we write $P \simeq P'$. Equivalent tableaux enjoy the property that they can be transformed into each other through the procedure of \emph{jeu de taquin}. This operation is described in terms of \emph{sliding moves}. We say that a semi-standard tableau $P$ is \emph{punctured} if we replace the entry of one or more of its cells with the symbol $\bullet$. From a punctured tableau $P$ we can remove the $\bullet$-cells as follows. 
Assume that cells $(i,j),(i+1,j)$ and $(i,j+1)$ have respectively entries $\bullet, a$ and $b$. Then if $a \le b$ we exchange the labels of cells $(i,j)$ and $(i+1,j)$, while if $a>b$ we swap the labels at $(i,j)$ and $(i,j+1)$. This single move is the \emph{inward sliding} and graphically we have
\begin{equation*}
    \ytableausetup{aligntableaux = center}
    \begin{ytableau}
        \bullet & b  
        \\
        a
    \end{ytableau}
    \xrightarrow{\hspace{10pt} \text{if } a\le b \hspace{10pt}}
    \begin{ytableau}
        a & b  
        \\
        \bullet
    \end{ytableau}
    \hspace{50pt}
    \begin{ytableau}
        \bullet & b  
        \\
        a
    \end{ytableau}
    \xrightarrow{\hspace{10pt} \text{if } a > b \hspace{10pt}}
    \begin{ytableau}
        b & \bullet  
        \\
        a
    \end{ytableau}
    \,
    .
\end{equation*}
When either cell $(i+1,j)$ or $(i,j+1)$ are not part of the shape of the tableaux we think of their value as infinite, while when the $\bullet$-cell reaches an external corner we simply erase it. From a tableau $P$ of shape $\lambda/\mu$ let $c$ be an external corner of $\mu$. The \emph{outward jeu de taquin} $J_{c}(P)$ is the tableau obtained puncturing the cell $c$ of $P$, and sliding out the $\bullet$-cell.  
    
The sliding moves can be also defined in the opposite direction, moving the $\bullet$-cells inwards.
If in a punctured tableaux the cells $(i,j), (i,j-1)$ and $(i-1,j)$ have labels $\bullet, a$ and $b$ we will slide $b$ up in case $a\le b$, while $a$ is shifted rightward when $a > b$, as in
\begin{equation*}
    \ytableausetup{aligntableaux = center}
    \begin{ytableau}
        \none & b  
        \\
        a & \bullet 
    \end{ytableau}
    \xrightarrow{\hspace{10pt} \text{if } a\le b \hspace{10pt}}
    \begin{ytableau}
        \none & \bullet  
        \\
        a & b
    \end{ytableau}
    \hspace{50pt}
    \begin{ytableau}
        \none & b
        \\
        a &\bullet
    \end{ytableau}
    \xrightarrow{\hspace{10pt} \text{if } a > b \hspace{10pt}}
    \begin{ytableau}
        \none & b  
        \\
        \bullet & a
    \end{ytableau}
    \,
    .
\end{equation*}
After a number of inward slides the $\bullet$-cell will reach an inner corner and in that case it is erased. This procedure defines inward jeu de taquin transformations. If $P$ is a skew tableau of shape $\lambda/\mu$ and $c\notin \lambda$ is such that $c-\mathbf{e}_1,c-\mathbf{e}_2 \in \lambda$, then $J_c(P)$ is the tableau obtained from $P$ puncturing the cell $c$ and sliding inward the $\bullet$-cell. We do not differentiate the notation between inward and outward jeu de taquin as the choice of the cell $c$ dictates the direction of sliding.
    
The jeu de taquin can be employed to associate to skew-shaped semi-standard tableaux canonical straight shaped ones. Given a tableau $P$ of shape $\lambda/ \mu$ we fill the empty shape $\mu$ with $\bullet$-cells and subsequently we slide them all out. The result is a tableaux of straight shape $\widetilde{\lambda}$ called \emph{jeu de taquin rectification} of $P$ and denoted as $\mathrm{rect}(P)$. It is a theorem of Sch{\"u}tzenberger \cite[Theorem 3.7.7]{sagan2001symmetric} that the rectification is independent of the order of sliding moves.
    
The following classical theorem states the relation between Knuth equivalence of tableaux and jeu de taquin.
    
\begin{theorem}[\cite{schutzenberger_RS},\cite{sagan2001symmetric} Theorem 3.7.8] \label{thm:Knuth_equiv_jdt}
Two tableaux $P,P'$ are equivalent if and only if their jeu de taquin rectification is equal. In particular $P \simeq P'$ if and only if they can be transformed into each other through a finite sequence of jeu de taquin moves.
\end{theorem}

\begin{remark}
    An equivalent definition of the Knuth equivalence between tableaux $P,P'$ can be given requiring that their \emph{column reading words} $w^{\mathrm{col}}_P \simeq w^{\mathrm{col}}_{P'}$. More in general, in \cite{fomin_greene_1993} authors discuss a full class or reading orders for tableaux, which include row and column ones, producing equivalent theories of Knuth equivalence.
\end{remark}

We close this subsection stating several simple but crucial properties that endow the skew $\RSK$ map with its many symmetries.
    
\begin{proposition} \label{prop:Knuth_eq_int_ins}
Let $P$ be a semi-standard skew tableau and $P' = \mathcal{R}_{[r]}(P)$ for some row $r$. Then $P \simeq P'$.
\end{proposition}
\begin{proof}
    The tableau $P'$ is obtained from $P$ vacating the rightmost cell at row $r$ and inserting the entry of the vacated cell in the row below. The fact that the insertion algorithm is reproduced at the level of row words by a sequence of Knuth moves yields the proof. 
\end{proof}
    
\begin{proposition} \label{prop:RSK_via_jdt}
    Let $P,Q$ be semi-standard tableaux with same skew shape and take $(P',Q')=\RSK(P,Q)$. Then $P \simeq P'$ and $Q \simeq Q'$.
\end{proposition}

\begin{proof}
    By the result stated in \cref{prop:RSK_from_n_int_ins} $P'$ and $Q'$ are obtained respectively from $P$ and $Q$ after a sequence of internal insertions. Then, by \cref{prop:Knuth_eq_int_ins} we have $P\simeq P'$ and $Q\simeq Q'$.
\end{proof}
    
\begin{proposition} \label{prop:jdt_of_PQ}
    Let $\overline{\pi}\in \overline{\mathbb{A}}_{n,n}$ and consider the pair of tableaux $(P,Q) \xleftrightarrow[]{\skwRSK \,} (\overline{\pi} ; \nu)$, for some partition $\nu$. Define $\pi' = p(\overline{\pi}^{\na})$ and $\pi''=p((\overline{\pi}^{-1})^{\na})$. Then
    \begin{equation}
        P \simeq P(\pi')
        \qquad
        \text{and}
        \qquad
        Q \simeq P(\pi''). 
    \end{equation}
\end{proposition}

\begin{proof}
    From the definition of the Sagan-Stanely correspondence we understand that the timetable ordering $\overline{\pi}^{\na}$ records the``times" at which each entry of the $P$ tableau is inserted in its 1st row. More precisely the word $\pi'=p(\overline{\pi}^{\na})$ is the list of such entries in order of insertion. By \cref{prop:Knuth_eq_int_ins}, $\pi'$ is Knuth equivalent to the row word of $P$ and therefore $P \simeq P(\pi')$. The alternative statement for the $Q$ tableau is proven analogously using the swap symmetry of \cref{prop:RSK_swap_symmetry}.
\end{proof}

\subsection{Dual equivalence} 

Two words $\pi, \pi'$ are \emph{dual equivalent} if their $Q$-tableaux under RSK correspondence are equal. We denote dual equivalence by $\pi \stackrel{*}{\simeq} \pi'$. We say that two skew tableaux $P,P'$ of the same shape are \emph{dual equivalent} if their column reading words are dual equivalent and in this case we write $P \stackrel{*}{\simeq} P'$. Again by a result of \cite{fomin_greene_1993}, the notion of dual equivalence does not depend on the reading order of the tableaux and in literature often the row reading is used. The theory of dual equivalence was started by Haiman in \cite{HAIMAN1992} and below we present two classical results that will be relevant to us.

\begin{theorem}[\cite{HAIMAN1992},\cite{sagan2001symmetric} Theorem 3.8.8] \label{thm:haiman_dual_eq}
Two tableaux $P_1,P_2$ are dual equivalent if and only if for any sequence of jeu de taquin slides $J$, the tableaux $J(P_1), J(P_2)$ have the same shape.
\end{theorem}

\begin{theorem}[\cite{HAIMAN1992} Theorem 2.13] \label{thm:uniqueness_dual_eq_non_dual_eq}
    Let $P,P'$ be two semi-standard tableaux with same skew shape. Then $P\simeq P'$ and $P \stackrel{*}{\simeq} P'$ if and only if $P=P'$.
\end{theorem}

In the following theorem we use the notion of Kashiwara operators defined in \cref{subs:classical_crystals}.

\begin{theorem}[\cite{lothaire_2002}, Theorem 5.5.1] \label{thm:kash_op_equiv}
    Let $h$ be anyone between $\E{i},\F{i},i=1,\dots,n-1$ and $\pi \in \mathcal{A}_n^*$ such that $h(\pi) \neq \varnothing$. Then
    \begin{enumerate}
        \item $h(\pi) \stackrel{*}{\simeq} \pi $;
        \item if $\pi' \simeq \pi$ then $h(\pi') \simeq h(\pi)$.
    \end{enumerate}
\end{theorem}

\subsection{Generalized Knuth relations for weighted words} \label{subs:generalized_Knuth_rel}

We recall that a weighted word is just a weighted biword $\overline{\pi}$ having $q$-word $q(\overline{\pi})=12 \cdots k$, where $k$ is the lenght of $\overline{\pi}$. Borrowing a notation used in \cite{sagan1990robinson} we will write such a $\overline{\pi}$ as a word in the \emph{weighted alphabet} $\overline{\mathcal{A}}_n$ consisting of symbols $a^{(w)}$, where $a\in \{1,\dots ,n\}$ and $w\in \mathbb{Z}$. In this more compact notation, for instance, we write
\begin{equation}
    \overline{\pi}=
    \left( \begin{matrix}
        1 & 2 & 3 & 4 & 5
        \\
        2 & 1 & 3 & 1 & 2
        \\
        1 & -1 & 0 & 0 & 1
    \end{matrix} \right)
    \qquad
    \text{as}
    \qquad
    \overline{\pi}=
    2^{(1)} 1^{(-1)} 3^{(0)} 1^{(0)} 2^{(1)}.
\end{equation}
On $\overline{\mathcal{A}}_n$ we introduce the total ordering $\prec$ defined as
\begin{equation} \label{eq:total_ordering}
        a^{(w)} \prec b^{(w')}
        \qquad
        \text{if}
        \quad
        w > w',
        \qquad
        \text{or}
        \quad
        w=w' , a< b.
\end{equation}
The following definition generalizes the classical Knuth relations recalled in \eqref{eq:Knuth_rel}.
\begin{definition}[Generalized Knuth relations]
    The generalized Knuth relations $\overline{\pi} \simeq_{\mathrm{g}} \overline{\pi}'$ is the equivalence relation on weighted words generated by the transformations
\begin{equation} \label{eq:gen_Knuth_rel}
\begin{split}
    \alpha \, x^{(w_x)} \, z^{(w_z)} \, y^{(w_y)} \, \beta \rightleftharpoons \alpha \,z^{(w_z)} \, x^{(w_x)} \, y^{(w_y)} \, \beta,
    \qquad
    &\text{if } x^{(w_x)} \preceq y^{(w_y)} \prec z^{(w_z)},
    \\
    \alpha \, y^{(w_y)} \, z^{(w_z)} \, x^{(w_x)} \, \beta \rightleftharpoons \alpha \, y^{(w_y)} \, z^{(w_z)} \,x^{(w_x)} \, \beta,
    \qquad
    &\text{if } x^{(w_x)} \prec y^{(w_y)} \preceq z^{(w_z)},
\end{split}
\end{equation}
where $\alpha,\beta$ are generic weighted words. If a transformation swaps the $i$-th and the $i+1$-th letter of weighted word we say that such transformation has \emph{type i}. For a graphical interpretation of \eqref{eq:gen_Knuth_rel} see \cref{fig:Knuth_rel}. 
\end{definition}

\begin{figure}[h]
        \centering
        \subfloat[]{\includegraphics[scale=.78]{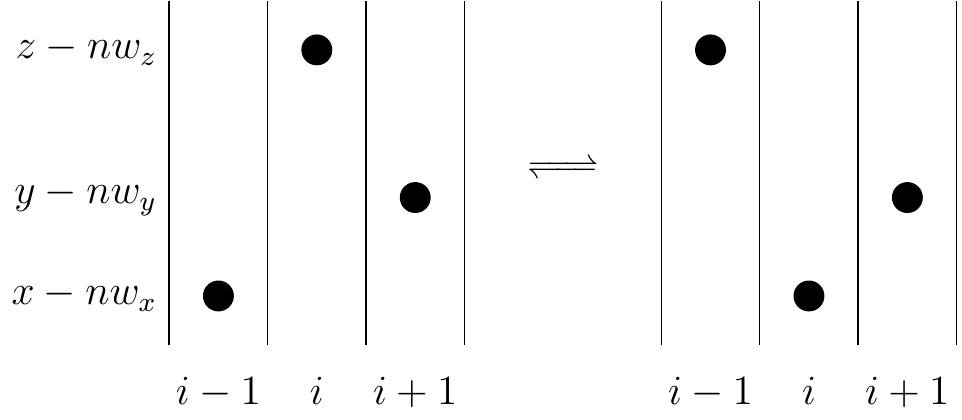}}%
        \hspace{1cm}
        \subfloat[]{{\includegraphics[scale=.78]{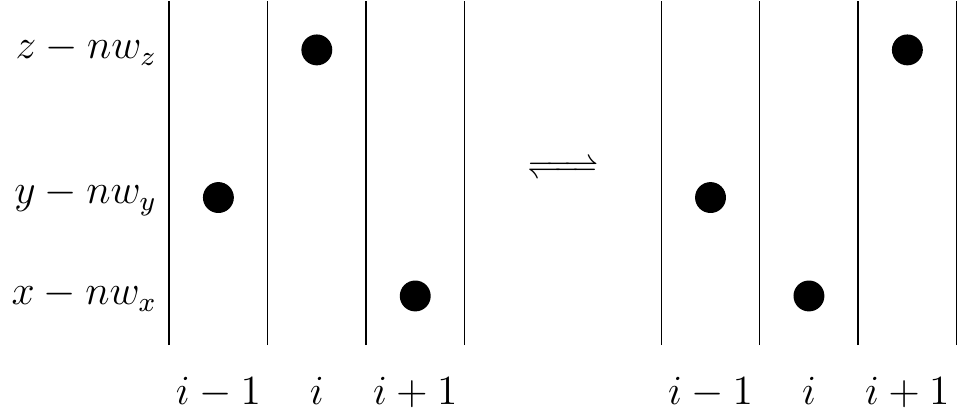} }}%
        \caption{Representing weighted words as point configurations on the twisted cylinder $\mathscr{C}_n$, the two relations \eqref{eq:gen_Knuth_rel} correspond respectively to the left and right panel above.}
        \label{fig:Knuth_rel}
\end{figure}

The following theorem offers a characterization of generalized Knuth equivalence classes, that partially extends \cref{thm:Knuth_eq_P}. Given a weighted biword $\overline{\pi}$ we denote with $(P_t(\overline{\pi}),Q_t(\overline{\pi}))$ the skew $\RSK$ dynamics with initial data $(P,Q)\xrightarrow[]{\skwRSK\,}\overline{\pi}$.

\begin{theorem} \label{thm:gen_Knuth_rel}
    Consider a pair of weighted words $\overline{\pi} \simeq_\mathrm{g} \overline{\pi}'$. Then for all $t \in \mathbb{Z}$ we have $P_t(\overline{\pi}) = P_t(\overline{\pi}')$. 
\end{theorem}

The proof of \cref{thm:gen_Knuth_rel} is based on a simple quasi-commutation relation between Knuth relations and the Viennot map.

\begin{lemma} \label{lemma:quasi_comm_Knuth_Viennot}
    Let $\overline{\pi}$ and $\overline{\pi}'$ be weighted permutations differing by a single Knuth transformation of type $i$. Then $\mathbf{V}(\overline{\pi})$ and $\mathbf{V}(\overline{\pi}')$ also differ by a single Knuth transformation whose type is either $i-1$, $i$ or $i+1$.
\end{lemma}

\begin{proof}

Our statement is best proven through a graphical argument. In \eqref{eq:Knuth_viennot_1}, \eqref{eq:Knuth_viennot_2}, \eqref{eq:Knuth_viennot_3} we give a schematic representation of all possible cases that could present while performing a generalized Knuth transformation. We focus only on the sector $\mathsf{S}_{i,i+2}=\{i,i+1,i+2\} \times \mathbb{Z} \subseteq \mathscr{C}_n$ where the transformation takes place as the remaining part of the construction is determined by unaffected points outside $\mathsf{S}_{i,i+2}$ and by the heights $j_1,j_2,j_3$ and $j_1',j_2',j_3'$ of the horizontal segments originating and terminating inside $\mathsf{S}_{i,i+2}$. 
In \eqref{eq:Knuth_viennot_1}, \eqref{eq:Knuth_viennot_2} we show the cases where the first transformation of \eqref{eq:gen_Knuth_rel} is applied  
\begin{equation} \label{eq:Knuth_viennot_1}
    \begin{minipage}{.4\linewidth}
        \frame{\includegraphics[scale=.95]{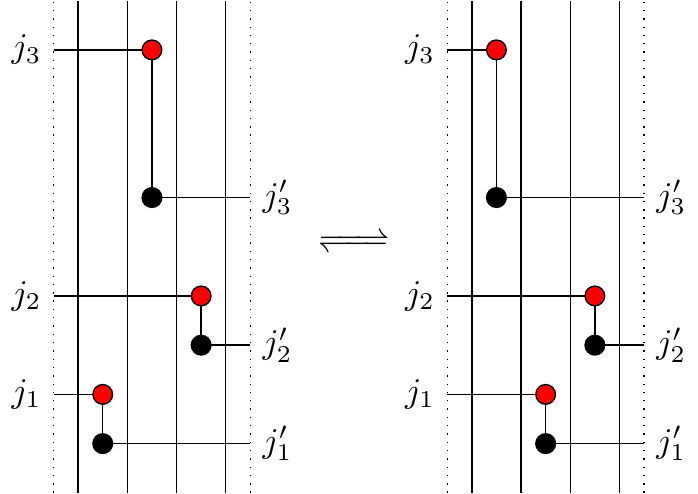}}
    \end{minipage}%
    \hspace{.5cm}
    ,
    \hspace{.5cm}
    \begin{minipage}{.4\linewidth}
        \frame{\includegraphics[scale=.95]{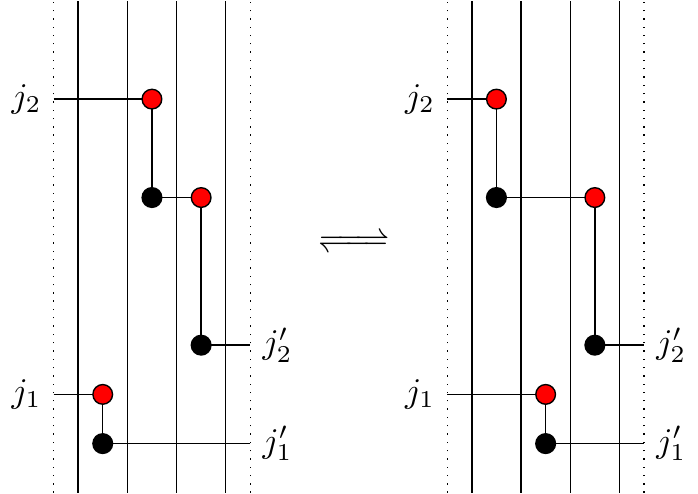}}
    \end{minipage}
    \hspace{.2cm},
\end{equation}
\begin{equation} \label{eq:Knuth_viennot_2}
    \begin{minipage}{.45\linewidth}
        \frame{\includegraphics[scale=.95]{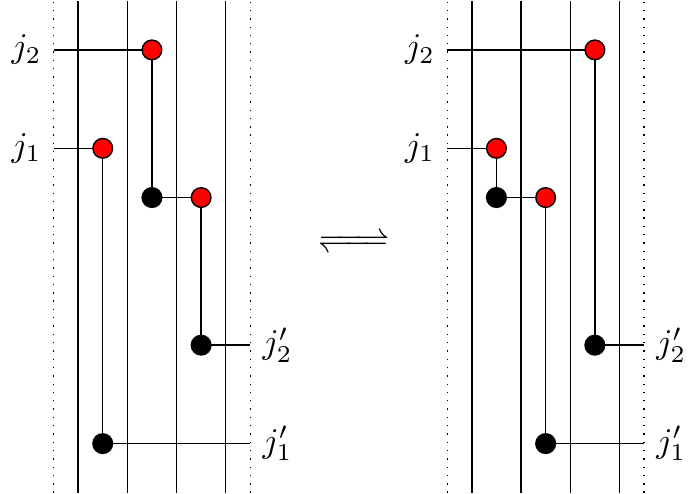}}
    \end{minipage}
    .
\end{equation}
The cases where the second transformation of \eqref{eq:gen_Knuth_rel} is applied are shown below
\begin{equation} \label{eq:Knuth_viennot_3}
    \begin{minipage}{.4\linewidth}
        \frame{\includegraphics[scale=.95]{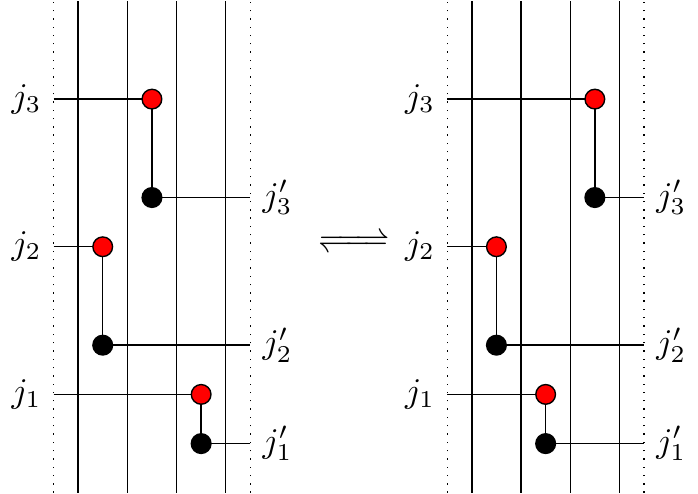}}
    \end{minipage}%
    \hspace{.5cm}
    ,
    \hspace{.5cm}
    \begin{minipage}{.4\linewidth}
        \frame{\includegraphics[scale=.95]{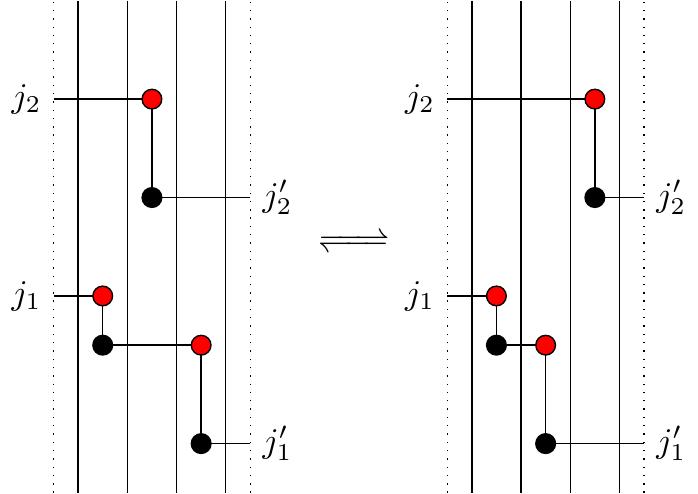}}
    \end{minipage}
    \hspace{.2cm} .
\end{equation}

In all cases, except for \eqref{eq:Knuth_viennot_2}, the generalized Knuth transformation separating $\mathbf{V}(\overline{\pi})$ and $\mathbf{V}(\overline{\pi}')$ is the same separating $\overline{\pi}$ and $\overline{\pi}'$. 
The only non trivial assumption made in the figures above is that the heights $j_1,j_2,j_3$ of shadow lines entering the sector  $\mathsf{S}_{i,i+2}$ are not affected by the transformations. This is a consequence of the fact that the position of horizontal lines exiting $\mathsf{S}_{i,i+2}$ at heights $j_1',j_2',j_3'$ also does not change while swapping highest and lowest point of the triple. 
\end{proof}

\begin{proof}[Proof of \cref{thm:gen_Knuth_rel}]

It is sufficient to prove this theorem in case $\overline{\pi}$ is a weighted permutation as this result would extend to weighted words via standardization. We show that if $\overline{\pi}$ and $\overline{\pi}'$ are separated by a single generalized Knuth relation then $P_t(\overline{\pi}) = P_t(\overline{\pi}')$ for all $t$. We recall that $P_t(\overline{\pi})$ is obtained reading values of east edges $\mathsf{E}(n,n(t-1)+1),\dots,\mathsf{E}(n,nt)$ of the edge configuration $\mathcal{E}(\overline{\pi})$. On the other hand, if $\{\overline{\pi}^{(t)}\}_t$ is the Viennot dynamics with initial data $\overline{\pi}^{(1)}=\overline{\pi}$, then $\mathcal{E}(\overline{\pi})$ is completely determined by the shadow line constructions of transitions $\overline{\pi}^{(t)} \to \overline{\pi}^{(t+1)}$. In fact the segments of each shadow line construction determine all edges having a same fixed value. By arguments presented in the proof of \cref{lemma:quasi_comm_Knuth_Viennot} if the Knuth transformation $\overline{\pi}\to \overline{\pi}'$ involves $i,i+1$ and $i+2$-th letters of both weighted permutations, than this is also the case for $\mathbf{V}(\overline{\pi}) \to  \mathbf{V} (\overline{\pi}')$. Moreover the shadow line construction is unaffected at columns different than $i,i+1,i+2$. This implies that edge configurations $\mathcal{E}(\overline{\pi})$ and $\mathcal{E}(\overline{\pi}')$ differently at edges corresponding to columns $i,i+1,i+2$ and moreover east edges $\mathsf{E}(n,j)$ are common for all $j$. This concludes the proof.
\end{proof}

\begin{remark}
It is easy to see that if two weighted words $\overline{\pi}$ and $\overline{\pi}'$ have the same $P$ tableaux under Sagan-Stanley correspondence then they are not necessarily connected by a sequence of generalized Knuth relations. For example $1^{(1)}2^{(0)}$ and $2^{(0)}1^{(1)}$ have the same $P$ tableau $\begin{ytableau} \, & 2 \\ 1 \end{ytableau}$, but they cannot be transformed into each other via \eqref{eq:gen_Knuth_rel}.
\end{remark}

\begin{remark}
It is also not true that if two words $\overline{\pi},\overline{\pi}'$ are such that $P_t(\overline{\pi})=P_t(\overline{\pi}')$ then they are Knuth equivalent in the generalized sense. For instance the $P$ tableaux of the skew $\RSK$ dynamics corresponding to $1^{(2)}2^{(0)}$ and $2^{(0)}1^{(2)}$ are equal for all $t$, but these words are not generalized Knuth equivalent.
\end{remark}

We conclude this subsection proposing a partial generalization of the notion of dual Knuth relation.

\begin{definition}[Generalized dual Knuth relations]
    The generalized dual Knuth relations $\overline{\pi} \stackrel{*}{\simeq}_\mathrm{g} \overline{\pi}'$ is the equivalence relation on weighted words generated by the transformations
\begin{equation} \label{eq:gen_dual_Knuth_rel}
\begin{split}
    \cdots k^{(w)} \cdots (k+2)^{(w')} \cdots (k+1)^{(w'')} \cdots \rightleftharpoons \cdots (k+1)^{(w)} \cdots (k+2)^{(w')} \cdots k^{(w'')} \cdots,
    \\
    \cdots (k+1)^{(w)} \cdots k^{(w')} \cdots (k+2)^{(w'')} \cdots \rightleftharpoons \cdots (k+2)^{(w)} \cdots k^{(w')} \cdots (k+1)^{(w'')} \cdots,
\end{split}
\end{equation}
where $w \ge w' \ge w''$. For a graphical representation of these transformation see \cref{fig:Knuth_rel_dual}.
\end{definition}

\begin{figure}[h]
        \centering
        \subfloat[]{{\includegraphics[scale=.78]{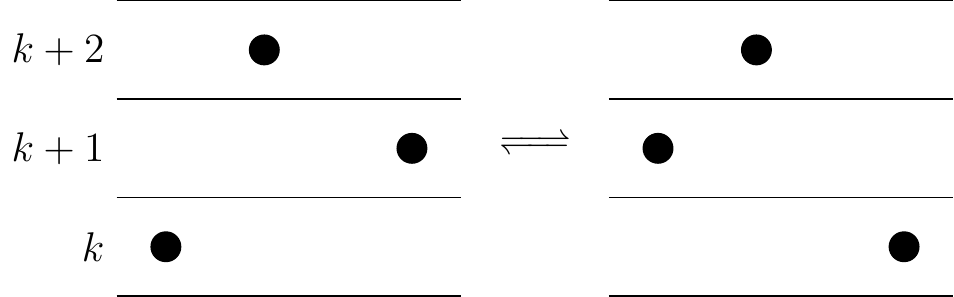}} }%
        \hspace{1cm}
        \subfloat[]{{\includegraphics[scale=.78]{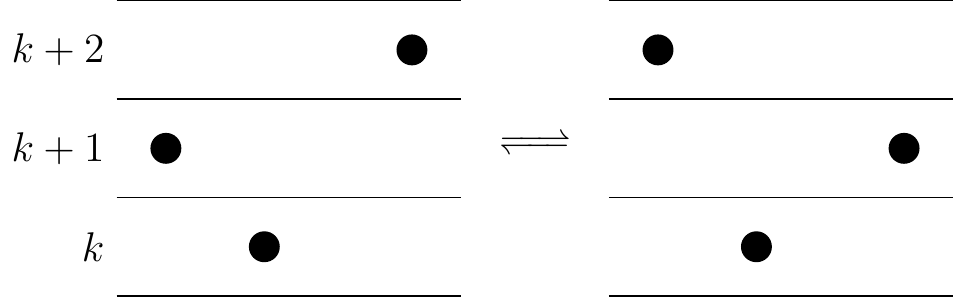} }}%
        \caption{Visualization of dual generalized Knuth relations.}
        \label{fig:Knuth_rel_dual}
\end{figure}

In the following theorem we report a statement dual to \cref{thm:Knuth_eq_P}. We omit the proof, as the argument are equivalent to those presented right above.

\begin{theorem}  \label{thm:gen_dual_knuth_rel}
    Consider a pair of weighted permutations $\overline{\pi} \stackrel{*}{\simeq}_\mathrm{g} \overline{\pi}'$. Then for all $t \in \mathbb{Z}$ we have $Q_t(\overline{\pi}) = Q_t(\overline{\pi}')$. 
\end{theorem}

\begin{theorem} \label{thm:Greene_invariants_gen_knuth_rel}
    Consider weighted permutations $\overline{\pi},\overline{\pi}'$ such that $\overline{\pi} \simeq_\mathrm{g} \overline{\pi'}$ or $\overline{\pi} \stackrel{*}{\simeq}_\mathrm{g} \overline{\pi}'$. Then, for all $k$, we have
    \begin{equation}
        I_k(\overline{\pi}) = I_k(\overline{\pi}')
        \qquad
        \text{and}
        \qquad
        D_k(\overline{\pi}) = D_k(\overline{\pi}').
    \end{equation}
\end{theorem}

\begin{proof}
    Let $(P,Q),(P',Q)'$ be pairs of tableaux such that $(P,Q)\xrightarrow[]{\skwRSK\,} \overline{\pi}$ and $(P,Q)\xrightarrow[]{\skwRSK\,} \overline{\pi}'$.
    We have shown in \cref{thm:gen_Knuth_rel} that if $\overline{\pi} \simeq_\mathrm{g} \overline{\pi'}$, then $P_t=P_t'$ for all $t$. In particular, this shows that the asymptotic increment $\mu$ is common for both pairs $(P,Q)$ and $(P',Q')$. By \cref{thm:asymptotic_shape_RSK} this implies that $I_k,D_k$ are invariant under generalized Knuth relations. The same statement for dual generalized Knuth relations can be proven analogously.  
\end{proof}

\section{Proof of \cref{thm:subsequences_crystals}} \label{app:proof_thm}

We will proceed by direct inspection. Arguments implemented here can be thought as affine generalizations of those originally presented in \cite{Danilov_Kolshevoy_bicrystals} and \cite{vanLeeuwen_crystal_matrices}. We organize the proof of \cref{thm:subsequences_crystals} in a number of lemmas. The first basic property we prove is that the inversion $\overline{\pi} \to \overline{\pi}^{-1}$ preserves increasing and localized decreasing subsequences.

\begin{lemma} \label{lemma:inversion_prererves_is_lds}
    For any $\overline{\pi} \in \overline{\mathbb{A}}_{n,n}$ and any $k$ we have
    \begin{equation}
        I_k(\overline{\pi}^{-1}) = I_k(\overline{\pi})
        \qquad
        \text{and}
        \qquad
        D_k(\overline{\pi}^{-1}) = D_k(\overline{\pi}).
    \end{equation}
\end{lemma}

\begin{proof}
    We prove that inversion $\overline{\pi} \to \overline{\pi}^{-1}$, correponding to the transposition $\overline{M}_{i,j}(k) \to \overline{M}_{j,i}(k)$, preserves both increasing and localized increasing subsequences. Let us start with localized decreasing subsequences. A path $\varsigma=(\varsigma_j: j=1,\dots,s ) \subset \mathscr{C}_n$ is a strict down-right loop if and only if its points have coordinates
    \begin{gather*}
        \varsigma_1=(j_1,i_1-nw), \dots, \varsigma_r=(j_r,i_r-nw),
        \\
        \varsigma_{r+1}=(j_{r+1},i_{r+1}-n(w+1)), \dots, \varsigma_s=(j_s,i_s-n(w+1)),
    \end{gather*}
    for some $r,w$ and numbers $i_k,j_k$ such that
    \begin{equation}\label{eq:localized_drp_condition}
        \begin{split}
            &j_1 < \cdots < j_r < j_{r+1} < \cdots < j_s,
            \\
            &i_{r+1} > \cdots > i_s > i_1 > \cdots > i_r.
        \end{split}
    \end{equation}
    Denote with $\varpi^T$ the image under transposition $(j,i-nw)\to(i,j-nw)$ of the path $\varpi$. Then its coordinates are
    \begin{gather*}
        \varpi_r^T=(i_r,j_r-nw), \dots, \varpi_1^T=(i_1,j_1-nw),
        \\
        \varpi_s^T=(i_s,j_s-n(w+1)), \dots, \varpi_{r+1}^T=(i_{r+1},j_{r+1}-n(w+1)),
    \end{gather*}
which, by \eqref{eq:localized_drp_condition}, form again a strict down-right loop.

The proof that transposition maps increasing subsequences $\varpi$ into increasing subsequences $\varpi^T$ is also straightforward and therefore we omit it.
\end{proof}

\begin{lemma} \label{lemma:iota_preserves_I_D}
    We have $I_k(\iota_\epsilon(\overline{\pi})) = I_k(\overline{\pi})$ and $D_k(\iota_\epsilon(\overline{\pi})) = D_k(\overline{\pi})$ for all $\overline{\pi}\in \overline{\mathbb{A}}_{n,n}$, $\epsilon=1,2$ and $k=1,2,\dots$.
    Equivalently, recalling the shift $T_\epsilon$ of \eqref{eq:shift_T}, $I_k(T_\epsilon(\overline{M})) = I_k(\overline{M})$ and $D_k(T_\epsilon(\overline{M})) = D_k(\overline{M})$ for all $\overline{M}\in \overline{\mathbb{M}}_{n\times n}$, $\epsilon=1,2$, $k=1,2,\dots$.
\end{lemma}
\begin{proof}
    This is obvious since shifts $T_1,T_2$ preserve up-right paths and strict down-right loops.
\end{proof}

As a result of \cref{lemma:iota_preserves_I_D,lemma:inversion_prererves_is_lds} and of definition \eqref{eq:Kashiwara_op_matrices_second}, \eqref{eq:Kashiwara_op_weighted_biwords_second} of family $\widetilde{E}^{(2)}_i,\widetilde{F}^{(2)}_i$, in order to prove \cref{thm:subsequences_crystals}, it suffices to show that
\begin{equation}
    I_k(\widetilde{E}^{(1)}_i(\overline{\pi})) = I_k(\overline{\pi}),
    \qquad
    D_k(\widetilde{E}^{(1)}_i(\overline{\pi})) = D_k(\overline{\pi}),
\end{equation}
for all $i=1,\dots,n-1$. In the remaining lemmas below these are indeed the only situation we consider. We start by showing that classical Kashiwara operators preserve increasing subsequences. 

\begin{lemma} \label{lemma:Kashiwara_is}
    Let $\overline{\pi} \in \overline{\mathbb{A}}_{n,n}$ and $i\in\{1,\dots,n-1\}$ such that  $\widetilde{E}^{(1)}_i(\overline{\pi})$ exists. Then $I_k(\widetilde{E}^{(1)}_i(\overline{\pi}))=I_k(\overline{\pi})$.
\end{lemma}

\begin{proof}
    Define $\overline{\pi}' = \widetilde{E}^{(1)}_i(\overline{\pi})$ and let $\overline{\sigma} = \overline{\sigma}^{(1)} \cupdot \cdots \cupdot \overline{\sigma}^{(k)}$ be a $k$-increasing subsequence of $\overline{\pi}$. We show that we can always find a $k$-increasing subsequence $\overline{\xi} \subset \overline{\pi}'$, such that $|\overline{\xi}| = |\overline{\sigma}|$ and this clearly implies the claim of the lemma. Weighted biword $\overline{\pi}'$ is obtained from $\overline{\pi}$ by replacing one entry $\left( \begin{smallmatrix} \widehat{j} \\ i+1 \\ \widehat{w} \end{smallmatrix} \right)$ by $\left( \begin{smallmatrix} \widehat{j} \\ i \\ \widehat{w} \end{smallmatrix} \right)$, where $\widehat{j}, \widehat{w}$ are selected through the signature rule \eqref{eq:signature_rule_matrix}. We denote the corresponding cells $\widehat{c}=(\widehat{j} - n \widehat{w},i+1)$ and $\widehat{c}'=\widehat{c} - \mathbf{e}_2$. In case $\widehat{c}\notin \overline{\sigma}$, then $\overline{\sigma}$ is not affected by transformation $\widetilde{E}^{(1)}_i$ and we simply take $\overline{\xi} = \overline{\sigma}$. Alternatively, assume $\widehat{c} \in \overline{\sigma}$ and without loss of generality let $\widehat{c} \in \overline{\sigma}^{(1)}$. We write
    \begin{equation}
        \overline{\sigma}^{(1)} = a_1 \to \cdots \to a_K \to \widehat{c} \to b_1 \to \cdots \to b_J,
    \end{equation}
    for some increasing subsequences $A=a_1 \to \dots \to a_K$ and $B=b_1 \to \cdots \to b_J$. In case $A \to \widehat{c}'$ is still an increasing sequence we define $\overline{\xi}^{(1)} = A \to \widehat{c}' \to B$ and $\overline{\xi} = \overline{\xi}^{(1)} \cupdot \overline{\sigma}^{(2)} \cupdot \cdots \cupdot \overline{\sigma}^{(k)}$ is the desired $k$-increasing subsequence of $\overline{\pi}'$. 
    
    The only non-trivial case to treat therefore occurs when $A \to \widehat{c}'$ is no longer an increasing sequence, as in \cref{fig:Kashiwara_Greene_1}. This happens only when point $a_S,\dots, a_K$ have coordinates $a_s = (j_s, i+1)$ for $s=S,\dots ,K$. With no loss of generality we assume that $a_{S-1}$ does not lie on the strip $\mathbb{Z} \times \{i+1\}$. By signature rule \eqref{eq:signature_rule_matrix} there exists a set of points $U = \{ u_S, \dots, u_K\}$ with coordinates $u_s = (j_s',i)$ such that $j_s < j_s' \le \widehat{j}$. If $U \cap \overline{\sigma} = \varnothing$ we define 
    \begin{equation}
        \overline{\xi}^{(1)}=a_1 \to \cdots \to a_{S-1} \to u_S \to \cdots \to u_K \to \widehat{c}' \to B
    \end{equation}
    and again $\overline{\xi} = \overline{\xi}^{(1)} \cupdot \overline{\sigma}^{(2)} \cupdot \cdots \cupdot \overline{\sigma}^{(k)}$ has the desired properties. This procedure is given in \cref{fig:Kashiwara_Greene_1}, where the solid red path denotes $\overline{\sigma}^{(1)}$, while the dotted one denotes $\overline{\xi}^{(1)}$.
    \begin{figure}
        \centering
        \includegraphics{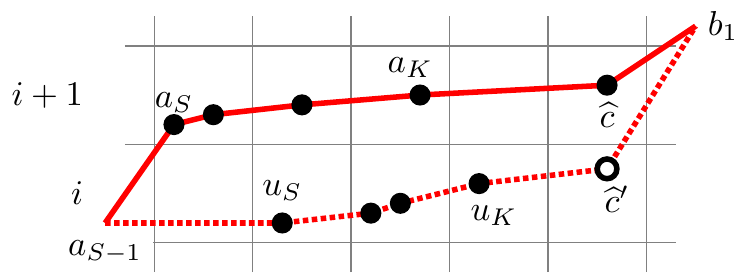}
        \caption{An instance of the relabeling procedure of \cref{lemma:Kashiwara_is}.}
        \label{fig:Kashiwara_Greene_1}
    \end{figure}
    Otherwise we assume that  $U\cap \overline{\sigma} \neq \varnothing$ and let
    \begin{equation}
        J= \max \{s \in \{S,\dots,K\} : u_s \in \overline{\sigma} \}.
    \end{equation}
    With no loss of generality we can write
     \begin{equation}
         \overline{\sigma}^{(2)} = A' \to u_J \to B',
     \end{equation}
     for two increasing subsequences $A',B'$. 
     This situation is reported in
     panel (a) and (c) of  \cref{fig:Kashiwara_Greeene_2345}
     depending on if $\widehat{c}' \to B'$ is increasing or not. 
     There red and blue lines denote $\overline{\sigma}^{(1)}$ and $\overline{\sigma}^{(2)}$. When $\widehat{c}' \to B'$ is increasing we set, as in \cref{fig:Kashiwara_Greeene_2345} panel (b)
     \begin{equation}
         \overline{\xi}^{(1)} = A \to B,
         \qquad
         \overline{\xi}^{(2)} = A' \to u_J \to \widehat{c}' \to B'
     \end{equation}
     and $\overline{\xi} = \overline{\xi}^{(1)} \cupdot \overline{\xi}^{(2)} \cupdot \overline{\sigma}^{(3)} \cupdot \cdots \overline{\sigma}^{(k)}$ is the $k$-increasing subsequence of $\overline{\pi}$ we are looking for.
     \begin{figure}[ht]
            \centering
            \subfloat[]{{\includegraphics[width=.45\linewidth]{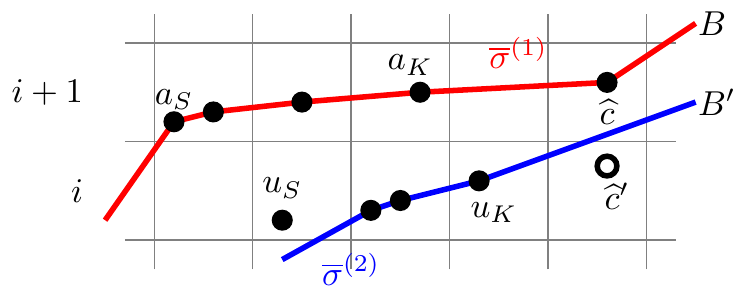} } }%
            \hspace{1cm}
            \subfloat[]{{\includegraphics[width=.45\linewidth]{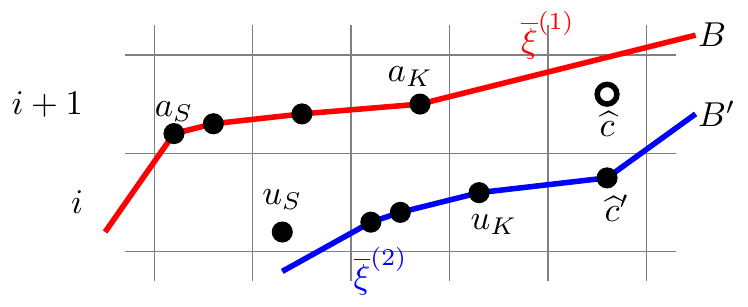} }}%
            \hspace{1cm}
            \subfloat[]{{\includegraphics[width=.45\linewidth]{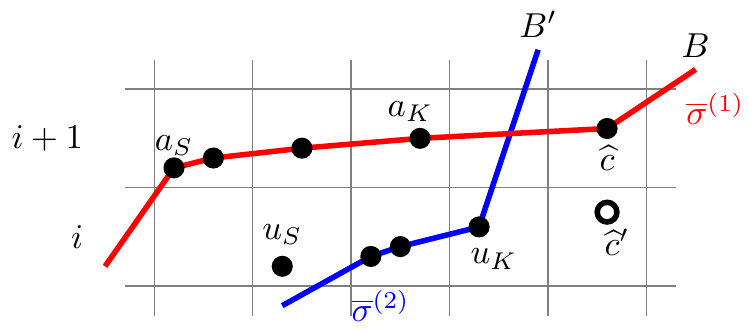} } }%
            \hspace{1cm}
            \subfloat[]{{\includegraphics[width=.45\linewidth]{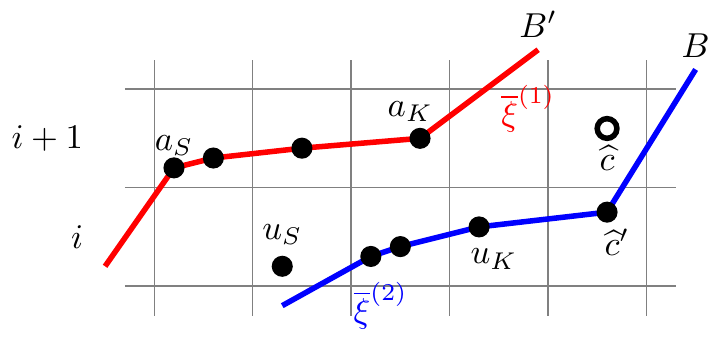} }}%
            \caption{Some of the possible relabeling procedures described in the proof of \cref{lemma:Kashiwara_is}.}
            \label{fig:Kashiwara_Greeene_2345}
        \end{figure}
     Otherwise, if $\widehat{c}' \to B'$ is not increasing, we define
     \begin{equation}
         \overline{\xi}^{(1)} = a_1 \to \cdots \to a_J \to B',
         \qquad
         \overline{\xi}^{(2)} = A' \to u_J \to u_{J+1} \to \cdots \to u_S \to \widehat{c}' \to B,
     \end{equation}
     as in \cref{fig:Kashiwara_Greeene_2345} panels (c) and (d).
     Also in this case $\overline{\xi} = \overline{\xi}^{(1)} \cupdot \overline{\xi}^{(2)} \cupdot \overline{\sigma}^{(3)} \cupdot \cdots \overline{\sigma}^{(k)}$ has the desired properties. This check exhausts all possibilities and concludes the proof.
\end{proof}

The proof of the fact that Kashiwara operators preserve the length of the longest localized decreasing subsequences is slighly more involved than the preservation property for increasing subsequences. We articulate the analysis of this case in the following three lemmas. For the next statement we need to recall how the shadow line construction one draws for the transition $\overline{\pi} \mapsto \mathbf{V}(\overline{\pi})$ defines an ensemble of down-right loops; see \cref{subs:periodic_shadow_line}.

\begin{lemma} \label{lemma:shadow_lines_number_loops}
    Let $\overline{\pi}$ be a $k$-localized decreasing sequence. Then the shadow line construction of $\overline{\pi}$ consists of at most $k$ down right loops. 
\end{lemma}

\begin{figure}
    \centering
    \subfloat[]{{\includegraphics[]{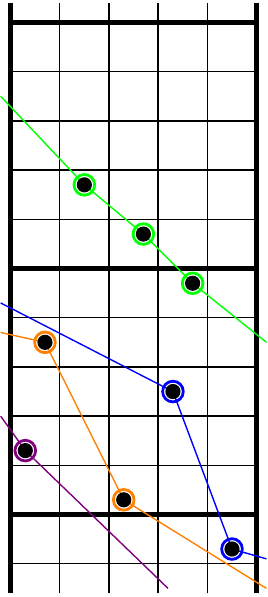} } }%
    \hspace{1cm}
    \subfloat[]{{\includegraphics[]{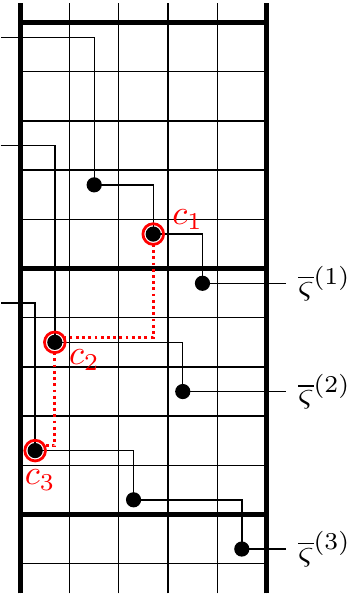} }}%
    \caption{An example of the construction presented in the proof of \cref{lemma:shadow_lines_number_loops}. In the left panel we see the weighted biword $\overline{\pi}$ represented as an union of four LDS's. In the right panel we see that the shadow line construction produces three down right loops.}
    \label{fig:Viennot_loops_and_increasing_subs}
\end{figure}

\begin{proof}
    We show that if the shadow line construction of $\overline{\pi}$ consists of $\widetilde{k}$ down right loops $\varsigma^{(1)},\dots,\varsigma^{(\widetilde{k})}$, then there exists an \emph{increasing} subsequence $\overline{\xi}\subset \overline{\pi}$ of length $\widetilde{k}$. With this assumption let $\overline{\xi} = c_1\to \cdots \to c_{\widetilde{k}}$ be contained in an up-right path in $\mathscr{C}_n$. Writing $\overline{\pi}=\overline{\sigma}^{(1)} \cupdot \cdots \cupdot \overline{\sigma}^{(k)}$, where $\overline{\sigma}^{(j)}$'s are localized decreasing subsequences, no two $c_i$ can belong to the same $\overline{\sigma}^{(j)}$ forcing $k \ge \widetilde{k}$.
    
    Representing $\mathscr{C}_n$ as an infinite vertical strip, we assume that loops are ordered from the topmost $\varsigma^{(1)}$ to the bottom-most $\varsigma^{(\widetilde{k})}$; see \cref{fig:Viennot_loops_and_increasing_subs}. We are left to show that an increasing subsequence of length $\widetilde{k}$ always exists. For this we show that if $c_i \in \overline{\pi} \cap \varsigma^{(i)}$ is a point of the configuration lying on the $i$-th loop of the shadow line construction, then we can always find $c_{i+1} \in \overline{\pi} \cap \varsigma^{(i+1)}$ such that $c_{i+1} \to c_{i}$ is an increasing sequence. To find such $c_{i+1}$ start walking southward from $c_i=(a,b-nw)$ until the first intersection with an horizontal segment of the loop $\varsigma^{(i+1)}$, which happens at a location $(a,b'-nw')$ for some $b'-nw' \le b-nw$. From there travel eastward along $\varsigma^{(i+1)}$ until the first occurrence of a point in $c_{i+1} \in \overline{\pi}$, which might happen after winding around the cylinder. In \cref{fig:Viennot_loops_and_increasing_subs} (b) we represented such walks with red dotted segments. By construction $c_{i+1}\to c_i$ forms an increasing sequence and by induction we conclude the proof.
\end{proof}

\begin{lemma} \label{lemma:kashiwara_2lds}
    Consider a 2-localized decreasing sequence $\overline{\pi}=\overline{\sigma}^{(1)} \cupdot \overline{\sigma}^{(2)}$ and assume $\overline{\pi}'=\widetilde{E}^{(1)}_i(\overline{\pi}) \neq \varnothing$ for some $i\in\{1,\dots,n-1\}$. Then we also have $\overline{\pi}'=\overline{\xi}^{(1)} \cupdot \overline{\xi}^{(2)}$, for two localized decreasing subsequences $\overline{\xi}^{(1)}, \overline{\xi}^{(2)}$.
\end{lemma}

\begin{proof}
    Weighted biword $\overline{\pi}'$ is obtained from $\overline{\pi}$ replacing an entry $\left( \begin{smallmatrix} \widehat{j} \\ i+1 \\ \widehat{w} \end{smallmatrix} \right)$ with $\left( \begin{smallmatrix} \widehat{j} \\ i \\ \widehat{w} \end{smallmatrix} \right)$, where $\widehat{j},\widehat{w}$ are selected through the signature rule \eqref{eq:signature_rule_matrix}. Call $\widehat{c}=(\widehat{j}, i+1 - \widehat{w} n)$ and $\widehat{c}' = \widehat{c} - \mathbf{e}_2$. With no loss of generality assume $\widehat{c} \in \overline{\sigma}^{(1)}$ and write
    \begin{equation}
        \overline{\sigma}^{(1)} = a_1 \to \cdots \to a_M
        \qquad
        \overline{\sigma}^{(2)}=
        b_1 \to \cdots \to b_N,
    \end{equation}
    with $\widehat{c}=a_K$ for some $K$. Representing $\mathscr{C}_n$ as an infinite vertical strip, we assume cells $a_k,b_k$ to be ordered from top to bottom. Define \begin{equation}
        \overline{\theta} = a_1 \to \dots \to a_{K-1} \to \widehat{c}' \to a_{K+1} \to \dots \to a_M.
    \end{equation}
    If $\overline{\theta}$ is an LDS, we set $\overline{\xi}^{(1)}=\overline{\theta}$ and $\overline{\xi}^{(2)}=\overline{\sigma}^{(2)}$ and this yields the desired decomposition of $\overline{\pi}'$. To check the remaining cases we consider two possibilities.
    \medskip
    
    \textbf{Case 1: $\overline{\theta}$ is not localized.} This only happens if $\widehat{c}=a_M$ and $a_1=(j_1,i-n(\widehat{w}-1))$, for some $1\le j_1 < \widehat{j}$. From the signature rule \eqref{eq:signature_rule_matrix_2}  this implies that there exists a point $\widehat{d} \in \overline{\sigma}^{(2)}$ such that $\widehat{d}=(j_2, i+1 - n (\widehat{w}-1))$, with $1 \le j_2 < j_1$, or $\widehat{d} = (\tilde{j}_2,i+1- n \widehat{w})$ with $\widehat{j}<\tilde{j}_2 \le n$.
    We treat these two additional subcases separately. 
    \begin{enumerate}[align=left]
        \item[\textbf{Case 1.1 : $\widehat{d}=(j_2, i+1 - n (\widehat{w}-1))$.}] We can write subsequence $\overline{\sigma}^{(2)} = U \to V$, where $U=u_1 \to \cdots \to u_J$ is the LDS of all $u_j \in \overline{\sigma}^{(2)}$ such that $u_j \to a_1$ is an LDS; see \cref{fig:Kashiwara_Greene_IS_1} (a) for an example. Notice that $u_J=\widehat{d}$. Define sector $S=\{ 1,\dots,n \} \times \{ i+1 - n \widehat{w} , \dots , i-1 - n (\widehat{w}-1) \}$ as in \cref{fig:Kashiwara_Greene_IS_1} (b). By construction $S$ contains all points of $\overline{\pi}$ except for $U$ and $a_1$. Moreover all points in $S$ are contained in two LDS : $V$ and $a_2 \to \cdots \to a_M$. We draw the shadow line construction, restricted to the sector $S$, of all points contained in $S$, as in \cref{fig:Kashiwara_Greene_IS_1} (b). By an adaptation of \cref{lemma:shadow_lines_number_loops}, such shadow line construction consists of at most two down right broken lines we call $\overline{\varsigma}^{(1)}, \overline{\varsigma}^{(2)}$. With no loss of generality we assume that $\widehat{c}\in \overline{\varsigma}^{(1)}$ and this forces $\overline{\varsigma}^{(2)}$ to contain only points of $\overline{\pi}$ that are weakly to east of $a_2$ and weakly north of $b_N$. Define now $W^{(2)}$ selecting all points of $\overline{\pi} \cap \overline{\varsigma}^{(2)}$, without multiplicity and $\overline{\xi}^{(2)} = U \to a_1 \to \eta^{(2)}$. Then $\overline{\xi}^{(2)}$ is an LDS. Subsequently define $W^{(1)}$ taking all points of $\overline{\pi} \cap \overline{\varsigma}^{(1)}$ minus $\widehat{c}$ and set $\overline{\xi}^{(1)} = W^{(1)} \to \widehat{c}'$. Again $\overline{\xi}^{(1)}$ is an LDS and $\overline{\xi}^{(1)}, \overline{\xi}^{(2)}$ provide the desired decomposition of $\overline{\pi}'$
        \begin{figure}[ht]
            \centering
            \subfloat[]{{\includegraphics[height=4.7cm]{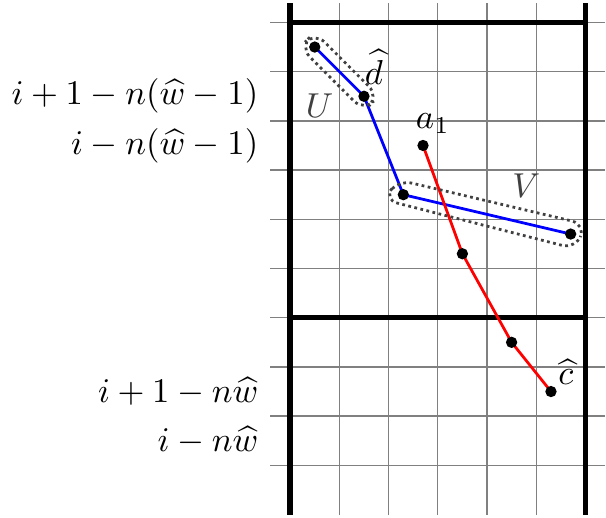} } }%
            \hspace{1cm}
            \subfloat[]{{\includegraphics[height=4.7cm]{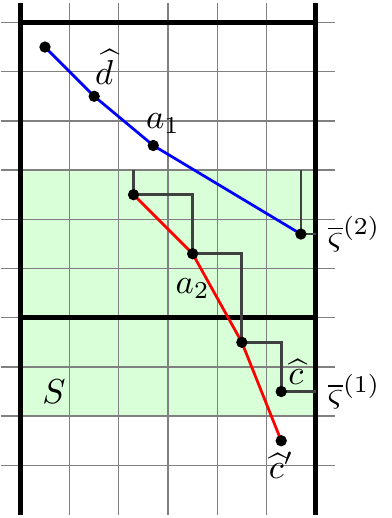} }}%
            \caption{Relabeling procedure corresponding to Case 1.1 in the proof of \cref{lemma:kashiwara_2lds}. Red and blue LDS's in the left panel are $\overline{\sigma}^{(1)},\overline{\sigma}^{(2)}$. In the right panel red and blue LDS's are $\overline{\xi}^{(1)},\overline{\xi}^{(2)}$}
            \label{fig:Kashiwara_Greene_IS_1}
        \end{figure}
        
        \item[\textbf{Case 1.2 : $\widehat{d}=(\tilde{j}_2, i+1 - n \widehat{w})$.}] This case is represented in \cref{fig:Kashiwara_Greene_IS_2} (a).
        We write $\overline{\sigma}^{(2)} = U \to V$, with $U$ having all points located weakly north of $\widehat{c}$ and $V$ having all points located south-east of $\widehat{c}'$. In this case consider the sector $S = \{1,\dots, n\}\times \{ i+1-n\widehat{w}, \dots , i - n (\widehat{w}-1) \}$ as in \cref{fig:Kashiwara_Greene_IS_2} (b). Then all points of $\overline{\pi}$ except those in $V$ are contained in $S$. By signature rule $a_1$ is the northernmost point of the configuration as no other point can be of the form $(j,i-n(\widehat{w}-1))$. Moreover points in $S$ all belong to the union of down-right paths $U$ and $\overline{\sigma}^{(1)}$. Therefore, the inverse shadow line construction, restricted to the sector $S$, of the points within $S$, drawn in \cref{fig:Kashiwara_Greene_IS_2} (b), will consist, again by an adaptation of \cref{lemma:shadow_lines_number_loops}, of exactly two separate broken lines $\overline{\varsigma}^{(1)},\overline{\varsigma}^{(2)}$. With no loss of generality we assume that $\widehat{c} \in \overline{\varsigma}^{(1)}$, while $\widehat{d},a_1 \in \overline{\varsigma}^{(2)}$, which implies that $\overline{\varsigma}^{(1)}$ is contained in the region weakly south of $b_1$. We now define LDS $\overline{\xi}^{(2)}$ taking points $\overline{\pi} \cap \overline{\varsigma}^{(2)}$ without multiplicity. Define also $W$ taking points $\overline{\pi} \cap \overline{\varsigma}^{(1)}$, without mulitplicity and excluding $\widehat{c}$. Finally we set $\overline{\xi}^{(1)}=W \to \widehat{c}' \to V$. Also in this case the decomposition $\overline{\pi}'=\overline{\xi}^{(1)}\cupdot \overline{\xi}^{(2)}$ has the desired properties.
    \begin{figure}
            \centering
            \subfloat[]{{\includegraphics[height=4.7cm]{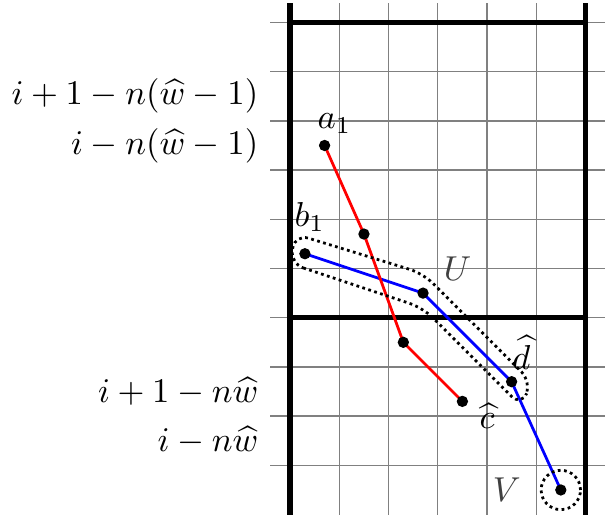} } }%
            \hspace{1cm}
            \subfloat[]{{\includegraphics[height=4.7cm]{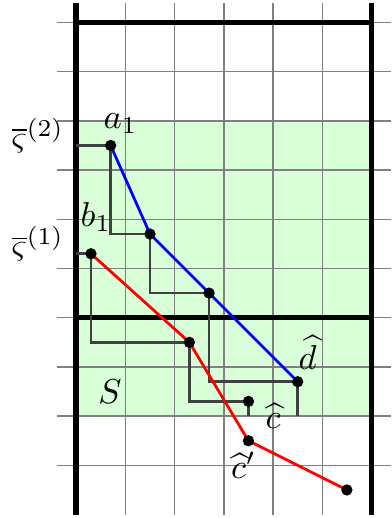} }}%
            \caption{Relabeling procedure corresponding to Case 1.2 in proof of \cref{lemma:kashiwara_2lds}. In the left panel red and blue LDS's are $\overline{\sigma}^{(1)}, \overline{\sigma}^{(2)}$, while in the right panel they are $\overline{\xi}^{(1)}, \overline{\xi}^{(2)}$.}
            \label{fig:Kashiwara_Greene_IS_2}
        \end{figure}
    \end{enumerate}
    
        \textbf{Case 2: $\overline{\theta}$ is not strictly down-right.} This can only happen if $\widehat{c}=a_K$ and $a_{K+1}=(j_1, i-n\widehat{w})$ for some $j_1\in \{ \widehat{j}+1,\cdots ,n \}$ and some $K$. See \cref{fig:Kashiwara_Greene_IS_5} (a) for an example. By the signature rule there must exist $\widehat{d}\in \overline{\sigma}^{(2)}$ such that $\widehat{d}=(j_2,i+1-n\widehat{w})$ with $j_2 \in \{\widehat{j},\cdots j_1-1\}$ and moreover no other point occupies the segment of vertical coordinate $i - n \widehat{w}$. We draw the inverse shadow line construction of the point configuration, which consists in two down-right paths $\overline{\varsigma}^{(1)}, \overline{\varsigma}^{(2)}$, by \cref{lemma:shadow_lines_number_loops}. Because of their relative position we must have, $a_K \in \overline{\varsigma}^{(1)}$, while $a_{K+1}, \widehat{d} \in \overline{\varsigma}^{(2)}$. Then we define $\overline{\xi}^{(1)}$ taking points in $\overline{\pi} \cap \overline{\varsigma}^{(1)}$, without multiplicity, and replacing $\widehat{c}$ by $\widehat{c}'$; see \cref{fig:Kashiwara_Greene_IS_5} (b). Define also $\overline{\xi}^{(2)}$ taking point in $\overline{\pi} \cap \overline{\varsigma}^{(2)}$. Both $\overline{\xi}^{(1)}, \overline{\xi}^{(2)}$ are LDS's and their union is $\overline{\pi}'$. The analysis of this case exhausts all possible configurations of $\overline{\sigma}^{(1)}, \overline{\sigma}^{(2)}$ and completes the proof.
    \begin{figure}[h]
    \centering
    \subfloat[]{{\includegraphics[height=4.7cm]{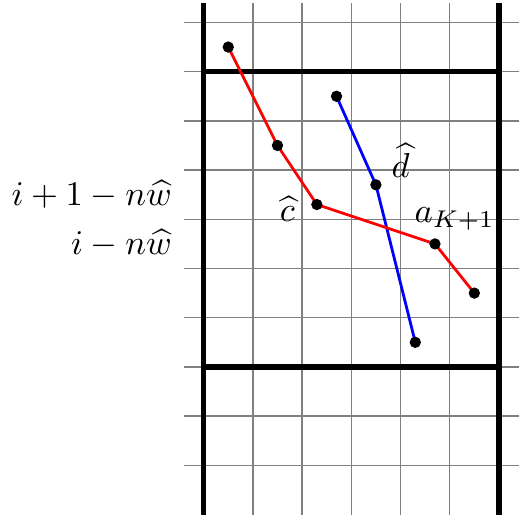} } }%
    \hspace{1cm}
    \subfloat[]{{\includegraphics[height=4.7cm]{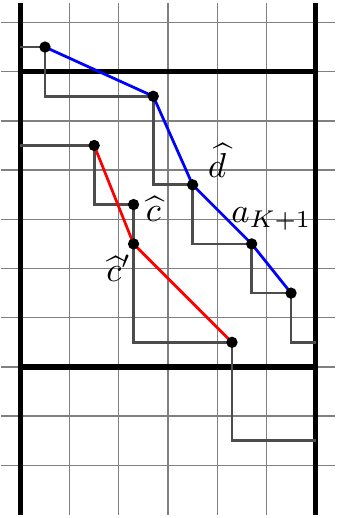}}}
    \caption{Depiction of the relabeling described in Case 2 in the proof of \cref{lemma:kashiwara_2lds}.}
    \label{fig:Kashiwara_Greene_IS_5}
\end{figure}

\end{proof}

\begin{lemma} \label{lemma:Kashiwara_lds}
    Let $\overline{\pi} \in \overline{\mathbb{A}}_{n,n}$ and 
    $\overline{\pi}'=\widetilde{E}^{(1)}_i (\overline{\pi}) \neq \varnothing$ for some $i\in\{1,\dots,n-1\}$. Then, for all $k$, we have $D_k(\overline{\pi}')=D_k(\overline{\pi})$.
\end{lemma}

\begin{proof}
    As discussed several times above $\overline{\pi}'$ differs from  $\overline{\pi}$ by a replacement of $\left( \begin{smallmatrix} \widehat{j} \\ i+1 \\ \widehat{w} \end{smallmatrix} \right)$ with $\left( \begin{smallmatrix} \widehat{j} \\ i \\ \widehat{w} \end{smallmatrix} \right)$. For this proof it is convenient to parameterize $\mathscr{C}_n$ as the infinite horizontal strip \eqref{eq:twisted_cyl_horizontal} and we define points $\widehat{c}=(\widehat{j}+ \widehat{w} n, i+1)$ and $\widehat{c}' = \widehat{c}-\mathbf{e}_2$. Let $\overline{\sigma}=\overline{\sigma}^{(1)} \cupdot \cdots \cupdot \overline{\sigma}^{(k)}$ be a $k$-LDS of $\overline{\pi}$. We show that we can find a $k$-LDS $\overline{\xi}= \overline{\xi}^{(1)} \cupdot \cdots \cupdot \overline{\xi}^{(k)} \subseteq \overline{\pi}'$ with the same length of $\overline{\sigma}$ and this statement clearly implies the result of the lemma.
    
    Clearly, if $\widehat{c}\notin \overline{\sigma}$ we take $\overline{\xi} = \overline{\sigma}$ and there is nothing to prove. We then assume, with no loss of generality that $\widehat{c} \in \overline{\sigma}^{(1)}$ and in particular $\widehat{c}=[\overline{\sigma}^{(1)}_K]$ for some $K$. As in proof of \cref{lemma:kashiwara_2lds}, define 
    \begin{equation}
        \overline{\theta} = [\overline{\sigma}^{(1)}_1] \to \cdots \to [\overline{\sigma}^{(1)}_{K-1}] \to \widehat{c}' \to [\overline{\sigma}^{(1)}_{K+1}] \to \cdots .   
    \end{equation}
    If $\overline{\theta}$ is an LDS, then we set $\overline{\xi}=\overline{\theta} \cupdot \overline{\sigma}^{(2)} \cupdot \cdots \cupdot \overline{\sigma}^{(K)}$, producing the desired $k$-LDS $\overline{\xi} \subseteq \overline{\pi}'$. If, on the other hand $\overline{\theta}$ is not an LDS, then there exists $\widetilde{c} \in \overline{\sigma}^{(1)}$ such that $\widetilde{c}=(\widetilde{j},i)$ for some $\widehat{j} < \widetilde{j}$. Define $\Omega^{(i)}(\overline{\pi}; \overline{\sigma})$ as the set of $a = (m,i+1) \in \overline{\pi}$ such that
    \begin{equation} \label{eq:set_Omega_pi_sigma}
        \begin{minipage}{.9\linewidth}
            \begin{enumerate}
                \item $m \ge \widehat{j}$;
                \item $a \notin \overline{\sigma}$ or $a \in \overline{\sigma}^{(\ell)}$ for some $\ell$ but $(m',i)\notin \overline{\sigma}^{(\ell)}$ for any $m' \in \mathbb{Z}$.
            \end{enumerate}
        \end{minipage}
    \end{equation}
    Then, by the signature rule $\Omega^{(i)}(\overline{\pi}; \overline{\sigma})$ is not empty and we consider its element $\widehat{d}=(\widehat{m},i+1)$ situated furthest to the west. By this we mean that $(m,i+1)\notin \Omega^{(i)}(\overline{\pi}; \overline{\sigma})$ for any $m < \widehat{m}$. Based on the position of $\widehat{d}$ we distinguish two cases.
    
    \medskip
    
    \textbf{Case 1: $\widehat{d}$ lies between $\widehat{c}$ and $\widetilde{c}$.} More precisely we assume $\widehat{j}\le \widehat{m} < \widetilde{j}$. Then:
    \begin{itemize}
        \item if $\widehat{d}\notin \overline{\sigma}$ we can define $\overline{ \xi}^{(1)}$ replacing, in $\overline{\sigma}^{(1)}$, $\widehat{c}$ with $\widehat{d}$. This produces the desires $k$-LDS $\overline{\xi}'=\overline{\xi}^{(1)}\cupdot \overline{\sigma}^{(2)} \cupdot \cdots \cupdot \overline{\sigma}^{(k)}\subseteq \overline{\pi}'$.
        \item if $\widehat{d} \in \overline{\sigma}$ and with no loss of generality, we assume $\widehat{d} \in \overline{\sigma}^{(2)}$, then weighted biword $\overline{\sigma}^{(1)} \cupdot \overline{\sigma}^{(2)}$ satisfies hypothesis of \cref{lemma:kashiwara_2lds}. By the same lemma, we can find a decomposition $\overline{\xi}^{(1)} \cupdot \overline{\xi}^{(2)} = \widetilde{E}^{(1)}_i (\overline{\sigma}^{(1)} \cupdot \overline{\sigma}^{(2)})$ and  $\overline{\xi}=\overline{\xi}^{(1)} \cupdot \overline{\xi}^{(2)} \cupdot \overline{\sigma}^{(3)} \cupdot \cdots \cupdot \overline{\sigma}^{(k)}$ yields the desired $k$-LDS of $\overline{\pi}'$. This situation is reported in \cref{fig:LDS_horiz_strip_case_1}.
        \begin{figure}
            \centering
            \includegraphics{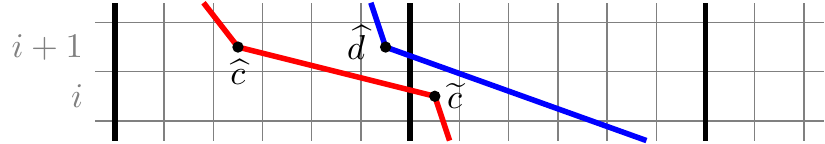}
            \caption{An example of configuration where $\widehat{d} \in \overline{\sigma}^{(2)}$ lies between $\widehat{c}$ and $\widetilde{c}$.}
            \label{fig:LDS_horiz_strip_case_1}
        \end{figure}
    \end{itemize} 
    
    \medskip
    
    \textbf{Case 2: $\widehat{d}$ does not lie between $\widehat{c}$ and $\widetilde{c}$.} When $\widehat{m}\ge \widetilde{j}$ we want to show that through a series of reshuffling of elements of $\overline{\sigma}$ and $\overline{\pi}$ we can always ``transport" the point $\widehat{d}$ in the region between $\widehat{c}$ and $\widetilde{c}$, falling back into \textbf{Case 1}. By the signature rule and by definition of $\widehat{d}$ it is easy to conclude that there exist points $a=(j_a,i+1)$, $b=(j_b,i)$ such that
    \begin{equation}
        \begin{minipage}{.9\linewidth}
            \begin{enumerate}
                \item $a,b \in \overline{\sigma}^{(s)}$ for some $s$.
                \item $a$ lies between $\widehat{c}$ and $\widehat{d}$, while $b$ lies to east of $\widehat{d}$. More precisely $\widehat{j} \le j_a < \widehat{m} < j_b$.
            \end{enumerate}
        \end{minipage}
    \end{equation}
    \begin{figure}[h]
        \centering
        \includegraphics[scale=1.1]{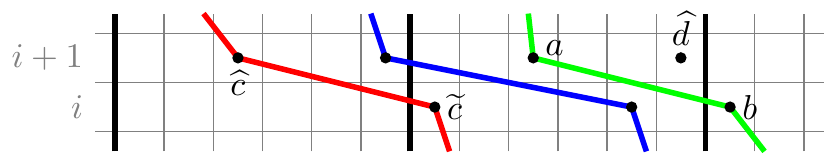}
        \caption{An example of a configuration where $\widehat{d}$ does not lie between $\widehat{c}$ and $\widetilde{c}$.}
        \label{fig:LDS_horiz_strip_case_2}
    \end{figure}
    For an example see \cref{fig:LDS_horiz_strip_case_2}.
    In case $\widehat{d} \notin \overline{\sigma}$, we replace in $\overline{\sigma}^{(s)}$, $a$ with $\widehat{d}$ yielding a new LDS $\widetilde{\sigma}^{(s)}$ and hence a new $k$-LDS $\widetilde{\sigma}$ obtained from $\overline{\sigma}$ interchanging $\widetilde{\sigma}^{(s)}$ and $\overline{\sigma}^{(s)}$. After such replacement the set $\Omega^{(i)}(\overline{\pi}, \widetilde{\sigma})$ differs from $\Omega^{(i)}(\overline{\pi}, \overline{\sigma})$ and in particular its westernmost element is $a$, rather than $\widehat{d}$.
    Alternatively assume $\widehat{d}\in \overline{\sigma}$, say $\widehat{d} \in \overline{\sigma}^{(s+1)}$ for some $s$. Then as in \textbf{Case 1} weighted biword $\overline{\sigma}^{(s)}\cupdot \overline{\sigma}^{(s+1)}$ fulfills hypothesis of \cref{lemma:kashiwara_2lds}. This implies that we can find a decomposition $\overline{\eta}^{(s)} \cupdot \overline{\eta}^{(s+1)} = \overline{\sigma}^{(s)}\cupdot \overline{\sigma}^{(s+1)}$ such that, if $a \in \overline{\eta}^{(s)}$, then $(m,i)\notin \overline{\eta}^{(s)}$ for any $m \in \mathbb{Z}$. Define now $\widetilde{\sigma}$ replacing in $\overline{\sigma}$ LDS's $\overline{\sigma}^{(s)}, \overline{\sigma}^{(s+1)}$ with $\overline{\eta}^{(s)}, \overline{\eta}^{(s+1)}$. Then also in this case the westernmost element of set $\Omega^{(i)}(\overline{\pi},\widetilde{\sigma})$ is no longer $\widehat{d}$, but $a$. This shows that, inductively we can move west, through reshuffling, the westernmost element of $\Omega^{(i)}(\overline{\pi},\overline{\sigma})$ until it lies between $\widehat{c}$ and $\widetilde{c}$, which is then treated by \textbf{Case 1}. This concludes the proof. 
\end{proof}

We finally come to the proof of \cref{thm:subsequences_crystals}.

\begin{proof}[Proof of \cref{thm:subsequences_crystals}]
    It follows by combining results of lemmas enumerated in this appendix. Thanks to \cref{lemma:inversion_prererves_is_lds} the family of Kashiwara operators $\widetilde{E}^{(2)}_i,\widetilde{F}^{(2)}_i$ preserves quantities $I_k,D_k$ only if the first family $\widetilde{E}^{(1)}_i,\widetilde{F}^{(1)}_i$ does. Further, through  \cref{lemma:iota_preserves_I_D} one simply has to check that $\widetilde{E}^{(1)}_i,\widetilde{F}^{(1)}_i$ preserve $I_k,D_k$ for $i=1,\dots,n-1$. This last case is then handled by \cref{lemma:Kashiwara_is} and \cref{lemma:Kashiwara_lds}. Notice that if $\widetilde{E}^{(1)}_i$ preserves $I_k,D_k$, then also $\widetilde{F}^{(1)}_i$ does, being its inverse.
\end{proof}



\bibliographystyle{alpha}
\bibliography{bib}

\end{document}